\documentclass[11pt]{amsart}

\usepackage{epigamath}

\usepackage[english]{babel}

\numberwithin{equation}{section}
\numberwithin{figure}{section}

\usepackage{amstext}
\usepackage{amsthm}
\usepackage{amssymb}

\usepackage{tikz}
\usepackage{pgfplots}

\usepackage{mathtools}
\usepackage{array}
\usepackage[shortlabels,inline]{enumitem}
\usepackage{longtable}
\usepackage{xltabular}

\theoremstyle{plain}
\newtheorem{thm}{\protect\theoremname}[section]
\newtheorem{lem}[thm]{\protect\lemmaname}
\newtheorem{prop}[thm]{\protect\propositionname}
\newtheorem{cor}[thm]{\protect\corollaryname}

\theoremstyle{remark}
\newtheorem{rem}[thm]{\protect\remarkname}
\newtheorem{example}[thm]{\protect\examplename}

\theoremstyle{definition}
\newtheorem{defn}[thm]{\protect\definitionname}

\providecommand{\corollaryname}{Corollary}
\providecommand{\definitionname}{Definition}
\providecommand{\examplename}{Example}
\providecommand{\lemmaname}{Lemma}
\providecommand{\propositionname}{Proposition}
\providecommand{\remarkname}{Remark}
\providecommand{\theoremname}{Theorem}

\providecommand{\tabularnewline}{\\}

\newcommand{\CC}{\mathbb{C}}

\newcommand{\FF}{\mathbb{F}}

\newcommand{\NN}{\mathbb{N}}
\newcommand{\PP}{\mathbb{P}}

\newcommand{\RR}{\mathbb{R}}
\newcommand{\ZZ}{\mathbb{Z}}

\newcommand{\OO}{\mathcal{O}}

\renewcommand{\a}{\alpha}
\renewcommand{\b}{\beta}
\renewcommand{\d}{\delta}
\newcommand{\e}{\epsilon}
\newcommand{\g}{\gamma}
\newcommand{\G}{\Gamma}
\renewcommand{\l}{\lambda}
\renewcommand{\L}{\Lambda}
\newcommand{\s}{\sigma}
  
\newcommand{\aut}{\operatorname{Aut}}

\newcommand{\Ns}{\operatorname{NS}}

\newcommand{\NS}{\operatorname{NS}(}
\newcommand{\Pic}{\operatorname{Pic}(}
\newcommand{\rk}{\operatorname{rk}(}

\newcommand{\cu}{\operatorname{(-2)}}

\newcommand{\zde}{\ensuremath{{\scriptscriptstyle (\mathbb{Z}/2\mathbb{Z})^{2}}}}%
\newcommand{\zd}{\ensuremath{{\scriptscriptstyle \mathbb{Z}/2\mathbb{Z}}}}%
\newcommand{\tss}{\ensuremath{{\scriptscriptstyle \mathfrak{S}_{3}\times\mathbb{Z}/2\mathbb{Z}}}}%
 
\newcommand{\triv}{\ensuremath{{\scriptscriptstyle \{1\}}}}%

\EpigaVolumeYear{6}{2022} \EpigaArticleNr{19}
\ReceivedOn{April 17, 2020}
\InFinalFormOn{April 26, 2022}
\AcceptedOn{June 21, 2022}

\title{An atlas of K3 surfaces with finite automorphism group}
\titlemark{An atlas of K3 surfaces with finite automorphism group}

\author{Xavier Roulleau}
\address{Aix-Marseille Universit\'e, CNRS, Centrale Marseille,
I2M UMR 7373,  
13453 Marseille,
France} 
\email{Xavier.Roulleau@univ-amu.fr}

\authormark{X.~Roulleau}

\AbstractInEnglish
{We study the geometry of K3 surfaces with finite automorphism group
and Picard number at least $3$. We describe these surfaces classified
by Nikulin and Vinberg as double covers of simpler surfaces or as embedded
in a projective space. We study, moreover, the configurations of their
finite set of $(-2)$-curves.}

\MSCclass{14J28}

\KeyWords{K3 surfaces, finite automorphism groups, finite number of $(-2)$-curves, Vinberg algorithms}

\acknowledgement{The author is thankful to the Max Planck Institute
for Mathematics in Bonn for financial support.}

\begin{document}


\maketitle

\begin{prelims}

\DisplayAbstractInEnglish

\bigskip

\DisplayKeyWords

\medskip

\DisplayMSCclass

\end{prelims}


\newpage

\setcounter{tocdepth}{1}

\tableofcontents


\section{Introduction}

Algebraic K3 surfaces $X$ over $\CC$ with finite automorphism group
were classified by their Picard lattices $\NS X)$ and Picard
number $\rho=\rk\NS X))\geq3$ by
Nikulin for $\rho=5,6,\ldots$
in 1981, \textit{cf.}\ \cite{Nikulin}, and for $\rho=3$ in 1985, \textit{cf.}\ \cite{Nikulin2}, and
by Vinberg for $\rho=4$ (in 1981, published in 2007 in \cite{Vinberg}).
The result of their classification is a list of $118$ N\'eron--Severi
lattices. 

In the present paper, we exhibit for each of these surfaces a geometric
construction, as a double plane, a double cover of the Hirzebruch
surface $\mathbf{F}_{n}$, $n\in\{2,3,4\}$ (which part has also been
more or less explicitly done in \cite{AN} and \cite{Zhang}) or a
complete intersection surface of degree $k\in\{4,6,8\}$ in $\PP^{\frac{1}{2}k+1}$.
Let us recall that a K3 surface with Picard number $\rho>2$ has 
finite automorphism group if and only if it contains only a 
non-zero finite number of $\cu$-curves. The $\cu$-curves play a key role
for K3 surfaces; for example, one can describe their ample cone using
these curves. We are thus especially interested in the configuration
of these curves on the K3 surfaces. In \cite{Nikulin} and especially
in \cite{Nikulin2} for $\rho=3$, Nikulin described the number and
some configurations of these curves. Our contribution to the subject
comes for higher-rank Picard lattices, when there are more $\cu$-curves
than expected. Among the K3 surfaces we describe, there are two series
which we find most remarkable. 

The first series are the surfaces with N\'eron--Severi lattice of type
$U(2)\oplus\mathbf{A}_{1}^{\oplus n}$, for $n\in\{2,\dots,7\}$.  Such
a surface $X$ is the desingularization of the double cover of the
plane branched over a sextic curve with $n+1$ nodes in general
position. The surface $X$ is also naturally the double cover $X\to Z$
of a degree $8-n$ del Pezzo surface $Z$ branched over a curve in the
linear system $|-2K_{Z}|$. The pull-back on $X$ of the $(-1)$-curves
on $Z$ are $(-2)$-curves, and one has a description of the
configuration of the $\cu$-curves on $X$ from the well-known
configuration of $(-1)$-curves on the del Pezzo surface. In
particular, for $n=2,\dots,7$, the number of $\cu$-curves on the K3
surface is $6,10,16,27,56,240$, respectively.  Let us describe the
case with $240$ $\cu$-curves. 

\begin{thm}\label{thm:1}
Let $X$ be a general K3 surface with N\'eron--Severi lattice isometric
to $U(2)\oplus\mathbf{A}_{1}^{\oplus7}$. There exists a double cover
$f_{1}\colon X\to\PP^{2}$ branched over a sextic curve with $8$ nodes
$p_{1},\dots,p_{8}$ such that the images of the $240$ $\cu$-curves
on $X$ are
\begin{itemize}
\item the $8$ points $p_{1},\dots,p_{8}$, 
\item the $28$ lines through $p_{i},p_{j}$ with $i\neq j$,
\item the $56$ conics that go through $5$ points in $\{p_{1},\dots,p_{8}\}$, 
\item the $56$ cubics that go through $7$ points $p_{j}$ with a double
point at $1$ of these points,
\item the $56$ quartics through the $8$ points $p_{j}$ with double points
  at $3$ of them,
\item the $28$ quintics through the $8$ points $p_{j}$ with double points
at $2$ of them,
\item the $8$ sextics through the $8$ points with double points at all
except a single point with multiplicity $3$.
\end{itemize}

Let $Z\to\PP^{2}$ be the blow-up at the points
$p_{k},\,k\in\{1,\dots,8\}$; the surface $X$ is also a double cover of
$Z$. The pull-back on $X$ of the pencil $|-K_{Z}|$ gives another
double cover $f_{2}\colon X\to\PP^{2}$.  Its branch locus is a smooth
sextic curve to which $120$ conics are tangent at every intersection
point. These $120$ conics are the images of the $240$ $\cu$-curves on
$X$.

The two involutions corresponding to the covers $f_{1},f_{2}$ generate
the automorphism group $(\ZZ/2\ZZ)^{2}$ of $X$. 
\end{thm}

The second series are the surfaces with N\'eron--Severi lattice of type
$U\oplus\mathbf{E}_{8}\oplus\mathbf{A}_{1}^{\oplus4}$, $U\oplus\mathbf{D}_{8}\oplus\mathbf{A}_{1}^{\oplus3}$,
$U\oplus\mathbf{D}_{4}^{\oplus2}\oplus\mathbf{A}_{1}^{\oplus2}$,
$U\oplus\mathbf{D}_{4}\oplus\mathbf{A}_{1}^{\oplus5}$ or $U\oplus\mathbf{A}_{1}^{\oplus8}$.
These surfaces contain, respectively, $27,39,59,90,145$ $\cu$-curves.
Let us give an example of the results obtained for the case of $U\oplus\mathbf{A}_{1}^{\oplus8}$.
Let $\mathbf{F}_{4}$ be the Hirzebruch surface with a section $s$
such that $s^{2}=-4$. Let $f$ be a fiber of the unique fibration,
and let $L=4f+s$, $L'=5f+s$. 

\begin{thm}\label{thm:2}
Let $X$ be a general K3 surface with N\'eron--Severi lattice isometric
to $U\oplus\mathbf{A}_{1}^{\oplus8}$. There exists a double cover
$f_{1}\colon X\to\mathbf{F}_{4}$ branched over a curve $B=s+b$, where
$b$ is a curve in $|3L|$ with $8$ nodes $p_{1},\dots,p_{8}$ and
$s\cap b=\emptyset$; that double cover is such that the images of
the $145$ $\cu$-curves on $X$ are
\begin{itemize}
\item the $8$ points $p_{1},\dots,p_{8}$,
\item the section s,
\item the $8$ fibers through the $8$ points $p_{1},\dots,p_{8}$,
\item the $8$ curves in $|L'|$ going through $7$ of the $8$ points $p_{1},\dots,p_{8}$,
\item the $56$ sections in the linear system $|L|$ that pass through $5$
of the $8$ points $p_{1},\dots,p_{8}$,
\item the $56$ curves in the linear system $|2L|$ that pass through $5$
of the $8$ points with multiplicity $1$ and through the $3$ remaining
with multiplicity $2$,
\item the $8$ curves in the linear system $|3L|$ that pass through the $8$ points with multiplicity $2$, except at $1$ point where the 
multiplicity is $3$. 
\end{itemize}

There exists, moreover, a double cover of $\,\PP^{2}$ branched over a
sextic curve with one node. Through that node, there are $8$ lines
that are tangent to the sextic at another intersection point. There
are $64$ conics that are $6$-tangent to the sextic. The $72=8+64$
curves and the node are the images of the $145$ $\cu$-curves on
$X$. 

The two involutions corresponding to the two covers we described generate
the automorphism group $(\ZZ/2\ZZ)^{2}$ of~$X$. 
\end{thm}

The other members of the family have a similar description as double
covers of $\mathbf{F}_{2}$ or  $\mathbf{F}_{4}$, and one can 
describe their $\cu$-curves similarly. As above, these surfaces have another
model as a double cover of $\PP^{2}$ branched over a sextic curve.
In most cases, the set of $\cu$-curves can be naturally decomposed
as the union of two sets of $\cu$-curves having a different behavior:
one set generates the N\'eron--Severi lattice; it has a dual graph which
is directly related to the name of the lattice and can be easily represented. 
The second set is a set of extra $\cu$-curves which, when they exist,
form an interesting configuration, sometimes very symmetric. Often
the curves of the second set occur as pull-backs under a double cover
of conics or lines in $\PP^{2}$ that are tangent to the sextic branch
curve at an intersection point. 

In this paper, we also describe the configurations of the $\cu$-curves
contained in the K3 surfaces. The perhaps nicest configuration is
as follows.

\begin{thm}\label{thm:3}
Let $X$ be a general K3 surface with N\'eron--Severi lattice $U(4)\oplus A_{1}^{\oplus3}$.
The surface $X$ is a double cover of $\,\PP^{2}$ branched over a smooth
sextic curve $C_{6}$ such that there exist $12$ conics that are
tangent to the sextic at an intersection point; the $24$ $\cu$-curves
on $X$ are mapped in pairs to the $12$ conics. 

There exists a partition of the $24$ $\cu$-curves into three sets
$S_{1},S_{2},S_{3}$ of $\,8$ curves each, such that for curves $B,B'$
in two different sets $S,S'$, one has $BB'\in\{0,4\}$ and for any
$B\in S$, there are exactly $4$ curves $B'$ in $S'$ such that $BB'=4$,
and symmetrically for $B'$. The sets $S$ and $S'$ form an $8_{4}$
configuration called the M\"obius configuration.

The moduli space of K3 surfaces polarized by $U(4)\oplus A_{1}^{\oplus3}$
is unirational.
\end{thm}

The following graph is the Levi graph of the M\"obius configuration,
where vertices in red are curves in $S$, vertices in blue are curves
in $S'$, and an edge links a red curve to a blue curve if and only
if their intersection number is $4$:

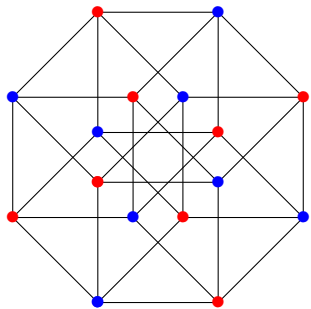
\begin{center}
\begin{tikzpicture}[scale=0.8]


\draw (0,0) rectangle (2,2);
\draw (-1.414,1.414) rectangle (2-1.414,2+1.414);
\draw (1.414,1.414) rectangle (2+1.414,2+1.414);
\draw (0,2*1.414) rectangle (2,2+2*1.414);
\draw (0,0) -- ++(1.414,1.414) -- ++(-1.414,1.414) --++ (-1.414,-1.414) -- cycle;
\draw (0,2) -- ++(1.414,1.414) -- ++(-1.414,1.414) --++ (-1.414,-1.414) -- cycle;
\draw (2,0) -- ++(1.414,1.414) -- ++(-1.414,1.414) --++ (-1.414,-1.414) -- cycle;
\draw (2,2) -- ++(1.414,1.414) -- ++(-1.414,1.414) --++ (-1.414,-1.414) -- cycle;

\draw [color=red] (0,2) node {$\bullet$};
\draw [color=red] (-1.414,1.414) node {$\bullet$};
\draw [color=red] (1.414,1.414) node {$\bullet$};
\draw [color=red] (0,2) node {$\bullet$};
\draw [color=red] (0,2+2*1.414) node {$\bullet$};
\draw [color=red] (2-1.414,2+1.414) node {$\bullet$};
\draw [color=red] (2,0) node {$\bullet$};
\draw [color=red] (2+1.414,2+1.414) node {$\bullet$};
\draw [color=red] (2,2*1.414) node {$\bullet$};

\draw [color=blue] (0,0) node {$\bullet$};
\draw [color=blue] (-1.414,3.414) node {$\bullet$};
\draw [color=blue] (2-1.414,1.414) node {$\bullet$};
\draw [color=blue] (2+1.414,1.414) node {$\bullet$};
\draw [color=blue] (2,2) node {$\bullet$};
\draw [color=blue] (0,2*1.414) node {$\bullet$};
\draw [color=blue] (2,2+1.414*2) node {$\bullet$};
\draw [color=blue] (1.414,2+1.414) node {$\bullet$};
\draw [color=blue] (0,0) node {$\bullet$};


\end{tikzpicture}
\end{center} 

From that graph, we can moreover read off the intersection numbers
of the curves in $S$ (and $S'$) as follows:

For any red curve $B$, there are four blue curves linked to it by an
edge. Consider the complementary set of blue curves; this is another
set of four blue curves, all linked through an edge to the same
red curve $B'$. Then we have $BB'=6$, and for any other red curve
$B''\neq B'$, we have $BB''=2$. Symmetrically, the intersection numbers
between the blue curves follow the same rule.

Nikulin also studied that configuration of $24$ curves, which he
obtained by lattice considerations (see \cite[Section~8.3]{Nikulin}),
showing a very nice relation with the $24$ roots of $\mathbf{D}_{4}$. 

Another result is about the ``famous'' $95$ families of K3 hypersurfaces
in weighted projective threefolds. In Section \ref{subsec:About-the-famous95},
we remark that among these K3 surfaces, many are surfaces with finite
automorphism group, and their moduli space is unirational. We also study
the unirationality of many other moduli spaces of K3 surfaces with
finite automorphism group. 

At the end of the paper, in Section~\ref{sec:table}, we give the list
of the K3 surfaces with finite automorphism group and their number of
$\cu$-curves, for Picard number at least $3$. In this paper is missing
the classification for  Picard number $4$: that is the subject of
another paper \cite{ACR} with Artebani and Correa Diesler.

In \cite{Kondo}, Kondo classifies the automorphism group of general
K3 surfaces in each of the $118$ families. For $105$ of these families,
Kondo proves that the automorphism group of a general K3 surface is either
trivial or the group $\ZZ/2\ZZ$, leaving indeterminate which cases
actually happen. Our work clarifies the situation for each of these
families; the automorphism group of a general K3 surface of each
family can be found in Section~\ref{sec:table}.

It is not possible for the author to cite every contribution on the
subject of K3 surfaces with finite automorphism group. However, let
us give some details on \cite{AN}, where Alexeev and Nikulin classify
log del Pezzo surfaces of index at most $2$ (\textit{i.e.}, normal surfaces $X$
with quotient singularities such that the anti-canonical divisor $-K_{X}$
is ample and $2K_{X}$ is a Cartier divisor). They also describe possible
configurations of exceptional curves on a minimal resolution of $X$,
\textit{i.e.}, smooth rational curves with negative self-intersection. 

The paper of  Alexeev and Nikulin gives the description of the $(-2)$-curves in Theorems
\ref{thm:1} and~\ref{thm:2}. Table 3 in \cite{AN} gives these configurations for the ``special'' resolutions of the most degenerate
del Pezzo surfaces, those with the largest configuration of $(-2)$-curves.
But  the procedure for the general case is also provided in \cite{AN}:
if a certain ADE configuration disappears on a deformation, then one
should add the corresponding Weyl group orbit of the $(-1)$-curves.
The $(-2)$-curves on the double cover K3 surface are the preimages
of the $(-1)$-, $(-2)$- and $(-4)$-curves on these resolutions.

For example, for the double covers of the degree $1$ del Pezzo surfaces:
On the most degenerate ``almost'' del Pezzo
surface, there is an ${\bf E}_{8}$ configuration of $(-2)$-curves
and a single $(-1)$-curve. The corresponding K3 surface has $16+1=17$
$(-2)$-curves. On the most general del Pezzo surfaces, there are no
$(-2)$-curves, and there is a single $W({\bf E}_{8})$-orbit of $(-1)$-curves,
which has cardinality $|W({\bf E}_{8})/W({\bf E}_{7})|=240$. Indeed,
this description of the $(-1)$-curves on a del Pezzo surface is well
known. The corresponding K3 surface has $240$ $(-2)$-curves.

Table 3 in \cite{AN} contains this information for all $2$-elementary
hyperbolic lattices where the K3 surfaces have finite automorphism
group, excluding the $(19,1,1)$ case which has $g=1$ (see the notation
in \cite{AN}).
The lattices in Theorems~\ref{thm:1} and~\ref{thm:2} are $2$-elementary. 

To be more precise, the paper \cite{AN} computes the fundamental
chambers $P^{2,4}$ for all the $2$-elementary lattices in question,
for the full reflection group. That group is generated by the reflections
in the $(-2)$-vectors and in the $2$-divisible $(-4)$-vectors.
And by \cite[Proposition 2.4.1]{AN}, the walls of the fundamental
chamber $P^{2}$ for the Weyl group $W^{2}$ are the orbits of the
$(-2)$-walls of $P^{2,4}$ under the group generated by the $(-4)$-walls of $P^{2,4}$. In terms of graphs, the set of walls of $P^{2}$
(that is, the $(-2)$-curves on $X$) is the union of the orbits of
the white vertices in the Coxeter graph of $P^{2,4}$ (see \cite[Table 3, p.~93]{AN})
under the Weyl group corresponding to the black subgraph. 

On the subject of K3 surfaces with finite automorphism group, let
us also mention the paper \cite{Zhang}, where Zhang classifies the
quotients of K3 surfaces by an involution. 

At last, let us also cite \cite{AHL} by Artebani, Hausen and Laface,
and \cite{McK} by McKernan, where  the Cox ring of some
K3 surfaces is studied. It turns out that a K3 surface has a finitely generated
Cox ring if and only if its cone of effective divisor classes is polyhedral
or, equivalently, if its automorphism group is finite. K3 surfaces
with finite automorphism group and Picard number $\rho=2$ have been
described by Piatetskii-Shapiro and Shafarevich \cite{PS} and also 
studied in \cite{GLP} (see also \cite{AHL,Ottem} for their
Cox rings).  The Cox rings of K3 surfaces with a non-symplectic involution
and $2\leq\rho\leq5$ have been described in \cite{AHL}. Also in
\cite{AHL} is studied the Cox ring for all K3 surfaces which are
general double covers of del Pezzo surfaces. Finally, let us mention
(see \cite{PS}) that the general  K3 surface with Picard number $1$
has trivial automorphism group, unless its N\'eron--Severi lattice is
generated by $D$ with $D^{2}=2$, in which case the automorphism
group is generated by the non-symplectic involution. 

\subsection*{Acknowledgements} The author is very grateful to the referees
for many suggestions, comments, ideas improving the paper; one of
the referees explained to us how to finish the computation of the
number of $\cu$-curves for certain difficult cases (see Section \ref{subsec:About-the-computations}).
The author also thanks Alessandra Sarti and Alan Thompson for very
useful email exchanges about moduli spaces of K3 surfaces. Part
of this paper was written during the author's stay at Max Planck Institute
for Mathematics in Bonn, to which the author is grateful for its hospitality. The computations have been done using Magma
software; \textit{cf.}\ \cite{Magma}. 

\section{Notation, Preliminaries}

\subsection{Notation and conventions}

We work over the complex numbers.  Linear equivalence between divisors is denoted by $\equiv$. On
a K3 surface, linear and numerical equivalences coincide. A $\cu$-curve
is unique in its numerical equivalence class, so we will often not
distinguish between a $\cu$-curve and its equivalence class.

The configuration of a set of $\cu$-curves $C,C',\ldots$ is given by
its dual graph, where a vertex represents a $\cu$-curve and two vertices
are linked by an edge if the intersection $C\cdot C'$ (sometimes also denoted by $CC'$) of the corresponding curves
$C$, $C'$ satisfies $C\cdot C'>0$. Unless explicitly stated, when $n=C\cdot C'$, we label
the edge $n$: 
\begin{center}
\begin{tikzpicture}[scale=1]

\draw (0,0) -- (1,0);

\draw (0,0) node {$\bullet$};
\draw (1,0) node {$\bullet$};

\draw (0.5,0) node [above]{$n$};

\end{tikzpicture}
\end{center} 
Moreover, a thick edge
\begin{center}
\begin{tikzpicture}[scale=1]

\draw [very thick] (0,0) -- (1,0);

\draw (0,0) node {$\bullet$};
\draw (1,0) node {$\bullet$};


\end{tikzpicture}
\end{center} 
between the two vertices $C$, $C'$ means that $C\cdot C'=2$, and no label
means that $C\cdot C'=1$. 

We denote by ${\bf a}_{n}$, ${\bf d}_{n}$, ${\bf e}_{n}$ the Du Val curve
singularities, also called simple curve singularities (see \cite[Chapter II, Section 8]{BHPV}).  
Let $C_{6}\hookrightarrow\PP^{2}$ be a reduced plane sextic curve
with at most $\mathbf{a}\mathbf{d}\mathbf{e}$ singularities. Let
$X$ be the K3 surface which is the minimal desingularization of the
double cover branched over $C_{6}$. We denote by $\eta\colon X\to\PP^{2}$
the natural map. We say that a line is tritangent to $C_{6}$ if the
multiplicities at the intersection points of the line and $C_{6}$
are even. Similarly, we say that a conic is $6$-tangent to $C_{6}$
if the  multiplicities of the intersection points at the conic and $C_{6}$
are even. The following result is well known. 

\begin{lem}
\label{lem:splitingsPullBackTritang}Let $R_{d}$ be a tritangent
line $(d=1)$ or a $6$-tangent conic $(d=2)$. The pull-back on $X$
of $R_{d}$ splits: $\eta^{*}R_{d}=A+B$, where $A$, $B$ are two smooth
rational curves with intersection number $d^{2}+2$. One has
$D_{2}\cdot A=D_{2}\cdot B=d$ and $A+B\equiv dD_{2}$, where $D_{2}$ is the
pull-back by $\eta$ of a line. 

Conversely, if two $\cu$-curves $A$, $B$ are such that there exist a nef,
base-point free divisor $D_{2}$ of square $2$ and a $d\in\{1,2\}$
such that $dD_{2}=A+B$, then the image of $A$ and $B$ by the natural
map obtained from $|D_{2}|$ is a rational curve of degree $d$. 
\end{lem}

For a symmetric matrix $Q$ with integral entries, we denote by $[Q]$
the lattice with Gram matrix $Q$. We denote by $U$ the lattice
$\begin{bsmallmatrix}0 & 1\\1 & 0\end{bsmallmatrix}$. If $L$ is a
  lattice and $m\in\ZZ$, then $L(m)$ is the lattice with the same
  group as $L$ but with Gram matrix multiplied by $m$.

We denote by ${\bf A}_{n}$, ${\bf D}_{n}$, ${\bf E}_{n}$ the negative-definite
lattices that correspond to the root systems denoted with the same letters. 
 A double cover branched over a curve with a singularity ${\bf b}_{n}$ 
($\bf b$ in $\bf a,\, \bf d,\,  \bf e$)
is a singular surface with $B_n$ singularity ($B$ in $A,\, D,\,  E$), see \cite[Chapter III, Section 6]{BHPV}. 
Its minimal resolution is by a union of $(-2)$-curves, which curves
 generate a lattice ${\bf B}_{n}$.

In this paper, by an elliptic fibration of a K3 surface $X$, we mean
a morphism $X\to\PP^{1}$ with connected fibers. We will frequently 
use implicitly the linear equivalence relations obtained from (singular)
fibers of an elliptic fibration; such information can be read off the
dual graph of the $\cu$-curves. The Kodaira classification of singular
fibers of elliptic fibrations and their dual graph with their weight
can be found, \textit{e.g.}, in \cite[Section~V.7]{BHPV}. We denote
by $\tilde{{\bf A}}_{n}$, $\tilde{{\bf D}}_{n}$, $\tilde{{\bf E}}_{n}$
the types of singular fibers of an elliptic fibration  in Kodaira's
classification. 

Let $N$ be a lattice of signature $(1,\rho-1)$. We denote by $\mathcal{M}_{N}$
the moduli of K3 surfaces $X$ polarized by a primitive embedding
$N\hookrightarrow\NS X)$. 

\subsection{Some results of Saint-Donat on linear systems of K3 surfaces}

To be self-contained, let us recall the following results of Saint-Donat
\cite{SaintDonat}. 

\begin{thm}\label{thm:4}
  \label{thm:SaintDonat-1}Let $D$ be a divisor on a K3 surface $X$.
  \begin{enumerate}[a)]
\item (\cite[Proposition 2.6]{SaintDonat}).  If D is effective and non-zero and $D^{2}=0$, then $D=aE$, where
$|E|$ is a free pencil and $a\in\NN$.
\item (\cite[(2.7.3)]{SaintDonat}).  If $D$ is big and nef, then $h^{0}(D)=2+\frac{1}{2}D^{2}$ and
either $|D|$ has no fixed part, or $D=aE+\Gamma$, where $|E|$ is
a free pencil and $\G$ is an effective $(-2)$-class such that $\G \cdot E=1$.
\item (\cite[Section 4.1]{SaintDonat}). If $D$ is big and nef and $|D|$ has no fixed part, then $|D|$
is base-point free and either $\varphi_{|D|}$ is $2$-to-$1$ onto
its image $($hyperelliptic case$)$, or it maps $X$ birationally onto its
image, contracting the $(-2)$-curves $\G$ such that $D\cdot \G=0$ to
singularities of type $ADE$. 
\end{enumerate}
  \end{thm}

For a divisor $D$ such that the linear system $|D|$ has no base
points, we denote by $\varphi_{D}$ the morphism associated to $|D|$.
In case the linear system $|D|$ is hyperelliptic, we have the following. 

\begin{thm}[See \protect{\cite[Proposition 5.7]{SaintDonat}} and its proof\,]
\label{thm:SaintDonat-2} Let $|D|$ be a complete linear system on the K3 surface
$X$. Suppose that $|D|$ has no fixed components. Then $|D|$ is
hyperelliptic if and only if either
  \begin{enumerate}[a)]
\item  $D^{2}\geq4$, and there is a fiber of an elliptic fibration  $F$
  such that $F\cdot D=2$; or  
\item
  $D^{2}\geq4$, and there is an irreducible curve $D'$ with $D'^{2}=2$
  and $D\equiv2D'$ $($thus in that case $D^{2}=8)$;

  then the image of
the map $\varphi_{D}$ is the Veronese surface; or 
\item $D^{2}=2$;

  then $\varphi_{D}$ is a double cover of $\,\PP^{2}$.
\end{enumerate}
In case a), the image of the associated map $\varphi_{D}\colon X\to\PP^{\frac{1}{2}{D^{2}}+1}$
is a rational normal scroll of degree $\frac{1}{2}{D^{2}}$, except in
the following three cases:
\begin{enumerate}[i)]
\item  $D\equiv4F+2\G$, where $F$ is a fiber and $\G$ is a $\cu$-curve
  such that $F\G=1$.

  In that case, $D^{2}=8$, and $\varphi_{D}(X)$
is a cone in $\PP^{5}$ over a rational normal quartic curve in a
hyperplane $\PP^{4}\subset\PP^{5}$. The map $\varphi_{D}$ factors
through $X\stackrel{\varphi}{\to}\mathbf{F}_{4}\to\PP^{5}$, where
$\varphi$ is a morphism onto the Hirzebruch surface $\mathbf{F}_{4}$
and $\mathbf{F}_{4}\to\PP^{5}$ is the contraction map of the unique
section $s$ such that $s^{2}=-4$. The branch locus of $\varphi$
is the union of $s$ and a reduced curve $B'$ in the linear system
$|3s+12f|$ such that $s\cap B'=\emptyset$ $($here $f$ is a fiber of
the unique fibration $\mathbf{F}_{4}\to\PP^{1})$. 
\item  $D\equiv3F+2\G_{0}+\G_{1}$, where $\G_{0}$, $\G_{1}$ are $\cu$-curves
  such that $\G_{0}\cdot F=1$, $\G_{1}\cdot F=0$ and $\G_{0}\cdot \G_{1}=1$.

  In that
case, $D^{2}=6$, and $\varphi_{D}(X)$ is a cone in $\PP^{4}$ over
a rational normal cubic curve in a hyperplane $\PP^{3}\subset\PP^{4}$.
There is a factorization of $\varphi_{D}$ through $X\stackrel{\varphi}{\to}\mathbf{F}_{3}\to\PP^{4}$.
The branch locus of $\varphi$ is the union of the unique section
$s$ such that $s^{2}=-3$ and a reduced curve $B'$ in the linear
system $|3s+10f|$ such that $sB'=1$.
\item
\begin{enumerate*}
  \item[u)] $D\equiv2F+\G_{0}+\G_{1}$, where $\G_{0}$, $\G_{1}$ are $\cu$-curves
such that $\G_{0}\cdot F=\G_{1}\cdot F=1$ and $\G_{0}\cdot \G_{1}=0$ $($then $D^{2}=4)$.
\item[v)] $D\equiv2F+2\G_{0}+2\G_{1}+\dots+2\G_{n}+\G_{n+1}+\G_{n+2}$, where
  $n\geq0$ and the $\G_{i}$ are $\cu$-curves, such that $D^{2}=4$.
  \end{enumerate*}
\end{enumerate}
In cases iii) u) and iii) v), $\varphi_{D}(X)$ is a quadric cone
in $\PP^{3}$, and there is a factorization $X\stackrel{\varphi}{\to}\mathbf{F}_{2}\to\varphi_{D}(X)$.
In these two cases, the branch locus $B$ of $\varphi$ is in $|4s+8f|$.
In case u), $B$ does not contain the section $s$ such that $s^{2}=-2$; 
in case v), one has $B=s+B'$, with $B'$ reduced and $sB'=2$. 
\end{thm}

\begin{rem}
For brevity, we will say that a divisor $D$ is base-point free or
is hyperelliptic if the associated linear system $|D|$ is; we hope
this will not induce any confusion to the reader. 
\end{rem}

\subsection{\label{subsec:About-the-computations}About the computations}

Let $X$ be a K3 surface with finite automorphism group and N\'eron--Severi
lattice $L$ of rank $\rho>2$. We did our search of $\cu$-curves
on $X$ using an algorithm of Shimada \cite[Section 3]{Shimada} as
described in \cite{Roulleau}. 

The inputs of that algorithm are the Gram matrix of a basis of the
N\'eron--Severi lattice and an ample class. For each N\'eron--Severi lattice
$L$ involved, we computed such an ample class $D$ as follows. The lattices $L$ in this paper are mainly of the form 
\[
L=U(k)\oplus K\quad\text{or}\quad L=[m]\oplus K\quad(k,m\in\NN^{*}),
\]
where $K$ is a direct sum of ${\bf ADE}$ lattices. To obtain an
ample divisor $D\in L$, we construct a divisor $D'\in L$ such that
on a set of roots of $K$ which is also a basis of $K$, the divisor
$D'$ has positive intersection. Then we add a suitable multiple $a$
of an element $u$ in the $U[k]$- or $[m]$-part with positive square
so that one gets a divisor $D''=au+D'\in L$ with $(D'')^{2}>0$. We
then check that the divisor $D''$ is ample by verifying that the
negative-definite orthogonal complement $D''^{\perp}$ does not contain
roots, \textit{i.e.}, elements $v$ such that $v^{2}=-2$ (if this is not the
case, then we increase the parameter $a$). Then, the transitivity of
the action of the Weyl group
\[
W=\left\langle s_{\d}\colon x\to x+(x\cdot \d)\d\,\,|\,\,\d\in\Delta\right\rangle 
\]
(with $\Delta=\{\d\in L\,|\,\d^{2}=-2\}$) on the chambers of the positive
cone (see \cite[Proposition 8.2.6]{Huybrechts}) allows us to choose
$D=D''$ as an ample class for the K3 surfaces with lattice $L$. 

In many cases, the first-found ample divisor $D$ is such that $D^{2}$
is large. By computing the $\cu$-curves on the K3 surface and by
using Shimada's algorithm in \cite{Roulleau}, we are able to find
ample classes with smaller self-intersection. For each lattice, we
give the ample class $D$ with the smallest $D^{2}$ we found.

With that knowledge of an ample class $D$, one can run Shimada's
algorithm for the computation of the (classes of) the $\cu$-curves
$C$ on $X$ which have degree $C\cdot D$ less than or equal to a fixed bound.
We then use the test function in \cite{ACL}  to check that
we obtain the complete list of $\cu$-curves, which worked well (and
confirmed in another way the already known cases by Nikulin), except
for the cases 
\[
U(2)\oplus{\bf A}_{1}^{\oplus7},\quad U\oplus{\bf A}_{1}^{\oplus8},\quad U\oplus{\bf D}_{4}\oplus{\bf A}_{1}^{\oplus5},\quad U\oplus{\bf D}_{4}^{\oplus2}\oplus{\bf A}_{1}^{\oplus2},\quad U\oplus{\bf D}_{8}\oplus{\bf A}_{1}^{\oplus3},
\]
where the computation were too heavy to finish (there are too many
facets on the effective cone), and thus we have only a lower bound.
For these cases, one can obtain the exact number of $(-2)$-curves
using the following approach communicated to us by one of the referees. 

One starts by embedding the lattice $\NS X)$ into another N\'eron--Severi
lattice $\Ns'$ of larger rank for which one has already determined
the set of classes of $(-2)$-curves. For example, let us take
\[
II_{1,17}=U\oplus{\bf E}_{8}\oplus{\bf E}_{8}
\]
for $\Ns'$, as in Section \ref{subsec:The-latticeUE8E8}. The set
of $(-2)$-curves $A_{1},\dots,A_{19}$ and their configuration had
been determined in the classical work of Vinberg \cite{Vinberg2}.
Let $\mathcal{P}_{18}$ be the positive cone of $II_{1,17}\otimes\RR$
containing an ample class and $\mathcal{N}_{18}$ the set of the
closures of connected components of the complement of the union of
all hyperplanes $(r)^{\perp}$ of $\mathcal{P}_{18}$ defined by $(-2)$-vectors
$r$. Suppose that we have a primitive embedding
\[
\NS X)\hookrightarrow\Ns'=II_{1,17}, 
\]
and let $\mathcal{P}_{X}$ be the positive cone of $\NS X)\otimes S$
containing an ample class. We assume that $\mathcal{P}_{X}$ is embedded
into $\mathcal{P}_{18}$ and regard $\mathcal{P}_{X}$ as a subspace
of $\mathcal{P}_{18}$. We consider the closed subsets 
\[
\mathcal{P}_{X}\cap N'\quad (N'\in\mathcal{N}_{18}) 
\]
of $\mathcal{P}_{X}$ that contain a non-empty open subset of $\mathcal{P}_{X}$.
In \cite{ShimadaIMRN}, these closed subsets $\mathcal{P}_{X}\cap N'$
are called induced chambers. Let $N_{X}$ be the closure of the ample
cone of $\NS X)$. Since $N_{X}$ has only finitely many 
walls, the cone $N_{X}$ is tessellated by a finite number of induced
chambers. By the method in \cite{ShimadaIMRN}, we can determine the
set of induced chambers contained in $N_{X}$, and hence the set of
walls of $N_{X}$, that is, the classes of $(-2)$-curves on X. More
precisely, let $G$ be the subgroup of $O(\NS X))$ consisting of the 
isometries $g\in O(\NS X))$ that preserve $N_{X}$ and extend to
an isometry of $\Ns'=II_{1,17}$. We can calculate a complete set
of representatives of orbits of the action of $G$ on the set of induced
chambers contained in $N_{X}$. 

For each $\NS X)$ in the above five lattices, we embed $\NS X)$
into $II_{1,17}$ primitively as follows. Let $\rho$ be the rank
of $\NS X)$, and let $A_{i_{1}},\dots,A_{i_{18-\rho}}$ be the first
$18-\rho$ elements of the sequence
\[
A_{3},\,A_{5},\,A_{7},\,A_{9},\,A_{1},\,A_{17},\,A_{15},\,A_{13},\,A_{11}
\]
of $(-2)$-curves $A_{1},\dots,A_{19}$ in Section \ref{subsec:The-latticeUE8E8}.
Then $\NS X)$ is isometric to the orthogonal complement of these
classes $A_{i_{1}},\dots,A_{i_{18-\rho}}$ in $II_{1,17}$. The order
of $G$ and the orbit decomposition of induced chambers in $N_{X}$
under the action of $G$ are given in the following table, where,
for example, $42456960$ and $[1,55]$, $[2,7]$ in the first line mean
that there exist $55$ orbits of size $|G|$ and $7$ orbits of size
$|G|/2$ (with stabilizer group of order $2$), and hence there exist
\[
\left(55+\frac{7}{2}\right)|G|=42456960
\]
induced chambers contained in $N_{X}$. 

\begin{table}[ht]
\centering
{ \setlength\extrarowheight{2pt}
\begin{tabular}{|c|c|c|c|c|c|}
\hline 
$\rho$ & $\NS X)$ & $\#(-2)$ & $|G|$ & Chambers & Orbits\tabularnewline
\hline 
\hline 
$9$ & $U(2)\oplus{\bf A}_{1}^{\oplus7}$ & $240$ & $725760$ & $42456960$ & $[1,55],[2,7]$\tabularnewline
\hline 
$10$ & $U\oplus{\bf A}_{1}^{\oplus8}$ & $145$ & $80640$ & $30683520$ & $[1,372],[2,17]$\tabularnewline
\hline 
$11$ & $U\oplus{\bf D}_{4}\oplus{\bf A}_{1}^{\oplus5}$ & $90$ & $10080$ & $10800720$ & $[1,1061],[2,21]$\tabularnewline
\hline 
$12$ & $U\oplus{\bf D}_{4}^{\oplus2}\oplus{\bf A}_{1}^{\oplus2}$ & $59$ & $1440$ & $2400480$ & $[1,1649],[2,36]$\tabularnewline
\hline 
$13$ & $U\oplus{\bf D}_{8}\oplus{\bf A}_{1}^{\oplus3}$ & $39$ & $240$ & $376200$ & $[1,1557],[2,21]$\tabularnewline
\hline 
\end{tabular}
}
\end{table}

We give at the end of this atlas a table with the number of $\cu$-curves
in the $118$ different cases. 

One can also compute the finite set of all fibrations on $X$ as follows. By \cite[Proposition 2.4]{Kovacs}, if the class of a fiber $F$
in $\NS X)$ is indecomposable (\textit{i.e.}, is not the sum of two effective
divisors), then it is an extremal class in the closure of the effective
cone. But by \cite[Theorem 6.1]{Kovacs}, since there are only a finite
number of $\cu$-curves in our surface $X$, the $\cu$-curves are
the extremal classes. Therefore, there is a singular fiber of the elliptic
fibration  which is the sum of $\cu$-curves on $X$.

Now Kodaira's classification gives a finite number of possibilities
for the types $\tilde{{\bf A}}_{n}$, $\tilde{{\bf D}}_{n}$, $\tilde{{\bf E}}_{n}$
of the reducible fibers $F$ on $X$. Since a fiber $F$ in $\NS X)$
is a sum of $\cu$-curves (which are finitely many), there is a finite number
of possible degrees $FD$ of fibers $F$ with respect to $D$. Thus
there is an upper bound for $FD$, so that one can compute all fibrations
on $X$ using Shimada's algorithm (see \cite{Roulleau}).

For the computations, we used Magma \cite{Magma}. Our algorithms are
available as ancillary files of the submission of this paper on arXiv.

\subsection{Algorithms for double planes}

In this paper, many double plane models of K3 surfaces are studied.
Algorithms to obtain geometric properties of the double covering (singularities
of the branch curve $B$, configuration of irreducible components
of $B$, enumeration of splitting curves, \ldots) from the lattice data
are developed in the paper \cite{ShimadaLatt}. They can be modified
to the case where the linear system has a fixed component. Many arguments
in the present paper
(expressing a class $D_{2}$ such that
$(D_{2})^{2}=2$ as various sums of classes of $\cu$-curves with
non-negative coefficients) can be automated by these algorithms.

\subsection{Nikulin star-shaped dual graphs and lattices of type $\boldsymbol{U\oplus K}$}

In \cite{Nikulin}, Nikulin obtained the following result. 

\begin{prop}[\protect{\cite[Corollary 1.6.5)]{Nikulin}}; see
also \protect{\cite[Lemma 3.1]{Kondo}}]
\label{prop:STAR-Nikulin-Kondo} Let $X$ be a K3 surface. Assume that
$\NS X)\simeq U\oplus K$.  
\begin{enumerate}[i)]
\item There is an elliptic pencil $\pi\colon X\to\PP^{1}$ with a section. A fiber and the section generate the lattice isometric to~$U$.
\item If $K=\mathbf{G}\oplus K'$ with the lattice $\mathbf{G}$ generated
by irreducible elements of square $-2$, then $\pi$ has a singular
fiber of type $\tilde{\mathbf{G}}$, where $\mathbf{G}$ is among
the lattices 
\[
\mathbf{A}_{n},\;n\geq1,\quad \mathbf{D}_{n},\;n\geq4,\quad\mathbf{E}_{6},\,\mathbf{E}_{7},\,\mathbf{E}_{8}.
\]
\end{enumerate}
\end{prop}

Assume  $\NS X)\simeq U\oplus\bigoplus_{i\in I}\mathbf{G}^{(i)}$,
where the $\mathbf{G}^{(i)}$ are lattices generated by irreducible elements of
square~$-2$.
The dual graph of the $\cu$-curves contained
in the singular fibers of type $\tilde{\mathbf{\mathbf{G}}}^{(i)}$
and the section form a (so-called) star-shaped graph (see \cite{Nikulin}),
which we denote by ${\rm St}(\NS X))$.

\begin{thm}[See \protect{\cite[Theorem~0.2.2 and Corollary~3.9.1]{Nikulin}}]
\label{thm:Not-2-elementaryAllMinus2}
Let $X$ be a K3 surface with N\'eron--Severi lattice isometric to $U\oplus K$,
where $K$ is in the following list: 
\[
\begin{array}{c}
\mathbf{A}_{2};\quad\mathbf{A}_{1}\oplus\mathbf{A}_{2},\,\mathbf{A}_{3};\quad\mathbf{A}_{1}^{\oplus2}\oplus\mathbf{A}_{2},\,\mathbf{A}_{2}^{\oplus2},\,\mathbf{A}_{1}\oplus\mathbf{A}_{3},\,\mathbf{A}_{4};\\
\mathbf{A}_{1}\oplus\mathbf{A}_{2}^{\oplus2},\,\mathbf{A}_{1}^{\oplus2}\oplus\mathbf{A}_{3},\,\mathbf{A}_{2}\oplus\mathbf{A}_{3},\,\mathbf{A}_{1}\oplus\mathbf{A}_{4},\,\mathbf{A}_{5},\,\mathbf{D}_{5};\\
\mathbf{A}_{2}^{\oplus3},\,\mathbf{A}_{3}^{\oplus2},\,\mathbf{A}_{2}\oplus\mathbf{A}_{4},\,\mathbf{A}_{1}\oplus\mathbf{A}_{5},\,\mathbf{A}_{6},\,\mathbf{A}_{2}\oplus\mathbf{D}_{4},\,\mathbf{A}_{1}\oplus\mathbf{D}_{5},\,\mathbf{E}_{6};\quad\mathbf{A}_{7},\\
\mathbf{A}_{3}\oplus\mathbf{D}_{4},\,\mathbf{A}_{2}\oplus\mathbf{D}_{5},\,\mathbf{D}_{7},\,\mathbf{A}_{1}\oplus\mathbf{E}_{6};\quad\mathbf{A}_{2}\oplus\mathbf{E}_{6};\quad\mathbf{A}_{2}\oplus\mathbf{E}_{8};\quad\mathbf{A}_{3}\oplus\mathbf{E}_{8}.
\end{array}
\]
The K3 surface has finite automorphism group, and ${\rm St}(\NS X))$ is
the dual graph of all $\cu$-curves on $X$. 
\end{thm}

Let $X$ be a K3 surface such that $\NS X)\simeq U\oplus\bigoplus_{i\in I}\mathbf{G}^{(i)}$, where the $\mathbf{G}^{(i)}$ are lattices generated by irreducible elements of square $-2$. Let $F$ be a fiber of the natural fibration $\pi\colon X\to\PP^{1}$ and $E$ be the section as in Proposition~\ref{prop:STAR-Nikulin-Kondo}.  The divisor $D_{2}=2F+E$ is nef of square $2$, with base points since $D_{2}F=1$. By Theorem~\ref{thm:SaintDonat-2}, the divisor $D_{8}=2D_{2}$ is base-point free and hyperelliptic, and it defines a morphism $X\to\PP^{5}$ which factors through the Hirzebruch surface $\mathbf{F}_{4}$ so that the branch locus of $\eta\colon X\to\mathbf{F}_{4}$ is the disjoint union of the unique section $s$ such that $s^{2}=-4$ and $B$, a curve in the linear system $|3s+12f|$, where $f$ is a fiber of the unique fibration of $\mathbf{F}_{4}$. We  immediately have the folowing. 

\begin{prop}
\label{prop:consequenceOfNikulinStar}The image by $\eta$ of the
section $E$ is the section $s$; the pull-back on $X$ of the pencil
$|f|$ is the pencil of elliptic curves in the elliptic fibration  $\pi\colon X\to\PP^{1}$.
A singular fiber of $\pi$ of type $\tilde{{\bf A}}_{n}$,  $\tilde{{\bf D}}_{n}$, $\tilde{{\bf E}}_{n}$
 is mapped onto a fiber of $\,\mathbf{F}_{4}\to\PP^{1}$
that cuts $B$ at, respectively, an ${\bf a}_{n}$, ${\bf d}_{n}$, ${\bf e}_{n}$ singularity
of $B$.
\end{prop}

\subsection{\label{subsec:About-the-famous95}About the famous 95, their moduli
spaces and K3 surfaces with finite automorphism group}

In \cite{Belcastro} are studied the N\'eron--Severi lattices of the
so-called ``famous $95$'' families of K3 surfaces. These ``famous
$95$'' have been constructed by Reid (unpublished); the list of
these families appeared in Yomemura \cite{Yonemura} from the point
of view of singularity theory. The K3 surfaces involved are (singular)
anti-canonical divisors in weighted projective threefolds $\mathbb{WP}^{3}=\mathbb{WP}^{3}(\bar{a})$
(here $\bar{a}$ is the weight $\bar{a}=(a_{1},\dots,a_{4})$). As
we will see, many of these K3 surfaces have  finite automorphism
group.

Let $d_{\bar{a}}$ be the degree of an anti-canonical divisor in $\mathbb{WP}^{3}$.
It turns out that each general  degree $d_{\bar{a}}$ surface $\bar{X}$
in $\mathbb{WP}^{3}$ has the same singularities, and its minimal
desingularization is a K3 surface $X$. The main result of  \cite{Belcastro}
is the computation of the N\'eron--Severi lattice $\NS X)$ for a general
member of each of the $95$ families.

For $\bar{a}$ among the $95$ possible weights, let $L_{\bar{a}}\simeq\NS X)$
be the lattice of the N\'eron--Severi lattice of a general K3 surface
$X$ with singular model $\bar{X}\subset\mathbb{WP}^{3}(\bar{a})$,
and let $\mathcal{M}_{\bar{a}}$ be the moduli space of $L_{\bar{a}}$-polarized
K3 surfaces. 

\begin{prop}
\label{prop:famous95}There is a birational map between $\mathcal{M}_{\bar{a}}$
and the moduli space of degree $d_{\bar{a}}$ surfaces in $\mathbb{WP}^{3}(\bar{a})$
modulo automorphisms. 
\end{prop}

\begin{proof}
For each of the $95$ cases of $\bar{a}$, we compute the dimension
of the quotient space $\mathcal{Q}=\PP(H^{0}(\mathbb{WP}^{3}(\bar{a}),\mathcal{O}(d_{\bar{a}}))^{*})/\aut(\mathbb{WP}^{3}(\bar{a}))$,
using the formula 
\[
\dim\aut(\mathbb{WP}^{3}(\bar{a}))=-1+\sum h^{0}(\mathbb{WP}^{3},\mathcal{O}(a_{k})).
\]
The $95$ weights $\bar{a}$ are given in \cite[Table 3]{Belcastro}; 
the degrees $d_{\bar{a}}$ can be found in \cite[Table 4.6]{Yonemura}.
It turns out that this quotient has dimension $20-\rho$, where $\rho=\text{rank}(L_{\bar{a}})$
is the Picard number. There is a natural injective map from an open
set in $\mathcal{Q}$ to the moduli space $\mathcal{M}_{\bar{a}}$; 
since $\mathcal{M}_{\bar{a}}$ is also $(20-\rho)$-dimensional, both  
spaces are birational.
\end{proof}
A direct consequence of Proposition \ref{prop:famous95} is the following. 

\begin{cor}
The moduli spaces of the famous $95$ polarized K3 surfaces are unirational. 
\end{cor}

Comparing with Nikulin and Vinberg's lists, among the lattices $L_{\bar{a}}$
associated to the $95$ weights $\bar{a}$, at least $43$ lattices
are such that the general  K3 surface $X$ with $\NS X)\simeq L_{\bar{a}}$
(and a model $\bar{X}\subset\mathbb{WP}^{3}(\bar{a})$) has only a
finite number of automorphisms. However, there are some weights $\bar{a}$, $\bar{b}$
among the $95$ such that $L_{\bar{a}}\simeq L_{\bar{b}}$ (then there
are two singular models of the same K3 surface; see \cite{KM}). Without
counting the repetitions, one get (at least) $29$ lattices which
are lattices among the famous $95$ and are also lattices of K3 surfaces
with finite automorphism group. These lattices are 
\[
\begin{array}{c}
[2],\,[4],\,U,\,U(2),\,U\oplus{\bf A}_{1},\,U(2)\oplus{\bf A}_{1},\,U\oplus{\bf A}_{2},\,U(2)\oplus{\bf D}_{4},\,U\oplus{\bf D}_{4},\,\hfill\\
U\oplus{\bf D}_{4}\oplus{\bf A}_{1},\,U\oplus{\bf D}_{5},\,U\oplus{\bf A}_{2}\oplus{\bf D}_{4},\,U\oplus{\bf E}_{6},\,U\oplus{\bf E}_{7},\,U\oplus{\bf A}_{1}\oplus{\bf E}_{6},\,\hfill\\
U\oplus{\bf E}_{8},\,U\oplus{\bf E}_{7}\oplus{\bf A}_{1},\,U\oplus{\bf D}_{4}^{\oplus2},\,U\oplus{\bf A}_{2}\oplus{\bf E}_{6},\,U\oplus{\bf E}_{8}\oplus{\bf A}_{1},\,\hfill\\
U(2)\oplus{\bf D}_{4}^{\oplus2},\,U\oplus{\bf A}_{2}\oplus{\bf E}_{8},\,U\oplus{\bf E}_{8}\oplus{\bf A}_{1}^{\oplus3},\,U\oplus{\bf D}_{8}\oplus{\bf A}_{1}^{\oplus3},\,U\oplus{\bf D}_{8}\oplus{\bf D}_{4},\,\hfill\\
U\oplus{\bf E}_{8}\oplus{\bf D}_{6},\,U\oplus{\bf E}_{8}\oplus{\bf E}_{7},\,U\oplus{\bf E}_{8}\oplus{\bf E}_{8},\,U\oplus{\bf E}_{8}\oplus{\bf E}_{8}\oplus{\bf A}_{1},\hfill
\end{array}
\]
plus the lattice $\begin{psmallmatrix}2 & 1\\1 & -2\end{psmallmatrix}$. For these lattices $L_{\bar{a}}$, we thus have a complete description
of the K3 surfaces $X$ with finite-order automorphism group with
$\NS X)\simeq L_{\bar{a}}$ as singular model(s) in weighted projective
space(s).

It is worth mentioning that the $95$ families have mirrors (see \cite{Belcastro}); not all mirrors are among the $95$ families, but many mirrors are
K3 surfaces with finite automorphism group. The families that do
not already appear in the above list are
\[
\begin{array}{c}
U\oplus{\bf A}_{1}^{\oplus2},\,U(3)\oplus{\bf A}_{2},\,U\oplus{\bf A}_{1}\oplus{\bf A}_{2},\,U\oplus{\bf A}_{2}^{\oplus2},\\
U\oplus{\bf A}_{2}\oplus{\bf A}_{3},\,U\oplus{\bf A}_{5},\,U\oplus{\bf E}_{8}\oplus{\bf A}_{3}.
\end{array}
\]

\subsection{On the automorphism groups\label{subsec:About-the-automorphism}}

In \cite{Kondo}, Kondo studies the automorphism group of a general
K3 surface with a finite number of $\cu$-curves and Picard number
at least $3$. The main result of Kondo's paper \cite[Table 1]{Kondo}
is that for a certain list of $12$ lattices among the $118$ lattices,
the automorphism group is $(\ZZ/2\ZZ)^{2}$, that for $U\oplus{\bf E}_{8}^{\oplus2}\oplus A_{1}$
it is $\mathfrak{S}_{3}\times\ZZ/2\ZZ$ and that otherwise, for the
remaining $105$ families, it is either the trivial group or $\ZZ/2\ZZ$. 

Our study enables us to construct a hyperelliptic involution for surfaces in $90$ out of $105$ families, and we prove that the surfaces in the remaining $15$ families have trivial automorphism group, thus completing in that way the results of Kondo. These results are summarized in the table in Section~\ref{sec:table}.

\begin{rem}
For the ``general'' assumption on the surface, which can be made more
precise, we refer to the introduction of Kondo's paper \cite{Kondo}.
That hypothesis is important since Kondo constructed special K3 surfaces
with finite automorphism group isomorphic to $\ZZ/42\ZZ$ or $\ZZ/66\ZZ$
(see \cite{Kondo86}). When computing the automorphism group, the
K3 surfaces we consider are always supposed general.
\end{rem}

When the K3 surface $X$ has automorphism group isomorphic to $(\ZZ/2\ZZ)^{2}$
and Picard number at most $14$, we describe the two hyperelliptic involutions
generating the automorphism group $\aut(X)$;
when the Picard number is larger, we refer to
the description in \cite{Kondo}. 

For computing the automorphism group of a K3 surface $X$, it is important
to know the image of the natural map 
\[
\varphi\colon \aut(X)\to O(\NS X)).
\]
It turns out that when the automorphism group is not trivial, 
there is always a hyperelliptic involution~$\s$. Using our description
of the set of $\cu$-curves, one can understand when an involution
$\s$ is in the kernel of~$\varphi$. One may also use Proposition
\ref{prop:consequenceOfNikulinStar} in the case when the K3 lattice is
of type $U\oplus\bigoplus\tilde{\mathbf{G}}_{j}$ as follows. When
a K3 surface is the double cover branched over a curve with only  singularities  of type  
\[
\mathbf{a}_{1},\,\mathbf{d_{4}},\,\mathbf{d_{6}},\,\mathbf{d_{8}},\,\mathbf{e_{7}},\,\mathbf{e_{8}},
\]
the action on $\cu$-curves in the star of $U\oplus\bigoplus\tilde{\mathbf{G}}_{j}$
is trivial, and if there are any singularities of type 
\[
\mathbf{a}_{2},\dots,\mathbf{a}_{5},\,\mathbf{d}_{5},\,\mathbf{d}_{7},\,\mathbf{e}_{6},
\]
 then the involution acts non-trivially on the set of $\cu$-curves
in the star of $U\oplus\bigoplus\tilde{\mathbf{G}}_{j}$. Since the
N\'eron--Severi lattice is always generated by these $\cu$-curves, one
can then understand when an involution is in the kernel of $\varphi$.
In the table in Section~\ref{sec:table}, we indicate when the action
of a hyperelliptic involution on the N\'eron--Severi lattice is not
trivial.

Let $X$ be a K3 surface for which Kondo proved that $\aut(X)$ is
trivial or $\ZZ/2\ZZ$. 

\begin{prop}
\label{prop:AUTOTrivial}Suppose that $X$ is general and the rank
of $\,\NS X)$ is less than or equal to $8$ and
\begin{enumerate}[i)]
\item there is no $(-2)$-curve $A$ and fiber $F$ such that $A\cdot F=1$; 
\item there is no big and nef divisor $D$ such that $D\cdot F=2$ for a
fiber $F$; 
\item there is no big and nef divisor $D$ such that $D^{2}=2$.
  \end{enumerate}
Then the automorphism group of $\,X$ is trivial.
\end{prop}

Here, a fiber $F$ means an irreducible curve with $F^{2}=0$.

\begin{proof}
By Theorems \ref{thm:SaintDonat-1} and \ref{thm:SaintDonat-2},  the
hypotheses i), ii) and iii) imply that there is no hyperelliptic involution
acting on $X$. By \cite[Table 1 and Lemma 2.3]{Kondo}, since the
rank of $\NS X)$ is less than or equal to $8$, the automorphism group
is either trivial or generated by a hyperelliptic involution; thus
it must be trivial. 
\end{proof}

\subsection{\label{subsec:On-the-irreducibility}On the irreducibility of the
moduli spaces}

The aim of this section is to prove the following result.

\begin{prop}
\label{prop:The-118-moduli-irred}The $118$ moduli spaces of K3 surfaces
with Picard number at least $3$ and finite automorphism group are irreducible.
\end{prop}

Let us recall that if $L$ is an even lattice of rank $\rho$ and
signature $(1,\rho-1)$, we denote by $\mathcal{M}_{L}$ the moduli
space of K3 surfaces $X$ polarized by a primitive embedding $j_{X}\colon L\hookrightarrow\NS X)$.
The moduli space $\mathcal{M}_{L}$ may depend upon the choice of the embedding
of $L$ in the K3 lattice $\Lambda_{K3}=U^{\oplus3}\oplus{\bf E}_{8}^{\oplus2}$:
two non-isometric embeddings will give two different moduli spaces
(see \cite{NikulinFinite} or \cite{Dolga}). For that question, one
can use the following result. 

\begin{thm}[See \protect{\cite[Theorem 14.1.12]{Huybrechts}}]
\label{thm:existencePlobngement}
Let $\L$ be an even unimodular lattice of signature $(n_{+},n_{-})$
and $M$ be an even lattice of signature $(m_{+},m_{-})$. If $m_{+}<n_{+}$, $m_{-}<n_{-}$
and
\begin{equation}
\ell(M)+2\leq\rk\L)-\rk M),\label{eq:cond-unique-emb}
\end{equation}
then there exists a primitive embedding $M\hookrightarrow\L$, which
is unique up to automorphisms of $\L$. 
\end{thm}

Here $\ell(M)$ is the minimal number of generators of the discriminant
group of $M$ and $\rk M)$ is its rank. 

\begin{rem}
\label{rem:easy-embed}Since $\ell(M)\leq\rk M)$, if $\rk M)\leq\frac{1}{2}(\rk\L)-2)$,
then condition \eqref{eq:cond-unique-emb} is verified. 
\end{rem}

In our situation, when $\L=\L_{K3}$ and $M=L$ is an even lattice
of signature $(1,\rk L)-1)$, one has the following. 

\begin{cor}
\label{cor:unicity-primitive-embedding}Suppose $\ell(L)\leq20-\rk L)$.
Then the primitive embedding $L\hookrightarrow\L_{K3}$ is unique up to
automorphisms. 
\end{cor}

 Suppose that the embedding $j\colon L\to\L_{K3}$ is unique up to automorphisms.
The following criterion of Dolgachev may be used  to check
if the moduli space $\mathcal{M}_{L}$ is irreducible. 

\begin{thm}[See \protect{\cite[Proposition 5.6]{Dolga}}]
\label{thmModuliIrreducible}
Suppose that $L^{\perp}\subset\L_{K3}$ contains a sublattice isometric
to $U$ or $U(2)$. Then the moduli $\mathcal{M}_{L}$ is irreducible.
\end{thm}

Let $L$ be an even lattice of signature $(1,\rk L)-1)$ such that
$\ell(L)\leq18-\rk L)$. 

\begin{cor}
\label{cor:IrreductibleM(L)}The moduli space $\mathcal{M}_{L}$ is
irreducible.
\end{cor}

\begin{proof}
By Theorem \ref{thm:existencePlobngement}, one can find a primitive
embedding of $L$ into the sublattice $U^{\oplus2}\oplus{\bf E}_{8}^{\oplus2}$
of $\L_{K3}$. Then the orthogonal complement of $L$ in $\L_{K3}$
contains a copy of $U$, and therefore by Theorem \ref{thmModuliIrreducible},
the moduli space $\mathcal{M}_{L}$ is irreducible. 
\end{proof}

Using that  $\ell(L)\leq\rk L)$, one obtains the following. 

\begin{cor}
Suppose moreover that $L$ has rank at most $9$. Then the moduli space
$\mathcal{M}_{L}$ is irreducible.
\end{cor}

The following table gives the discriminant groups of ADE lattices:

\begin{table}[ht]
  \centering
{ \setlength\extrarowheight{2pt}
\begin{tabular}{|c|c|c|c|c|c|c|}
\hline 
$L$ & ${\bf A}_{n}\,(n\geq1)$ & ${\bf D}_{2n}\,(n\geq2)$ & ${\bf D}_{2n+1}\,(n\geq2)$ & ${\bf E}_{6}$ & ${\bf E}_{7}$ & ${\bf E}_{8}$\tabularnewline
\hline 
$L^{*}/L$ & $\ZZ/(n+1)\ZZ$ & $(\ZZ/2\ZZ)^{2}$ & $\ZZ/4\ZZ$ & $\ZZ/3\ZZ$ & $\ZZ/2\ZZ$ & $\{0\}$\tabularnewline
\hline 
\end{tabular}
}
\end{table}

There are $28$ lattices $L$ of rank at least $10$ such that a K3 surface
$X$ with $\NS X)\simeq L$ has finite automorphism group. For $L$
among these lattices with $L\neq U\oplus{\bf E}_{8}^{\oplus2}\oplus{\bf A}_{1}$,
using the above table and Corollary~\ref{cor:IrreductibleM(L)}, one
obtains that the moduli $\mathcal{M}_{L}$ is irreducible. For $L=U\oplus{\bf E}{}_{8}^{\oplus2}\oplus{\bf A}_{1}$,
the embedding in $U^{\oplus3}\oplus{\bf E}_{8}^{\oplus2}$ is unique
up to automorphisms, and one sees that its orthogonal complement contains
a copy of~$U$; therefore, $\mathcal{M}_{L}$ is also irreducible.
We have therefore proved Proposition \ref{prop:The-118-moduli-irred}.

\section{Rank 3 lattices}

\subsection{Rank 3 cases}

In \cite{Nikulin2}, Nikulin classifies rank $3$ lattices that are
N\'eron--Severi lattices of K3 surfaces with a finite number of automorphisms.
Let us describe the classification when the fundamental domain of
the associated Weyl group is not compact (see \cite[Section 2]{Nikulin2}). 

Let $S_{1,1,1}$ be the rank $3$ lattice generated by vectors $a$, $b$, $c$
with intersection matrix 
\[
\left(\begin{array}{ccc}
-2 & 0 & 1\\
0 & -2 & 2\\
1 & 2 & -2
\end{array}\right).
\]
For $r$, $s$, $t\in\ZZ$, let $S_{r,s,t}$ denote the sublattice of $S_{1,1,1}$
generated by $ra$, $sb$, $tc$. Let also define $S_{4,1,2}'$, the lattice
generated by $2a+c$, $b$, $2c$, and $S_{6,1,2}'$, the lattice generated
by $6a+c$, $b$, $2c$. In \cite[Theorem 2.5]{Nikulin2}, Nikulin gives a
list of lattices which are N\'eron--Severi lattices of K3 surfaces with
finite automorphism group. These lattices are of the form $S_{r,s,t}$
or $S_{r,s,t}'$. In \cite{NikulinElliptic}, Nikulin observes that
some lattices in the list are isometric, thus giving the same families
of K3 surfaces. One has 
\[
S_{2,1,2}\simeq S_{4,1,1},\quad S_{4,1,2}\simeq S_{8,1,1},\quad S_{6,1,2}\simeq S_{12,1,1},\quad S_{6,1,2}'\simeq S_{6,1,1},
\]
and when the fundamental domain is not compact, there are only $20$
distinct cases.  In the compact case, there are $6$ lattices $S_{1},\dots,S_{6}$,
see \cite{Nikulin2}, and their geometric description is given in
\cite{Roulleau} (see also \cite{ACL} for their Cox ring). So the
total number of (isomorphism classes of) lattices of rank $3$ which
are N\'eron--Severi lattices of K3 surfaces with finite automorphism
group is $26$. For each case, Nikulin gives the number and sometimes
the configurations of the $\cu$-curves. 

\subsection{The rank 3 and compact cases}

In \cite{Roulleau}, we studied the six lattices $S_{1},\dots,S_{6}$
of rank $3$ such that the fundamental domain associated to the Weyl
group is compact. These lattices are 
\[
\begin{array}{lll}
S_{1}=[6]\oplus\mathbf{A}_{1}^{\oplus2},\quad &S_{2}=[36]\oplus\mathbf{A}_{2},\quad & S_{3}=[12]\oplus\mathbf{A}_{2},\\
S_{4}\subset[60]\oplus\mathbf{A}_{2},\quad &S_{5}=[4]\oplus\mathbf{A}_{2},\quad& S_{6}\subset[132]\oplus\mathbf{A}_{2},
\end{array}
\]
where the two inclusions have index $3$. For completeness, let us
recall the obtained results.

\begin{thm}
\label{thm:Main1} The K3 surfaces of type $S_{1}$, $S_{4}$, $S_{5}$, $S_{6}$
are double covers of the plane branched over a smooth sextic curve
$C_{6}$ and such that 
\begin{itemize}
\item if the N\'eron--Severi lattice is isometric to $S_{1}$, the six
  $(-2)$-curves on $X$ are pull-backs of three conics that are
  $6$-tangent to the sextic $C_{6}$;
\item if the N\'eron--Severi lattice is isometric to $S_{4}$, the four
  $(-2)$-curves on $X$ are pull-backs of a line tritangent to $C_{6}$ and a
  conic $6$-tangent to $C_{6}$;
\item if the N\'eron--Severi lattice is isometric to $S_{5}$, the four
  $(-2)$-curves on $X$ are pull-backs of two lines tritangent to
  $C_{6}$; 
\item if the N\'eron--Severi lattice is isometric to $S_{6}$, the six
  $(-2)$-curves on $X$ are pull-backs of one $6$-tangent conic and two
  cuspidal cubics that cut $C_{6}$ tangentially and at their cusps.
\end{itemize}
Let $X$ be a K3 surface that has a N\'eron--Severi lattice isometric
to $S_{2}$. There are three quadrics in $\PP^{3}$ such that each
intersection with $X\hookrightarrow\PP^{3}$ is the union of two smooth
degree $4$ rational curves. These six rational curves are the only
$(-2)$-curves on $X$.

Let $X$ be a K3 surface that has a N\'eron--Severi lattice isometric
to $S_{3}$. There exist two hyperplanes sections such that each hyperplane
section is a union of two smooth conics. These four conics are the
only $(-2)$-curves on $X$.
\end{thm}

The cases $S_{3}$ and $S_{4}$ are linked to the surfaces $S_{1,1,4}$
and $S_{1,1,3}$; see  Sections \ref{subsec:The-latticeS114}
and \ref{subsec:The-latticeS113} below, respectively.

We will only add the following result. 

\begin{prop}
A general K3 surface $X$ with a N\'eron--Severi lattice isometric to
$S_{2}$ or $S_{3}$ has trivial automorphism group. 
\end{prop}

\begin{proof}
The Hilbert bases of the nef cones of these K3 surfaces are described
in \cite{ACL}. Then as in Proposition \ref{prop:The-automorphism-groupTRIVIAL}
below, one can check that there is no hyperelliptic involution and
conclude that the automorphism group is trivial. 
\end{proof}

\subsection{The lattice $\boldsymbol{S_{1,1,1}}$}

Let $X$ be a K3 surface with  rank $3$ N\'eron--Severi lattice and
intersection form 
\[
\left(\begin{array}{ccc}
-2 & 0 & 1\\
0 & -2 & 2\\
1 & 2 & -2
\end{array}\right).
\]
The surface $X$ contains three $(-2)$-curves $A_{1}$, $A_{2}$, $A_{3}$,
with intersection matrix as above; their dual graph is 

\begin{center}
\begin{tikzpicture}[scale=1]

\draw [very thick] (1,0) -- (2,0);
\draw (0,0) -- (1,0);

\draw (0,0) node {$\bullet$};
\draw (1,0) node {$\bullet$};
\draw (2,0) node {$\bullet$};

\draw (0,0) node [above]{$A_{1}$};
\draw (1,0) node [above]{$A_{3}$};
\draw (2,0) node [above]{$A_{2}$};

\end{tikzpicture}
\end{center} 

The divisor 
\[
D_{22}=3A_{1}+6A_{2}+7A_{3}
\]
is ample, of square $22$, with $D_{22} \cdot A_{1}=D_{22} \cdot A_{3}=1$, $D_{22} \cdot A_{2}=2$.
The divisor 
\begin{equation}
D_{2}=A_{1}+2A_{2}+2A_{3}\label{eq:S111}
\end{equation}
is nef, of square $2$, with base points since $F=A_{2}+A_{3}$ is
a fiber of an elliptic fibration  and $D_{2}F=1$. It satisfies $D_{2} \cdot A_{1}=D_{2} \cdot A_{2}=0$,
$D_{2} \cdot A_{3}=1$. The divisor $D_{8}=2D_{2}$ is base-point free and
hyperelliptic. By Theorem \ref{thm:SaintDonat-2}, case~i), we have the following. 

\begin{prop}
The linear system $|D_{8}|$ defines a map $\varphi\colon X\to\mathbf{F}_{4}$
onto the Hirzebruch surface $\mathbf{F}_{4}$ such that the branch
locus of $\varphi$ is the disjoint union of the unique section $s$
such that $s^{\ensuremath{2}}=-4$ and a reduced curve $B'$ in the
linear system $|3s+12f|$. 
\end{prop}

By equation \eqref{eq:S111}, the image of $A_{1}$ is the curve $s$, and 
the image of $A_{3}$ is the fiber through the point $q$ onto which
$A_{2}$ is contracted. The curve $B'$ has a unique singularity,
which is a node at $q$; its geometric genus is therefore $9$. 

\subsection{The lattice $\boldsymbol{S_{1,1,2}}$}

Let $X$ be a K3 surface with  rank $3$ N\'eron--Severi lattice of
type $S_{1,1,2}$. The surface $X$ contains three $(-2)$-curves $A_{1}$, $A_{2}$ , $A_{3}$,
with dual graph 

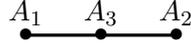
\begin{center}
\begin{tikzpicture}[scale=1]

\draw [very thick] (0,0) -- (2,0);

\draw (0,0) node {$\bullet$};
\draw (1,0) node {$\bullet$};
\draw (2,0) node {$\bullet$};

\draw (0,0) node [above]{$A_{1}$};
\draw (1,0) node [above]{$A_{3}$};
\draw (2,0) node [above]{$A_{2}$};

\end{tikzpicture}
\end{center} 

The divisor 
\[
D_{14}=2A_{1}+2A_{2}+3A_{3}
\]
is ample, of square $14$, with $D_{14} \cdot A_{1}=D_{14} \cdot A_{2}=2$ and $D_{14} \cdot A_{3}=4$.
The divisor 
\[
D_{2}=A_{1}+A_{2}+A_{3}
\]
is nef, of square $2$, base-point free, with $D_{2} \cdot A_{1}=D_{2} \cdot A_{2}=0$
and $D_{2} \cdot A_{3}=2$. Thus, the following holds.  

\begin{prop}
The K3 surface is the double cover of $\,\PP^{2}$ branched over a sextic
curve with two nodes $p$, $q$. The curves $A_{1}$, $A_{2}$ are contracted
to $p$, $q$, and the image of $A_{3}$ is the line through $p$, $q$. That
line cuts the sextic curve transversally in two other points. 
\end{prop}

The Severi variety of plane curves of degree 6 with two nodes is rational,
so the moduli space $\mathcal{M}_{S_{1,1,2}}$ of K3 surfaces with N\'eron--Severi
lattice isometric to $S_{1,1,2}$ is unirational.

\subsection{The lattice $\boldsymbol{S_{1,1,3}}$\label{subsec:The-latticeS113}}

Let $X$ be a K3 surface with N\'eron--Severi lattice of type $S_{1,1,3}$.
The surface $X$ contains four $(-2)$-curves $A_{1}$, $A_{2}$, $A_{3}$, $A_{4}$,
with intersection matrix 
\[
\left(\begin{array}{cccc}
-2 & 3 & 2 & 0\\
3 & -2 & 0 & 2\\
2 & 0 & -2 & 6\\
0 & 2 & 6 & -2
\end{array}\right).
\]
The curves $A_{1}$, $A_{2}$, $A_{3}$ generate the N\'eron--Severi lattice.
The divisor 
\[
D_{2}=A_{1}+A_{2}
\]
 is ample, of square $2$, base-point free, with $D_{2} \cdot A_{1}=D_{2} \cdot A_{2}=1$
and $D_{2} \cdot A_{3}=D_{2} \cdot A_{4}=2$. We have $2D_{2}\equiv A_{3}+A_{4}$;
thus we have obtained the first part of the following proposition. 

\begin{prop}
\label{prop:The-surface-case S113}The surface $X$ is the double
cover of $\,\PP^{2}$ branched over a smooth sextic $C_{6}$ which has
a tritangent line and a $6$-tangent conic. For general  $X$, the
sextic has equation 
\[
C_{6}\colon \ell qg-f^{2}=0,
\]
where $\ell$, $q$, $g$ and $f$ are forms of degree $1$, $2$, $3$ and $3$,
respectively. The line $\ell=0$ is the tritangent to the sextic, and
the curve $q=0$ is the $6$-tangent conic. The moduli space $\mathcal{M}_{S_{1,1,3}}$
of K3 surfaces $X$ with $\NS X)\simeq S_{1,1,3}$ is unirational.
\end{prop}

\begin{proof}
Let $C_{6}\colon \ell qg-f^{2}=0$ be a sextic curve as above, and let $Y\to\PP^{2}$
be the double cover branched over $C_{6}$. Above the line $\ell=0$
are two $\cu$-curves $A_{1}$, $A_{2}$ such that $A_{1} \cdot A_{2}=3$, and
above the conic $q=0$ are two $\cu$-curves $A_{3}$, $A_{4}$ such that
$A_{3} \cdot A_{4}=6$. It remains to understand the intersections $A_{j} \cdot A_{k}$
for $j\in\{1,2\}$ and $k\in\{3,4\}$. This is done by computing 
Example \ref{exa:S113} below, for which the intersection is (up to
permutation of $A_{1}$, $A_{2}$ and $A_{3}$, $A_{4}$) the above intersection
matrix of the four $\cu$-curves on a K3 surface $X$ with $\NS X)\simeq S_{1,1,3}$.
The intersection numbers remain the same for the $\cu$-curves in
that flat family of smooth surfaces. 

If the equation of $C_{6}$ is general , the Picard number of $Y$
is $3$. The lattice generated by the $\cu$-curves is $S_{1,1,3}$,
of discriminant $18$. The unique over-lattice containing $S_{1,1,3}$
is $S_{1,1,1}$, but a surface with N\'eron--Severi lattice isometric
to $S_{1,1,1}$ contains only three $\cu$-curves; thus $\NS Y)\simeq S_{1,1,3}$. 

In order to prove that a general K3 surface $X$ with $\NS X)\simeq S_{1,1,3}$
is branched over a sextic curve with an equation of the form $\ell qg-f^{2}=0$,
let us consider the map
\[
\Phi\colon H^{0}(\PP^{2},\OO(1))\oplus H^{0}(\PP^{2},\OO(2))\oplus H^{0}(\PP^{2},\OO(3))^{\oplus2}\to H^{0}(\PP^{2},\OO(6))
\]
defined by 
\[
w:=(\ell,q,g,f)\mapsto f_{6,w}:=\ell qg-f^{2}.
\]
It is invariant under the action of the transformations $\G\colon (\ell,q,g,f)\mapsto(\a\ell,\b q,\g g,f)$
for $\a\b\g=1$. Suppose 
\begin{gather*}
\ell qg-f^{2}=\ell'q'g'-f'^{2}
\end{gather*}
for $f'\neq\pm f$ and that the forms are chosen general  so that the double
cover $Y$ branched over the sextic $f_{6,w}=0$ has $\NS Y)\simeq S_{1,1,3}$.
Since the curves $\ell=0$, $q=0$ are the images of the four $\cu$-curves
on $Y$, up to rescaling by using a transformation $\G$, one can
suppose $\ell'=\ell$, $q'=q$, and then one obtains the relation 
\[
\ell q(g-g')=(f-f')(f+f').
\]
Up to changing the sign of $f'$, we can suppose that $\ell$ divides
$f+f'$. If $q$ does not divide $f+f'$, then one gets a codimension
$1$ family of such sextic curves; thus we can suppose that there
exists a scalar $\a$ such that $\ell q=\a(f+f')$, and by solving
the equations, we obtain that
\[
f'=\frac{1}{\a}\ell q-f,\quad g'=g+\frac{1}{\a^{2}}lq-\frac{2}{\a}f.
\]
The dimension of the (unirational) moduli space of sextic curves with an
equation of the form $\ell qg-f^{2}=0$  is
\[
(3+6+2\cdot10)-(9+2+1)=17;
\]
since $\mathcal{M}_{S_{1,1,3}}$ also has  dimension $17$, 
both spaces are birational.
\end{proof}

\begin{example}
\label{exa:S113}Let us take 
\[
\begin{array}{l}
\ell=13x+10y+z,\quad q=4x^{2}+6xy+26xz+6yz+23z^{2},\\
g=9x^{3}+16x^{2}y+5xy^{2}+24y^{3}+22x^{2}z+26xyz+4y^{2}z+23xz^{2}+20yz^{2}+9z^{3},\\
f=28x^{3}+8x^{2}y+2xy^{2}+23y^{3}+2x^{2}z+23xyz+19y^{2}z+24xz^{2}+20yz^{2}+17z^{3}.
\end{array}
\]
Let $X$ be the associated K3 surface, and let $X_{p}$ be its reduction
modulo a prime $p$. Using the Tate conjectures, one finds that the
K3 surfaces $X_{23}$ and $X_{29}$ have Picard number $4$. By the
Artin--Tate conjectures, one computes that 
\[
\begin{array}{l}
|\text{Br}(X_{23})|\cdot|\text{disc}(\NS X_{23}))| =3^{2}\cdot491,\\
|\text{Br}(X_{29})|\cdot|\text{disc}(\NS X_{29}))| =3^{3}\cdot5^{2}\cdot53,
\end{array}
\]
and using Van Luijk's
trick (see \cite{vL}), we conclude that the
Picard number of $X$ is $3$ since the ratio of the two integers above 
is not a square (here $\text{Br}$ is the Brauer group, and $\text{disc}$
denotes the discriminant group).
\end{example}

\begin{rem}
In \cite{Roulleau}, we study K3 surfaces with finite automorphism
group and compact fundamental domain. The surface
with N\'eron--Severi
lattice of type $S_{4}$ (Nikulin's notation) is also a double cover
of $\PP^{2}$ branched over a smooth sextic which has one tritangent
line and one $6$-tangent conic. But the intersection matrix of the
four $\cu$-curves above the line and the conic (the only $\cu$ curves
on that surface) is 
\[
\left(\begin{array}{cccc}
-2 & 3 & 1 & 1\\
3 & -2 & 1 & 1\\
1 & 1 & -2 & 6\\
1 & 1 & 6 & -2
\end{array}\right).
\]
\end{rem}

\begin{rem}
By Section \ref{subsec:On-the-irreducibility}, we know that the unique
moduli space $\mathcal{M}_{S_{1,1,3}}$ of $S_{1,1,3}$-polarized K3 surfaces
is irreducible. By Proposition \ref{prop:The-surface-case S113},
there is a copy of $\mathcal{M}_{S_{1,1,3}}$ in the moduli space $\mathcal{M}_{[2]}$
of K3 surfaces with an ample divisor of square $2$. Using the Hilbert
basis of the nef cone, one can check that $D_{2}$ is the unique nef
divisor of square $2$ in the N\'eron--Severi group; therefore, the copy
of $\mathcal{M}_{S_{1,1,3}}$ in $\mathcal{M}_{[2]}$ is unique. 
\end{rem}

\subsection{The lattice $\boldsymbol{S_{1,1,4}}$\label{subsec:The-latticeS114}}

Let $X$ be a K3 surface with N\'eron--Severi lattice of type $S_{1,1,4}$.
The surface $X$ contains four $(-2)$-curves $A_{1}$, $A_{2}$, $A_{3}$, $A_{4}$,
with intersection matrix 
\[
M_{1}=\left(\begin{array}{cccc}
-2 & 4 & 0 & 2\\
4 & -2 & 2 & 0\\
0 & 2 & -2 & 4\\
2 & 0 & 4 & -2
\end{array}\right).
\]
The curves $A_{1}$, $A_{2}$, $A_{3}$ generate the N\'eron--Severi lattice.
The divisors $F_{1}=A_{1}+A_{4}$ and $F_{2}=A_{2}+A_{3}$ are fibers
of two distinct elliptic fibrations. The divisor 
\[
D_{4}=A_{1}+A_{2}\equiv A_{3}+A_{4}
\]
is ample, of square $4$, with $D_{2} \cdot A_{j}=2$ for $j\in\{1,\dots,4\}$
and is non-hyperelliptic.

\begin{prop}
The surface $X$ is a quartic in $\PP^{3}$ which has two hyperplane
sections which are unions of two conics. The general  surface has a
projective model of the form
\[
X\colon \ell_{1}\ell_{2}q_{3}-q_{1}q_{2}=0\hookrightarrow\PP^{3},
\]
where $\ell_{1}$, $\ell_{2}$ are linear forms and $q_{1}$, $q_{2}$, $q_{3}$
are quadrics. The moduli space $\mathcal{M}_{S_{1,1,4}}$ of K3 surfaces
$X$ with $\NS X)\simeq S_{1,1,4}$ is unirational.
\end{prop}

\begin{proof}
Consider the map
\[
\Phi\colon H^{0}(\PP^{3},\OO(1))^{\oplus2}\oplus H^{0}(\PP^{3},\OO(2))^{\oplus3}\to H^{0}(\PP^{3},\OO(4))
\]
defined by 
\[
w:=(\ell_{1},\ell_{2},q_{1},q_{2},q_{3})\mapsto Q_{4,w}:=\ell_{1}\ell_{2}q_{1}-q_{2}q_{3}.
\]
By computing the differential $d\Phi_{w}$ at a randomly chosen point,
we find that it has rank $33$; thus the image of~$\Phi$ is
$33$-dimensional, and the quotient $W$ of that image by $GL_{4}(\CC)$
is at least $17$-dimensional. The space $W$ is unirational. For
a  general  $w$, the quartic $Y\colon Q_{4,w}=0$ is non-singular. The
curves $A_{1}\colon \ell_{1}=q_{2}=0$, $A_{2}\colon \ell_{1}=q_{3}=0$, $A_{3}\colon \ell_{2}=q_{2}=0$,
$A_{4}\colon \ell_{2}=q_{3}=0$ are $\cu$-curves on $Y$ with (up to
permutation of $A_{3}$ and $A_{4}$) intersection matrix $M_{1}$.
We thus obtain an injective map from $W$ to the moduli space $\mathcal{M}_{S_{1,1,4}}$.
Since $\mathcal{M}_{S_{1,1,4}}$ has dimension $17$, this moduli
space is unirational.
\end{proof}

\begin{rem}

a)~ One can construct, more geometrically, a member of $\mathcal{M}_{S_{1,1,4}}$
by considering two disjoint conics $C_{1}$ and $C_{3}$ in $\PP^{3}$
and taking a  general  quartic among the $16$-dimensional linear system
of quartics containing them. We also obtain in that way a dominant
map from a rational space to $\mathcal{M}_{S_{1,1,4}}$.

b)~ In \cite{Roulleau}, we study K3 surfaces with finite automorphism
group and compact fundamental domain. One of these surfaces, namely
the surface with N\'eron--Severi lattice of type $S_{3}$ (Nikulin's notation)
is also a quartic surface in $\PP^{3}$ with two hyperplane sections
which are the union of two conics. But the intersection matrix of the
four $\cu$-curves (which are the only $\cu$ curves on that surface)
is 
\[
\left(\begin{array}{cccc}
-2 & 4 & 1 & 1\\
4 & -2 & 1 & 1\\
1 & 1 & -2 & 4\\
1 & 1 & 4 & -2
\end{array}\right).
\]
By using a construction as in part~a), one finds that the moduli space
of these surfaces is also unirational. 
\end{rem}

\begin{prop}
\label{prop:The-automorphism-groupTRIVIAL}The automorphism group
of a general K3 surface $X$ with $\NS X)\simeq S_{1,1,4}$ is trivial.
\end{prop}

\begin{proof}
The walls of the nef cone are $A_{k}^{\perp}$, $k\in\{1,\dots,4\}$.
A Hilbert basis (see Definition \ref{def:HilbertBase}) of the nef
cone is given in \cite{ACL}; it is as follows. The divisors $H_{1}=A_{1}+A_{4}$
and $H_{2}=A_{2}+A_{3}$ are fibers of the two distinct elliptic fibrations
on the K3 surface $X$; the seven remaining classes $H_{3},\dots,H_{9}$
in the Hilbert basis are (in the basis $A_{1}$, $A_{2}$, $A_{3}$) 
\[
(1,1,0),(1,1,1),(1,2,0),(2,1,0),(2,1,1),(2,2,-1),(2,3,-1).
\]
 The intersection matrix $M_{H}=(H_{i}\cdot H_{j})_{1\leq i,j\leq9}$ is
\[
M_{H}=\left(\begin{array}{ccccccccc}
0 & 8 & 4 & 8 & 8 & 4 & 8 & 4 & 8\\
8 & 0 & 4 & 4 & 4 & 8 & 8 & 8 & 8\\
4 & 4 & 4 & 6 & 6 & 6 & 8 & 6 & 8\\
8 & 4 & 6 & 6 & 10 & 8 & 8 & 12 & 16\\
8 & 4 & 6 & 10 & 6 & 12 & 16 & 8 & 8\\
4 & 8 & 6 & 8 & 12 & 6 & 8 & 10 & 16\\
8 & 8 & 8 & 8 & 16 & 8 & 8 & 16 & 24\\
4 & 8 & 6 & 12 & 8 & 10 & 16 & 6 & 8\\
8 & 8 & 8 & 16 & 8 & 16 & 24 & 8 & 8
\end{array}\right).
\]
Let us check the condition of Proposition \ref{prop:AUTOTrivial}:

\begin{enumerate}[(i)]
\item Since the intersection matrix of the fibers $H_{1}$, $H_{2}$ with
$A_{1},\dots,A_{4}$ is
\[
\left(\begin{array}{cccc}
0 & 4 & 4 & 0\\
4 & 0 & 0 & 4
\end{array}\right),
\]
there are no $(-2)$-curves $A$ and fibers $F$ such that $AF=1$.

\item A nef divisor $D$ is a positive linear combination of the elements
$H_{1},\dots,H_{9}$ of the Hilbert basis of the nef cone. From the
two first lines of the Gram matrix $M_{H}$, we see that there is 
no big and nef divisor $D$ such that $D\cdot F=2$ for any fiber $F=H_{1}$
or $H_{2}$.

\item From the matrix $M_{H}$, we see that there is no big and nef
  divisor $D$ such that $D^{2}=2$.
 \end{enumerate}
Therefore, one can apply Proposition \ref{prop:AUTOTrivial} and conclude
that the automorphism group of $X$ is trivial. 
\end{proof}

\begin{defn}
\label{def:HilbertBase}Let us recall that if $C\subset\RR^{d}$ is
a polyhedral cone generated by integral vectors, a Hilbert basis $H(C)$
of $C$ is a subset of integral vectors such that
\begin{enumerate}[a)]
  \item\label{def:HilbertBase-a} each element of $C\cap\ZZ^{d}$ can be written as a non-negative
integer combination of elements of $H(C)$, and
\item\label{def:HilbertBase-b}  $H(C)$ has minimal cardinality with respect to all subsets of
  $C\cap\ZZ^{d}$ for which part~\ref{def:HilbertBase-a} holds.
  \end{enumerate}
\end{defn}

\subsection{The lattice $\boldsymbol{S_{1,1,6}}$}

Let $X$ be a K3 surface with N\'eron--Severi lattice of type $S_{1,1,6}$.
The surface $X$ contains six $(-2)$-curves $A_{1},\dots,A_{6}$,
with intersection matrix 
\[
\left(\begin{array}{cccccc}
-2 & 6 & 2 & 6 & 6 & 2\\
6 & -2 & 6 & 2 & 2 & 6\\
2 & 6 & -2 & 18 & 0 & 16\\
6 & 2 & 18 & -2 & 16 & 0\\
6 & 2 & 0 & 16 & -2 & 18\\
2 & 6 & 16 & 0 & 18 & -2
\end{array}\right).
\]
The curves $A_{1}$, $A_{3}$, $A_{5}$ generate the N\'eron--Severi lattice.
We have 
\[
A_{2}=A_{1}-2A_{3}+2A_{5},\quad A_{4}=4A_{1}-5A_{3}+4A_{5},\quad A_{6}=4A_{1}-4A_{3}+3A_{5}.
\]
 The divisor 
\[
D_{2}=A_{1}-A_{3}+A_{5}
\]
is ample, of square $2$, with $D_{2}\cdot A_{1}=D_{2}\cdot A_{2}=2$ and $D_{2}\cdot A_{j}=4$
for $j\in\{3,\dots,6\}$. We have 
\[
2D_{2}\equiv A_{1}+A_{2},\quad 4D_{2}\equiv A_{3}+A_{4}\equiv A_{5}+A_{6}. 
\]
Therefore, the following holds. 

\begin{prop}
The K3 surface $X$ is the double cover $\eta\colon X\to\PP^{2}$ of $\PP^{2}$
branched over a smooth sextic curve which has a $6$-tangent conic
$C$, and there are two rational cuspidal quartics $Q_{1},Q_{2}$ such
that their three cusps are on the sextic curve and their remaining
intersection points are tangent to the sextic $($there are nine such
points$)$. The image by $\eta$ of $A_{1}+A_{2}$ is $C$, the image
of $A_{3}+A_{4}$ is $Q_{1}$, and the image of $A_{5}+A_{6}$ is $Q_{2}$.
\end{prop}

\subsection{The lattice $\boldsymbol{S_{1,1,8}}$}

Let $X$ be a K3 surface with N\'eron--Severi lattice of type $S_{1,1,8}$.
The surface $X$ contains eight $(-2)$-curves $A_{1},\dots,A_{8}$,
with intersection matrix 
\[
\begin{pmatrix}-2 & 2 & 0 & 8 & 16 & 8 & 18 & 14\\
2 & -2 & 8 & 0 & 8 & 16 & 14 & 18\\
0 & 8 & -2 & 14 & 18 & 2 & 16 & 8\\
8 & 0 & 14 & -2 & 2 & 18 & 8 & 16\\
16 & 8 & 18 & 2 & -2 & 14 & 0 & 8\\
8 & 16 & 2 & 18 & 14 & -2 & 8 & 0\\
18 & 14 & 16 & 8 & 0 & 8 & -2 & 2\\
14 & 18 & 8 & 16 & 8 & 0 & 2 & -2
\end{pmatrix}.
\]
The curves $A_{1}$, $A_{2}$, $A_{3}$ generate the N\'eron--Severi lattice.
We have 
\[
\begin{array}{lll}
A_{4}=-2A_{1}+2A_{2}+A_{3}, \quad & A_{5}=-5A_{1}+3A_{2}+3A_{3},\quad & A_{6}=-3A_{1}+A_{2}+3A_{3},\\
A_{7}=-6A_{1}+3A_{2}+4A_{3}, \quad & A_{8}=-5A_{1}+2A_{2}+4A_{3}.&
\end{array}
\]
The divisor 
\[
D_{6}=-A_{1}+A_{2}+A_{3}
\]
is ample, of square $6$, base-point free and non-hyperelliptic, with
$D_{6}\cdot A_{1}=D_{6}\cdot A_{2}=4$, $D_{6}\cdot A_{3}=D_{6}\cdot A_{4}=6$, $D_{6}\cdot A_{5}=D_{6}\cdot A_{6}=10$,
$D_{6}\cdot A_{7}=D_{6}\cdot A_{8}=12$. The K3 surface $X$ is a smooth complete
intersection in $\PP^{4}$. We remark that 
\[
A_{3}+A_{4}=2D_{6}.
\]

\begin{prop}
The automorphism group of a general K3 surface $X$ with $\NS X)\simeq S_{1,1,4}$
is trivial.
\end{prop}

\begin{proof}
We proceed as in the proof of Proposition \ref{prop:The-automorphism-groupTRIVIAL}.
The Hilbert basis of the nef cone of the K3 surface $X$ is a set
of $60$ classes (see \cite[Table 1]{ACL}), among which are the fibers
\[
A_{2}+A_{3},\quad 4A_{1}+5A_{2}-7A_{3},\quad 4A_{1}+A_{2}-3A_{3},\quad 8A_{1}+5A_{2}-11A_{3}.
\]
Then one can check that any big and nef divisor $D$ on $X$ is base-point
free, with $D^{2}\geq6$ or $D^{2}=0$. Moreover, there is no fiber 
$F$ of an elliptic fibration such that $FD=2$; thus by Theorems
\ref{thm:SaintDonat-1} and \ref{thm:SaintDonat-2}, there are no
hyperelliptic involutions acting on $X$. From \cite[Table 1 and Lemma 2.3]{Kondo},
the automorphism group of $X$ is therefore trivial. 
\end{proof}

\subsection{The lattice $\boldsymbol{S_{1,2,1}}$}

Let $X$ be a K3 surface with N\'eron--Severi lattice of type $S_{1,2,1}$.
The surface $X$ contains three $(-2)$-curves $A_{1},\dots,A_{3}$,
with dual graph 

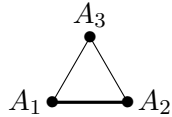
\begin{center}
\begin{tikzpicture}[scale=1]

\draw [very thick] (0,0) -- (1,0);
\draw (0,0) -- (0.5,0.866);
\draw (1,0) -- (0.5,0.866);

\draw (0,0) node {$\bullet$};
\draw (1,0) node {$\bullet$};
\draw (0.5,0.866) node {$\bullet$};

\draw (0,0) node [left]{$A_{1}$};
\draw (1,0) node [right]{$A_{2}$};
\draw (0.5,0.866) node [above]{$A_{3}$};

\end{tikzpicture}
\end{center} 

These curves generate the N\'eron--Severi lattice. The divisor 
\[
D_{6}=2A_{1}+2A_{2}+A_{3}
\]
 is ample, of square $6$, with $D_{6}\cdot A_{1}=D_{6}\cdot A_{2}=1$ and $D_{6}\cdot A_{3}=2$.
Since $D_{6}\cdot (A_{1}+A_{2})=2$ and $A_{1}+A_{2}$ is a fiber, $D_{6}$
is hyperelliptic. The divisor 
\[
D_{2}=A_{1}+A_{2}+A_{3}
\]
is nef, of square $2$, base-point free, with $D_{2} \cdot A_{1}=D_{2} \cdot A_{2}=1$
and $D_{2} \cdot A_{3}=0$. 

\begin{prop}
The K3 surface $X$ is a double cover of $\,\PP^{2}$ branched over
a sextic curve $C_{6}$ which has a node at a point $q$ and a line
$L$ containing the node, which is bitangent to $C_{6}$ at two points
$p_{1}$, $p_{2}$. The moduli space $\mathcal{M}_{1,2,1}$ is unirational.
\end{prop}

\begin{proof}
It remains to prove the assertion on the moduli space $\mathcal{M}_{1,2,1}$.
We identify this space with the moduli space of sextic curves with
a node and a line through the node which is bitangent to the sextic.
The imposition of a node and the tangency conditions  are linear conditions
on the space $H^{0}(\PP^{2},\OO_{\PP^{2}}(6))$ of sextic curves; 
thus the moduli space of such curves is unirational. 
\end{proof}

\subsection{The lattice $\boldsymbol{S_{1,3,1}}$}

Let $X$ be a K3 surface with N\'eron--Severi lattice of type $S_{1,3,1}$.
The surface $X$ contains three $(-2)$-curves $A_{1}$, $A_{2}$, $A_{3}$,
with dual graph 

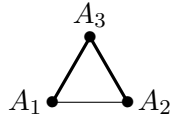
\begin{center}
\begin{tikzpicture}[scale=1]

\draw (0,0) -- (1,0);
\draw [very thick] (0,0) -- (0.5,0.866);
\draw [very thick] (1,0) -- (0.5,0.866);

\draw (0,0) node {$\bullet$};
\draw (1,0) node {$\bullet$};
\draw (0.5,0.866) node {$\bullet$};

\draw (0,0) node [left]{$A_{1}$};
\draw (1,0) node [right]{$A_{2}$};
\draw (0.5,0.866) node [above]{$A_{3}$};

\end{tikzpicture}
\end{center} 

These curves generate the N\'eron--Severi lattice. The divisor 
\[
D_{4}=A_{1}+A_{2}+A_{3}
\]
 is very ample of square $4$, with $D_{4}\cdot A_{1}=D_{4}\cdot A_{2}=1$ and
$D_{4}\cdot A_{3}=2$. 

\begin{prop}
The K3 surface $X$ is a quartic surface in $\PP^{3}$ with a hyperplane
section which is the union of two lines and a conic. The general  surface
$X$ with $\NS X)\simeq S_{1,3,1}$ has an equation of the form
\[
\ell_{1}f+q_{1}\ell_{2}\ell_{3}=0.
\]
The moduli space $\mathcal{M}_{S_{1,3,1}}$ of such surfaces is unirational. 
\end{prop}

\begin{proof}
Let $Y\colon \ell_{1}f+q_{1}\ell_{2}\ell_{3}=0$ be a quartic surface as
in the proposition. The hyperplane section $\ell_{1}$ cuts $Y$ into
a union of two lines $A_{1}\colon \ell_{1}=\ell_{2}=0$, $A_{2}\colon \ell_{1}=\ell_{3}=0$
and a conic $A_{3}\colon \ell_{1}=q_{1}=0$. Their intersection graph
is as above. 

If the equation of $Y$ is general , the Picard number of $Y$ is $3$; 
the lattice generated by the $\cu$-curves is $S_{1,3,1}$, of discriminant
$18$. The unique over-lattice containing $S_{1,3,1}$ is $S_{1,1,1}$,
but the configuration of the $\cu$-curves in a surface with N\'eron--Severi
lattice isometric to $S_{1,1,1}$ is different; thus $\NS Y)\simeq S_{1,3,1}$. 

Let us consider the map
\[
\Phi\colon H^{0}(\PP^{3},\OO(1))^{\oplus3}\oplus H^{0}(\PP^{3},\OO(2))\oplus H^{0}(\PP^{3},\OO(3))\to H^{0}(\PP^{3},\OO(4))
\]
defined by 
\[
w:=(\ell_{1},\ell_{2},\ell_{3},q_{1},f)\mapsto Q_{4,w}:=\ell_{1}f+q_{1}\ell_{2}\ell_{3}.
\]
It is invariant under the action of the transformations $\G\colon (\ell_{1},\ell_{2},\ell_{3},q_{1},f)\mapsto(\nu\ell_{1},\a\ell_{2},\b\ell_{3},\g q_{1},\mu f)$
for $\nu\mu=1$ and $\a\b\g=1$. Suppose 
\begin{equation}
\ell_{1}f+q_{1}\ell_{2}\ell_{3}=\ell_{1}'f'+q_{1}'\ell_{2}'\ell_{3}'\label{eq:S131}
\end{equation}
 and that the forms are chosen general  so that the associated quartic satisfies
$\NS Y)\simeq S_{1,1,3}$. Since the lines $A_{1}\colon \ell_{1}=\ell_{2}=0$,
$A_{2}\colon \ell_{1}=\ell_{3}=0$ and the conic $A_{3}\colon \ell_{1}=q_{1}=0$
are the only $\cu$-curves on $X$, we have (up to exchanging $\ell_{2}'$, $\ell_{3}'$)
\[
A_{1}\colon \ell_{1}'=\ell_{2}'=0,\quad A_{2}\colon \ell_{1}'=\ell_{3}'=0,\quad A_{3}\colon \ell_{1}'=q_{1}'=0.
\]
The consequence is that---up to rescaling using a transformation $\G$---one can suppose that there exist scalars $\a_{1}$, $\a_{2}$, $\a_{3}$ and
a linear form $\ell$ such that 
\begin{equation}
\ell_{1}'=\ell_{1},\quad \ell_{2}'=\a_{2}\ell_{1}+\ell_{2},\quad\ell_{3}'=\a_{3}\ell_{1}+\ell_{2},\quad q_{1}'=\a_{1}q_{1}+\ell_{1}\ell.\label{eq:S131No2}
\end{equation}
Substituting this into the formula \eqref{eq:S131} and taking the equation
modulo $\ell_{1}$, one finds that necessarily $\a_{1}=1$. Conversely,
given forms as in equation \eqref{eq:S131No2}, with $\a_{1}=1$,
one can find a cubic form $f'$ such that the linear forms $\ell_{1}'$, $\ell_{2}'$, $\ell_{3}'$,
the quadric $q_{1}'$ and the cubic $f'$ satisfy equation \eqref{eq:S131}.
Therefore the dimension of the (unirational) moduli space $\mathcal{M}_{4}$ of quartic
surfaces with an equation of type $\ell_{1}f+q_{1}\ell_{2}\ell_{3}=0$ is  
\[
(3\cdot4+10+20)-(16+1+2+6)=17.
\]
An open set of $\mathcal{M}_{4}$ is a subspace of the
moduli space $\mathcal{M}_{S_{1,3,1}}$, which  also has dimension
$17$; thus both spaces are birational. 
\end{proof}

\begin{prop}
The automorphism group of a general K3 surface $X$ with $\NS X)\simeq S_{1,3,1}$
is trivial.
\end{prop}

\begin{proof}
We proceed as in the proof of Proposition \ref{prop:The-automorphism-groupTRIVIAL}.
The Hilbert basis of the nef cone of the K3 surface $X$ is a set
of four classes:
\[
A_{2}+A_{3},\quad A_{1}+A_{3},\quad A_{1}+A_{2}+A_{3},\quad 2A_{1}+2A_{2}+A_{1}; 
\]
the first two classes are fibers of elliptic fibrations. The intersection
matrix of these four classes is 
\[
\left(\begin{array}{cccc}
0 & 3 & 3 & 6\\
3 & 0 & 3 & 6\\
3 & 3 & 4 & 6\\
6 & 6 & 6 & 4
\end{array}\right). 
\]
From that one can check that any big and nef divisor $D$ on $X$
is base-point free, with $D^{2}\geq4$ or $D^{2}=0$, and there is
no fiber $F$ of an elliptic fibration such that $FD=2$. Thus we
conclude as in Proposition \ref{prop:The-automorphism-groupTRIVIAL}
that the automorphism group of $X$ is trivial. 
\end{proof}

\subsection{The lattice $\boldsymbol{S_{1,4,1}}$}

Let $X$ be a K3 surface with N\'eron--Severi lattice of type $S_{1,4,1}$.
The surface $X$ contains four $(-2)$-curves $A_{1}$, $A_{2}$, $A_{3}$, $A_{4}$,
with intersection matrix
\[
\left(\begin{array}{cccc}
-2 & 3 & 2 & 1\\
3 & -2 & 1 & 2\\
2 & 1 & -2 & 11\\
1 & 2 & 11 & -2
\end{array}\right).
\]
The curves $A_{1}$, $A_{2}$, $A_{3}$ generate the N\'eron--Severi lattice.
The divisor 
\[
D_{2}=A_{1}+A_{2}
\]
is ample, of square $2$. We have 
\[
A_{3}+A_{4}\equiv3D_{2}; 
\]
therefore, using the linear system $|D_{2}|$, we obtain the following. 

\begin{prop}
The K3 surface $X$ is the double cover $\eta\colon X\to\PP^{2}$ of $\PP^{2}$
branched over a smooth sextic curve $C_{6}$ which has a tritangent
line and a cuspidal cubic with a cusp on $C_{6}$ and such that the
cuspidal cubic and $C_{6}$ are tangent at every other intersection
points. 
\end{prop}

The divisor $D_{6}=A_{1}+A_{2}+A_{3}$ is very ample of square $6$,
with $D_{6} \cdot A_{j}=3,2,1,14$ for $j=1,\dots,4$, so that $A_{1}$, $A_{2}$, $A_{3}$
are, respectively, a rational cubic, a conic and a line in a hyperplane
of $\PP^{4}$. That leads us to the following proposition. 

\begin{prop}
The moduli space $\mathcal{M}_{S_{1,4,1}}$ of K3 surfaces with $\NS X)\simeq S_{1,4,1}$
is unirational. 
\end{prop}

\begin{proof}
We can construct a member of $\mathcal{M}_{S_{1,4,1}}$ as follows:
Let $H_{3}\subset\PP^{4}$ be a hyperplane, let $B_{3}\subset H_{3}$
be a line, let $p_{0}$, $p_{1}$, $p_{3}$ be three point on $B_{3}$, let
$H_{2}\subset H_{3}$ be a plane intersecting $B_{3}$ at $p_{0}$,
let $B_{2}\subset H_{2}$ be a smooth conic containing $p_{0}$, and let
$q_{1}$, $q_{2}$, $q_{3}$ be three  general  points on $B_{2}$ (thus different
from $p_{0}$). The moduli space of normal cubic curves passing through
the points $p_{1}$, $p_{2}$, $q_{1}$, $q_{2}$, $q_{3}$ is unirational. Let $B_{1}\subset H_{3}$
be such a rational normal cubic curve. By construction, for $j\neq k$
in $\{1,2,3\}$, the degree of the intersection scheme of $B_{j}$
and $B_{k}$ equals $A_{j} \cdot A_{k}.$

There is a unique quadric in $H_{3}$ that contains the curves $B_{1}$, $B_{2}$, $B_{3}$.  The linear system of cubics in $H_{3}$ that contain these curves is $3$-dimensional. Therefore, the linear systems $L_{2}$, $L_{3}$ of quadrics and cubics in $\PP^{4}$ containing the three curves $B_{1}$, $B_{2}$, $B_{3}$ are, respectively, $4$- and $18$-dimensional.

Let $X\hookrightarrow\PP^{4}$ be the degree $6$ K3 surface which
is the complete intersection of a  general  quadric in $L_{2}$ and a  general 
cubic in $L_{3}$. It has Picard number at least $3$, and there is an open
subspace for which the Picard number is $3$ since K3 surfaces with N\'eron--Severi
lattice $S_{1,4,1}$ belong to that space. Thus for a  general  choice,
the N\'eron--Severi lattice of $X$ has rank $3$ and contains the lattice
$S_{1,4,1}$ generated by the $\cu$-curves $B_{1}$, $B_{2}$, $B_{3}$.
The only over-lattices of $S_{1,4,1}$ are $S_{1,1,1}$ and $S_{1,2,1}$.
The K3 surfaces with such a N\'eron--Severi lattice contain only three
$\cu$-curves which do not have the same intersection matrix as the curves
$B_{k}$, $k=1,\dots,3$. Thus $\NS X)\simeq S_{1,4,1}$, and by 
the above construction, the moduli space $\mathcal{M}_{S_{1,4,1}}$
is unirational. 
\end{proof}

\begin{rem} \label{rem:Rat-Norm-Curves}
  
a)~ We recall that a rational normal curve
is a smooth, rational curve of degree $n$ in projective $n$-space
$\PP^{n}$. Given $n+3$ points in $\PP^{n}$ in linear general position
(that is, with no $n+1$ lying in a hyperplane), there is a unique
rational normal curve $C$ passing through them. From, \textit{e.g.},  \cite[Theorem 1.18]{Harris},
the coefficients of the equations of that curve are rational functions
in the coordinates of the $n+3$ points. In the above proof, for the
unirationality of $\mathcal{M}_{S_{1,4,1}}$, we implicitly used the
fact that the construction of a rational cubic curve passing through
the four points in general position is rational in the coordinates
of the points.
 
b)~ There are $\begin{psmallmatrix}
n+2\\
2
\end{psmallmatrix}-2n-1$ independent quadrics that generate the ideal of a degree $n$ rational
  normal curve in $\PP^{n}$.
\end{rem}

\subsection{The lattice $\boldsymbol{S_{1,5,1}}$}

Let $X$ be a K3 surface with N\'eron--Severi lattice of type $S_{1,5,1}$.
The surface $X$ contains six $(-2)$-curves $A_{1},\dots,A_{6}$,
with intersection matrix
\[
\left(\begin{array}{cccccc}
-2 & 6 & 4 & 2 & 2 & 14\\
6 & -2 & 2 & 4 & 14 & 2\\
4 & 2 & -2 & 11 & 1 & 23\\
2 & 4 & 11 & -2 & 23 & 1\\
2 & 14 & 1 & 23 & -2 & 66\\
14 & 2 & 23 & 1 & 66 & -2
\end{array}\right).
\]
The curves $A_{1}$, $A_{3}$, $A_{5}$ generate the N\'eron--Severi lattice.
The divisor 
\[
D_{2}=2A_{1}+2A_{3}-A_{5}
\]
 is ample, of square $2$, with  
\[
2D_{2}\equiv A_{1}+A_{2},\quad 3D_{2}\equiv A_{3}+A_{4},\quad 8D_{2}\equiv A_{5}+A_{6}. 
\]
Thus, by using the linear system $|D_{2}|$, we obtain the following. 

\begin{prop}
The surface $X$ is a double cover of $\,\PP^{2}$ branched over a smooth
sextic curve $C_{6}$ which has a $6$-tangent conic, a tangent cuspidal
cubic and a tangent rational cuspidal octic such that the cusps are
on $C_{6}$, and at each intersection point of the octic or the cubic
with $C_{6}$, the multiplicity is even. 
\end{prop}

The divisor $D_{4}=A_{1}+A_{3}$ is very ample of square $4$, with
$D_{4}\cdot A_{j}=2,8,2,13,3,37$ for $j=1,\dots,6$.  The divisor $D_{8}=A_{1}+A_{3}+A_{5}$
is very ample of square $8$, with $D_{8}\cdot A_{j}=4,22,3,36,1,103$ for
$j=1,\dots,6$. This last model enables us to construct the surfaces
with $\NS X)\simeq S_{1,5,1}$ and to obtain the following proposition. 

\begin{prop}
The moduli space $\mathcal{M}_{S_{1,5,1}}$ is unirational.
\end{prop}

\begin{proof}
Let us fix a hyperplane $H_{4}\subset\PP^{5}$, and let $B_{3}\hookrightarrow H_{4}$
be a normal quartic curve. Let $H_{3}$ be a  general  hyperplane of
$H_{4}$, let $q_{1},\dots,q_{4}$ be the intersection points of $B_{3}$
with $H_{3}$. Let $B_{5}\hookrightarrow H_{4}$ be a line passing
through two  general  points $p_{1}$, $p_{2}$ of $B_{3}$, and let $q_{0}$
the intersection point of $B_{5}$ with $H_{3}$. Let $B_{1}\hookrightarrow H_{3}$
be a rational normal cubic curve containing the points $q_{0},q_{1},\dots,q_{4}$.

By construction, the curves $B_{1}$, $B_{3}$, $B_{5}$ are such that the
degree of the $0$-cycle $B_{j}\cdot B_{k}$ equals $A_{j} \cdot A_{k}$
for $j\neq k$ in $\{1,3,5\}$. The linear system of quadrics in $H_{4}$
containing the curves $B_{1}$, $B_{2}$, $B_{5}$ is a net (a $2$-dimensional
linear space); thus the linear system $\mathcal{L}$ of quadrics in
$\PP^{5}$ containing these curves is $8$-dimensional. A  general 
net of quadrics in $\mathcal{L}$ defines a smooth K3 surface $X$
such that $X\cdot H_{4}=B_{1}+B_{3}+B_{5}$, and the curves $B_{k}$
are $\cu$-curves on $X$. The curves $B_{1}$, $B_{3}$, $B_{5}$ generate
a lattice isometric to $S_{1,5,1}$, and for a  general  choice, $\NS X)\simeq S_{1,5,1}$.
That construction and Remark \ref{rem:Rat-Norm-Curves} on the parametrization
of rational normal curves show that the moduli space $\mathcal{M}_{S_{1,5,1}}$
is unirational.
\end{proof}

\subsection{The lattice $\boldsymbol{S_{1,6,1}}$}

Let $X$ be a K3 surface with N\'eron--Severi lattice of type $S_{1,6,1}$.
The surface $X$ contains four $(-2)$-curves $A_{1},\dots,A_{4}$,
with intersection matrix
\[
\left(\begin{array}{cccc}
-2 & 5 & 2 & 1\\
5 & -2 & 1 & 2\\
2 & 1 & -2 & 5\\
1 & 2 & 5 & -2
\end{array}\right). 
\]
The curves $A_{1}$, $A_{2}$, $A_{3}$ generate the N\'eron--Severi lattice.
The divisor 
\[
D_{6}=A_{1}+A_{2}\equiv A_{3}+A_{4}
\]
is ample, of square $6$, with $D_{6} \cdot A_{j}=3$ for $j\in\{1,\dots,4\}$.
It is base-point free and non-hyperelliptic, and therefore the surface
$X$ is a degree $6$ surface in $\PP^{4}$ with two hyperplane sections
that split as the union of two rational normal cubic curves. 

\begin{prop}
The linear system $|D_{6}|$ gives an embedding of $X$ as a complete
intersection in $\PP^{4}$ with two hyperplanes sections which split
as the unions of two rational cubic curves. The moduli space $\mathcal{M}_{S_{1,6,1}}$
of K3 surfaces $X$ with $\NS X)\simeq S_{1,6,1}$ is unirational.
\end{prop}

\begin{proof}
One can construct these surfaces by taking two degree $3$ rational
normal curves $C_{1}$, $C_{4}$ in two different hyperplanes $H_{1}$, $H_{2}$
 such that the curves $C_{1}$, $C_{4}$ meet transversely in one point.
The linear system $\mathcal{Q}$ of quadrics containing $C_{1}$ and
$C_{4}$ is a pencil, and the linear system $\mathcal{C}$ of cubics
containing $C_{1}$ and $C_{4}$ has dimension $15$. Let $X$ be
the intersection of a  general  element in $\mathcal{Q}$ and a  general 
element in $\mathcal{C}$. The intersections of $X$ with $H_{1}$, $H_{2}$
break down into $C_{1}+C_{2}$ and $C_{3}+C_{4}$, where $C_{2}$, $C_{3}$
are two rational cubic normal curves. Using that $C_{1}+C_{2}$, $C_{3}+C_{4}$
are hyperplane sections and that we know that $C_{1}C_{4}=1$, one obtains
that the curves $C_{1},\dots,C_{4}$ have the above intersection matrix
and therefore generate a lattice isometric to $S_{1,6,1}$, which
is equal to the N\'eron--Severi lattice for a  general  choice of $X$.
The construction shows that the moduli space $\mathcal{M}_{S_{1,6,1}}$
is unirational.
\end{proof}

\begin{prop}
The automorphism group of a general K3 surface $X$ with $\NS X)\simeq S_{1,6,1}$
is trivial.
\end{prop}

\begin{proof}
We proceed as in the proof of Proposition \ref{prop:The-automorphism-groupTRIVIAL}.
The Hilbert basis of the nef cone of the K3 surface $X$ contains
$17$ classes (see \cite[Table 1]{ACL}). One can check that any big
and nef divisor $D$ on $X$ is base-point free, with $D^{2}\geq6$
or $D^{2}=0$; moreover, there is no fiber $F$ of an elliptic fibration
such that $FD=2$. Thus by Theorems \ref{thm:SaintDonat-1} and \ref{thm:SaintDonat-2},
there are no hyperelliptic involutions acting on $X$. From \cite[Table 1 and Lemma~2.3]{Kondo},
the automorphism group of $X$ is therefore trivial. 
\end{proof}

\subsection{The lattice $\boldsymbol{S_{1,9,1}}$}

Let $X$ be a K3 surface with N\'eron--Severi lattice of type $S_{1,9,1}$.
The surface $X$ contains nine $(-2)$-curves $A_{1},\dots,A_{9}$,
with intersection matrix
\[
\left(\begin{array}{ccccccccc}
-2 & 10 & 2 & 8 & 10 & 26 & 2 & 26 & 8\\
10 & -2 & 8 & 2 & 10 & 2 & 26 & 8 & 26\\
2 & 8 & -2 & 1 & 26 & 37 & 25 & 46 & 37\\
8 & 2 & 1 & -2 & 26 & 25 & 37 & 37 & 46\\
10 & 10 & 26 & 26 & -2 & 8 & 8 & 2 & 2\\
26 & 2 & 37 & 25 & 8 & -2 & 46 & 1 & 37\\
2 & 26 & 25 & 37 & 8 & 46 & -2 & 37 & 1\\
26 & 8 & 46 & 37 & 2 & 1 & 37 & -2 & 25\\
8 & 26 & 37 & 46 & 2 & 37 & 1 & 25 & -2
\end{array}\right). 
\]
The family $(A_{1},A_{3},A_{4})$ is a basis of the N\'eron--Severi lattice. 
In that basis, the coordinates of the other $(-2)$-curves are
\[
\begin{array}{lll}
A_{2}=(1,-2,2), \quad & A_{5}=(5,-6,4), \quad & A_{6}=(6,-9,7),\\
A_{7}=(6,-5,3), \quad & A_{8}=(8,-11,8), \quad & A_{9}=(8,-8,5).
\end{array}
\]
 The divisor 
\[
D_{4}=A_{1}-A_{3}+A_{4}
\]
is ample, of square $4$, and the degrees of the curves $A_{1},\dots,A_{9}$
are 
\[
4,\,4,\,5,\,5,\,10,\,14,\,14,\,17,\,17.
\]
The divisor $D_{4}$ is very ample,  and one has
\[
2D_{4}\equiv A_{1}+A_{2}.
\]
The divisor $D_{10}\equiv2A_{1}-A_{3}+A_{4}$ is very ample of square
$10$.

\begin{prop}
The automorphism group of a general K3 surface $X$ with $\NS X)\simeq S_{1,9,1}$
is trivial.
\end{prop}

\begin{proof}
We proceed as in the proof of Proposition \ref{prop:The-automorphism-groupTRIVIAL}.
The Hilbert basis of the nef cone of the K3 surface $X$ is a set
of $72$ classes (see \cite[Table 1]{ACL}). 
\end{proof}

\subsection{The lattice $\boldsymbol{S_{4,1,1}}$}

Let $X$ be a K3 surface with N\'eron--Severi lattice of type $S_{4,1,1}$.
The surface $X$ contains three $(-2)$-curves $A_{1}$, $A_{2}$, $A_{3}$,
with dual graph 

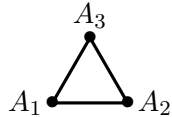
\begin{center}
\begin{tikzpicture}[scale=1]

\draw [very thick] (0,0) -- (1,0);
\draw [very thick] (0,0) -- (0.5,0.866);
\draw [very thick] (1,0) -- (0.5,0.866);

\draw (0,0) node {$\bullet$};
\draw (1,0) node {$\bullet$};
\draw (0.5,0.866) node {$\bullet$};

\draw (0,0) node [left]{$A_{1}$};
\draw (1,0) node [right]{$A_{2}$};
\draw (0.5,0.866) node [above]{$A_{3}$};

\end{tikzpicture}
\end{center} 

These curves generate the N\'eron--Severi lattice. The divisor 
\[
D_{6}=A_{1}+A_{2}+A_{3}
\]
 is ample, base-point free, non-hyperelliptic, of square $6$, with
$D_{6} \cdot A_{j}=2$ for $j\in\{1,2,3\}$, so that the surface $X$ is
a degree $6$ surface in $\PP^{4}$ with a hyperplane section containing
three conics. That leads to the following. 

\begin{prop}
The K3 surface $X$ is degree $6$ complete intersection in $\PP^{4}$
with a hyperplane section which is the union of three conics. The
moduli space $\mathcal{M}_{S_{4,1,1}}$ is unirational.
\end{prop}

\begin{proof}
Let $P^{3}$ be a hyperplane in $\PP^{4}$, and let $P_{1}$, $P_{2}$, $P_{3}$
be three planes in $P^{3}$. For $\{i,j,k\}=\{1,2,3\}$, let us denote
by $L_{k}$ the line $P_{i}\cap P_{j}$. For each line $L_{k}$, let
us choose two  general  points $p_{k}\neq q_{k}$ on $L_{k}$. For $\{i,j,k\}=\{1,2,3\}$,
let $C_{k}$ be a smooth conic in $P_{k}$ passing through $p_{i}$, $p_{j}$, $q_{i}$, $q_{j}$.
The linear system of quadrics in $P^{3}$ containing the conics $C_{1}$, $C_{2}$, $C_{3}$
is $0$-dimensional, and the linear system of cubics containing $C_{1}$, $C_{2}$, $C_{3}$
is $4$-dimensional. Thus the linear system of quadrics (resp.\ cubics)
in $\PP^{4}$ containing the three conics has dimension $5$ (resp.\ $19$). A  general  complete intersection of such a quadric and cubic
is a K3 surface with $\NS X)\simeq S_{4,1,1}$. By that construction,
we see that the moduli space $\mathcal{M}_{S_{4,1,1}}$ is unirational.
\end{proof}
By \cite[Theorem 3.9]{ACL}, the surface $X$ has equations of the
form
\[
X\colon q_{2}=\ell_{1}\ell_{2}\ell_{3}+\ell_{4}g_{2}=0,
\]
where $q_{2}$, $g_{2}$ are quadrics and $\ell_{1}$, $\ell_{2}$, $\ell_{3}$, $\ell_{4}$
are independent linear forms. 

\begin{prop}
The automorphism group of a general K3 surface $X$ with $\NS X)\simeq S_{4,1,1}$
is trivial.
\end{prop}

\begin{proof}
We proceed as in the proof of Proposition \ref{prop:The-automorphism-groupTRIVIAL}.
The Hilbert basis of the nef cone of the K3 surface $X$ is (see \cite[Table 1]{ACL}) 
\[
A_{2}+A_{3},\,A_{1}+A_{3},\,A_{1}+A_{2},\,A_{1}+A_{2}+A_{3}.
\]
Their intersection matrix is 
\[
\left(\begin{array}{cccc}
0 & 4 & 4 & 4\\
4 & 0 & 4 & 4\\
4 & 4 & 0 & 4\\
4 & 4 & 4 & 6
\end{array}\right); 
\]
the result follows. 
\end{proof}

\subsection{The lattice $\boldsymbol{S_{5,1,1}}$}

Let $X$ be a K3 surface with N\'eron--Severi lattice of type $S_{5,1,1}$.
The surface $X$ contains four $(-2)$-curves $A_{1},\dots,A_{4}$,
with intersection matrix
\[
\left(\begin{array}{cccc}
-2 & 3 & 2 & 2\\
3 & -2 & 2 & 2\\
2 & 2 & -2 & 18\\
2 & 2 & 18 & -2
\end{array}\right). 
\]
The curves $A_{1}$, $A_{2}$, $A_{3}$ generate the N\'eron--Severi lattice.
The divisor 
\[
D_{2}=A_{1}+A_{2}
\]
is ample, of square $2$, with $D_{2}\cdot A_{1}=D_{2}\cdot A_{2}=1$ and $D_{2}\cdot A_{3}=D_{2}\cdot A_{4}=4$.
We have
\[
4D_{2}\equiv A_{3}+A_{4}; 
\]
thus, by using the linear system $|D_{2}|$, we obtain the following. 

\begin{prop}
The surface $X$ is a double cover of $\,\PP^{2}$ branched over a smooth
sextic curve $C_{6}$ which has a tritangent line and such that there
is a rational cuspidal quartic curve $Q_{4}$ such that its three cusps
are on $C_{6}$ and the intersection points of $Q_{4}$ and $C_{6}$
have even multiplicities. The moduli space $\mathcal{M}_{S_{5,1,1}}$
of K3 surfaces $X$ with $\NS X)\simeq S_{5,1,1}$ is unirational.
\end{prop}

\begin{proof}
The divisor $D_{8}=A_{1}+A_{2}+A_{3}$ is very ample of square $8$,
with $D_{8} \cdot A_{j}=3,3,2,22$. Thus the curves $A_{1}$, $A_{2}$, $A_{3}$
are, respectively, two rational normal cubics and a conic. The rational
normal cubics cut each others in three points and cut the conic in
two points.

Let $P_{4}\subset\PP^{5}$ be a hyperplane, and let $P_{3}$, $P_{3}'$
be $3$-dimensional projective subspaces of $H_{4}$. Let $P_{0}$
be the plane $P_{0}=P_{3}\cap P_{3}'$. Let $P_{2}\subset P_{4}$
be a plane such that $P_{0}$ and $P_{2}$ meet at one point only.
Let us define the lines $L_{1}=P_{2}\cap P_{3}$, $L_{1}'=P_{2}\cap P_{3}'$, 
and let us fix two points $p_{1}$, $p_{2}$ (resp.\ $p_{1}'$, $p_{2}'$)
on $L_{1}$ (resp.\ $L_{1}'$) and three points $q_{1}$, $q_{2}$, $q_{3}$
in $P_{0}$. We now fix a smooth rational normal cubic curve $B_{1}$
(resp.\ $B_{2}$) on $P_{3}$ (resp.\ $P_{3}'$) passing through $p_{1}$, $p_{2}$
(resp.\ $p_{1}'$, $p_{2}'$) and $q_{1}$, $q_{2}$, $q_{3}$. We also fix an
irreducible conic $B_{3}$ in $P_{2}$ passing through $p_{1}$, $p_{2}$, $p_{1}'$, $p_{2}'$,
so that 
\[
\text{Degree}(B_{j}\cap B_{k})=A_{j} \cdot A_{k}
\]
for $j\neq k$ in $\{1,2,3\}$. The linear system $\mathcal{Q}$ of
quadrics containing the curve $B_{1}+B_{2}+B_{3}$ is $11$-dimensional;
the complete intersection surface obtained by the intersection of
the quadrics in a  general  net of $\mathcal{Q}$ is a K3 surface $X$
containing the $\cu$-curves $B_{1}$, $B_{2}$, $B_{3}$ which generate the
lattice $S_{5,1,1}$. In fact, one has $\NS X)\simeq S_{5,1,1}$ since
the only over-lattice containing $S_{5,1,1}$ is $S_{1,1,1}$. From
that construction and Remark \ref{rem:Rat-Norm-Curves} on the construction
of rational normal curves, we obtain that the moduli space $\mathcal{M}_{S_{5,1,1}}$
is unirational. 
\end{proof}

\subsection{The lattice $\boldsymbol{S_{6,1,1}}$}

Let $X$ be a K3 surface with N\'eron--Severi lattice of type $S_{6,1,1}$.
The surface $X$ contains four $(-2)$-curves $A_{1},\dots,A_{4}$,
with intersection matrix
\[
\left(\begin{array}{cccc}
-2 & 4 & 2 & 2\\
4 & -2 & 2 & 2\\
2 & 2 & -2 & 10\\
2 & 2 & 10 & -2
\end{array}\right).
\]
The curves $A_{1}$, $A_{2}$, $A_{3}$ generate the N\'eron--Severi lattice.
The divisor 
\[
D_{4}=A_{1}+A_{2}
\]
is ample, of square $4$, with $D_{4} \cdot A_{1}=D_{4} \cdot A_{2}=2$ and $D_{4} \cdot A_{3}=D_{4} \cdot A_{4}=4$.
We have $2D_{4}=A_{3}+A_{4}$. The linear system $|D_{4}|$ is base-point
free, non-hyperelliptic. This leads to the following. 

\begin{prop}
The surface $X$ is a quartic in $\PP^{3}$ with a hyperplane section
which is the union of two conics and a quadric section which is the
union of two degree $4$ smooth rational curves. The moduli space
$\mathcal{M}_{S_{6,1,1}}$ of K3 surfaces $X$ with $\NS X)\simeq S_{6,1,1}$
is unirational.
\end{prop}

\begin{proof}
Let us construct these K3 surfaces. Let $A_{3}$ be a smooth degree
$4$ rational curve in $\PP^{3}$, let $p_{1},p_{2}$ be two points
on it, let $P$ be a  general  plane containing the points $p_{1}$, $p_{2}$,
and let $A_{1}$ be an irreducible conic contained in $P$ and containing
the points $p_{1}$, $p_{2}$, so that the degree of the intersection $A_{1}\cdot A_{3}$
is $2$. The linear system of quartics containing $A_{1}$ and $A_{3}$
is $10$-dimensional. Let $X$ be such a  general  quartic; the intersection
of $X$ and $P$ contains $A_{1}$; the residual curve $A_{2}$ is
another smooth conic. By \cite[Exercise IV.6.1]{Hartshorne}, there
exists a unique quadric $Q_{2}$ (which is moreover smooth) containing
the curve $A_{3}$. The intersection of $X$ and $Q_{2}$ is the union
of $A_{3}$ and another degree $4$ smooth rational curve $A_{4}$.
Since 
\[
8=2HA_{3}=(A_{3}+A_{4})A_{3},
\]
we get $A_{3} \cdot A_{4}=10$, and therefore from $(A_{3}+A_{4})^{2}=16$,
one obtains $A_{4}^{2}=-2$. From the construction, the curves $A_{1},\dots,A_{4}$
have the above intersection matrix; thus the N\'eron--Severi lattice
of $X$ contains the lattice $S_{6,1,1}$, and by  the general assumption $\NS X)\simeq S_{6,1,1}$.
That construction shows that the moduli space $\mathcal{M}_{S_{6,1,1}}$
is unirational. 
\end{proof}

\begin{prop}
The automorphism group of a general K3 surface $X$ with $\NS X)\simeq S_{6,1,1}$
is trivial.
\end{prop}

\begin{proof}
We proceed as in the proof of Proposition \ref{prop:The-automorphism-groupTRIVIAL}.
\end{proof}

\subsection{The lattice $\boldsymbol{S_{7,1,1}}$}

Let $X$ be a K3 surface with N\'eron--Severi lattice of type $S_{7,1,1}$.
The surface $X$ contains six $(-2)$-curves $A_{1},\dots,A_{6}$,
with intersection matrix
\[
\left(\begin{array}{cccccc}
-2 & 5 & 2 & 5 & 16 & 2\\
5 & -2 & 2 & 5 & 2 & 16\\
2 & 2 & -2 & 16 & 26 & 26\\
5 & 5 & 16 & -2 & 2 & 2\\
16 & 2 & 26 & 2 & -2 & 26\\
2 & 16 & 26 & 2 & 26 & -2
\end{array}\right). 
\]
The curves $A_{1}$, $A_{2}$, $A_{3}$ generate the N\'eron--Severi lattice, 
and 
\[
A_{4}\equiv3A_{1}+3A_{2}-2A_{3},\quad A_{5}\equiv4A_{1}+6A_{2}-3A_{3},\quad A_{6}\equiv6A_{1}+4A_{2}-3A_{3}.
\]
The divisor 
\[
D_{6}=A_{1}+A_{2}
\]
is very ample of square $6$. For $j\in\{1,\dots,6\}$, we have $D_{6} \cdot A_{j}$
equal to, respectively, $3,3,4,10,18$. The divisors 
\[
2A_{1}+2A_{2}-A_{3},\, 3A_{1}+4A_{2}-2A_{3},\, 4A_{1}+3A_{2}-2A_{3}
\]
are also very ample  of square $6$. 

\begin{prop}
The automorphism group of a general K3 surface $X$ with $\NS X)\simeq S_{6,1,1}$
is trivial.
\end{prop}

\begin{proof}
We proceed as in the proof of Proposition \ref{prop:The-automorphism-groupTRIVIAL}.
\end{proof}

\subsection{The lattice $\boldsymbol{S_{8,1,1}}$\label{subsec:The-latticeS811}}

Let $X$ be a K3 surface with N\'eron--Severi lattice of type $S_{8,1,1}$.
The surface $X$ contains four $(-2)$-curves $A_{1},\dots,A_{4}$,
with intersection matrix
\[
\left(\begin{array}{cccc}
-2 & 6 & 2 & 2\\
6 & -2 & 2 & 2\\
2 & 2 & -2 & 6\\
2 & 2 & 6 & -2
\end{array}\right).
\]
The curves $A_{1}$, $A_{2}$, $A_{3}$ generate the N\'eron--Severi lattice.
The divisor 
\[
D_{8}=A_{1}+A_{2}\equiv A_{3}+A_{4}
\]
is ample, of square $8$, base-point free, non-hyperelliptic, with
$D_{8} \cdot A_{j}=4$ for $j\in\{1,\dots,4\}$. 

\begin{prop}
The K3 surface is a complete intersection in $\PP^{5}$ with two hyperplane
sections which are each the union of two degree $4$ rational curves.
The moduli space $\mathcal{M}_{S_{8,1,1}}$ is unirational.
\end{prop}

\begin{proof}
One can construct these surfaces by taking two degree $4$ rational
normal curves $C_{1}$, $C_{3}$ in two different hyperplanes $H_{1}$, $H_{3}$
but such that the curves $C_{1}$, $C_{3}$ meet transversely in two fixed
points. Let $X$ be a general quartic that contains $C_{1}$ and $C_{3}$; 
then the intersections of $X$ with $H_{1}$, $H_{2}$ are $C_{1}+C_{2}$
and $C_{3}+C_{4}$, where $C_{2}$, $C_{4}$ are two degree $4$ rational
normal curves. Curves $C_{1},\dots,C_{4}$ generate a lattice isometric
to $S_{8,1,1}$. That construction shows that the moduli space $\mathcal{M}_{S_{8,1,1}}$
is unirational.
\end{proof}

\begin{prop}
The automorphism group of a general K3 surface $X$ with $\NS X)\simeq S_{8,1,1}$
is trivial.
\end{prop}

\begin{proof}
We proceed as in the proof of Proposition \ref{prop:The-automorphism-groupTRIVIAL}. 
\end{proof}

\subsection{The lattice $\boldsymbol{S_{10,1,1}}$}

Let $X$ be a K3 surface with N\'eron--Severi lattice of type $S_{10,1,1}$.
The surface $X$ contains eight $(-2)$-curves $A_{1},\dots,A_{8}$,
with intersection matrix
\[
\left(\begin{array}{cccccccc}
-2 & 18 & 8 & 8 & 2 & 22 & 2 & 22\\
18 & -2 & 8 & 8 & 22 & 2 & 22 & 2\\
8 & 8 & -2 & 18 & 2 & 22 & 22 & 2\\
8 & 8 & 18 & -2 & 22 & 2 & 2 & 22\\
2 & 22 & 2 & 22 & -2 & 38 & 18 & 18\\
22 & 2 & 22 & 2 & 38 & -2 & 18 & 18\\
2 & 22 & 22 & 2 & 18 & 18 & -2 & 38\\
22 & 2 & 2 & 22 & 18 & 18 & 38 & -2
\end{array}\right).
\]
The curves $A_{1}$, $A_{3}$, $A_{5}$ generate the N\'eron--Severi lattice.
The divisor 
\[
D_{2}=A_{1}+A_{3}-A_{5}
\]
is ample, base-point free, of square $2$, with $D_{2} \cdot A_{j}=4$ for
$j\in\{1,2,3,4\}$ and $D_{2} \cdot A_{j}=6$ for $j\in\{5,6,7,8\}$. We
have 
\[
4D_{2}\equiv A_{1}+A_{2}\equiv A_{3}+A_{4},\quad 6D_{2}\equiv A_{5}+A_{6}\equiv A_{7}+A_{8}.
\]
By using the linear system $|D_{2}|$, we obtain the following. 

\begin{prop}
The K3 surface is a double cover of $\,\PP^{2}$ branched over a smooth
sextic curve $C_{6}$ such that there are two quartic cuspidal rational
curves $Q_{4}$, $Q_{4}'$ and two sextic cuspidal rational curves $Q_{6}$, $Q_{6}'$
such that the cusps are on $C_{6}$ and the intersection multiplicities
of these curves with $C_{6}$ are even at all intersection points. 
\end{prop}

The divisor $D_{8}=2A_{1}+A_{3}-A_{5}$ is very ample, of square $8$,
with $D_{8} \cdot A_{j}=2,22,12,12,8,28,8,28$ for $j=1,\dots,8$. 

\subsection{The lattice $\boldsymbol{S_{12,1,1}}$}

Let $X$ be a K3 surface with N\'eron--Severi lattice of type $S_{12,1,1}$.
The surface $X$ contains six $(-2)$-curves $A_{1},\dots,A_{6}$,
with intersection matrix
\[
\left(\begin{array}{cccccc}
-2 & 14 & 2 & 10 & 10 & 2\\
14 & -2 & 10 & 2 & 2 & 10\\
2 & 10 & -2 & 14 & 2 & 10\\
10 & 2 & 14 & -2 & 10 & 2\\
10 & 2 & 2 & 10 & -2 & 14\\
2 & 10 & 10 & 2 & 14 & -2
\end{array}\right). 
\]
The curves $A_{1}$, $A_{3}$, $A_{5}$ generate the N\'eron--Severi lattice.
The divisor 
\[
D_{6}=A_{1}-A_{3}+A_{5}
\]
is very ample, of square $6$, with $D_{2} \cdot A_{j}=6$ for $j\in\{1,\dots,6\}$.
We have 
\[
2D_{6}\equiv A_{1}+A_{2}\equiv A_{3}+A_{4}\equiv A_{5}+A_{6};
\]
thus we obtain the first part of the following proposition. 

\begin{prop}
The surface $X$ is a degree $6$ surface in $\PP^{4}$ such that
there are three quadric sections, each of which splits as the union
of two degree $6$ smooth rational curves.

The automorphism group of a general K3 surface $X$ with $\NS X)\simeq S_{12,1,1}$
is trivial.
\end{prop}

\begin{proof}
For the second part, we proceed as in the proof of Proposition \ref{prop:The-automorphism-groupTRIVIAL}.
\end{proof}

\subsection{The lattice $\boldsymbol{S_{4,1,2}'}$}

Let $X$ be a K3 surface with N\'eron--Severi lattice of type $S_{4,1,2}'$.
The surface $X$ contains four $(-2)$-curves $A_{1},\dots,A_{4}$,
with intersection matrix
\[
\left(\begin{array}{cccc}
-2 & 6 & 2 & 2\\
6 & -2 & 2 & 2\\
2 & 2 & -2 & 6\\
2 & 2 & 6 & -2
\end{array}\right).
\]
This is the same intersection matrix as the four $\cu$-curves in
Section~\ref{subsec:The-latticeS811}, but here the curves $A_{1}$, $A_{2}$, $A_{3}$, $A_{4}$
generate only an index $2$ subgroup of the N\'eron--Severi lattice.
There is a basis $e_{1}$, $e_{2}$, $e_{3}$ of the N\'eron--Severi lattice
such that the intersection matrix of $e_{1}$, $e_{2}$, $e_{3}$ is 
\[
\left(\begin{array}{ccc}
-6 & 2 & 0\\
2 & -2 & 4\\
0 & 4 & -8
\end{array}\right).
\]
 In that basis, the classes of the curves are 
\[
\begin{array}{ll}
A_{1}=(2,3,1), & A_{2}=(0,1,1),\\
A_{3}=(2,3,2), & A_{4}=(0,1,0),
\end{array}
\]
and the divisor 
\[
D_{2}=(1,2,1)
\]
is ample, of square $2$, base-point free, with $D_{2} \cdot A_{j}=2$ for
$j\in\{1,2,3,4\}$. We have 
\[
2D_{2}\equiv A_{1}+A_{2}\equiv A_{3}+A_{4}.
\]
By using the linear system $|D_{2}|$, we obtain the following. 

\begin{prop}
The K3 surface is a double cover of $\,\PP^{2}$ branched over a smooth
sextic curve which has two $6$-tangent conics. 
\end{prop}

Let $D_{4}$ be the divisor $D_{4}=(1,3,1)$. It is nef, base-point
free, with $D_{4} \cdot A_{j}=4,4,8,0$ for $j=1,\dots,4$. The linear system
$|D_{4}|$ defines a singular model $Y$ of $X$ which is a quartic
in $\PP^{3}$ with a node. Since 
\[
2D_{4}\equiv A_{1}+A_{2}\equiv A_{3}+A_{4},
\]
 there are two quadric sections $Q_{1}$, $Q_{2}$ such that $Q_{1}$
is the union of two smooth rational degree $4$ curves which are the images
of $A_{1}$, $A_{2}$ and $Q_{2}$ is a degree $8$ rational curve (the
image of $A_{3}$) which contains the node (the image of $A_{4}$).

\section{Rank 4 lattices}

\subsection*{Vinberg's classification}

The reference for the classification of rank $4$ lattices that are
 N\'eron--Severi lattices of K3 surfaces with finite automorphism
group is the article of Vinberg \cite{Vinberg}. There are two lattices
such that the fundamental domain of the Weyl group is compact, and
$12$ for which the domain is not compact. Geometrically, the fundamental
domain is compact if and only if the K3 surface has no elliptic fibration.

\subsection{The rank 4 and compact cases}

We studied in \cite{Roulleau} the two lattices $L(24)$, $L(27)$ of
rank $4$ such that the K3 surfaces have no elliptic fibrations. The
Gram matrices of the lattices $L(24)$, $L(27)$ are, respectively,
\[
\left(\begin{array}{cccc}
2 & 1 & 1 & 1\\
1 & -2 & 0 & 0\\
1 & 0 & -2 & 0\\
1 & 0 & 0 & -2
\end{array}\right),\,\,\,\left(\begin{array}{cccc}
12 & 2 & 0 & 0\\
2 & -2 & 1 & 0\\
0 & 1 & -2 & 1\\
0 & 0 & 1 & -2
\end{array}\right).
\]
The K3 surfaces with such N\'eron--Severi lattices are double covers of
the plane branched over a smooth sextic curve $C_{6}$. For completeness,
let us recall the results obtained in \cite{Roulleau}. 

\begin{thm}
Let $X$ be a K3 surface with N\'eron--Severi lattice isometric to $L(24)$.
The six $(-2)$-curves on $X$ are pull-backs of three lines tritangent
to $C_{6}$.

Let $X$ be a K3 surface with N\'eron--Severi lattice isometric to $L(27)$.
The eight $(-2)$-curves on $X$ are pull-backs of one line tritangent to $C_{6}$ and  three conics $6$-tangent  to $C_{6}$.
\end{thm}

The K3 surfaces with N\'eron--Severi lattice isometric to $L(24)$ have
some connection with K3 surfaces of type $S_{0}\oplus\mathbf{A}_{2}$,
where $S_{0}=\begin{bsmallmatrix}0 & -3\\
-3 & -2
\end{bsmallmatrix}$; see \cite{ACR}.

\subsection{The rank 4 and non-compact cases}

This is the subject of another paper \cite{ACR}. The lattices are
\[
\begin{array}{c}
[8]\oplus\mathbf{A}_{1}^{\oplus3},\,U\oplus\mathbf{A}_{1}^{\oplus2},\,U(2)\oplus\mathbf{A}_{1}^{\oplus2},\,U(3)\oplus\mathbf{A}_{1}^{\oplus2},\,U(4)\oplus\mathbf{A}_{1}^{\oplus2},\\
U\oplus\mathbf{A}_{2},\,U(2)\oplus\mathbf{A}_{2},\,U(3)\oplus\mathbf{A}_{2},\,U(6)\oplus\mathbf{A}_{2},\\
S_{0}\oplus\mathbf{A}_{2},\,[4]\oplus[-4]\oplus\mathbf{A}_{2},\,U(4)\oplus A_{3}.
\end{array}
\]
We give the number of $\cu$-curves in the different cases in the
table in Section~\ref{sec:table}.  

For the Hirzebruch surface $\mathbb{F}_{n}$, $n\ge1$, there is a unique negative
curve $s$ which is such that $s^{2}=-n$. We denote by $f$ a fiber
of the natural fibration $\FF_{n}\to\PP^{1}$. For completeness, we
summarize the constructions obtained in \cite{ACR}. 

\begin{thm}\quad
  
  \begin{itemize}
\item A K3 surface with $\NS X)\simeq[8]\oplus\mathbf{A}_{1}^{\oplus3}$
is a double cover of $\,\PP^{2}$ branched over a smooth sextic curve
which has six $6$-tangent conics. 
\item A K3 surface with $\NS X)\simeq U\oplus\mathbf{A}_{1}^{\oplus2}$
is a double cover of a Hirzebruch surface, $\eta\colon X\to\mathbf{F}_{2}$.
It is branched over the section $s$ and a curve $B\in|3s+8f|$, so that
$sB=2$.
\item A K3 surface $X$ with $\NS X)\simeq U(2)\oplus\mathbf{A}_{1}^{\oplus2}$
is the minimal desingularization of the double cover of $\,\PP^{2}$
branched over a sextic curve with three nodal singularities.
\item A K3 surface $X$ with $\NS X)\simeq U(3)\oplus\mathbf{A}_{1}^{\oplus2}$
is a double cover of $\,\PP^{2}$ branched over a smooth sextic curve
with two tritangent lines and two $6$-tangent conics. 
\item A K3 surface $X$ with $\NS X)\simeq U(4)\oplus\mathbf{A}_{1}^{\oplus2}$
is a quartic in $\PP^{3}$ which has four hyperplane sections that
decompose into the union of two conics.
\item A K3 surface $X$ with $\NS X)\simeq U\oplus\mathbf{A}_{2}$ is the
minimal desingularization of the double cover of $\,\mathbf{F}_{3}$
with branch locus the section $s$ and a reduced curve in $B\in|3s+10f|$,
so that $Bs=1$. 
\item A K3 surface $X$ with $\NS X)\simeq U(2)\oplus\mathbf{A}_{2}$ $X$
is the minimal resolution of the double cover of $\,\PP^{2}$ branched
over a sextic with two nodes such that the line through the two nodes
is tangent to the sextic curve in a third point. 
\item A K3 surface with $\NS X)\simeq U(3)\oplus\mathbf{A}_{2}$ is a quartic
in $\PP^{3}$ which has a hyperplane section that is the union of
four lines.
\item A K3 surface $X$ with $\NS X)\simeq U(6)\oplus\mathbf{A}_{2}$ is
a quartic in $\PP^{3}$ which has three hyperplane sections that decompose
into the union of two conics.
\item A K3 surface $X$ with $\NS X)\simeq\begin{bsmallmatrix}0 & -3\\
-3 & -2
\end{bsmallmatrix}\oplus\mathbf{A}_{2}$ is a double cover of $\,\PP^{2}$ branched over a smooth sextic curve
which has three tritangent lines.
\item A K3 surface $X$ with $\NS X)\simeq[4]\oplus[-4]\oplus\mathbf{A}_{2}$
is a double cover of $\,\PP^{2}$ branched over a smooth sextic curve
which has two tritangent lines and one $6$-tangent conic. 
\item A K3 surface $X$ with $\NS X)\simeq[4]\oplus\mathbf{A}_{3}$ is the
minimal resolution of a double cover of $\,\PP^{2}$ branched over a
sextic curve with one node; through that node go two lines that are
tangent to the sextic at every other intersection point.
\end{itemize}
\end{thm}

We will only add the following result. 

\begin{prop}
A general K3 surface $X$ with a N\'eron--Severi lattice isometric to
one of the  lattices
\[
U(4)\oplus\mathbf{A}_{1}^{\oplus2},\,U(3)\oplus\mathbf{A}_{2},\,U(6)\oplus\mathbf{A}_{2}
\]
has trivial automorphism group. 
\end{prop}

\begin{proof}
For each of these lattices, the Hilbert basis of its nef cone is described
in \cite{ACR}. Then as in Proposition~\ref{prop:The-automorphism-groupTRIVIAL},
one can check that there are no hyperelliptic involutions and conclude
that the automorphism group is trivial. 
\end{proof}

\section{Rank 5 lattices}

\subsection*{Nikulin's classification for higher ranks}

The list of lattices $L$ of rank at least $5$ such that the K3 surfaces
with $\NS X)\simeq L$ have finite automorphism group is given in
\cite{NikulinElliptic}; that list is obtained from the paper \cite{Nikulin}. 

\subsection{The lattice $\boldsymbol{U\oplus\mathbf{A}_{1}^{\oplus3}}$}

There exist seven $(-2)$-curves $A_{1},\dots,A_{7}$ on $X$, with
dual graph 

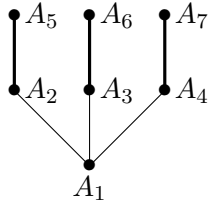
\begin{center}
\begin{tikzpicture}[scale=1]

\draw (0.5,0) -- (0.5,1);
\draw (0.5,0) -- (1.5,1);
\draw (0.5,0) -- (-0.5,1);

\draw [very thick] (0.5,2) -- (0.5,1);
\draw [very thick] (1.5,2) -- (1.5,1);
\draw [very thick] (-0.5,2) -- (-0.5,1);


\draw (0.5,0) node {$\bullet$};
\draw (0.5,1) node {$\bullet$};
\draw (1.5,1) node {$\bullet$};
\draw (-0.5,1) node {$\bullet$};
\draw (0.5,2) node {$\bullet$};
\draw (1.5,2) node {$\bullet$};
\draw (-0.5,2) node {$\bullet$};

\draw (0.5,0) node [below]{$A_{1}$};
\draw (-0.5,1) node [right]{$A_{2}$};
\draw (0.5,1) node [right]{$A_{3}$};
\draw (1.5,1) node [right]{$A_{4}$};
\draw (-0.5,2) node [right]{$A_{5}$};
\draw (0.5,2) node [right]{$A_{6}$};

\draw (1.5,2) node [right]{$A_{7}$};

\end{tikzpicture}
\end{center} 

The curves $A_{1},\dots,A_{5}$ generate the N\'eron--Severi lattice.
The divisor 
\[
D_{18}=3A_{1}+5A_{2}+A_{3}+A_{4}+4A_{5}
\]
is ample, with $D_{18} \cdot A_{j}=1$ for $j\in\{1,2,3,4\}$ and $D_{18} \cdot A_{j}=2$
for $j\in\{5,6,7\}$. The divisor 
\[
D_{2}=2A_{1}+2A_{2}+A_{3}+A_{4}+A_{5}
\]
is nef, base-point free, of square $2$, with $D_{2} \cdot A_{j}=0$ for $j\in\{1,2,3,4\}$
and $D_{2} \cdot A_{j}=2$ for $j\in\{5,6,7\}$. We also have 
\[
D_{2}\equiv2A_{1}+A_{2}+2A_{3}+A_{4}+A_{6}\equiv2A_{1}+A_{2}+A_{3}+2A_{4}+A_{7}.
\]
By using the linear system $|D_{2}|$, we obtain the following. 

\begin{prop}
The K3 surface is a double cover of $\,\PP^{2}$ branched over a sextic
curve with a $\mathbf{d}_{4}$ singularity $q$. 
\end{prop}

The three curves $A_{5}$, $A_{6}$, $A_{7}$ are mapped by the double cover
map to the three lines that are the tangents to the three branches of
the singularity $q$. 

\subsection{The lattice $\boldsymbol{U(2)\oplus\mathbf{A}_{1}^{\oplus3}}$ and the del Pezzo surface of degree 5}

The K3 surface $X$ is the minimal resolution of the double cover
of $\PP^{2}$ branched over a sextic curve $C_{6}$ with four nodes
$p_{1},\dots,p_{4}$ in general position. It is also the double cover
branched over the strict transform of $C_{6}$ in the degree $5$
del Pezzo surface $Z$ which is the blow-up at $p_{1},\dots,p_{4}$.
The surface $X$ contains $10$ $\cu$-curves denoted by $A_{i,j}$,
for $\{i,j\}\subset\{1,\dots,5\}$ with $i<j$ (the pull-back of the
$10$ $(-1)$-curves on the del Pezzo surface). One has 
\[
A_{ij} \cdot A_{st}=2
\]
if and only if $\#\{i,j,s,t\}=4$; else $A_{ij} \cdot A_{st}=0$ or $-2$.
The dual graph of the configuration is the Petersen graph 

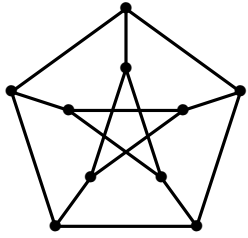
\begin{center}
\begin{tikzpicture}[scale=0.8]



\draw [very thick] ( 0.95106, 0.30901 )  -- ( -0.95103, 0.30910); 
\draw [very thick] ( 0, 1.0000 ) -- ( -0.58788, -0.80894 );
\draw [very thick] ( -0.95103, 0.30910 ) -- ( 0.58765, -0.80911);     
\draw [very thick] ( -0.58788, -0.80894 ) -- ( 0.95106, 0.30901); 
\draw [very thick] ( 0.58765, -0.80911 ) --( 0, 1.0000 );

\draw [very thick] ( 1.9021, 0.61802 )--( 0, 2 );
\draw [very thick] ( 0, 2 )--( -1.9021, 0.61819 );
\draw [very thick] ( -1.9021, 0.61819 )--( -1.1758, -1.6179 );
\draw [very thick] ( -1.1758, -1.6179 )--( 1.1753, -1.6182 );
\draw [very thick] ( 1.1753, -1.6182 )--( 1.9021, 0.61802 );

\draw [very thick] ( 0.95106, 0.30901 ) -- ( 1.9021, 0.61802 );
\draw [very thick] ( 0, 1.0000 ) -- ( 0, 2 );
\draw [very thick] ( -0.95103, 0.30910 ) -- ( -1.9021, 0.61819 );
\draw [very thick] ( -0.58788, -0.80894 ) -- ( -1.1758, -1.6179 );
\draw [very thick] ( 0.58765, -0.80911 ) -- ( 1.1753, -1.6182 );

\draw ( 0.95106, 0.30901 )     node {$\bullet$};
\draw ( 0, 1.0000 ) node {$\bullet$};
\draw ( -0.95103, 0.30910 )      node {$\bullet$};
\draw ( -0.58788, -0.80894 )      node {$\bullet$};
\draw ( 0.58765, -0.80911 ) node {$\bullet$};

\draw ( 1.9021, 0.61802 ) node {$\bullet$};
\draw ( 0, 2 ) node {$\bullet$};
\draw ( -1.9021, 0.61819 ) node {$\bullet$};
\draw ( -1.1758, -1.6179 ) node {$\bullet$};
\draw ( 1.1753, -1.6182 ) node {$\bullet$};


\end{tikzpicture}
\end{center} 
with weight $2$ on the edges. 

\begin{rem}
If $C$ is a  general  curve of genus $6$, then there exists a map $C\to\PP^{2}$
with image a sextic curve with four nodes. Using that property, Artebani
and Kondo describe in \cite{AK} the moduli space of genus $6$ curves and
their link with K3 surfaces and the quintic del Pezzo surface. They
study that moduli space as a quotient of a bounded symmetric domain. 
\end{rem}

\subsection{The lattice $\boldsymbol{U(4)\oplus\mathbf{A}_{1}^{\oplus3}}$}

The K3 surface $X$ contains $24$ $\cu$-curves $A_{1},\dots,A_{24}$.
Up to permutation, one can suppose that $A_{1}$, $A_{3}$, $A_{5}$, $A_{7}$, $A_{9}$
have the following intersection matrix: 
\[
\begin{pmatrix}-2 & 0 & 0 & 0 & 0\\
0 & -2 & 0 & 0 & 4\\
0 & 0 & -2 & 2 & 2\\
0 & 0 & 2 & -2 & 2\\
0 & 4 & 2 & 2 & -2
\end{pmatrix}.
\]
These five curves generate $\NS X)$. The divisor 
\[
D_{2}=-A_{1}+A_{3}+A_{9}
\]
is ample, of square $2$, with $D_{2} \cdot A_{j}=2$ for $j\in\{1,\dots,24\}$
and, up to permutation of the indices, 
\[
2D_{2}\equiv A_{2k-1}+A_{2k}
\]
 for any $k\in\{1,\dots,12\}$. The classes of $A_{11}, A_{13},\dots,A_{23}$
in the basis $A_{1}$, $A_{3}$, $A_{5}$, $A_{7}$, $A_{9}$ are 
\[
\begin{array}{ll}
A_{11}=(0,-1,1,1,0), & A_{13}=(-1,0,1,1,0),\\
A_{15}=(-1,0,0,1,1), & A_{17}=(-2,1,0,1,1),\\
A_{19}=(-2,0,1,1,1), & A_{21}=(-1,0,1,0,1),\\
A_{23}=(-2,1,1,0,1),
\end{array}
\]
so that we know the $24$ classes of $(-2)$-curves in $X$. By using
the linear system $|D_{2}|$, we obtain the following. 

\begin{prop}
The K3 surface $X$ is a double cover of $\,\PP^{2}$ branched over
a smooth sextic curve which has $12$ $6$-tangent conics.
\end{prop}

There exists a partition of the $24$ $\cu$-curves into three sets
$S_{1}$, $S_{2}$, $S_{3}$ of $8$ curves each such that for curves $B$, $B'$
in two different sets $S$, $S'$, one has $BB'=0$ or $4$ and for any
$B\in S$, there are exactly $4$ curves $B'$ in $S'$ such that $BB'=4$,
and symmetrically for $B'$. Therefore, $S$ and $S'$ form an $8_{4}$
configuration. This is the so-called M\"obius configuration (see \cite{Coxeter}).
The following graph is the Levi graph of that $8_{4}$ configuration;
this is the graph of the $4$-dimensional hypercube (see \cite{Coxeter}). 
Vertices in red are curves in $S$, vertices in blue are curves in
$S'$, and an edge links a red curve to a blue curve if and only if
their intersection number is $4$. 

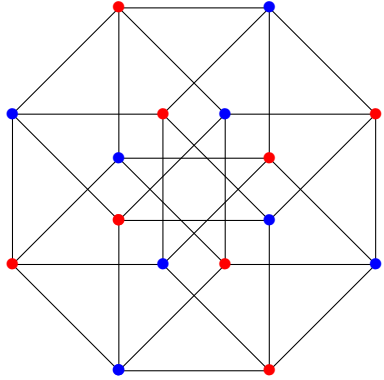
\begin{center}
\begin{tikzpicture}[scale=1]


\draw (0,0) rectangle (2,2);
\draw (-1.414,1.414) rectangle (2-1.414,2+1.414);
\draw (1.414,1.414) rectangle (2+1.414,2+1.414);
\draw (0,2*1.414) rectangle (2,2+2*1.414);
\draw (0,0) -- ++(1.414,1.414) -- ++(-1.414,1.414) --++ (-1.414,-1.414) -- cycle;
\draw (0,2) -- ++(1.414,1.414) -- ++(-1.414,1.414) --++ (-1.414,-1.414) -- cycle;
\draw (2,0) -- ++(1.414,1.414) -- ++(-1.414,1.414) --++ (-1.414,-1.414) -- cycle;
\draw (2,2) -- ++(1.414,1.414) -- ++(-1.414,1.414) --++ (-1.414,-1.414) -- cycle;

\draw [color=red] (0,2) node {$\bullet$};
\draw [color=red] (-1.414,1.414) node {$\bullet$};
\draw [color=red] (1.414,1.414) node {$\bullet$};
\draw [color=red] (0,2) node {$\bullet$};
\draw [color=red] (0,2+2*1.414) node {$\bullet$};
\draw [color=red] (2-1.414,2+1.414) node {$\bullet$};
\draw [color=red] (2,0) node {$\bullet$};
\draw [color=red] (2+1.414,2+1.414) node {$\bullet$};
\draw [color=red] (2,2*1.414) node {$\bullet$};

\draw [color=blue] (0,0) node {$\bullet$};
\draw [color=blue] (-1.414,3.414) node {$\bullet$};
\draw [color=blue] (2-1.414,1.414) node {$\bullet$};
\draw [color=blue] (2+1.414,1.414) node {$\bullet$};
\draw [color=blue] (2,2) node {$\bullet$};
\draw [color=blue] (0,2*1.414) node {$\bullet$};
\draw [color=blue] (2,2+1.414*2) node {$\bullet$};
\draw [color=blue] (1.414,2+1.414) node {$\bullet$};
\draw [color=blue] (0,0) node {$\bullet$};


\end{tikzpicture}
\end{center} 

From that graph, we can moreover read the intersection numbers of
the curves in $S$ (and $S'$) as follows. For any red curve $B$, there are four blue curves linked to it by
an edge. Consider the complementary set of blue curves; this is another
set of four blue curves, all  linked through an edge to the same
red curve $B'$. Then we have $BB'=6$, and for any other red curve
$B''\notin\{B,B'\}$, we have $BB''=2$. Symmetrically,
the intersection
numbers between the blue curves follow the same rule. 

In \cite[Proof of Theorem 8.1.1]{Nikulin}, Nikulin studies the lattice $U(4)\oplus\mathbf{A}_{1}^{\oplus3}$ in detail, obtaining that it contains only $24$ $\cu$-curves, and he describes the fundamental polygon of the action of the Weyl group by an embedding in the $4$-dimensional Euclidian space. There,  the polyhedral formed by the $24$ vertices is the dual of the polyhedron formed by the roots of type $\mathbf{D}_{4}$.

There is a second geometric model which is as follows. The divisor
\[
D_{4}=-A_{1}+2A_{3}+A_{9}
\]
is a nef, non-hyperelliptic divisor of square $4$, with  $D_{4} \cdot A_{j}=0$
if and only if $j=3$. We have, moreover,  
\[
D_{4}\equiv A_{1}+A_{10}\equiv A_{5}+A_{22}\equiv A_{7}+A_{16}\equiv A_{13}+A_{20},
\]
and the intersection number of these eight curves with $D_{4}$ is
$2$. Therefore, the linear system $|D_{4}|$ gives a singular model
of $X$ as a quartic in $\PP^{3}$ with a unique node and four hyperplane
sections which are unions of two conics. One can check that the eight
$\cu$-curves which are mapped to the conics and the $\cu$-curve
which is contracted to the node generate the lattice $U(4)\oplus\mathbf{A}_{1}^{\oplus3}$.

\subsection{The lattice $\boldsymbol{U\oplus\mathbf{A}_{1}\oplus\mathbf{A}_{2}}$}

The K3 surface $X$ contains six $(-2)$-curves $A_{1},\dots,A_{6}$; 
their configuration is 

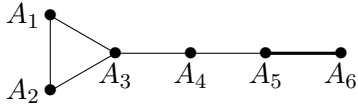
\begin{center}
\begin{tikzpicture}[scale=1]

\draw (0,0) -- (2,0);
\draw (0,0) -- (-0.866,0.5);
\draw (0,0) -- (-0.866,-0.5);
\draw (-0.866,0.5) -- (-0.866,-0.5);
\draw [very thick] (2,0) -- (3,0);

\draw (0,0) node {$\bullet$};
\draw (1,0) node {$\bullet$};
\draw (2,0) node {$\bullet$};
\draw (3,0) node {$\bullet$};
\draw (-0.866,-0.5) node {$\bullet$};
\draw (-0.866,0.5) node {$\bullet$};

\draw (0,0) node [below]{$A_{3}$};
\draw (1,0) node [below]{$A_{4}$};
\draw (2,0) node [below]{$A_{5}$};
\draw (3,0) node [below]{$A_{6}$};
\draw (-0.866,-0.5) node [left]{$A_{2}$};
\draw (-0.866,0.5) node [left]{$A_{1}$};

\end{tikzpicture}
\end{center} 

These curves generate the N\'eron--Severi lattice. The divisor 
\[
D_{20}=5A_{1}+5A_{2}+6A_{3}+3A_{4}+A_{5}
\]
is ample, of square $20$, with $D_{20} \cdot A_{j}=1$ for $j\leq5$ and $D_{20} \cdot A_{6}=2$.
The divisor 
\[
D_{4}=2A_{1}+2A_{2}+3A_{3}+2A_{4}+A_{5}
\]
is nef, of square $4$, base-point free and hyperelliptic since 
\[
D_{4}(A_{1}+A_{2}+A_{3})=2,
\]
and $A_{1}+A_{2}+A_{3}\equiv A_{5}+A_{6}$ is a fiber of an elliptic
fibration. One has $D_{4} \cdot A_{1}=D_{4} \cdot A_{2}=1$, $D_{4} \cdot A_{6}=2$ and
$D_{4} \cdot A_{j}=0$ for $j\in\{3,4,5\}$. Moreover,  
\[
D_{4}\equiv A_{3}+2A_{4}+3A_{5}+2A_{6}.
\]

\begin{prop}
The double cover induced by $|D_{4}|$ factors through $\eta\colon X\to\mathbf{F}_{2}$.
The image by $\eta$ of $A_{4}$ is $s$ $($where $s$ is the section
such that $s^{2}=-2)$. The branch locus is the union of the section
$s$ and $B\in|3s+8f|$, so that $Bs=2$. Let $f_{1},f_{2}$ be
the fibers through the points $p$ and $q$, the intersection points
of $s$ and $B$. The curve $f_{1}$ cuts $B$ with multiplicity $2$
at another point, and the image by $\eta$ of the curves $A_{1}$, $A_{2}$
is $f_{1}$, the image of $A_{6}$ is $f_{2}$, and $A_{3}$, $A_{5}$
are contracted to $p$ and $q$, respectively. 
\end{prop}

\begin{proof}
We apply Theorem \ref{thm:SaintDonat-2}, case a) iii) v).
\end{proof}

\subsection{The lattice $\boldsymbol{U\oplus\mathbf{A}_{3}}$}

The K3 surface $X$ contains five $(-2)$-curves $A_{1},\dots,A_{5}$; 
their configuration is 

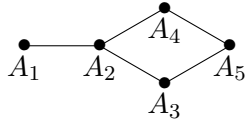
\begin{center}
\begin{tikzpicture}[scale=1]

\draw (0,0) -- (-1,0);
\draw (0,0) -- (0.866,0.5);
\draw (0,0) -- (0.866,-0.5);
\draw (0.866,0.5) -- (0.866*2,0);
\draw (0.866,-0.5) -- (0.866*2,0);

\draw (0,0) node {$\bullet$};
\draw (-1,0) node {$\bullet$};
\draw (0.866,0.5) node {$\bullet$};
\draw (0.866,-0.5) node {$\bullet$};
\draw (2*0.866,0) node {$\bullet$};

\draw (0,0) node [below]{$A_{2}$};
\draw (-1,0) node [below]{$A_{1}$};
\draw (2*0.866,0) node [below]{$A_{5}$};
\draw (0.866,-0.5) node [below]{$A_{3}$};
\draw (0.866,0.5) node [below]{$A_{4}$};

\end{tikzpicture}
\end{center} 

These curves generate the N\'eron--Severi lattice. In that basis, the
divisor
\[
D_{50}=(5,11,9,9,8)
\]
is ample, of square $50$, with $D_{50} \cdot A_{j}=1$ for $j\leq4$ and $D_{50} \cdot A_{5}=2$.
We have $D_{50} \cdot A_{j}=0$ for $j\leq4$ and $D_{50} \cdot A_{5}=2$. The divisor
\[
D_{4}=2A_{1}+4A_{2}+3A_{3}+3A_{4}+2A_{5}
\]
is nef, of square $4$, and $|D_{4}|$ is base-point free hyperelliptic
since $D_{4}(A_{2}+A_{3}+A_{4}+A_{5})=2$; one has $D_{4} \cdot A_{j}=0$
for $j\leq4$ and $D_{4} \cdot A_{5}=2$. 

\begin{prop}
The double cover induced by $|D_{4}|$ factors through $\eta\colon X\to\mathbf{F}_{2}$.
The image by $\eta$ of $A_{1}$ is the section $s$. The branch locus
is the union of the section $s$ and $B\in|3s+8f|$, so that $Bs=2$.
The intersection of $B$ and $s$ is tangent at one point $q$, forming
an $\mathbf{a}_{3}$ singularity. The curve $A_{5}$ is mapped onto
the fiber through $q$, and the curves $A_{2}$, $A_{3}$, $A_{4}$ are mapped
to $q$. 
\end{prop}

\begin{proof}
We apply Theorem \ref{thm:SaintDonat-2}, case a) iii) v).
\end{proof}

\subsection{The lattice $\boldsymbol{[4]\oplus\mathbf{D}_{4}}$}

The K3 surface $X$ contains five $(-2)$-curves $A_{1},\dots,A_{5}$; 
their configuration is 

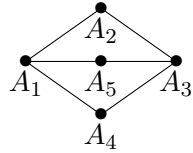
\begin{center}
\begin{tikzpicture}[scale=1]

\draw (0,0) -- (2,0);
\draw (0,0) -- (1,1.414/2);
\draw (0,0) -- (1,-1.414/2);
\draw (2,0) -- (1,1.414/2);
\draw (2,0) -- (1,-1.414/2);


\draw (0,0) node {$\bullet$};
\draw (1,1.414/2) node {$\bullet$};
\draw (1,-1.414/2) node {$\bullet$};
\draw (1,0) node {$\bullet$};
\draw (2,0) node {$\bullet$};

\draw (0,0) node [below]{$A_{1}$};
\draw (1,1.414/2) node [below]{$A_{2}$};
\draw (1,-1.414/2) node  [below]{$A_{4}$};
\draw (1,0) node  [below]{$A_{5}$};
\draw (2,0) node [below]{$A_{3}$};

\end{tikzpicture}
\end{center} 

These curves generate the N\'eron--Severi lattice. The divisor
\[
D_{22}=A_{1}+3\sum_{j=1}^{5}A_{j}
\]
is ample, of square $22$, with $D_{22} \cdot A_{3}=3$ and $D_{22} \cdot A_{j}=1$ for
$j\neq3$. The divisor 
\[
D_{2}=\sum_{j=1}^{5}A_{j}
\]
is nef, of square $2$, with  $D_{2} \cdot A_{1}=D_{2} \cdot A_{3}=1$, $D_{2} \cdot A_{j}=0$
for $j\in\{2,4,5\}$. By using the linear system $|D_{2}|$, we obtain the following. 

\begin{prop}
The K3 surface is a double cover of $\,\PP^{2}$ branched over a sextic
curve with three nodes on a line. 
\end{prop}

The image of $A_{1}$, $A_{3}$ is the line through the three nodes;
the curves $A_{2}$, $A_{4}$, $A_{5}$ are contracted to the nodes. 

The divisor $D_{4}=A_{1}+2A_{2}+3A_{3}+2A_{4}+2A_{5}$ is nef, non-hyperelliptic,
with $D_{4} \cdot A_{1}=4$ and $D_{4} \cdot A_{k}=0$ for $k\geq2$. It defines
a model of the K3 surface as a quartic in $\PP^{3}$ with a $\mathbf{D}_{4}$
singularity.

\subsection{The lattice $\boldsymbol{[8]\oplus\mathbf{D}_{4}}$}

The K3 surface $X$ contains seven $(-2)$-curves $A_{1},\dots,A_{7}$; 
their configuration is 

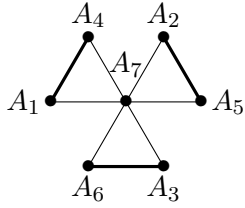
\begin{center}
\begin{tikzpicture}[scale=1]

\draw [very thick] (1,0) -- (0.5,0.86);
\draw [very thick] (-1,0) -- (-0.5,0.86);
\draw [very thick] (-0.5,-0.86) -- (0.5,-0.86);
\draw (0,0) -- (0.5,-0.86);
\draw (0,0) -- (-0.5,-0.86);
\draw (0,0) -- (-0.5,0.86);
\draw (-1,0) -- (1,0);
\draw (0,0) -- (0.5,0.86);

\draw (1,0) node {$\bullet$};
\draw (0.5,0.86) node {$\bullet$};
\draw (-1,0) node {$\bullet$};
\draw (-0.5,0.86) node {$\bullet$};
\draw (-0.5,-0.86) node {$\bullet$};
\draw (0.5,-0.86)  node {$\bullet$};
\draw (0,0) node {$\bullet$};

\draw (1,0) node  [right]{$A_{5}$};
\draw (0.5,0.86) node [above]{$A_{2}$};
\draw (-1,0) node  [left]{$A_{1}$};
\draw (-0.5,0.86) node  [above]{$A_{4}$};
\draw (-0.5,-0.86) node  [below]{$A_{6}$};
\draw (0.5,-0.86) node   [below]{$A_{3}$};
\draw (0,0.2) node  [above]{$A_{7}$};

\end{tikzpicture}
\end{center} 

The curves $A_{1}$, $A_{2}$, $A_{3}$, $A_{4}$, $A_{7}$ generate the N\'eron--Severi
lattice. The divisor
\[
D_{6}=2A_{1}+2A_{4}+A_{7}
\]
is ample, of square $6$, with $D_{6} \cdot A_{j}=1$ for $j\leq6$ and $D_{6} \cdot A_{7}=2$.
The divisor 
\[
D_{2}=A_{1}+A_{4}+A_{7}
\]
is nef, of square $2$, with $D_{6} \cdot A_{j}=1$ for $j\leq6$ and $D_{6} \cdot A_{7}=0$.
We have 
\[
D_{2}\equiv A_{2}+A_{5}+A_{7}\equiv A_{3}+A_{6}+A_{7}.
\]
By using the linear system $|D_{2}|$, we obtain the following. 

\begin{prop}
The surface is a double cover of $\,\PP^{2}$ branched over a sextic
curve with one node, and there are three lines through that node such
that each line is tangent to the sextic at its other intersection
points. 
\end{prop}

One can obtain another interesting geometric model of the K3 surface
$X$ as follows. The divisor $D_{8}=3A_{1}+2A_{2}+2A_{3}+A_{4}+4A_{7}$
is big and nef, base-point free, non-hyperelliptic, with $D_{8} \cdot A_{k}=0,0,0,8,8,8,0$
for $k=1,\dots,7$. The image of the K3 surface by the map associated
to $|D_{8}|$ is a degree $8$ surface in $\PP^{5}$ with a ${\bf D}_{4}$
singularity. 

\subsection{The lattice $\boldsymbol{[16]\oplus\mathbf{D}_{4}}$}

The K3 surface $X$ contains eight $(-2)$-curves $A_{1},\dots,A_{8}$; 
their intersection matrix is 
\[
\begin{pmatrix}-2 & 3 & 0 & 1 & 0 & 1 & 0 & 1\\
3 & -2 & 1 & 0 & 1 & 0 & 1 & 0\\
0 & 1 & -2 & 3 & 0 & 1 & 0 & 1\\
1 & 0 & 3 & -2 & 1 & 0 & 1 & 0\\
0 & 1 & 0 & 1 & -2 & 3 & 0 & 1\\
1 & 0 & 1 & 0 & 3 & -2 & 1 & 0\\
0 & 1 & 0 & 1 & 0 & 1 & -2 & 3\\
1 & 0 & 1 & 0 & 1 & 0 & 3 & -2
\end{pmatrix}. 
\]
Their configuration is 

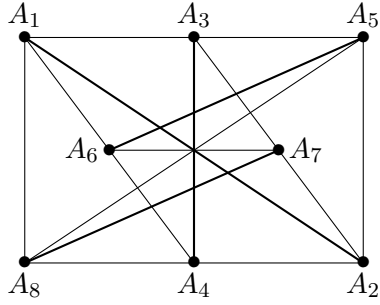
\begin{center}
\begin{tikzpicture}[scale=1.2]


\draw  (0,1*0.7) -- (0,3*0.7);
\draw  (1*0.7,0) -- (1*0.7,4*0.7);
\draw  (1*0.7,4*0.7) -- (3*0.7,4*0.7);
\draw  (0,3*0.7) -- (4*0.7,3*0.7);
\draw  (3*0.7,4*0.7) -- (3*0.7,0);
\draw  (4*0.7,3*0.7) -- (4*0.7,1*0.7);
\draw  (1*0.7,0) -- (3*0.7,0);
\draw  (0,1*0.7) -- (4*0.7,1*0.7);

\draw  (1*0.7 , 4*0.7) -- ( 4*0.7, 1*0.7);
\draw  ( 0, 3*0.7) -- (3*0.7 ,0 );
\draw  (0 ,1*0.7 ) -- (3*0.7 ,4*0.7 );
\draw  ( 1*0.7,0*0.7 ) -- ( 4*0.7,3*0.7 );

\draw  [thick] (0,0.7) -- (0.7,0);
\draw  [ thick] (0,3*0.7) -- (1*0.7,4*0.7);
\draw  [ thick] (3*0.7,0) -- (4*0.7,1*0.7);
\draw  [ thick] (3*0.7,4*0.7) -- (4*0.7,3*0.7);

\draw  (1*0.7,4*0.7) node {$\bullet$};
\draw  (3*0.7,4*0.7) node {$\bullet$};
\draw  (0,3*0.7) node {$\bullet$};
\draw  (0,1*0.7) node {$\bullet$};
\draw  (1*0.7,0) node {$\bullet$};
\draw  (3*0.7,0) node {$\bullet$};
\draw  (4*0.7,1*0.7) node {$\bullet$};
\draw  (4*0.7,3*0.7) node {$\bullet$};

\draw  (1*0.7,4*0.7) node [above]{$A_{1}$};
\draw  (3*0.7,4*0.7) node [above]{$A_{4}$};
\draw  (0,3*0.7) node [left]{$A_{2}$};
\draw  (0,1*0.7) node [left]{$A_{7}$};
\draw  (1*0.7,0) node [below]{$A_{8}$};
\draw  (3*0.7,0) node [below]{$A_{5}$};
\draw  (4*0.7,1*0.7) node [right]{$A_{6}$};
\draw  (4*0.7,3*0.7) node [right]{$A_{3}$};

\end{tikzpicture}
\end{center} 
where the thick lines have weight $3$. The curves $A_{1}$, $A_{2}$, $A_{3}$, $A_{5}$, $A_{7}$
generate $\NS X)$. The divisor 
\[
D_{2}=A_{1}+A_{2}
\]
is ample of square $2$, with $D_{2} \cdot A_{j}=1$ for $j\in\{1,\dots,8\}$.
We have also
\[
D_{2}\equiv A_{3}+A_{4}\equiv A_{5}+A_{6}\equiv A_{7}+A_{8}.
\]
By using the linear system $|D_{2}|$, we obtain the following. 

\begin{prop}
The surface is a double cover of $\,\PP^{2}$ branched over a smooth
sextic curve which has four tritangent lines. 
\end{prop}

One can obtain another interesting geometric model of the K3 surface
$X$ as follows. The divisor $D_{16}=4A_{1}+A_{2}+5A_{4}+3A_{5}+3A_{7}$
is nef, non-hyperelliptic, of square $16$, with $D_{16} \cdot A_{k}=0$
for curves $A_{1}$, $A_{4}$, $A_{5}$, $A_{7}$ and $D_{16} \cdot A_{k}=16$ for the
other curves. The image of the K3 surface under the map associated
to $|D_{16}|$ is a degree $16$ surface in $\PP^{9}$ with a ${\bf D}_{4}$
singularities, which is the image of the curves $A_{1}$, $A_{4}$, $A_{5}$, $A_{7}$. 

\subsection{The lattice $\boldsymbol{[6]\oplus\mathbf{A}_{2}^{\oplus2}}$}

The K3 surface $X$ contains $10$ $(-2)$-curves $A_{1},\dots,A_{10}$; 
their intersection matrix is 
\[
\begin{pmatrix}-2 & 3 & 1 & 0 & 0 & 1 & 0 & 1 & 2 & 0\\
3 & -2 & 0 & 1 & 1 & 0 & 1 & 0 & 0 & 2\\
1 & 0 & -2 & 3 & 0 & 1 & 0 & 1 & 2 & 0\\
0 & 1 & 3 & -2 & 1 & 0 & 1 & 0 & 0 & 2\\
0 & 1 & 0 & 1 & -2 & 3 & 1 & 0 & 0 & 2\\
1 & 0 & 1 & 0 & 3 & -2 & 0 & 1 & 2 & 0\\
0 & 1 & 0 & 1 & 1 & 0 & -2 & 3 & 0 & 2\\
1 & 0 & 1 & 0 & 0 & 1 & 3 & -2 & 2 & 0\\
2 & 0 & 2 & 0 & 0 & 2 & 0 & 2 & -2 & 6\\
0 & 2 & 0 & 2 & 2 & 0 & 2 & 0 & 6 & -2
\end{pmatrix}. 
\]
The configuration of the first eight curves is  

\begin{center}
\begin{tikzpicture}[scale=1.2]


\draw  (0,1*0.7) -- (0,3*0.7);
\draw  (1*0.7,0) -- (3*0.7,4*0.7);

\draw  (1*0.7,4*0.7) -- (3*0.7,4*0.7);
\draw  (0,3*0.7) -- (4*0.7,1*0.7);
\draw  (1*0.7,4*0.7) -- (3*0.7,0);

\draw  (4*0.7,3*0.7) -- (4*0.7,1*0.7);
\draw  (1*0.7,0) -- (3*0.7,0);
\draw  (0,1*0.7) -- (4*0.7,3*0.7);
\draw  (1*0.7 , 4*0.7) -- ( 4*0.7, 1*0.7);
\draw  ( 0, 3*0.7) -- (3*0.7 ,0 );
\draw  (0 ,1*0.7 ) -- (3*0.7 ,4*0.7 );
\draw  ( 1*0.7,0*0.7 ) -- ( 4*0.7,3*0.7 );

\draw  [thick] (0,0.7) -- (0.7,0);
\draw  [ thick] (0,3*0.7) -- (1*0.7,4*0.7);
\draw  [ thick] (3*0.7,0) -- (4*0.7,1*0.7);
\draw  [ thick] (3*0.7,4*0.7) -- (4*0.7,3*0.7);

\draw [color=red] (1*0.7,4*0.7) node {$\bullet$};
\draw [color=red] (3*0.7,4*0.7) node {$\bullet$};
\draw  (0,3*0.7) node {$\bullet$};
\draw  (0,1*0.7) node {$\bullet$};
\draw [color=red] (1*0.7,0) node {$\bullet$};
\draw [color=red] (3*0.7,0) node {$\bullet$};
\draw  (4*0.7,1*0.7) node {$\bullet$};
\draw  (4*0.7,3*0.7) node {$\bullet$};

\draw  (1*0.7,4*0.7) node [above]{$A_{1}$};
\draw  (3*0.7,4*0.7) node [above]{$A_{3}$};
\draw  (0,3*0.7) node [left]{$A_{2}$};
\draw  (0,1*0.7) node [left]{$A_{7}$};
\draw  (1*0.7,0) node [below]{$A_{8}$};
\draw  (3*0.7,0) node [below]{$A_{6}$};
\draw  (4*0.7,1*0.7) node [right]{$A_{5}$};
\draw  (4*0.7,3*0.7) node [right]{$A_{4}$};

\end{tikzpicture}
\end{center} 
where the thick lines have weight $3$, the curves $A_{k}$, $k\in\{1,3,6,8\}$
(with a red vertex) are such that $A_{k} \cdot A_{9}=2$, $A_{k} \cdot A_{10}=0$, and the
curves $A_{k}$, $k\in\{2,4,5,7\}$ (with a black vertex) are such that
$A_{k} \cdot A_{10}=2$, $A_{k} \cdot A_{9}=0$. The curves $A_{1}$, $A_{3}$, $A_{5}$, $A_{7}$, $A_{9}$
generate $\NS X)$. The divisor 
\[
D_{2}=A_{1}+A_{3}-A_{5}-A_{7}+A_{9}
\]
is ample, of square $2$, with $D_{2} \cdot A_{j}=1$ for $j\leq8$, $D_{2} \cdot A_{9}=D_{2} \cdot A_{10}=2$.
We have 
\[
D_{2}\equiv A_{2k-1}+A_{2k},\,\,k\in\{1,2,3,4\}
\]
 and $2D_{2}\equiv A_{9}+A_{10}$. By using the linear system $|D_{2}|$,
we obtain the following. 

\begin{prop}
The surface $X$ is the double cover of $\,\PP^{2}$ branched over a
smooth sextic curve which has four tritangent lines and one $6$-tangent
conic. 
\end{prop}

\begin{rem}
The N\'eron--Severi lattice is also generated by $A_{1}$, $A_{2}$, $A_{3}$, $A_{5}$, $A_{7}$.

It is interesting to compare this case with the previous lattice case, where
 four tritangent lines are also involved. 
\end{rem}

\begin{prop}
Let $X$ be a K3 surface with $\NS X)\simeq[6]\oplus\mathbf{A}_{2}^{\oplus2}$.
There exist linear forms $\ell_{1},\dots,\ell_{4}\in H^{0}(\PP^{2},\OO(1))$,
a quadric $q_{2}\in H^{0}(\PP^{2},\OO(2))$ and a cubic $f_{3}\in H^{0}(\PP^{2},\OO(3))$
such that $X$ is the double cover of $\,\PP^{2}$ branched over the
curve 
\[
C_{6}\colon \ell_{1}\ell_{2}\ell_{3}\ell_{4}q_{2}-f_{3}^{2}=0.
\]
The moduli space of K3 surfaces $X$ with $\NS X)\simeq[6]\oplus\mathbf{A}_{2}^{\oplus2}$
is unirational.
\end{prop}

\begin{proof}
Let us consider the map
\[
\Phi\colon H^{0}(\PP^{2},\OO(1))^{\oplus4}\oplus H^{0}(\PP^{2},\OO(2))\oplus H^{0}(\PP^{2},\OO(3))\to H^{0}(\PP^{2},\OO(6))
\]
defined by 
\[
w:=(\ell_{1},\ell_{2},\ell_{3},\ell_{4},q_{2},f_{3})\mapsto f_{6,w}:=\ell_{1}\ell_{2}\ell_{3}\ell_{4}q_{2}-f_{3}^{2}.
\]
Suppose that $w$ is  general , so that the double cover branched over
$C_{w}=\{f_{6,w}=0\}$ is a K3 surface. The pull-backs of $\ell_{k}=0$, $k=1,\dots,4$, 
and $q_{2}=0$ are pairs of $\cu$-curves, for which  Example
\ref{example=00005B6=00005D=00005Coplus=00005Cmathbf=00007BA=00007D_=00007B2=00007D^=00007B=00005Coplus2=00007D}
below shows that (for a suitable order) these curves intersect according
to the above matrix since intersection numbers are preserved for
flat families of surfaces.

The map $\Phi$ is invariant under the action of the transformations
\[(\ell_{1},\ell_{2},\ell_{3},\ell_{4},q_{2})\mapsto(\a\ell_{1},\b\ell_{2},\g\ell_{3},\d\ell_{4},\e q_{2},f_{3})\]
for $\a\b\g\d\e=1$. Since the curves $\ell_{j}=0$, $j\in\{1,\dots,4\}$, 
and $q_{2}=0$ are the images of the $10$ $\cu$-curves in the
double cover $X$ ramified over $f_{6,w}$, we have 
\[
\Phi(w)=\Phi(w')
\]
if and only if, up to permutation and the action of the above transformations,
$w=\l w'$ for some $\l\neq0$, where 
\[
\l\cdot(\ell_{1},\ell_{2},\ell_{3},\ell_{4},q_{2},f_{3})=(\l\ell_{1},\l\ell_{2},\l\ell_{3},\l\ell_{4},\l^{2}q_{2},\l^{3}f_{3}).
\]
The image of $\Phi$ modulo the action of $PGL_{3}$ is therefore
a unirational space of dimension
\[
4\cdot3+6+10-4-9=15, 
\]
with an open set which is birational to the moduli
space of K3 surfaces $X$ with $\NS X)\simeq[6]\oplus\mathbf{A}_{2}^{\oplus2}$.
\end{proof}

\begin{example}
\label{example=00005B6=00005D=00005Coplus=00005Cmathbf=00007BA=00007D_=00007B2=00007D^=00007B=00005Coplus2=00007D}Let
us take 
\[
\begin{array}{l}
\ell_{1}=16y+29z,\;\ell_{2}=22x+y+27z,\;\ell_{3}=17y+29z,\;\ell_{4}=25x+29y+11z,\\
q_{2}=31x^{2}+16xy+11y^{2}+11xz+9yz+17z^{2},\\
f_{3}=6x^{3}+8x^{2}y+17xy^{2}+5y^{3}+17x^{2}z+18xyz+19y^{2}z+16xz^{2}+8yz^{2}+23z^{3}.
\end{array}
\]
Let $X$ be the associated K3 surface, and let $X_{q}$ be its reduction
over the field $\FF_{q}$. Using the Tate conjecture, one computes
that the Picard numbers of $X_{17^{2}}$ and $X_{23^{2}}$ are both $6$,
and using the Artin--Tate conjecture and Van Luijk's trick, one obtains
that $X$ has Picard number $5$. One can check that the $10$ $\cu$-curves
above the lines $\ell_{k}=0$ and the conic $q_{2}=0$ have the above
intersection matrix and $\NS X)\simeq[6]\oplus\mathbf{A}_{2}^{\oplus2}$. 
\end{example}

\section{Rank 6 lattices}

\subsection{The lattice $\boldsymbol{U(3)\oplus\mathbf{A}_{2}^{\oplus2}}$}

Let $X$ be a K3 surface with N\'eron--Severi lattice 
\[
\NS X)=U(3)\oplus\mathbf{A}_{2}^{\oplus2}.
\]
The surface contains $12$ $\cu$-curves $A_{1},\dots,A_{12}$,  with
configuration 

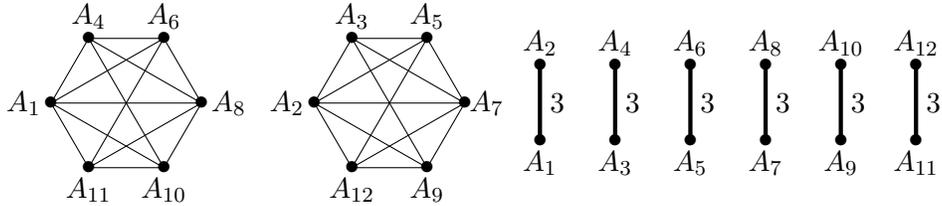
\begin{center}
\begin{tikzpicture}[scale=1]

\draw  [ultra thick] (5.5,-0.5) -- (5.5,0.5);
\draw  [ultra thick] (6.5,-0.5) -- (6.5,0.5);
\draw  [ultra thick] (7.5,-0.5) -- (7.5,0.5);
\draw  [ultra thick] (8.5,-0.5) -- (8.5,0.5);
\draw  [ultra thick] (9.5,-0.5) -- (9.5,0.5);
\draw  [ultra thick] (10.5,-0.5) -- (10.5,0.5);

\draw  (5.5,-0.5) node {$\bullet$};
\draw  (6.5,-0.5) node {$\bullet$};
\draw  (7.5,-0.5) node {$\bullet$};
\draw  (8.5,-0.5) node {$\bullet$};
\draw  (9.5,-0.5) node {$\bullet$};
\draw  (10.5,-0.5) node {$\bullet$};
\draw   (5.5,0.5) node {$\bullet$};
\draw   (6.5,0.5) node {$\bullet$};
\draw   (7.5,0.5) node {$\bullet$};
\draw   (8.5,0.5) node {$\bullet$};
\draw   (9.5,0.5) node {$\bullet$};
\draw   (10.5,0.5) node {$\bullet$};

\draw  (5.5,-0.5) node [below]{$A_{1}$};
\draw  (6.5,-0.5) node [below]{$A_{3}$};
\draw  (7.5,-0.5) node [below]{$A_{5}$};
\draw  (8.5,-0.5) node  [below]{$A_{7}$};
\draw  (9.5,-0.5) node [below]{$A_{9}$};
\draw  (10.5,-0.5) node [below]{$A_{11}$};
\draw   (5.5,0.5) node [above]{$A_{2}$};
\draw   (6.5,0.5) node [above]{$A_{4}$};
\draw   (7.5,0.5) node [above]{$A_{6}$};
\draw   (8.5,0.5) node [above]{$A_{8}$};
\draw   (9.5,0.5) node [above]{$A_{10}$};
\draw   (10.5,0.5) node  [above]{$A_{12}$};

\draw  (5.5,0) node [right]{$3$};
\draw  (6.5,0) node [right]{$3$};
\draw  (7.5,0) node [right]{$3$};
\draw  (8.5,0) node  [right]{$3$};
\draw  (9.5,0) node [right]{$3$};
\draw  (10.5,0) node [right]{$3$};

\draw (1,0) -- (0.5,0.86);
\draw (1,0) -- (-0.5,0.86);
\draw (1,0) -- (-0.5,-0.86);
\draw (0.5,0.86) -- (-0.5,0.86);
\draw (-1,0) -- (-0.5,0.86);
\draw (-1,0) -- (-0.5,-0.86);
\draw (-1,0) -- (0.5,0.86);
\draw (-1,0) -- (0.5,-0.86);
\draw (-1,0) -- (1,0);
\draw (-0.5,-0.86) -- (0.5,-0.86);
\draw (1,0) -- (0.5,-0.86);
\draw (-0.5,-0.86) -- (0.5,0.86);
\draw (-0.5,0.86) -- (0.5,-0.86);

\draw (1,0) node {$\bullet$};
\draw (0.5,0.86) node {$\bullet$};
\draw (-1,0) node {$\bullet$};
\draw (-0.5,0.86) node {$\bullet$};
\draw (-0.5,-0.86) node {$\bullet$};
\draw (0.5,-0.86)  node {$\bullet$};

\draw (1,0) node  [right]{$A_{8}$};
\draw (0.5,0.86) node [above]{$A_{6}$};
\draw (-1,0) node  [left]{$A_{1}$};
\draw (-0.5,0.86) node  [above]{$A_{4}$};
\draw (-0.5,-0.86) node  [below]{$A_{11}$};
\draw (0.5,-0.86) node   [below]{$A_{10}$};

\draw (1+3.5,0) -- (0.5+3.5,0.86);
\draw (1+3.5,0) -- (-0.5+3.5,0.86);
\draw (1+3.5,0) -- (-0.5+3.5,-0.86);
\draw (0.5+3.5,0.86) -- (-0.5+3.5,0.86);
\draw (-1+3.5,0) -- (-0.5+3.5,0.86);
\draw (-1+3.5,0) -- (-0.5+3.5,-0.86);
\draw (-1+3.5,0) -- (0.5+3.5,0.86);
\draw (-1+3.5,0) -- (0.5+3.5,-0.86);
\draw (-1+3.5,0) -- (1+3.5,0);
\draw (-0.5+3.5,-0.86) -- (0.5+3.5,-0.86);
\draw (1+3.5,0) -- (0.5+3.5,-0.86);
\draw (-0.5+3.5,-0.86) -- (0.5+3.5,0.86);
\draw (-0.5+3.5,0.86) -- (0.5+3.5,-0.86);

\draw (1+3.5,0) node {$\bullet$};
\draw (0.5+3.5,0.86) node {$\bullet$};
\draw (-1+3.5,0) node {$\bullet$};
\draw (-0.5+3.5,0.86) node {$\bullet$};
\draw (-0.5+3.5,-0.86) node {$\bullet$};
\draw (0.5+3.5,-0.86)  node {$\bullet$};

\draw (1+3.43,0) node  [right]{$A_{7}$};
\draw (0.5+3.5,0.86) node [above]{$A_{5}$};
\draw (-1+3.5,0) node  [left]{$A_{2}$};
\draw (-0.5+3.5,0.86) node  [above]{$A_{3}$};
\draw (-0.5+3.5,-0.86) node  [below]{$A_{12}$};
\draw (0.5+3.5,-0.86) node   [below]{$A_{9}$};

\end{tikzpicture}
\end{center} 
where for clarity we represented  the same curve several times.

The N\'eron--Severi lattice is generated by $A_{1}$, $A_{3}$, $A_{5}$, $A_{7}$, $A_{9}$, $A_{11}$.
The divisor 
\[
D_{2}=-A_{1}+A_{3}+A_{5}+A_{7}+A_{9}-A_{11}
\]
is ample, base-point free, of square $2$, and the $12$ $(-2)$-curves
$A_{1},\dots,A_{12}$ on $X$ are of degree $1$ for $D_{2}$. We
have 
\[
D_{2}\equiv A_{2k-1}+A_{2k}
\]
 for $k\in\{1,\dots,6\}$, and by using the linear system $|D_{2}|$,
we obtain the following. 

\begin{prop}
The surface $X$ is a double cover $\pi\colon X\to\PP^{2}$ of $\PP^{2}$
branched over a smooth sextic curve $C_{6}$. The $12$ $\cu$-curves
on $X$ are pull-backs of $6$ lines that are tritangent to the sextic
curve. There exist linear forms $\ell_{1},\dots,\ell_{6}\in H^{0}(\PP^{2},\OO(1))$
and a cubic $f_{3}\in H^{0}(\PP^{2},\OO(3))$ such that the curve
$C_{6}$ is given by 
\[
\ell_{1}\ell_{2}\ell_{3}\ell_{4}\ell_{5}\ell_{6}-f_{3}^{2}=0.
\]
The moduli space of K3 surfaces X with $\NS X)\simeq U(3)\oplus\mathbf{A}_{2}^{\oplus2}$
is unirational. 
\end{prop}

\begin{proof}
Let us consider the map
\[
\Phi\colon H^{0}(\PP^{2},\OO(1))^{\oplus6}\oplus H^{0}(\PP^{2},\OO(3))\to H^{0}(\PP^{2},\OO(6))
\]
defined by 
\[
w:=(\ell_{1},\ell_{2},\ell_{3},\ell_{4},\ell_{5},\ell_{6},f_{3})\mapsto f_{6,w}:=\ell_{1}\ell_{2}\ell_{3}\ell_{4}\ell_{5}\ell_{6}-f_{3}^{2}.
\]
When the sextic curve $C_{w}\colon f_{6,w}=0$ is smooth, the K3 surface
$Y_{w}$ which is the double cover of $\PP^{2}$ branched over $C_{w}$
contains $12$ $\cu$-curves over the $6$ tritangent lines $\{\ell_{k}=0\}$, $k=1,\dots,6$.
Example \ref{exa:As-an-exampleU3A2exp2} below shows that these $\cu$-curves
generate  a sublattice isometric to $U(3)\oplus\mathbf{A}_{2}^{\oplus2}$.
The map $\Phi$ is invariant under the action of the transformations
\[
(\ell_{1},\ell_{2},\ell_{3},\ell_{4},\ell_{5},\ell_{6})\mapsto(\a\ell_{1},\b\ell_{2},\g\ell_{3},\d\ell_{4},\e\ell_{5},\mu\ell_{6},f_{3})
\]
 for $\a\b\g\d\e\mu=1$. Since the curves $\ell_{j}=0$, $j\in\{1,\dots,6\}$, 
are the images of the $12$ $\cu$-curves in the double cover ramified
over $f_{6}$, we have 
\[
\Phi(w)=\Phi(w')
\]
if and only if, up to permutation and the action of the above transformations,
$w=\l\cdot w'$ for some $\l\neq0$, where 
\[
\l\cdot(\ell_{1},\ell_{2},\ell_{3},\ell_{4},\ell_{5},\ell_{6},f_{3})=(\l\ell_{1},\l\ell_{2},\l\ell_{3},\l\ell_{4},\l\ell_{5},\l\ell_{6},\l^{3}f_{3}).
\]
The image of $\Phi$ modulo the action of $PGL_{3}$ is therefore
a unirational space of dimension
\[
6\cdot3+10-5-9=14,
\]
with an open set which is birational to the moduli
space of K3 surfaces $X$ with $\NS X)\simeq U(3)\oplus\mathbf{A}_{2}^{\oplus2}$. 
\end{proof}

\begin{example}
\label{exa:As-an-exampleU3A2exp2}As an example, one can take 
\[
\begin{array}{l}
\ell_{1}=15x+5y+2z,\;\ell_{2}=20x+12y+17z,\;\ell_{3}=22x+22y+16z,\\
\ell_{4}=8x+y+7z,\;\ell_{5}=10x+15y+15z,\;\ell_{6}=6x+22y+20z,\\
f_{3}=18x^{2}y+13xy^{2}+16y^{3}+16x^{2}z+17xyz+22y^{2}z+6xz^{2}+10yz^{2}+20z^{3}; 
\end{array}
\]
this gives a smooth K3 surface $X$. Its reduction modulo $23$ is
a smooth surface $X_{23}$. Using the Artin--Tate conjecture, one finds
that its Picard number is $4$ and the discriminant of the N\'eron--Severi
lattice has order~$3^{4}$. Using the pull-back of the six lines
$\{\ell_{k}=0\}$, one can check that the N\'eron--Severi lattice contains
the lattice $U(3)\oplus\mathbf{A}_{2}^{\oplus2}$ with discriminant
$3^{4}$. Thus, one has $\NS X_{23})\simeq U(3)\oplus\mathbf{A}_{2}^{\oplus2}$,
and that implies that the N\'eron--Severi lattice of $X$ is also isometric
to $U(3)\oplus\mathbf{A}_{2}^{\oplus2}.$ 
\end{example}

\begin{rem}
In \cite[Proof of Theorem 6.4.1]{Nikulin}, Nikulin constructs some
surfaces $X$ with N\'eron--Severi lattice $U(3)\oplus\mathbf{A}_{2}^{\oplus2}$
as the minimal desingularization of a triple cover of a smooth quadric
$Q\subset\PP^{3}$ branched over a $(3,3)$-curve $C$ with two singularities
$\mathbf{a}_{1}$. In particular, the automorphism group of such a surface
$X$ has order at least $6$. According to \cite{Kondo}, the  general 
surface with $\NS X)\simeq U(3)\oplus\mathbf{A}_{2}^{\oplus2}$ has
automorphism group isomorphic to $\ZZ/2\ZZ$. 
\end{rem}

\subsection{The lattice $\boldsymbol{U(4)\oplus\mathbf{D}_{4}}$}

Let $X$ be a K3 surface with  N\'eron--Severi lattice 
\[
\NS X)=U(4)\oplus\mathbf{D}_{4}.
\]
The surface $X$ contains nine $(-2)$-curves $A_{1},\dots,A_{9}$, 
with configuration 

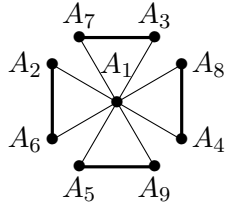
\begin{center}
\begin{tikzpicture}[scale=1]

\draw [very thick] (0.86,-0.5) -- (0.86,0.5);
\draw [very thick] (-0.5,-0.86) -- (0.5,-0.86);
\draw (0,0) -- (0.5,-0.86);
\draw (0,0) -- (-0.5,-0.86);
\draw (0,0) -- (0.86,0.5);
\draw (0,0) -- (0.86,-0.5);

\draw [very thick] (-0.86,0.5) -- (-0.86,-0.5);
\draw [very thick] (0.5,0.86) -- (-0.5,0.86);
\draw (0,0) -- (-0.5,0.86);
\draw (0,0) -- (0.5,0.86);
\draw (0,0) -- (-0.86,-0.5);
\draw (0,0) -- (-0.86,0.5);

\draw (0.5,0.86) node {$\bullet$};
\draw (-0.5,0.86) node {$\bullet$};
\draw (-0.5,-0.86) node {$\bullet$};
\draw (0.5,-0.86)  node {$\bullet$};
\draw (0,0) node {$\bullet$};
\draw (0.86,-0.5) node {$\bullet$};
\draw (0.86,0.5) node {$\bullet$};
\draw (-0.86,-0.5) node {$\bullet$};
\draw (-0.86,0.5) node {$\bullet$};

\draw (0.5,0.86) node [above]{$A_{3}$};
\draw (-0.5,0.86) node  [above]{$A_{7}$};
\draw (-0.5,-0.86) node  [below]{$A_{5}$};
\draw (0.5,-0.86) node   [below]{$A_{9}$};
\draw (0,0.2) node  [above]{$A_{1}$};
\draw (0.86,-0.5) node [right]{$A_{4}$};
\draw (0.86,0.5) node [right]{$A_{8}$};
\draw (-0.86,-0.5) node [left]{$A_{6}$};
\draw (-0.86,0.5) node  [left]{$A_{2}$};

\end{tikzpicture}
\end{center} 

The class 
\[
D_{6}=A_{1}+2A_{2}+2A_{6}
\]
is ample with $D_{6}^{2}=6$. The divisor $E=A_{2}+A_{6}$ is the
class of a fiber, and $D_{6}=2E+A_{9}$. For $j>1$, one has $EA_{j}=0$
and $D_{6} \cdot A_{j}=1$; moreover, $EA_{1}=D_{6} \cdot A_{1}=2$. The curves $A_{2},\dots,A_{9}$
are such that the divisors
\[
A_{k}+A_{k+4},\,\,k\in\{2,3,4,5\}
\]
are the reducible fibers of the elliptic fibration  induced by $E$. 

The divisor 
\[
D_{2}=A_{1}+A_{2}+A_{6}=E+A_{1}
\]
is nef, of square $2$, base-point free, with $D_{2} \cdot A_{1}=0$ and
$D_{2} \cdot A_{j}=1$ for $j\geq2$. We have 
\[
D_{2}\equiv A_{1}+A_{3}+A_{7}\equiv A_{1}+A_{4}+A_{8}\equiv A_{1}+A_{5}+A_{9}.
\]
By using the linear system $|D_{2}|$, we obtain the following. 

\begin{prop}
The K3 surface $X$ is a double cover of the plane branched over a
nodal sextic curve such that there exist four lines through the node
that are tangent to the sextic at other intersection points. 
\end{prop}

\subsection{The lattice $\boldsymbol{U\oplus\mathbf{A}_{4}}$}

Let $X$ be a K3 surface with  N\'eron--Severi lattice 
\[
\NS X)=U\oplus\mathbf{A}_{4}.
\]
The surface $X$ contains six $(-2)$-curves $A_{1},\dots,A_{6}$
with dual graph

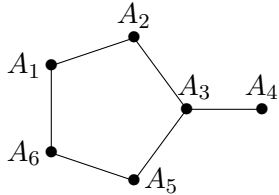
\begin{center}
\begin{tikzpicture}[scale=1]

\draw (0.3,-0.95) -- (1,0);
\draw (1,0) -- (0.3,0.95);
\draw (0.3,0.95) -- (-0.8,0.58);
\draw (-0.8,0.58) -- (-0.8,-0.58);
\draw (-0.8,-0.58) -- (0.3,-0.95);
\draw (2,0) -- (1,0);


\draw (1,0) node {$\bullet$};
\draw (0.3,-0.95) node {$\bullet$};
\draw (0.3,0.95) node {$\bullet$};
\draw (-0.8,-0.58) node {$\bullet$};
\draw (-0.8,0.58) node {$\bullet$};
\draw (2,0) node {$\bullet$};

\draw (1.1,0) node [above]{$A_{3}$};
\draw (0.3,0.95) node [above]{$A_{2}$};
\draw (-0.8,0.58) node [left]{$A_{1}$};
\draw (-0.8,-0.58) node [left]{$A_{6}$};
\draw (0.3,-0.95) node [right]{$A_{5}$};
\draw (2,0) node [above]{$A_{4}$};

\end{tikzpicture}
\end{center} 

These six curves generate  the N\'eron--Severi lattice. The class 
\[
D_{50}=(8,9,11,5,9,8)
\]
 in the basis $A_{1},\dots,A_{6}$ is ample, and $D_{50}^{2}=50$. The
six $(-2)$-curves have degree $1$ for $D_{50}$. The class 
\[
F=A_{1}+A_{2}+A_{3}+A_{5}+A_{6}
\]
is the class of a fiber of type $I_{5}$. The curve $A_{4}$ is a
section of the elliptic fibration. The divisor 
\[
D_{2}=2F+A_{4}
\]
 is nef, of square $2$, with $D_{2} \cdot A_{j}=0$ for $j\neq3$. It has
base points, but $D_{8}=2D_{2}$ is base-point free. 

\begin{prop}
The linear system $|D_{8}|$ defines a morphism $\varphi\colon X\to\mathbf{F}_{4}$ that is branched over $A_{4}$; the image of $A_{4}$ is the unique section $s$ of the Hirzebruch surface $\mathbf{F}_{4}$ with $s^2=-4$. The branch curve of $\varphi$ is the disjoint union of $s$ and a curve $B\in|3s+12f|$.  The curve $B$ has a unique singularity $q$ of type $\mathbf{a}_{4}$ which is the image of $A_{1}$, $A_{2}$,  $A_{5}$, $A_{6}$, and the image of $A_{3}$ by $\varphi$ is the fiber $f$ through $q$. The local intersection number of $f$ and $B$ at $q$ is $2$; in other words, $f$ is transverse to the branch of $B$ at $q$.
\end{prop}

\begin{proof}
We apply Theorem \ref{thm:SaintDonat-2}, case a) i).
\end{proof}

\subsection{The lattice $\boldsymbol{U\oplus\mathbf{A}_{1}\oplus\mathbf{A}_{3}}$}

Let $X$ be a K3 surface with  N\'eron--Severi lattice 
\[
\NS X)=U\oplus A_{1}\oplus A_{3}.
\]

The surface contains seven $(-2)$-curves $A_{1},\dots,A_{7}$ with
configuration

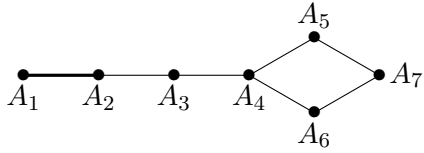
\begin{center}
\begin{tikzpicture}[scale=1]

\draw (0,0) -- (2,0);
\draw (2+0.866*2,0) -- (2+0.866*2-0.866,0.5);
\draw (2+0.866*2,0) -- (2+0.866*2-0.866,-0.5);
\draw [very thick] (-1,0) -- (0,0);

\draw (2,0) -- (2+0.866,0.5);
\draw (2,0) -- (2+0.866,-0.5);

\draw (0,0) node {$\bullet$};
\draw (1,0) node {$\bullet$};
\draw (2,0) node {$\bullet$};
\draw (-1,0) node {$\bullet$};
\draw (2+0.866,0.5) node {$\bullet$};
\draw (2+0.866,-0.5) node {$\bullet$};
\draw (2+0.866*2,0) node {$\bullet$};

\draw (0,0) node [below]{$A_{2}$};
\draw (1,0) node [below]{$A_{3}$};
\draw (-1,0) node [below]{$A_{1}$};
\draw (2,0) node [below]{$A_{4}$};
\draw (2+0.866,0.5) node [above]{$A_{5}$};
\draw (2+0.866,-0.5) node [below]{$A_{6}$};
\draw (2+0.866*2,0) node [right]{$A_{7}$};

\end{tikzpicture}
\end{center} 

The class
\[
D_{42}=6A_{1}+8A_{2}+5A_{3}+3A_{4}+A_{5}+A_{6}
\]
 is ample, of square $42$. The divisors $F_{1}=A_{4}+A_{5}+A_{6}+A_{7}$
and $F_{2}=A_{1}+A_{2}$ are two fibers of an elliptic fibration  for
which the class $A_{3}$ is a section. The class
\[
D_{2}=A_{1}+2A_{2}+2A_{3}+2A_{4}+A_{5}+A_{6}\equiv A_{2}+2A_{3}+3A_{4}+2A_{5}+2A_{6}+A_{7}
\]
is nef, of square $2$,  base-point free, with $D_{2} \cdot A_{1}=D_{2} \cdot A_{7}=2$ 
and $D_{2} \cdot A_{j}=0$ for $j\in\{2,3,4,5,6\}$. Let $\eta\colon X\to\PP^{2}$
be the associated double cover and $C_{6}$ be the branch curve. 

\begin{prop}
The curve $C_{6}$ has a $\mathbf{d}_{5}$ singularity $q$ onto which
the curves $A_{j}$, $j=2,\dots,6$ are contracted. From the above equivalence
relation between divisors, we see that the images of $A_{1}$ and
$A_{7}$ are the two tangents to the branches of the singularity $q$.
\end{prop}

\subsection{The lattice $\boldsymbol{U\oplus\mathbf{A}_{2}^{\oplus2}}$}

Let $X$ be a K3 surface with  N\'eron--Severi lattice 
\[
\NS X)=U\oplus\mathbf{A}_{2}^{\oplus2}.
\]
The surface contains seven $(-2)$-curves $A_{1},\dots,A_{7}$ with
configuration

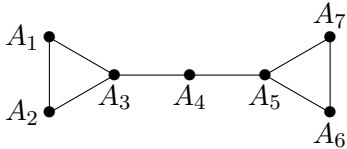
\begin{center}
\begin{tikzpicture}[scale=1]

\draw (0,0) -- (2,0);
\draw (0,0) -- (-0.866,0.5);
\draw (0,0) -- (-0.866,-0.5);
\draw (-0.866,0.5) -- (-0.866,-0.5);

\draw (2,0) -- (2+0.866,0.5);
\draw (2,0) -- (2+0.866,-0.5);
\draw (2+0.866,0.5) -- (2+0.866,-0.5);

\draw (0,0) node {$\bullet$};
\draw (1,0) node {$\bullet$};
\draw (2,0) node {$\bullet$};
\draw (-0.866,-0.5) node {$\bullet$};
\draw (-0.866,0.5) node {$\bullet$};
\draw (2+0.866,0.5) node {$\bullet$};
\draw (2+0.866,-0.5) node {$\bullet$};

\draw (0,0) node [below]{$A_{3}$};
\draw (1,0) node [below]{$A_{4}$};
\draw (2,0) node [below]{$A_{5}$};
\draw (-0.866,-0.5) node [left]{$A_{2}$};
\draw (-0.866,0.5) node [left]{$A_{1}$};
\draw (2+0.866,0.5) node [above]{$A_{7}$};
\draw (2+0.866,-0.5) node [below]{$A_{6}$};

\end{tikzpicture}
\end{center} 

The class 
\[
D_{20}=5A_{1}+5A_{2}+6A_{3}+3A_{4}+A_{5}
\]
is ample, of square $20$, and the curves $A_{j}$ have degree
$1$ for $D_{20}$. The divisor
\[
D_{4}=2A_{1}+2A_{2}+3A_{3}+2A_{4}+A_{5}\equiv A_{3}+2A_{4}+3A_{5}+2A_{6}+2A_{7}
\]
is nef, of square $4$, base-point free, with $D_{4} \cdot A_{j}=1$ for $j\in\{1,2,6,7\}$ and $D_{4} \cdot A_{j}=0$ for $j\in\{3,4,5\}$. Since $D_{4}F_{1}=2$,
the linear system $|D_{4}|$ is hyperelliptic. By using the linear
system $|D_{4}|$, we obtain the following. 

\begin{prop}
The surface is a double cover of $\,\mathbf{F}_{2}$ branched over the unique
section $s$ with $s^2=-2$ and a curve $B\in|3s+8f|$. The fibers $f_{p}$, $f_{q}$
through the two intersection points $p$, $q$ of $s$ and $B$ are tangent
to the curve $B$ at another intersection point with $B$. The image
of $A_{1}$, $A_{2}$ is $f_{p}$, the image of $A_{6}$, $A_{7}$ is $f_{q}$, 
and the curves $A_{3}$, $A_{5}$ are mapped to the points $p$, $q$. 
\end{prop}

\begin{proof}
This comes from Theorem \ref{thm:SaintDonat-2}, case a) iii) v).
\end{proof}

\subsection{The lattice $\boldsymbol{U\oplus\mathbf{A}_{1}^{\oplus2}\oplus\mathbf{A}_{2}}$}

Let $X$ be a K3 surface with  N\'eron--Severi lattice 
\[
\NS X)=U\oplus\mathbf{A}_{1}^{\oplus2}\oplus\mathbf{A}_{2}.
\]
The surface contains eight $(-2)$-curves $A_{1},\dots,A_{8}$, with
configuration 
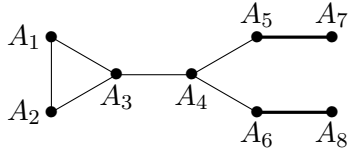
\begin{center}
\begin{tikzpicture}[scale=1]

\draw (0,0) -- (1,0);
\draw (0,0) -- (-0.866,0.5);
\draw (0,0) -- (-0.866,-0.5);
\draw (-0.866,0.5) -- (-0.866,-0.5);
\draw [very thick] (1+0.866,0.5) -- (2+0.866,0.5);
\draw [very thick] (1+0.866,-0.5) -- (2+0.866,-0.5);
\draw (1,0) -- (1+0.866,0.5);
\draw (1,0) -- (1+0.866,-0.5);

\draw (0,0) node {$\bullet$};
\draw (1,0) node {$\bullet$};
\draw (1+0.866,-0.5) node {$\bullet$};
\draw (1+0.866,0.5) node {$\bullet$};
\draw (-0.866,-0.5) node {$\bullet$};
\draw (-0.866,0.5) node {$\bullet$};
\draw (2+0.866,0.5) node {$\bullet$};
\draw (2+0.866,-0.5) node {$\bullet$};

\draw (0,0) node [below]{$A_{3}$};
\draw (1,0) node [below]{$A_{4}$};
\draw (1+0.866,0.5) node [above]{$A_{5}$};
\draw (1+0.866,-0.5) node [below]{$A_{6}$};
\draw (-0.866,-0.5) node [left]{$A_{2}$};
\draw (-0.866,0.5) node [left]{$A_{1}$};
\draw (2+0.866,0.5) node [above]{$A_{7}$};
\draw (2+0.866,-0.5) node [below]{$A_{8}$};

\end{tikzpicture}
\end{center} 

The class 
\[
D_{18}=4A_{1}+4A_{2}+5A_{3}+3A_{4}+A_{5}+A_{6}
\]
is ample, of square $18$, with $D_{18} \cdot A_{j}=1$ for $j\leq6$ and
$D_{18} \cdot A_{7}=D_{18} \cdot A_{8}=2$. The three divisors 
\[
F=A_{1}+A_{2}+A_{3}\equiv A_{5}+A_{7}\equiv A_{6}+A_{8}
\]
are classes of fibers of type $I_{3}$ or $IV$ and of type $I_{2}+I_{2}$
or $III$; 
the $(-2)$-curve $A_{4}$ is a section of the elliptic fibration.
The class
\[
D_{2}=A_{1}+A_{2}+2A_{3}+2A_{4}+A_{5}+A_{6}
\]
is of square $2$ and base-point free. Let $\eta\colon X\to\PP^{2}$ be
the associated double cover map, and let $C_{6}$ be the branch curve.
For $j=1,\dots,8$, one has 
\[
D_{2} \cdot A_{j}=1,1,0,0,0,0,2,2,
\]
respectively. This leads to the following. 

\begin{prop}
The morphism $\eta$ contracts the curves $A_{3}$, $A_{4}$, $A_{5}$, $A_{6}$ to
a $\mathbf{d}_{4}$ singularity $q$, the curve $A_{1}+A_{2}$ is
the pull-back of the tangent line of a branch of the singularity,
a line which is tangent to the sextic at another intersection point,
and the curves $A_{7}$, $A_{8}$ are pull-backs of the lines tangent to
other branches, each of which lines meets $C_{6}$ in two other points.
\end{prop}

\subsection{The lattice $\boldsymbol{U(2)\oplus\mathbf{A}_{1}^{\oplus4}}$ and del Pezzo surfaces of degree 4}

The class of square $8$ 
\[
D_{8}=(2,2,-1,-1,-1,-1)\in U(2)\oplus\mathbf{A}_{1}^{\oplus4}
\]
has no $\cu$-vectors perpendicular to it; thus we have a marking
$\NS X)\simeq U(2)\oplus\mathbf{A}_{1}^{\oplus4}$ that maps $D_{8}$
to an ample class. There are $16$ $(-2)$-curves $A_{1},\dots,A_{16}$
on $X$; all have degree $2$ with respect to $D_{8}$. The class
$D_{8}$ is base-point free and not hyperelliptic; moreover, the linear
system $|D_{8}|$ embeds $X$ as a degree $8$ surface in~$\PP^{5}$.
One can describe the configuration of the $\cu$-curves on $X$ as
follows:

Let $p_{1},\dots,p_{5}$ be five points in general position in $\PP^{2}$
(by which we mean that no three are on a line), and let $C_{6}$ be a sextic
curve that has only nodal singularities at each point $p_{j}$. Let
$\pi\colon X\to\PP^{2}$ be the minimal desingularization of the double
cover $Y\to\PP^{2}$ branched over $C_{6}$. This is a smooth K3 surface
which contains the following $16$ $(-2)$-curves:
\begin{itemize}
  \item the $(-2)$-curves $A_{1},\dots,A_{5}$ above the points $p_{1},\dots,p_{5}$,
\item the strict transform $A_{ij}$ in $X$ of the line through $p_{i}$
and $p_{j}$ for $1\leq i<j\leq5$, 
\item the strict transform $A_{0}$ of the unique conic passing through
the five points $p_{1},\dots,p_{5}$. 
\end{itemize}
From that description, one understands easily the configuration of
the $16$ $\cu$-curves $A_{j}$, $0\leq j\leq5$, and  \mbox{$A_{ij}$, $1\leq i<j\leq j\leq5$,} 
on $X$. The pull-backs of a line $L$ and the $(-2)$-curves $A_{0},\dots,A_{5}$
generate a lattice which is 
\[
[2]\oplus\mathbf{A}_{1}^{\oplus5}\simeq U(2)\oplus\mathbf{A}_{1}^{\oplus4}.
\]
The involution from the double cover fixes a smooth curve of genus
$10-5=5$. 

Let $p_{1},\dots,p_{5}$ be five points in general position in the
plane, and let $C_{6}$ be a  general  sextic curve with nodes at the
points $p_{i}$. 

\begin{prop}
The double cover of the blow-up of $\,\PP^{2}$ at the points $p_{i}$ branched
over the strict transform of $\,C_{6}$ is a K3 surface with $\NS X)\simeq U(2)\oplus\mathbf{A}_{1}^{\oplus4}$.
The moduli space of these K3 surfaces is therefore unirational.
\end{prop}

\subsection{The lattice $\boldsymbol{U\oplus\mathbf{A}_{1}^{\oplus4}}$}

Let $X$ be a K3 surface with N\'eron--Severi lattice 
\[
\NS X)=U\oplus\mathbf{A}_{1}^{\oplus4}.
\]
The surface contains nine $(-2)$-curves $A_{1},\dots,A_{9}$, with
dual graph
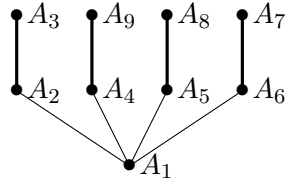
\begin{center}
\begin{tikzpicture}[scale=1]

\draw (0,0) -- (0.5,1);
\draw (0,0) -- (1.5,1);
\draw (0,0) -- (-0.5,1);
\draw (0,0) -- (-1.5,1);

\draw [very thick] (0.5,2) -- (0.5,1);
\draw [very thick] (1.5,2) -- (1.5,1);
\draw [very thick] (-0.5,2) -- (-0.5,1);
\draw [very thick] (-1.5,2) -- (-1.5,1);


\draw (0,0) node {$\bullet$};
\draw (0.5,1) node {$\bullet$};
\draw (1.5,1) node {$\bullet$};
\draw (-0.5,1) node {$\bullet$};
\draw (-1.5,1) node {$\bullet$};
\draw (0.5,2) node {$\bullet$};
\draw (1.5,2) node {$\bullet$};
\draw (-0.5,2) node {$\bullet$};
\draw (-1.5,2) node {$\bullet$};

\draw (0,0) node [right]{$A_{1}$};
\draw (-1.5,1) node [right]{$A_{2}$};
\draw (-1.5,2) node [right]{$A_{3}$};
\draw (-0.5,1) node [right]{$A_{4}$};
\draw (0.5,1) node [right]{$A_{5}$};
\draw (1.5,1) node [right]{$A_{6}$};
\draw (1.5,2) node [right]{$A_{7}$};
\draw (0.5,2) node [right]{$A_{8}$};
\draw (-0.5,2) node [right]{$A_{9}$};

\end{tikzpicture}
\end{center} 

The $\cu$-curves $A_{1},\dots,A_{6}$ generate the N\'eron--Severi lattice, 
and the classes of the remaining $(-2)$-curves are 
\[
A_{7}\equiv A_{2}+A_{3}-A_{6},\quad A_{8}\equiv A_{2}+A_{3}-A_{5},\quad A_{9}\equiv A_{2}+A_{3}-A_{4}.
\]
 The class 
\[
D_{16}=3A_{1}+4A_{2}+3A_{3}+A_{4}+A_{5}+A_{6}
\]
 is ample and of square $16$. The divisors
\[
E_{1}=A_{2}+A_{3},\, A_{4}+A_{9},A_{5}+A_{8},\,A_{6}+A_{7}
\]
are the reducible fibers of an elliptic fibration such that $A_{1}$
is a section since one has 
\[
A_{j} \cdot A_{1}=1\text{ for }j\in\{2,4,5,6\}.
\]
The class 
\[
E_{2}=2A_{1}+A_{2}+A_{4}+A_{5}+A_{6}
\]
is the unique reducible fiber of another fibration. The effective
divisor
\[
D_{2}=2A_{1}+A_{2}+A_{3}+A_{4}+A_{5}+A_{6}
\]
is nef, of square $D_{2}^{2}=2$. The system $|D_{2}|$ is base-point
free, and $D_{2} \cdot A_{j}=0$ if and only if $j\in\{1,3,4,5,6\}$; else
$D_{2} \cdot A_{j}=2$. Let $\pi\colon X\to\PP^{2}$ be the double cover map associated
to $|D_{2}|$. Since $D_{2} \cdot A_{j}=0$ for $j\in\{1,3,4,5,6\}$, the
image of $A_{2}$ by the double cover map is a line $L_{2}$ such
that $D_{2}=\pi^{*}L_{2}$. The intersection matrix of the curves
$A_{j}$ for $j\in\{1,3,4,5,6\}$ reveals that the sextic branch curve
has a $\mathbf{d}_{4}$ singularity and a node $\mathbf{a}_{1}$, and
$L_{2}$ is the line through these two singularities. The node is
resolved on the double cover by $A_{3}$; the $\mathbf{d}_{4}$ singularity
is resolved by the union of the curves $A_{1}$, $A_{4}$, $A_{5}$, $A_{6}$
with $A_{1} \cdot A_{4}=A_{1} \cdot A_{5}=A_{1} \cdot A_{6}=1$. We have
\[
\begin{array}{c}
D_{2}\equiv2A_{1}+A_{4}+A_{5}+2A_{6}+A_{7},\\
D_{2}\equiv2A_{1}+A_{4}+2A_{5}+A_{6}+A_{8},\\
D_{2}\equiv2A_{1}+2A_{4}+A_{5}+A_{6}+A_{9};\\
\end{array}
\]
thus the images of $A_{7}$, $A_{8}$, $A_{9}$ are the three  tangent lines
$L_{7}$, $L_{8}$, $L_{9}$ of the $\mathbf{d}_{4}$ singularity. 

Let $C_{6}$ be a  general  sextic plane curve with a $\mathbf{d}_{4}$
singularity and a node. We denote by $L_{2}$ the line through the
node and the $\mathbf{d}_{4}$ singularity. Let $L_{7}$, $L_{8}$, $L_{9}$
be the three lines tangent  to the sextic at the $\mathbf{d}_{4}$ singularity.
Let $Z\to\PP^{2}$ be the embedded desingularization of $C_{6}$: this
the blow-up of $\PP^{2}$ at the node of $C_{6}$ (with the exceptional
divisor denoted by $L_{3}$), and over the ${\bf d}_{4}$ singularity, 
there are four blow-ups, producing four exceptional curves $L_{1}$, $L_{4}$, $L_{5}$, $L_{6}$
such that 
\[
L_{1}^{2}=-4,\quad L_{1}\cdot L_{j}=1,\;L_{j}^{2}=-1,\;L_{j}\cdot L_{k}=0\quad\forall j,k\in\{4,5,6\}\text{ with } j\neq k.
\]
The strict transforms of the $L_j$ ($j\in\{2,7,8,9\}$) are disjoint curves $\bar{L}_{j}$, and (up
to re-ordering) 
\[
L_{4}\cdot \bar{L}_{9}=1,\quad L_{5}\cdot \bar{L}_{8}=1,\quad L_{6}\cdot \bar{L}_{7}=1.
\]
From the above discussion, we obtain the following. 

\begin{prop}
The curve $\bar{C}_{6}+L_{1}$ is the branch locus of a double cover
$Y\to Z$. The surface $Y$ is a smooth K3 surface, and the pull-backs 
of the curves $L_{1}$, $\bar{L}_{2}$, $L_{3}$, $L_{4}$ , $L_{5}$, $L_{6}$, $\bar{L}_{7}$, $\bar{L}_{8}$, $\bar{L}_{9}$
are $(-2)$-curves $A_{1},\dots,A_{9}$ such that the lattice generated
by these curves is $U\oplus\mathbf{A}_{1}^{\oplus4}$.
\end{prop}

\subsection{The lattice $\boldsymbol{U(2)\oplus\mathbf{D}_{4}}$}

Let $X$ be a K3 surface with N\'eron--Severi lattice 
\[
\NS X)=U(2)\oplus\mathbf{D}_{4}.
\]

The surface $X$ contains six $(-2)$-curves, with dual graph
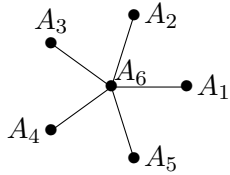
\begin{center}
\begin{tikzpicture}[scale=1]

\draw (0,0) -- (1,0);
\draw (0,0) -- (0.3,-0.95);
\draw (0,0) -- (0.3,0.95);
\draw (0,0) -- (-0.8,-0.58);
\draw (0,0) -- (-0.8,0.58);


\draw (0,0) node {$\bullet$};
\draw (1,0) node {$\bullet$};
\draw (0.3,-0.95) node {$\bullet$};
\draw (0.3,0.95) node {$\bullet$};
\draw (-0.8,-0.58) node {$\bullet$};
\draw (-0.8,0.58) node {$\bullet$};

\draw (1,0) node [right]{$A_{1}$};
\draw (-0.1,-0.1) node [above right]{$A_{6}$};
\draw (0.3,-0.95) node [right]{$A_{5}$};
\draw (0.3,0.95) node [right]{$A_{2}$};
\draw (-0.8,-0.58) node [left]{$A_{4}$};
\draw (-0.8,0.58) node [above]{$A_{3}$};

\end{tikzpicture}
\end{center} 
These six $(-2)$-curves generate the N\'eron--Severi lattice. The
class
\[
D_{22}=4A_{6}+3\sum_{j=1}^{6}A_{j}
\]
is ample of square $22$, and the curves $A_{j}$ have degree $1$
for that polarization.

For $j\in\{1,\dots,5\}$, let $F_{j}=A_{6}+\sum_{k\neq j}A_{k}$;
one has $F_{j}^{2}=0$, and $F_{j}$ is a singular fiber of type
$I_{0}^{*}$ of an elliptic fibration  $\phi_{j}$. Moreover, $F_{j} \cdot A_{j}=2$,
so that there are no sections. 

The class
\[
D_{2}=A_{6}+\sum_{j=1}^{6}A_{j}
\]
has square $2$, with $D_{2} \cdot A_{j}=0$ if and only if $j\in\{1,\dots,5\}$ and $D_{2} \cdot A_{6}=1$. By using the linear system $|D_{2}|$, we
obtain the following. 

\begin{prop}
The K3 surface $X$ is the double cover of $\,\PP^{2}$ branched over
a sextic curve which is the union of a line $L$ and a smooth quintic
$Q$ such that $L\cup Q$ has normal crossings. 
\end{prop}

The divisor $D_{2}$ is the pull-back of the line $L$. The six $(-2)$-curves
on $X$ come from the pull-backs of the line and the five exceptional
divisors above the nodes. Conversely, the six $\cu$-curves on a
K3 surface which is the minimal desingularization of the double cover
of the plane branched over the union of a line and a quintic 
necessarily have 
the same dual graph. The moduli space of K3 surfaces with
$\NS X)=U(2)\oplus\mathbf{D}_{4}$ is unirational. 

\subsection{The lattice $\boldsymbol{U\oplus\mathbf{D}_{4}}$}

Let $X$ be a K3 surface with N\'eron--Severi lattice 
\[
\NS X)=U\oplus\mathbf{D}_{4}.
\]
The dual graph of the six $\cu$-curves on $X$ is

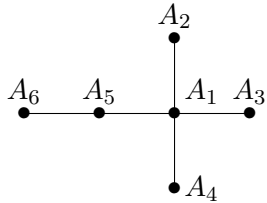
\begin{center}
\begin{tikzpicture}[scale=1]

\draw (2,0) -- (5,0);
\draw (4,1) -- (4,-1);


\draw (2,0) node {$\bullet$};
\draw (3,0) node {$\bullet$};
\draw (4,0) node {$\bullet$};
\draw (5,0) node {$\bullet$};
\draw (4,1) node {$\bullet$};
\draw (4,-1) node {$\bullet$};

\draw (2,0) node [above]{$A_{6}$};
\draw (3,0) node [above]{$A_{5}$};
\draw (4,0) node [above right]{$A_{1}$};
\draw (5,0) node [above]{$A_{3}$};
\draw (4,1) node [above]{$A_{2}$};
\draw (4,-1) node [right]{$A_{4}$};

\end{tikzpicture}
\end{center} 

The divisor 
\[
D_{70}=21A_{1}+10(A_{2}+A_{3}+A_{4})+13A_{5}+6A_{6}
\]
is ample, of square $70$; the curves $A_{1},\dots,A_{6}$ have degree
$1$ for $D_{70}$. The divisor
\[
F=2A_{1}+A_{2}+A_{3}+A_{4}+A_{5}
\]
is a fiber of type $I_{0}^{*}$ of an elliptic fibration of $X$ for
which $A_{6}$ is the unique section. The divisor 
\[
D_{2}=2F+A_{6}
\]
 is nef, of square $2$, and has base points. The divisor $D_{8}=2D_{2}$
is base-point free and hyperelliptic. 

\begin{prop}
The linear system $|D_{8}|$ defines a double cover $\varphi\colon X\to\mathbf{F}_{4}$,
where the branch locus is the disjoint union of the unique section $s$ with $s^2=-4$ and
$B'\in|3s+12f|$. The curve $B'$ has a $\mathbf{d}_{4}$ singularity
$q$; the curves $A_{1}$, $A_{2}$, $A_{3}$, $A_{4}$ are contracted to $q$
by $\varphi$;  the image of the curve $A_{5}$ is the fiber through
$q$.
\end{prop}

\begin{proof}
We apply Theorem \ref{thm:SaintDonat-2}, case a) i).
\end{proof}

\section{Rank 7 lattices}  

\subsection{The lattice $\boldsymbol{U\oplus\mathbf{D}_{4}\oplus\mathbf{A}_{1}}$}

Let $X$ be a K3 surface with  N\'eron--Severi lattice 
\[
\NS X)=U\oplus\mathbf{D}_{4}\oplus\mathbf{A}_{1}.
\]
There exist eight $(-2)$-curves $A_{1},\dots,A_{8}$ on $X$ with
 dual graph

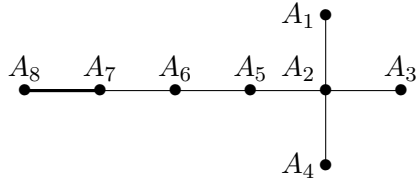
\begin{center}
\begin{tikzpicture}[scale=1]

\draw (0,0) -- (5,0);
\draw (4,1) -- (4,-1);

\draw [very thick] (0,0) -- (1,0);

\draw (0,0) node {$\bullet$};
\draw (1,0) node {$\bullet$};
\draw (2,0) node {$\bullet$};
\draw (3,0) node {$\bullet$};
\draw (4,0) node {$\bullet$};
\draw (5,0) node {$\bullet$};
\draw (4,1) node {$\bullet$};
\draw (4,-1) node {$\bullet$};

\draw (4,0) node [above left]{$A_{2}$};
\draw (0,0) node [above]{$A_{8}$};
\draw (1,0) node [above]{$A_{7}$};
\draw (2,0) node [above]{$A_{6}$};
\draw (3,0) node [above]{$A_{5}$};

\draw (5,0) node [above]{$A_{3}$};
\draw (4,1) node [left]{$A_{1}$};
\draw (4,-1) node [left]{$A_{4}$};

\end{tikzpicture}
\end{center} 

The divisor 
\[
D_{62}=(8,17,8,8,11,6,2)
\]
in the basis $A_{1},\dots,A_{7}$ is ample, of square $62$. The classes
\[
F_{1}=A_{1}+2A_{2}+A_{3}+A_{4}+A_{5},\quad F_{2}=A_{7}+A_{8}
\]
are singular fibers of type $I_{0}^{*}$ and $I_{2}$ or $III$, respectively,
of an elliptic fibration for which $A_{6}$ is a section (that determines
the class of $A_{8}$ in $\NS X)$). The divisor
\[
D_{2}=A_{1}+4A_{2}+2A_{3}+2A_{4}+3A_{5}+2A_{6}+A_{7}
\]
is base-point free, of square $2$, with $D_{2} \cdot A_{k}=0$ if and only
if $k\in\{2,3,4,5,6,7\}$, $D_{2} \cdot A_{1}=D_{j} \cdot A_{8}=2$ and
\[
D_{2}\equiv2A_{2}+A_{3}+A_{4}+2A_{5}+2A_{6}+2A_{7}+A_{8}.
\]

\begin{prop}
The linear system $|D_{2}|$ defines a double cover of the plane $\pi\colon X\to\PP^{2}$
branched over a sextic curve $C_{6}$ with a $\mathbf{d}_{6}$ singularity.
The image of the curve $A_{1}$ is the line that is tangent to two
branches of the $\mathbf{d}_{6}$ singularity, and the image of $A_{8}$
is tangent to the third branch. 
\end{prop}

\subsection{The lattice $\boldsymbol{U\oplus\mathbf{A}_{1}^{\oplus5}}$}

Let $X$ be a K3 surface with a N\'eron--Severi lattice 
\[
\NS X)=U\oplus\mathbf{A}_{1}^{\oplus5}.
\]
There exist $12$ $(-2)$-curves $A_{1},\dots,A_{12}$ on $X$, with
dual graph

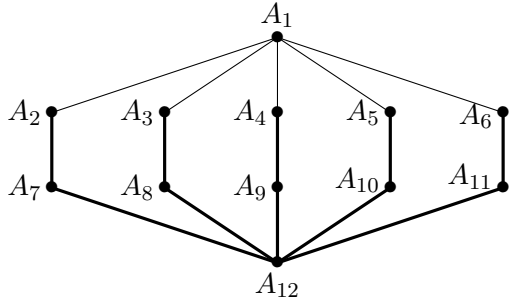
\begin{center}
\begin{tikzpicture}[scale=1]
\draw (0,1.5) -- (3,0.5);
\draw (0,1.5) -- (1.5,0.5);
\draw (0,1.5) -- (0,0.5);
\draw (0,1.5) -- (-1.5,0.5);
\draw (0,1.5) -- (-3,0.5);

\draw [very thick] (-3,0.5) -- (-3,-0.5);
\draw [very thick] (-1.5,0.5) -- (-1.5,-0.5);
\draw [very thick] (0,0.5) -- (0,-0.5);
\draw [very thick] (1.5,0.5) -- (1.5,-0.5);
\draw [very thick] (3,0.5) -- (3,-0.5);

\draw [very thick] (0,-1.5) -- (3,-0.5);
\draw [very thick] (0,-1.5) -- (1.5,-0.5);
\draw [very thick](0,-1.5) -- (0,-0.5);
\draw [very thick] (0,-1.5) -- (-1.5,-0.5);
\draw [very thick] (0,-1.5) -- (-3,-0.5);

\draw (0,1.5) node [above]{$A_{1}$};
\draw (3,0.45) node [left]{$A_{6}$};
\draw (1.5,0.5) node [left]{$A_{5}$};
\draw (0,0.5) node [left]{$A_{4}$};
\draw (-1.5,0.5) node [left]{$A_{3}$};
\draw (-3,0.5) node [left]{$A_{2}$};
\draw (3,-0.35) node [left]{$A_{11}$};
\draw (1.5,-0.4) node [left]{$A_{10}$};
\draw (0,-0.5) node [left]{$A_{9}$};
\draw (-1.5,-0.5) node [left]{$A_{8}$};
\draw (-3,-0.5) node [left]{$A_{7}$};
\draw (0,-1.5) node [below]{$A_{12}$};

\draw (0,1.5) node{$\bullet$};
\draw (3,0.5) node{$\bullet$};
\draw (1.5,0.5) node{$\bullet$};
\draw (0,0.5) node{$\bullet$};
\draw (-1.5,0.5) node{$\bullet$};
\draw (-3,0.5) node{$\bullet$};
\draw (3,-0.5) node{$\bullet$};
\draw (1.5,-0.5) node{$\bullet$};
\draw (0,-0.5) node{$\bullet$};
\draw (-1.5,-0.5) node{$\bullet$};
\draw (-3,-0.5) node{$\bullet$};
\draw (0,-1.5) node{$\bullet$};


\end{tikzpicture}
\end{center} 

The divisor 
\[
D_{14}=2A_{1}+2A_{2}+A_{7}+\sum_{i=1}^{7}A_{i}
\]
is ample of square $14$. The divisors 
\[
A_{2+j}+A_{7+j},\,\,j\in\{0,1,2,3,4\}, 
\]
are irreducible fibers of an elliptic fibration  $\phi_{0}$ of $X$.
The curve $A_{1}$ is a section of $\phi_{0}$, and the curve $A_{12}$
(of class $(2,0,1,1,1,1,-1)$) is a $2$-section. 

The divisor 
\[
D_{2}=2A_{1}+\sum_{j=2}^{6}A_{j}
\]
has square $2$, and $|D_{2}|$ is base-point free. It defines a $2$-to-$1$ cover $\pi\colon X\to\PP^{2}$ of the plane branched over a sextic curve $C_{6}$. For the curves $A_{j},\,\,j\in\{2,3,4,5,6,12\}$, one has $D_{2} \cdot A_{j}=0$; thus the images of these disjoint curves are six points in $\PP^{2}$, and the sextic has nodes at these points.  We have $D_{2} \cdot A_{1}=1$ and
\[
D_{2} \cdot A_{7}=D_{2} \cdot A_{8}=D_{2} \cdot A_{9}=D_{2} \cdot A_{10}=D_{2} \cdot A_{11}=2.
\]
Since $D_{2}=2A_{1}+\sum_{j=2}^{6}A_{j}$, there is a line $L_{1}$
in the branch locus such that $D_{2}=\pi^{*}L_{1}$, and $C_{6}$
is the union of $L_{1}$ and a nodal quintic. Conversely, we have the following. 

\begin{prop}
The minimal resolution of a surface which is the double cover of the
plane branched over the union of a line and a  general  quintic with
a node is a K3 surface of type $U\oplus\mathbf{A}_{1}^{\oplus5}$. 
\end{prop}

Such surfaces are studied in \cite[Section 3.3]{AK}. Clearly, the
moduli space of K3 surfaces with $\NS X)=U\oplus\mathbf{A}_{1}^{\oplus5}$
is unirational. 

\subsection{The lattice $\boldsymbol{U(2)\oplus\mathbf{A}_{1}^{\oplus5}}$ and cubic surfaces}

Let $C_{6}$ be a sextic curve in $\PP^{2}$ with six nodes in
general position. Let $Z\to\PP^{2}$ be the blow-up of the nodes;
it is a degree $3$ del Pezzo surface and contains $27$ $(-1)$-curves:
\begin{itemize}
\item the $6$ exceptional divisors $E_{i}$, $i=1,\dots,6$, 
\item the strict transforms $L_{ij}$ of the $15$ lines through $p_{i}$, $p_{j}$
($i\neq j$), 
\item the strict transforms $Q_{j}$ of the $6$ conics that go through
  points in $\{p_{1},\dots,p_{6}\}\setminus\{p_{j}\}$.
  \end{itemize}
The N\'eron--Severi lattice of $Z$ is generated by the pull-back $L'$
of a line and $E_{1},\dots,E_{6}$; it is the unimodular rank $7$
lattice 
\[
I_{1}\oplus I_{-1}^{\oplus6}.
\]
The anti-canonical divisor of the degree $3$ del Pezzo surface $Z$
is given by 
\[
-K_{Z}=3L'-\left(E_{1}+\dots+E_{6}\right); 
\]
this is an ample divisor. The linear system $|-K_{Z}|$ is base-point
free (see \cite[Section~3, Theorem 1]{DemazureIV}), and each of the
 $30$ divisors
\[
Q_{j}+L_{ij}+E_{i}\quad\text{ with }i\neq j,\;i,j\in\{1,\dots,6\}
\]
and $15$ divisors 
\[
L_{ij}+L_{kl}+L_{mn}\quad \text{ with }\{i,j,k,l,m,n\}=\{1,\dots,6\}
\]
belongs to the linear system $|-K_{Z}|$.

The double cover $f\colon Y\to Z$ branched over the strict transform $C_{6}'$
of $C_{6}$ is a smooth K3 surface. The pull-backs of the $27$ $(-1)$-curves
are $(-2)$-curves. We denote by $A_{1},\dots,A_{6}$ the pull-backs 
on $Y$ of the curves $E_{i}$ and by $L$ the pull-back of $L'$.
Naturally, the lattice $f^{*}\NS Z)$ is $\left(I_{1}\oplus I_{-1}^{\oplus6}\right)(2)$,
which is also the lattice generated by $L, A_{1},\dots,A_{6}$ and
is isometric to $U(2)\oplus\mathbf{A}_{1}^{\oplus5}$. 

Since $f\colon Y\to Z$ is finite, its pull-back $D_{6}=f^{*}(-K_{Z})$
is ample, base-point free, non-hyperelliptic, with $D_{6}^{2}=6$. We thus have the following. 

\begin{prop}
The image of $\,Y$ under the map $\pi\colon Y\to\PP^{4}$ obtained from $|D_{6}|$
is a degree $6$ complete intersection surface. There exist $45$
hyperplane sections of $Y$ which are each the union of $\,3$ conics.
\end{prop}

These $45$ hyperplane sections correspond to the $45$ tritangent
planes of a cubic, \textit{i.e.}, to the planes containing $3$ lines in the
cubic. In fact, by the above discussion, the strict transform in
the cubic surface $Z$ of $C_{6}$ is a degree $6$ curve in $\PP^{3}$
which is the complete intersection of the cubic surface $Y=\{f_{3}(x,y,z,t)=0\}$
and a quadric $\{q_{2}(x,z,y,t)=0\}$. The K3 surface $Y$ is the
complete intersection of $\{f_{3}(X,Y,Z,T)=0\}$ and the quadric $\{q_{2}(X,Y,Z,T)-W^{2}=0\}$
in $\PP^{4}$ (with coordinates $X,Y,Z,T,W$). From that discussion,
we see that the moduli space of K3 surfaces with $\NS X)\simeq U(2)\oplus{\bf A}_{1}^{\oplus5}$
is unirational.

\subsection{The lattice $\boldsymbol{U\oplus\mathbf{A}_{1}\oplus\mathbf{A}_{2}^{\oplus2}}$}

Let $X$ be a K3 surface with  N\'eron--Severi lattice 
\[
\NS X)=U\oplus\mathbf{A}_{1}\oplus\mathbf{A}_{2}^{\oplus2}.
\]
The dual graph of the nine $(-2)$-curves on $X$ is 

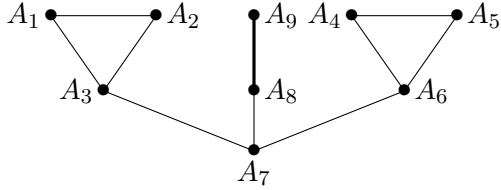
\begin{center}
\begin{tikzpicture}[scale=1]
\draw (-2.7,2) -- (-1.3,2);
\draw (-2.7,2) -- (-2,1);
\draw (-2,1) -- (-1.3,2);

\draw (-2,1) -- (0,0.2);
\draw (0,0.2) -- (0,1);
\draw [very thick] (0,1) -- (0,2);
\draw (2,1) -- (0,0.2);

\draw (2.7,2) -- (1.3,2);
\draw (2.7,2) -- (2,1);
\draw (2,1) -- (1.3,2);

\draw (0,0.2) node {$\bullet$};
\draw (0,1) node {$\bullet$};
\draw (0,2) node {$\bullet$};

\draw (-2.7,2) node {$\bullet$};
\draw (-1.3,2) node {$\bullet$};
\draw (-2,1) node {$\bullet$};

\draw (2.7,2) node {$\bullet$};
\draw (1.3,2) node {$\bullet$};
\draw (2,1) node {$\bullet$};

\draw (0,0.2) node [below]{$A_{7}$};
\draw (0,1) node [right]{$A_{8}$};
\draw (0,2) node [right]{$A_{9}$};

\draw (-2.7,2) node [left]{$A_{1}$};
\draw (-1.3,2) node [right]{$A_{2}$};
\draw (-2,1) node [left]{$A_{3}$};

\draw (2.7,2) node [right]{$A_{5}$};
\draw (1.3,2) node [left]{$A_{4}$};
\draw (2,1) node [right]{$A_{6}$};

\end{tikzpicture}
\end{center} 

The curves $A_{1}$, $A_{2}$, $A_{3}$, $A_{4}$, $A_{5}$, $A_{7}$, $A_{8}$ generate the
N\'eron--Severi lattice. The divisor 
\[
D_{18}=(5,5,6,-1,-1,3,1)
\]
in the above basis is ample, of square $18$. One has $D_{18} \cdot A_{j}=1$
for $j<8$ and $D_{18} \cdot A_{9}=2$. The divisors 
\[
F=A_{1}+A_{2}+A_{3},\,A_{4}+A_{5}+A_{6},\,A_{8}+A_{9}
\]
 are the reducible fibers of an elliptic fibration  of the surface.
The divisor 
\[
D_{2}=2F+A_{7}
\]
 is nef, of square $2$, with base points and $D_{2} \cdot A_{3}=D_{2} \cdot A_{6}=D_{2} \cdot A_{8}=1$,
$D_{2} \cdot A_{j}=0$ for $j\neq0$. By Theorem \ref{thm:SaintDonat-2},
case a) i), the base-point free linear system $|2D_{2}|$ defines
a morphism $\varphi\colon X\to\mathbf{F}_{4}$ which is branched over the
unique section $s$ of $mathbf{F}_{4}$ with $s^2=-4$ and a disjoint curve $B\in|2s+12f|$. The curve $B$
has a node $q$ (which is the image of $A_{9}$) and two $\mathbf{a}_{2}$
singularities $p$, $p'$ which are the images of $A_{1}$, $A_{2}$ and $A_{4}$, $A_{5}$.
The images of $A_{3}$, $A_{6}$, $A_{8}$ are the fibers through $p$, $p'$
and $q$, respectively.

One can give another description as follows. The divisor 
\[
D_{2}'=A_{9}+(A_{3}+A_{6}+2A_{7}+2A_{8})
\]
is nef of square $2$, and $|D_{2}'|$ is base-point free. Moreover, 
$D_{2} \cdot A_{j}=0$ if and only if $j\in\{3,6,7,8\}$, and $D_{2} \cdot A_{9}=2$ and 
$D_{2} \cdot A_{j}=1$ for $j\in\{1,2,4,5\}$. We have 
\[
\begin{array}{c}
D_{2}'\equiv A_{1}+A_{2}+(2A_{3}+A_{6}+2A_{7}+A_{8}),\\
D_{2}'\equiv A_{4}+A_{5}+(A_{3}+2A_{6}+2A_{7}+A_{8}).
\end{array}
\]
By using the linear system $|D_{2}'|$, we obtain the following. 

\begin{prop}
The K3 surface $X$ is the double cover $\eta\colon X\to\PP^{2}$ of $\,\PP^{2}$
branched over a sextic curve $C_{6}$ with a $\mathbf{d}_{4}$ singularity
$q$ onto which the curves $A_{3}$, $A_{6}$, $A_{7}$, $A_{8}$ are contracted. 
\end{prop}

The image of the curve $A_{9}$ is a line that is tangent to a branch
of the singularity $q$ and that meets the sextic in two other points.
The two lines $L_{1}$, $L_{2}$ that are tangent to the two other branches
of $q$ are tangent to another point of the sextic. The image of $A_{1}$, $A_{2}$
is $L_{1}$, and the image of $A_{4}$, $A_{5}$ is $L_{2}$.

\subsection{The lattice $\boldsymbol{U\oplus\mathbf{A}_{1}^{\oplus2}\oplus\mathbf{A}_{3}}$}

The dual graph of the nine $\cu$-curves $A_{1},\dots,A_{9}$ on
$X$ is 

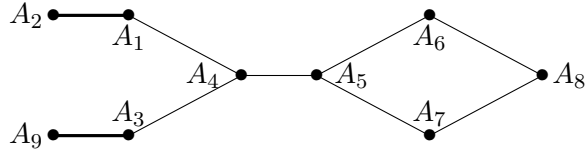
\begin{center}
\begin{tikzpicture}[scale=1]

\draw (0,0) -- (1.5,0.8);
\draw (0,0) -- (1.5,-0.8);
\draw (1.5,0.8) -- (3,0);
\draw (1.5,-0.8) -- (3,0);

\draw (-2.5,0.8) -- (-1,0);
\draw (-2.5,-0.8) -- (-1,0);
\draw (-1,0) -- (0,0);

\draw [very thick] (-2.5,-0.8) -- (-3.5,-0.8);
\draw [very thick] (-2.5,0.8) -- (-3.5,0.8);

\draw (0,0) node {$\bullet$};
\draw (3,0) node {$\bullet$};
\draw (1.5,0.8) node {$\bullet$};
\draw (1.5,-0.8) node {$\bullet$};
\draw (-1,0) node {$\bullet$};
\draw (-2.5,0.8) node {$\bullet$};
\draw (-2.5,-0.8) node {$\bullet$};
\draw (-3.5,0.8) node {$\bullet$};
\draw (-3.5,-0.8) node {$\bullet$};

\draw (0.1,0) node [right]{$A_{5}$};
\draw (3,0) node [right]{$A_{8}$};
\draw (1.5,0.8) node [below]{$A_{6}$};
\draw (1.5,-0.8) node [above]{$A_{7}$};
\draw (-1.15,0) node [left]{$A_{4}$};
\draw (-2.5,0.8) node [below]{$A_{1}$};
\draw (-2.5,-0.8) node [above]{$A_{3}$};
\draw (-3.5,-0.8) node [left]{$A_{9}$};
\draw (-3.5,0.8) node [left]{$A_{2}$};

\end{tikzpicture}
\end{center} 

The curves $A_{1},\dots,A_{7}$ generate the N\'eron--Severi lattice.
The divisor 
\[
D_{34}=(6,4,2,5,3,1,1)
\]
in the basis $A_{1},\dots,A_{7}$ is ample of square $34$. The divisors
\[
F_{1}=A_{1}+A_{2},\quad F_{2}=A_{3}+A_{9},\quad F_{3}=A_{5}+A_{6}+A_{7}+A_{8}
\]
 are the reducible fibers of an elliptic fibration  $\phi_{1}$ of
$X$ for which $A_{4}$ is a section (for that, one can deduce the
classes of curves $A_{8}$,  $A_{9}$).  The divisor 
\[
D_{2}=2F_{1}+A_{4}
\]
 is nef, of square $2$, with base points. By Theorem \ref{thm:SaintDonat-2}, 
case i) a), we have the following. 

\begin{prop}
The linear system $|D_{2}|$ defines a morphism $\varphi\colon X\to\mathbf{F}_{4}$
branched over the unique section $s$ with $s^2=-4$ and a curve $B\in|3s+12f|$. The curve
$B$ has two nodes $p$, $p'$ and an $\mathbf{a_{3}}$ singularity $q$.
The pull-backs of the fibers through $p$, $p'$, $q$ are the fibers $F_{1}$, $F_{2}$, $F_{3}$.
\end{prop}

\subsection{The lattice $\boldsymbol{U\oplus\mathbf{A}_{2}\oplus\mathbf{A}_{3}}$}

The dual graph of the eight $\cu$-curves $A_{1},\dots,A_{8}$ on $X$
is 

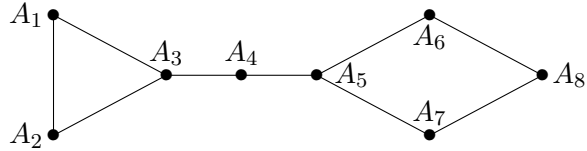
\begin{center}
\begin{tikzpicture}[scale=1]

\draw (0,0) -- (1.5,0.8);
\draw (0,0) -- (1.5,-0.8);
\draw (1.5,0.8) -- (3,0);
\draw (1.5,-0.8) -- (3,0);

\draw (-3.5,0.8) -- (-2,0);
\draw (-3.5,-0.8) -- (-2,0);
\draw (-1,0) -- (0,0);
\draw (-2,0) -- (-1,0);
\draw (-3.5,0.8) -- (-3.5,-0.8);


\draw (0,0) node {$\bullet$};
\draw (3,0) node {$\bullet$};
\draw (1.5,0.8) node {$\bullet$};
\draw (1.5,-0.8) node {$\bullet$};
\draw (-1,0) node {$\bullet$};
\draw (-2,0) node {$\bullet$};
\draw (-3.5,0.8) node {$\bullet$};
\draw (-3.5,-0.8) node {$\bullet$};

\draw (0.1,0) node [right]{$A_{5}$};
\draw (3,0) node [right]{$A_{8}$};
\draw (1.5,0.8) node [below]{$A_{6}$};
\draw (1.5,-0.8) node [above]{$A_{7}$};
\draw (-1,0) node [above]{$A_{4}$};
\draw (-2,0) node [above]{$A_{3}$};
\draw (-3.5,-0.8) node [left]{$A_{2}$};
\draw (-3.5,0.8) node [left]{$A_{1}$};

\end{tikzpicture}
\end{center} 

The seven $(-2)$-curves $A_{1},\dots,A_{7}$ generate  $\NS X)$, and
the divisor 
\[
D_{42}=(6,6,8,5,3,1,1)
\]
in the basis $A_{1},\dots,A_{7}$ is ample, of square $42$. The surface
has an elliptic fibration $\phi_{1}$ with reducible fibers the curves 
\[
F_{1}=A_{1}+A_{2}+A_{3},\quad F_{2}=A_{5}+A_{6}+A_{7}+A_{8},
\]
and $A_{4}$ is a section. The divisor 
\[
D_{2}=2F_{1}+A_{4}
\]
 is nef, of square $2$, with base points. By Theorem \ref{thm:SaintDonat-2}, 
case i) a), we have the following. 

\begin{prop}
The linear system $|D_{2}|$ defines a morphism $\varphi\colon X\to\mathbf{F}_{4}$
branched over the unique section $s$ with $s^2=-4$ and a curve $B\in|3s+12f|$. The curve
$B$ has a cusp $p$ and an $\mathbf{a_{3}}$ singularity $q$. The
pull-backs of the fibers through $p$, $q$ are the fibers $F_{1}$, $F_{2}$.
\end{prop}

\subsection{The lattice $\boldsymbol{U\oplus\mathbf{A}_{1}\oplus\mathbf{A}_{4}}$}

The dual graph of the eight $\cu$-curves $A_{1},\dots,A_{8}$ on
$X$ is 

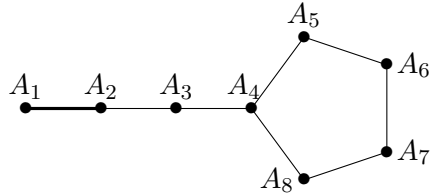
\begin{center}
\begin{tikzpicture}[scale=1]

\draw (-0.3,-0.95) -- (-1,0);
\draw (-1,0) -- (-0.3,0.95);
\draw (-0.3,0.95) -- (0.8,0.58);
\draw (0.8,0.58) -- (0.8,-0.58);
\draw (0.8,-0.58) -- (-0.3,-0.95);

\draw [very thick] (-4,0) -- (-3,0);
\draw (-1,0) -- (-4,0);

\draw (-4,0) node {$\bullet$};
\draw (-3,0) node {$\bullet$};
\draw (-2,0) node {$\bullet$};
\draw (-1,0) node {$\bullet$};
\draw (-0.3,-0.95) node {$\bullet$};
\draw (-0.3,0.95) node {$\bullet$};
\draw (0.8,-0.58) node {$\bullet$};
\draw (0.8,0.58) node {$\bullet$};

\draw (-4,0) node [above]{$A_{1}$};
\draw (-3,0) node [above]{$A_{2}$};
\draw (-2,0) node [above]{$A_{3}$};
\draw (-1.1,0) node [above]{$A_{4}$};
\draw (-0.3,0.95) node [above]{$A_{5}$};
\draw (0.8,0.58) node [right]{$A_{6}$};
\draw (0.8,-0.58) node [right]{$A_{7}$};
\draw (-0.3,-0.95) node [left]{$A_{8}$};

\end{tikzpicture}
\end{center} 

In the basis $A_{1},\dots,A_{7}$, the divisor 
\[
D_{42}=(7,9,5,2,0,-1,-1)
\]
is ample, of square $42$. The curves 
\[
F_{1}=A_{1}+A_{2},\quad F_{2}=A_{4}+A_{5}+A_{6}+A_{7}+A_{8}
\]
 are the reducible fibers of an elliptic fibration  for which $A_{3}$
is a section. The divisor 
\[
D_{2}=2F_{1}+A_{4}
\]
 is nef, of square $2$, with base points. By Theorem \ref{thm:SaintDonat-2}, 
case i) a), we have the following. 

\begin{prop}
The linear system $|2D_{2}|$ defines a morphism $\varphi\colon X\to\mathbf{F}_{4}$
branched over the unique section $s$ with $s^2=-4$ and a curve $B\in|3s+12f|$. The curve
$B$ has a node $p$ and an $\mathbf{a_{4}}$ singularity $q$. The
pull-backs of the fibers through $p$, $q$ are the fibers $F_{1}$, $F_{2}$.
\end{prop}

\subsection{The lattice $\boldsymbol{U\oplus\mathbf{A}_{5}}$}

The dual graph of the seven $\cu$-curves $A_{1},\dots,A_{7}$ on
$X$ is 

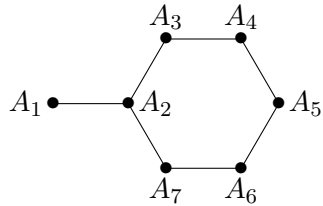
\begin{center}
\begin{tikzpicture}[scale=1]

\draw (1,0) -- (0.5,0.86);
\draw (0.5,0.86) -- (-0.5,0.86);
\draw (-1,0) -- (-0.5,0.86);
\draw (-1,0) -- (-0.5,-0.86);
\draw (-0.5,-0.86) -- (0.5,-0.86);
\draw (1,0) -- (0.5,-0.86);
\draw (-1,0) -- (-2,0);

\draw (1,0) node {$\bullet$};
\draw (0.5,0.86) node {$\bullet$};
\draw (-1,0) node {$\bullet$};
\draw (-0.5,0.86) node {$\bullet$};
\draw (-0.5,-0.86) node {$\bullet$};
\draw (0.5,-0.86)  node {$\bullet$};
\draw (-2,0) node {$\bullet$};

\draw (1,0) node  [right]{$A_{5}$};
\draw (0.5,0.86) node [above]{$A_{4}$};
\draw (-1,0) node  [right]{$A_{2}$};
\draw (-0.5,0.86) node  [above]{$A_{3}$};
\draw (-0.5,-0.86) node  [below]{$A_{7}$};
\draw (0.5,-0.86) node   [below]{$A_{6}$};
\draw (-2,0) node  [left]{$A_{1}$};

\end{tikzpicture}
\end{center} 

The seven $\cu$-curves $A_{1},\dots,A_{7}$ generate the N\'eron--Severi
lattice, and in that basis, the divisor 
\[
D_{84}=(7,15,12,10,9,10,12)
\]
is ample, of square $84$. We have $D_{84} \cdot A_{j}=1$ for $j\in\{1,\dots,7\}\setminus\{5\}$
and $D_{84} \cdot A_{5}=2$. The divisor 
\[
F=\sum_{j=2}^{7}A_{j}
\]
is the unique reducible fiber of an elliptic fibration  of $X$; the
curve $A_{1}$ is a section.

The divisor 
\[
D_{2}'=(2,4,3,2,1,2,3)
\]
is nef, base-point free, of square $2$. One has $D_{2} \cdot A_{j}=0$
for $j\in\{1,\dots,7\}\setminus\{5\}$ and $D_{2} \cdot A_{5}=2$. The surface
$X$ is a double cover of $\PP^{2}$. The branch locus is a sextic
curve with an $\mathbf{e}_{6}$ singularity; there exists a line $L$
through the $\mathbf{e}_{6}$ singularity which cuts the sextic transversely
in two points such that the image of $A_{5}$ is $L$. 

One may also describe that surface as a double cover of $\mathbf{F}_{4}$. 

\begin{prop}
The linear system $|4F+2A_{1}|$ defines a morphism $\varphi\colon X\to\mathbf{F}_{4}$
branched over the unique section $s$ with $s^2=-4$ and a curve $B\in|3s+12f|$. The curve
$B$ has an $\mathbf{a_{5}}$ singularity $q$. The pull-back of the
fiber through $q$ is the fiber $F$.
\end{prop}

\subsection{The lattice $\boldsymbol{U\oplus\mathbf{D}_{5}}$}

The dual graph of the seven $\cu$-curves $A_{1},\dots,A_{7}$ on
$X$ is 

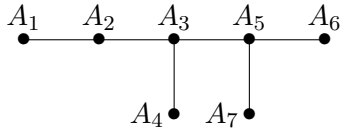
\begin{center}
\begin{tikzpicture}[scale=1]

\draw (0,0) -- (4,0);
\draw (2,0) -- (2,-1);
\draw (3,0) -- (3,-1);

\draw (0,0) node {$\bullet$};
\draw (1,0) node {$\bullet$};
\draw (2,0) node {$\bullet$};
\draw (2,-1) node {$\bullet$};
\draw (3,0) node {$\bullet$};
\draw (3,-1) node {$\bullet$};
\draw (4,0) node {$\bullet$};

\draw (0,0) node [above]{$A_{1}$};
\draw (1,0) node [above]{$A_{2}$};
\draw (2,0) node [above]{$A_{3}$};
\draw (2,-1) node [left]{$A_{4}$};
\draw (3,0) node [above]{$A_{5}$};
\draw (4,0) node [above]{$A_{6}$};
\draw (3,-1) node [left]{$A_{7}$};

\end{tikzpicture}
\end{center} 

The divisor 
\[
D_{114}=(8,17,27,13,25,12,12)
\]
 in the basis $A_{1},\dots,A_{7}$ is ample, of square $114$, and every
$\cu$-curve on $X$ has degree $1$ with respect to $D_{114}$.
The divisor 
\[
F=A_{2}+2A_{3}+A_{4}+2A_{5}+A_{6}+A_{7}
\]
is the unique reducible fiber of an elliptic fibration  of $X$ for
which $A_{1}$ is a section. By Theorem \ref{thm:SaintDonat-2},  case
i) a), we have the following. 

\begin{prop}
The linear system $|4F+2A_{1}|$ defines a morphism $\varphi\colon X\to\mathbf{F}_{4}$
branched over the unique section $s$ with $s^2=-4$ and a curve $B\in|3s+12f|$. The curve
$B$ has a $\mathbf{d_{5}}$ singularity $q$. The pull-back of the
fiber through $q$ is the fiber $F$.
\end{prop}

\section{Rank 8 lattices }

\subsection{The lattice $\boldsymbol{U\oplus\mathbf{D}_{6}}$}

The dual graph of the eight $\cu$-curves $A_{1},\dots,A_{8}$ on
$X$ is 

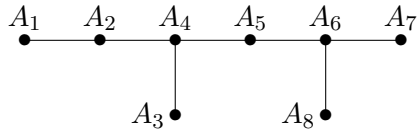
\begin{center}
\begin{tikzpicture}[scale=1]

\draw (0,0) -- (5,0);
\draw (2,0) -- (2,-1);
\draw (4,0) -- (4,-1);

\draw (0,0) node {$\bullet$};
\draw (1,0) node {$\bullet$};
\draw (2,0) node {$\bullet$};
\draw (2,-1) node {$\bullet$};
\draw (3,0) node {$\bullet$};
\draw (4,-1) node {$\bullet$};
\draw (4,0) node {$\bullet$};
\draw (5,0) node {$\bullet$};

\draw (0,0) node [above]{$A_{1}$};
\draw (1,0) node [above]{$A_{2}$};
\draw (2,0) node [above]{$A_{4}$};
\draw (2,-1) node [left]{$A_{3}$};
\draw (3,0) node [above]{$A_{5}$};
\draw (4,0) node [above]{$A_{6}$};
\draw (4,-1) node [left]{$A_{8}$};
\draw (5,0) node [above]{$A_{7}$};

\end{tikzpicture}
\end{center} 

The divisor 
\[
D_{220}=(11,24,18,38,35,33,16,16)
\]
 in the basis $A_{1},\dots,A_{8}$ is ample, of square $220$, with $D_{220} \cdot A_{J}=1$
if $j\notin\{1,3\}$ and $D_{220} \cdot A_{j}=2$ if $j\in\{1,3\}$. The
divisor 
\[
F=A_{2}+A_{3}+2A_{4}+2A_{5}+2A_{6}+A_{7}+A_{8}
\]
is a fiber of an elliptic fibration  for which $A_{1}$ is a section.
By Theorem \ref{thm:SaintDonat-2}, case i) a), we have the following. 

\begin{prop}
The linear system $|4F+2A_{1}|$ defines a morphism $\varphi\colon X\to\mathbf{F}_{4}$
branched over the unique section $s$ with $s^2=-4$ and a curve $B\in|3s+12f|$. The curve
$B$ has a $\mathbf{d_{6}}$ singularity $q$. The pull-back of the
fiber through $q$ is the fiber $F$.
\end{prop}

\subsection{The lattice $\boldsymbol{U\oplus\mathbf{D}_{4}\oplus\mathbf{A}_{1}^{\oplus2}}$}

The dual graph of the $10$ $\cu$-curves $A_{1},\dots,A_{10}$ on
$X$ is 

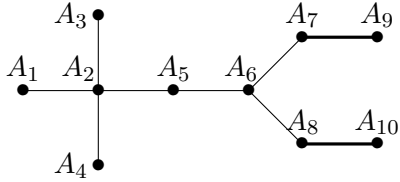
\begin{center}
\begin{tikzpicture}[scale=1]

\draw (0,0) -- (3,0);
\draw (1,0) -- (1,1);
\draw (1,0) -- (1,-1);
\draw (3,0) -- (3.71,0.71);
\draw (3,0) -- (3.71,-0.71);
\draw [very thick] (3.71,0.71) -- (4.71,0.71);
\draw [very thick] (3.71,-0.71) -- (4.71,-0.71);

\draw (0,0) node {$\bullet$};
\draw (1,0) node {$\bullet$};
\draw (2,0) node {$\bullet$};
\draw (3,0) node {$\bullet$};
\draw (1,1) node {$\bullet$};
\draw (1,-1) node {$\bullet$};
\draw (3.71,0.71) node {$\bullet$};
\draw (3.71,-0.71) node {$\bullet$};
\draw (4.71,0.71) node {$\bullet$};
\draw (4.71,-0.71) node {$\bullet$};

\draw (0,0) node [above]{$A_{1}$};
\draw (1.1,0) node [above left]{$A_{2}$};
\draw (2,0) node [above]{$A_{5}$};
\draw (1,1) node [left]{$A_{3}$};
\draw (1,-1) node [left]{$A_{4}$};
\draw (2.9,0) node [above]{$A_{6}$};
\draw (3.71,0.71) node [above]{$A_{7}$};
\draw (3.715,-0.71) node [above]{$A_{8}$};
\draw (4.71,0.71) node [above]{$A_{9}$};
\draw (4.71,-0.71) node [above]{$A_{10}$};

\end{tikzpicture}
\end{center} 

The first eight curves generate  the N\'eron--Severi lattice. The divisor
\[
D_{54}=(6,13,6,6,9,6,2,2)
\]
of square $54$ is ample with $D_{54} \cdot A_{j}=1$ for $j\leq6$, $D_{54} \cdot A_{j}=2$
for $j\in\{7,8\}$ and $D_{54} \cdot A_{j}=4$ for $j\in\{9,10\}$. The divisors
\[
F_{1}=A_{2}+\sum_{j=1}^{5} \cdot A_{j},\quad F_{2}=A_{7}+A_{9},\quad F_{3}=A_{8}+A_{10}
\]
 are fibers of an elliptic fibration  $\varphi_{2}$ such that $A_{6}$
is a section. For $j\in\{1,3,4\}$, the divisors 
\[
-A_{j}+A_{2}+A_{5}+A_{6}+\sum_{k=1}^{8}  A_{k}
\]
 are fibers of elliptic fibrations $\varphi_{j}$ for which $A_{j}$
and $A_{9}$, $A_{10}$ are $2$-sections. The divisor 
\[
D_{2}=A_{1}+2A_{2}+A_{3}+A_{4}+2A_{5}+2A_{6}+A_{7}+A_{8}
\]
is nef, of square $2$ and base-point free. The linear system $|D_{2}|$
contracts the curves $A_{j}$ for $j\in\{1,3,\dots,8\}$; moreover, $D_{2} \cdot A_{2}=1$ and 
$D_{2} \cdot A_{9}=D_{2} \cdot A_{10}=2$.

\begin{prop}
The branch curve is the union of a line $L$ and a nodal quintic $Q$; the line goes through the node transversely $($forming a $\mathbf{d}_{4}$
singularity on the sextic $L\cup Q$, resolved by $A_{5}$, $A_{6}$, $A_{7}$,  $A_{8})$
and cuts the quintic in three other points $($resolved by $A_{1}$, $A_{3}$, $A_{4})$. 
\end{prop}

We have $D_{2}=\pi^{*}L$, and the equivalences 
\[
D_{2}\equiv A_{5}+2A_{6}+A_{8}+A_{9}
\]
and 
\[
D_{2}\equiv A_{5}+2A_{6}+A_{7}+A_{10} 
\]
give that the images of $A_{9}$ and $A_{10}$ are lines that are tangent
to the branches of the node of the quintic $Q$. 

\begin{rem}
One may also describe that surface as the desingularization of the
double cover of $\mathbf{F}_{4}$ branched over the unique section $s$ with $s^2=-4$ and
a curve $B\in|3s+12f|$ with a $\mathbf{d}_{4}$ singularity and two
nodes. 
\end{rem}

\subsection{The lattice $\boldsymbol{U\oplus\mathbf{A}_{1}^{\oplus6}}$}

Let $X$ be a K3 surface with $\NS X)\simeq U\oplus\mathbf{A}_{1}^{\oplus6}$.
The surface $X$ contains $19$ $\cu$-curves; their dual graph is
represented by

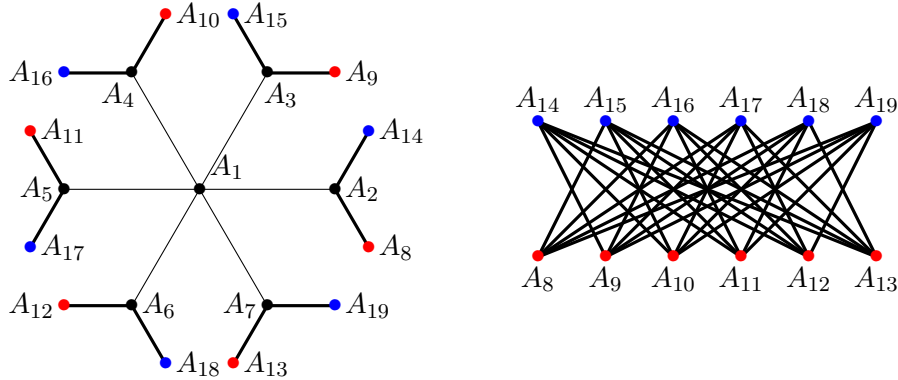
\begin{center}
\begin{tikzpicture}[scale=0.9]

\draw (0,0) -- (2,0);
\draw (0,0) -- (2*0.5,2*0.86);
\draw (0,0) -- (-2*0.5,2*0.86);
\draw (0,0) -- (-2,0);
\draw (0,0) -- (-2*0.5,-2*0.86);
\draw (0,0) -- (2*0.5,-2*0.86);

\draw [very thick] (2,0) -- (2.5,0.86);
\draw [very thick] (2,0) -- (2.5,-0.86);
\draw [very thick] (2*0.5,2*0.86) -- (2,1.72);
\draw [very thick] (2*0.5,2*0.86) -- (0.5,1.72+0.86);
\draw [very thick] (-2*0.5,2*0.86) -- (-0.5,1.72+0.86);
\draw [very thick] (-2*0.5,2*0.86) -- (-2,1.72);
\draw [very thick] (-2,0) -- (-2.5,0.86);
\draw [very thick] (-2,0) -- (-2.5,-0.86);
\draw [very thick] (-2*0.5,-2*0.86) -- (-2,-1.72);
\draw [very thick] (-2*0.5,-2*0.86) -- (-0.5,-1.72-0.86);
\draw [very thick] (2*0.5,-2*0.86) -- (0.5,-1.72-0.86);
\draw [very thick]  (2*0.5,-2*0.86) -- (2,-1.72);

\draw (0,0) node {$\bullet$};
\draw (2,0) node {$\bullet$};
\draw (2*0.5,2*0.86) node {$\bullet$};
\draw (-2*1,0) node {$\bullet$};
\draw (-2*0.5,2*0.86) node {$\bullet$};
\draw (-2*0.5,-2*0.86) node {$\bullet$};
\draw (2*0.5,-2*0.86)  node {$\bullet$};

\draw [color=red] (2.5,-0.86) node {$\bullet$};
\draw [color=blue] (2.5,0.86) node {$\bullet$};
\draw [color=red] (2,1.72) node {$\bullet$};
\draw [color=blue] (0.5,1.72+0.86) node {$\bullet$};
\draw [color=red] (-0.5,1.72+0.86) node {$\bullet$};
\draw [color=blue] (-2,1.72) node {$\bullet$};
\draw [color=red] (-2.5,0.86) node {$\bullet$};
\draw [color=blue] (-2.5,-0.86) node {$\bullet$};
\draw [color=red] (-2,-1.72) node {$\bullet$};
\draw [color=blue] (-0.5,-1.72-0.86) node {$\bullet$};
\draw [color=red] (0.5,-1.72-0.86) node {$\bullet$};
\draw [color=blue] (2,-1.72) node {$\bullet$};

\draw (0,0) node  [above right]{$A_{1}$};
\draw (2,0) node  [right]{$A_{2}$};
\draw (2*0.5+0.2,2*0.86) node [below]{$A_{3}$};
\draw (-2*0.5-0.2,2*0.86) node  [below]{$A_{4}$};
\draw (-2,0) node  [left]{$A_{5}$};
\draw (-2*0.5,-2*0.86) node  [right]{$A_{6}$};
\draw (2*0.5,-2*0.86) node   [left]{$A_{7}$};

\draw (2.5,-0.86) node [right]{$A_{8}$};
\draw  (2.5,0.86) node [right]{$A_{14}$};
\draw  (2,1.72) node [right]{$A_{9}$};
\draw  (0.5,1.72+0.86) node [right]{$A_{15}$};
\draw  (-0.5,1.72+0.86) node [right]{$A_{10}$};
\draw  (-2,1.72) node [left]{$A_{16}$};
\draw  (-2.5,0.86) node [right]{$A_{11}$};
\draw  (-2.5,-0.86) node [right]{$A_{17}$};
\draw  (-2,-1.72) node [left]{$A_{12}$};
\draw  (-0.5,-1.72-0.86) node [right]{$A_{18}$};
\draw  (0.5,-1.72-0.86) node [right]{$A_{13}$};
\draw  (2,-1.72) node [right]{$A_{19}$};

\draw [very thick] (5,1) -- (6,-1);
\draw [very thick] (5,1) -- (7,-1);
\draw [very thick] (5,1) -- (8,-1);
\draw [very thick] (5,1) -- (9,-1);
\draw [very thick] (5,1) -- (10,-1);
\draw [very thick] (6,1) -- (5,-1);
\draw [very thick] (6,1) -- (7,-1);
\draw [very thick] (6,1) -- (8,-1);
\draw [very thick] (6,1) -- (9,-1);
\draw [very thick] (6,1) -- (10,-1);
\draw [very thick] (7,1) -- (5,-1);
\draw [very thick] (7,1) -- (6,-1);
\draw [very thick] (7,1) -- (8,-1);
\draw [very thick] (7,1) -- (9,-1);
\draw [very thick] (7,1) -- (10,-1);
\draw [very thick] (8,1) -- (5,-1);
\draw [very thick] (8,1) -- (6,-1);
\draw [very thick] (8,1) -- (7,-1);
\draw [very thick] (8,1) -- (9,-1);
\draw [very thick] (8,1) -- (10,-1);
\draw [very thick] (9,1) -- (5,-1);
\draw [very thick] (9,1) -- (6,-1);
\draw [very thick] (9,1) -- (7,-1);
\draw [very thick] (9,1) -- (8,-1);
\draw [very thick] (9,1) -- (10,-1);
\draw [very thick] (10,1) -- (5,-1);
\draw [very thick] (10,1) -- (6,-1);
\draw [very thick] (10,1) -- (7,-1);
\draw [very thick] (10,1) -- (8,-1);
\draw [very thick] (10,1) -- (9,-1);

\draw [color=blue] (5,1) node {$\bullet$};
\draw [color=blue] (6,1) node {$\bullet$};
\draw [color=blue] (7,1) node {$\bullet$};
\draw [color=blue] (8,1) node {$\bullet$};
\draw [color=blue] (9,1) node {$\bullet$};
\draw [color=blue] (10,1) node {$\bullet$};
\draw [color=red] (5,-1) node {$\bullet$};
\draw [color=red] (6,-1) node {$\bullet$};
\draw [color=red] (7,-1) node {$\bullet$};
\draw [color=red] (8,-1) node {$\bullet$};
\draw [color=red] (9,-1) node {$\bullet$};
\draw [color=red] (10,-1) node {$\bullet$};

\draw (5,-1) node [below]{$A_{8}$};
\draw (6,-1) node [below]{$A_{9}$};
\draw (7,-1) node [below]{$A_{10}$};
\draw (8,-1) node [below]{$A_{11}$};
\draw (9,-1) node [below]{$A_{12}$};
\draw (10,-1) node [below]{$A_{13}$};

\draw (5,1) node [above]{$A_{14}$};
\draw (6,1) node [above]{$A_{15}$};
\draw (7,1) node [above]{$A_{16}$};
\draw (8,1) node [above]{$A_{17}$};
\draw (9,1) node [above]{$A_{18}$};
\draw (10,1) node [above]{$A_{19}$};

\end{tikzpicture}
\end{center} 
where we draw the graph of the configuration of the curves $A_{8},\dots,A_{19}$
in another part in order for it to be more readable. The first eight curves
$A_{1},\dots,A_{8}$ generate  the N\'eron--Severi lattice. The divisor
\[
D_{12}=3A_{1}+2A_{2}+A_{3}+A_{4}+A_{5}+A_{6}+A_{7}+A_{8}
\]
is ample, of square $12$, with $D_{12} \cdot A_{j}=1$ for $j\leq7$,
$D_{12} \cdot A_{j}=2$ for $8\leq j\leq13$ and $D_{12} \cdot A_{j}=4$ for $14\leq j\leq19$.
Let $j\neq k$ be two elements of $\{2,\dots,7\}$; the effective
divisor 
\[
F_{jk}=2A_{1}-A_{j}-A_{k}+\sum_{i=2}^{7}A_{i}
\]
is a fiber of type $IV$ of an elliptic fibration  $\varphi_{jk}$. The
divisors 
\[
A_{j+6}+A_{k+12},\,A_{k+6}+A_{j+12}
\]
are the other reducible fibers of that fibration. The divisors
\[
A_{2+k}+A_{8+k},\quad k\in\{0,\dots,5\}
\]
are reducible fibers of an elliptic fibration  $\varphi_{1}$. The divisors
\[
A_{2+k}+A_{14+k},\quad k\in\{0,\dots,5\}
\]
are reducible fibers of an elliptic fibration  $\varphi_{2}$. The
curve $A_{1}$ is a section of $\varphi_{1}$ and $\varphi_{2}$.
The divisor 
\[
D_{2}=2A_{1}+A_{2}+A_{3}+A_{4}+A_{5}+A_{6}
\]
has square $2$; it is base-point free, and it contracts $A_{2},\dots,A_{6}, A_{13}, A_{19}$.
Moreover, $D_{2} \cdot A_{1}=1$, and for the remaining curves, $D_{2} \cdot A_{j}=2$.
From the equivalence relations obtained by the elliptic fibration
 with fiber $A_{2}+A_{8}$ and the elliptic fibration  $\varphi_{27}$,
we obtain 
\[
D_{2}\equiv A_{j}+A_{j+6}+A_{19}\equiv A_{j}+A_{12+6}+A_{13},\quad\forall j\in\{2,\dots,7\}.
\]
Thus we see that there is a line $L$ in the branch locus which is
the image of $A_{1}$, and the curves $A_{2},\dots,A_{6}$ are mapped
to points $p_{2},\dots,p_{6}$ in the intersection of $L$ with residual
quintic $Q$ in the branch locus. The curve $Q$ has two nodes $p$, $q$
above which are $A_{13}$ and $A_{19}$. The curve $A_{7}$ is the
strict transform of the line through $p$ and $q$. For $j\in\{2,3,4,5,6\}$,
the curve $A_{j+6}$ (resp.\ $A_{j+12}$) is the pull-back of the
line through $p$, $p_{j}$ (resp.\ $q$, $p_{j}$). This leads to the following. 

\begin{prop}
The K3 surface $X$ is the minimal resolution of the double cover
of $\,\PP^{2}$ branched over the union of a line and a quintic with
two nodes.
\end{prop}

From that description, the moduli space of K3 surfaces $X$ with $\NS X)\simeq U\oplus\mathbf{A}_{1}^{\oplus6}$
is unirational. 

\subsection{The lattice $\boldsymbol{U(2)\oplus\mathbf{A}_{1}^{\oplus6}}$ and degree 2 del Pezzo surfaces}

Let $C_{6}$ be a sextic curve in $\PP^{2}$ with nodes through
seven points $p_{1},\dots,p_{7}$ in general position. Let $Z\to\PP^{2}$
the  blow-up of the nodes; it is a degree $2$ del Pezzo surface, and
it contains $56$ $(-1)$-curves:
\begin{itemize}
\item the $7$ exceptional divisors $E_{i}$, $i=1,\dots,7$,
\item the strict transforms $L_{ij}$ of the $21$ lines through $p_{i}$, $p_{j}$
($i\neq j$),
\item the strict transforms $Q_{rs}$ of the $21$ conics that go through
points in $\{p_{1},\dots,p_{7}\}\setminus\{p_{r},p_{s}\}$,
\item the strict transforms $CU_{j}$ of the $7$ cubics that go through
the $7$ points $p_{k}$, with a double point at one of these points
$p_{j}$. 
\end{itemize}
The N\'eron--Severi lattice of $Z$ is generated by the pull-back $L'$
of a line and $E_{1},\dots,E_{7}$; it is the unimodular rank $7$
lattice 
\[
I_{1}\oplus I_{-1}^{\oplus7}.
\]
The anti-canonical divisor of the del Pezzo surface $Z$ of degree
$2$ is given by 
\[
-K_{Z}=3L'-\left(E_{1}+\dots+E_{7}\right); 
\]
this is an ample divisor. The linear system $|-K_{Z}|$ is base-point
free (see \cite[Section~3, Theorem 1]{DemazureIV}), and we remark
that each of the  $28$ divisors 
\[
E_{i}+CU_{i},\;L_{ij}+Q_{ij},\quad i,j\in\{1,\dots,7\},\; i\neq j
\]
belongs to the system $|-K_{Z}|$. 

Let $Y\to Z$ be the double cover branched over the strict transform
$C'$ of $C_{6}$ in $Z$. Since $C'\equiv-2K_{Z}$, the surface $Y$
is a smooth K3 surface. Since for any $(-1)$-curve $B$, $C'B=-2K_{Z}B=2$,
the pull-backs on $Y$ of the curves $E_{i}$, $L_{ij}$, $Q_{rs}$, $CU_{j}$ are
$(-2)$-curves. We denote these curves by $A_{i}$, $B_{ij}$, $C_{rs}$, $D_{j}$, 
respectively. Let $L$ be the pull-back in $Y$ of $L'$. The lattice
generated by $L$, $A_{1},\dots,A_{7}$ is (isometric to) $U(2)\oplus\mathbf{A}_{1}^{\oplus6}=f^{*}\NS Z)=\left(I_{1}\oplus I_{-1}^{\oplus7}\right)(2)$.
Since $f\colon Y\to Z$ is finite and the linear system $|-K_{Z}|$ is ample
and base-point free, its pull-back $D_{4}=f^{*}(-K_{Z})$ is ample
and base-point free, with $D_{4}^{2}=4$. The linear system $|D_{4}|$
is $3$-dimensional, and the invariant part is the pencil $f^{*}|-K_{Z}|$. We thus have the following. 

\begin{prop}
\label{thm:Degree2delPezzo}The image of $Y$ under the map $\pi\colon Y\to\PP^{3}$
induced from $|D_{4}|$ is a smooth quartic surface. There exist $28$
hyperplane sections of $\,Y$ which are each the union of $\,2$ conics.
\end{prop}

\begin{rem}
The moduli space of K3 surfaces of type $U(2)\oplus\mathbf{A}_{1}^{\oplus6}$
is $12$-dimensional. In \cite{Kondo3}, Kondo studies the $6$-dimensional
moduli space of curves of genus $3$ via the periods of the K3 surfaces
\mbox{$X=\{t^{4}-f_{4}(x,y,z)=0\}$} which are quadruple covers of the plane
branched over smooth quartics curves $C=\{f_{4}(x,y,z)=0\}$ (thus
of genus $3$). The double cover of $\PP^{2}$ branched over $C$
is a degree $2$ del Pezzo surface. For such K3 surfaces, the pull-backs
of the $28$ bitangents of the curve $C$ give the $28$ hyperplanes
sections of Theorem \ref{thm:Degree2delPezzo}.
\end{rem}

\begin{prop}
The moduli space of $U(2)\oplus\mathbf{A}_{1}^{\oplus6}$-polarized
K3 surfaces is unirational.
\end{prop}

\begin{proof}
Seven points $p_{1},\dots,p_{7}$ in general position in $\PP^{2}$
form an open subscheme of $(\PP^{2})^{7}$. The space of sextics with 
 the points $p_{1},\dots,p_{7}$ as seven nodes is a $14$-dimensional
linear subspace of the space of sextics. Therefore, the moduli space
of K3 surfaces that are double covers of $\PP^{2}$ branched over a
$7$-nodal sextic is unirational.
\end{proof}

\subsection{The lattice $\boldsymbol{U\oplus\mathbf{A}_{2}^{\oplus3}}$}

The K3 surface contains $10$ $\cu$-curves $A_{1},\dots,A_{10}$
with dual graph

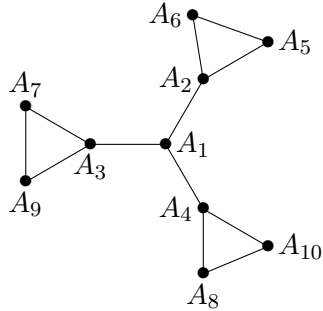
\begin{center}
\begin{tikzpicture}[scale=1]

\draw (0,0) -- (0.5,0.86);
\draw (0,0) -- (-1,0);
\draw (0,0) -- (0.5,-0.86);
\draw (-1,0) -- (-1.86,-0.5);
\draw (-1,0) -- (-1.86,0.5);
\draw (-1.86,0.5) -- (-1.86,-0.5);
\draw (0.5,0.86) -- (0.5+0.86,0.5+0.86);
\draw (0.5,0.86) -- (-0.5+0.86,2*0.86);
\draw (0.5+0.86,0.5+0.86) -- (-0.5+0.86,2*0.86);

\draw (0.5,-0.86) -- (0.5+0.86,-0.86-0.5);
\draw (0.5,-0.86) -- (0.5,-0.86*2);
\draw (0.5+0.86,-0.86-0.5) -- (0.5,-0.86*2);

\draw (0,0) node {$\bullet$};
\draw (-1,0) node {$\bullet$};
\draw (0.5,0.86) node {$\bullet$};
\draw (0.5,-0.86) node {$\bullet$};
\draw (0.5+0.86,0.5+0.86) node {$\bullet$};
\draw (-0.5+0.86,2*0.86) node {$\bullet$};
\draw (0.5+0.86,-0.86-0.5) node {$\bullet$};
\draw (0.5,-0.86*2) node {$\bullet$};
\draw (-1.86,0.5) node {$\bullet$};
\draw (-1.86,-0.5) node {$\bullet$};

\draw (0,0) node  [right]{$A_{1}$};
\draw (0.5,0.86) node [left]{$A_{2}$};
\draw (-1,0) node  [below]{$A_{3}$};
\draw (0.5,-0.86) node  [left]{$A_{4}$};
\draw (0.5+0.86,0.5+0.86) node [right]{$A_{5}$};
\draw (-0.5+0.86,2*0.86) node [left]{$A_{6}$};
\draw (0.5+0.86,-0.86-0.5) node [right]{$A_{10}$};
\draw (0.5,-0.86*2) node [below]{$A_{8}$};
\draw (-1.86,0.5) node [above]{$A_{7}$};
\draw (-1.86,-0.5) node [below]{$A_{9}$};

\end{tikzpicture}
\end{center} 

The first eight curves generate the N\'eron--Severi lattice. The divisor
\[
D_{18}=3A_{1}+5A_{2}+A_{3}+A_{4}+4A_{5}+4A_{6}
\]
is ample, of square $18$, with $D_{18} \cdot A_{j}=1$ for $j\in\{1,\dots,10\}$.
The divisors 
\[
A_{2}+A_{5}+A_{6},\,A_{3}+A_{7}+A_{9},\,A_{4}+A_{8}+A_{10}
\]
 are fibers of an elliptic fibration  of the K3 surface such that
$A_{1}$ is a section. For $i\in\{5,6\}$, $j\in\{7,9\}$, $k\in\{8,10\}$, 
the divisor 
\[
F_{ijk}=3A_{1}+2A_{2}+2A_{3}+2A_{4}+A_{i}+A_{j}+A_{k}
\]
 is the unique reducible fiber of type $IV^{*}$ of an elliptic fibration
$\varphi_{ijk}$. The divisor 
\[
D_{2}=2A_{1}+2A_{2}+A_{3}+A_{4}+A_{5}+A_{6}
\]
is nef, base-point free, of square $2$; it contracts $A_{1}$, $A_{2}$, $A_{3}$, $A_{4}$, 
and the other $(-2)$-curves have degree $1$ for that divisor. We
also have 
\[
\begin{array}{l}
D_{2}\equiv2A_{1}+A_{2}+2A_{3}+A_{4}+A_{7}+A_{9},\\
D_{2}\equiv2A_{1}+A_{2}+A_{3}+2A_{4}+A_{8}+A_{10}.
\end{array}
\]
By using the linear system $|D_{2}|$, we obtain the following. 

\begin{prop}
The surface $X$ is the double cover of $\,\PP^{2}$ branched over a
sextic curve with a $\mathbf{d}_{4}$ singularity. The three tangents
to the three branches of the singularity are tangent to another point
of the sextic curve, so that the pull-back of a tangent splits into two
$(-2)$-curves. The $6=3\cdot 2$ $(-2)$-curves above the tangent lines 
 are $A_{5}+A_{6}$, $A_{7}+A_{9}$ and $A_{8}+A_{10}$. 
\end{prop}

One can also construct this surface as a double cover of $\mathbf{F}_{4}$
branched over the unique section $s$ with $s^2=-4$ and a curve $B\in|3s+12f|$ with three
cusps singularities. 

\subsection{The lattice $\boldsymbol{U\oplus\mathbf{A}_{3}^{\oplus2}}$}

The K3 surface contains nine $\cu$-curves; the dual graph is 

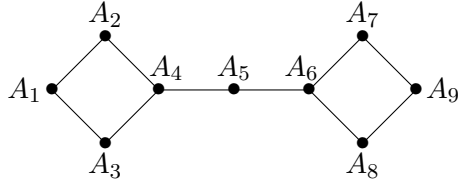
\begin{center}
\begin{tikzpicture}[scale=1]

\draw (-1,0) -- (1,0);
\draw (1,0) -- (1.71,0.71);
\draw (1,0) -- (1.71,-0.71);
\draw (1+1.42,0) -- (1.71,0.71);
\draw (1+1.42,0) -- (1.71,-0.71);
\draw (-1,0) -- (-1.71,0.71);
\draw (-1,0) -- (-1.71,-0.71);
\draw (-1-1.42,0) -- (-1.71,0.71);
\draw (-1-1.42,0) -- (-1.71,-0.71);


\draw (0,0) node {$\bullet$};
\draw (1,0) node {$\bullet$};
\draw (1.71,-0.71) node {$\bullet$};
\draw (1.71,0.71) node {$\bullet$};
\draw (1+1.42,0) node {$\bullet$};
\draw (-1,0) node {$\bullet$};
\draw (-1.71,-0.71) node {$\bullet$};
\draw (-1.71,0.71) node {$\bullet$};
\draw (-1-1.42,0) node {$\bullet$};

\draw (0,0) node [above]{$A_{5}$};
\draw (0.9,0) node [above]{$A_{6}$};
\draw (1.71,-0.71) node [below]{$A_{8}$};
\draw (1.71,0.71) node [above]{$A_{7}$};
\draw (1+1.42,0) node [right]{$A_{9}$};
\draw (-0.9,0) node [above]{$A_{4}$};
\draw (-1.71,-0.71) node [below]{$A_{3}$};
\draw (-1.71,0.71) node [above]{$A_{2}$};
\draw (-1-1.42,0) node [left]{$A_{1}$};

\end{tikzpicture}
\end{center} 

The eight $(-2)$-curves $A_{1},\dots,A_{8}$ generate the N\'eron--Severi
lattice. The divisor 
\[
D_{40}=(5,6,6,8,5,3,1,1)
\]
in the basis $A_{1},\dots,A_{8}$ is ample, of square $40$, with  $D_{40} \cdot A_{1}=D_{40} \cdot A_{9}=2$
and $D_{40} \cdot A_{j}=1$ for $j\in\{2,\dots,8\}$. The divisors 
\[
F_{1}=A_{1}+A_{2}+A_{3}+A_{4},\quad F_{2}=A_{6}+A_{7}+A_{8}+A_{9}
\]
 are reducible fibers of an elliptic fibration, where $A_{5}$ is
a section. By Theorem \ref{thm:SaintDonat-2}, case i) a), we have the following. 

\begin{prop}
The linear system $|4F_{1}+2A_{5}|$ defines a morphism $\varphi\colon X\to\mathbf{F}_{4}$
branched over the unique section $s$ with $s^2=-4$ and a curve $B\in|3s+12f|$. The curve
$B$ has two $\mathbf{a_{3}}$ singularities $p$, $q$. The pull-backs 
of the fibers through $p$, $q$ are the fibers $F_{1}$, $F_{2}$.
\end{prop}

\subsection{The lattice $\boldsymbol{U\oplus\mathbf{A}_{2}\oplus\mathbf{A}_{4}}$}

The surface $X$ contains nine $(-2)$-curves, with dual graph
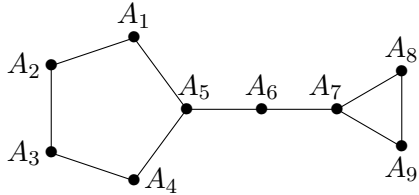
\begin{center}
\begin{tikzpicture}[scale=1]

\draw (0.3,-0.95) -- (1,0);
\draw (1,0) -- (0.3,0.95);
\draw (0.3,0.95) -- (-0.8,0.58);
\draw (-0.8,0.58) -- (-0.8,-0.58);
\draw (-0.8,-0.58) -- (0.3,-0.95);
\draw (3,0) -- (1,0);
\draw (3+0.86,0.5) -- (3,0);
\draw (3+0.86,-0.5) -- (3,0);
\draw (3+0.86,-0.5) -- (3+0.86,0.5);


\draw (1,0) node {$\bullet$};
\draw (0.3,-0.95) node {$\bullet$};
\draw (0.3,0.95) node {$\bullet$};
\draw (-0.8,-0.58) node {$\bullet$};
\draw (-0.8,0.58) node {$\bullet$};
\draw (2,0) node {$\bullet$};
\draw (3,0) node {$\bullet$};
\draw (3+0.86,-0.5) node {$\bullet$};
\draw (3+0.86,0.5) node {$\bullet$};

\draw (1.1,0) node [above]{$A_{5}$};
\draw (0.3,0.95) node [above]{$A_{1}$};
\draw (-0.8,0.58) node [left]{$A_{2}$};
\draw (-0.8,-0.58) node [left]{$A_{3}$};
\draw (0.3,-0.95) node [right]{$A_{4}$};
\draw (2,0) node [above]{$A_{6}$};
\draw (2.85,0) node [above]{$A_{7}$};
\draw (3+0.86,0.5) node [above]{$A_{8}$};
\draw (3+0.86,-0.5) node [below]{$A_{9}$};

\end{tikzpicture}
\end{center} 

The curves $A_{1},\dots,A_{8}$ generate  the N\'eron--Severi lattice.
The divisor 
\[
D_{42}=(7,6,6,7,9,5,2,0)
\]
 in the basis $A_{1},\dots,A_{8}$ is ample, of square $42$, with $D_{42} \cdot A_{j}=1$
for $j\in\{1,\dots,7\}$ and $D_{42} \cdot A_{8}=D_{42} \cdot A_{9}=2$. The divisors
\[
F_{1}=A_{1}+\dots+A_{5},\quad F_{2}=A_{7}+A_{8}+A_{9}
\]
are reducible fibers of an elliptic fibration  for which $A_{6}$
is a section. By Theorem \ref{thm:SaintDonat-2}, case i) a), we have the following. 

\begin{prop}
The linear system $|4F_{1}+2A_{6}|$ defines a morphism $\varphi\colon X\to\mathbf{F}_{4}$
branched over the unique section $s$ with $s^2=-4$ and a curve $B\in|3s+12f|$. The curve
$B$ has one $\mathbf{a_{4}}$ singularity $p$ and a cusp $q$. The
pull-backs of the fibers through \mbox{$p$, $q$} are the fibers $F_{1}$, $F_{2}$.
\end{prop}

\subsection{The lattice $\boldsymbol{U\oplus\mathbf{A}_{1}\oplus\mathbf{A}_{5}}$}

The surface $X$ contains nine $(-2)$-curves, with dual graph

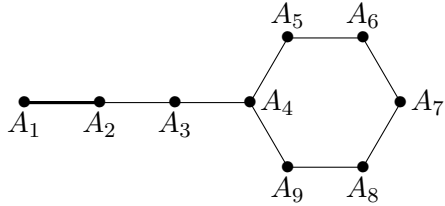
\begin{center}
\begin{tikzpicture}[scale=1]

\draw (1,0) -- (0.5,0.86);
\draw (0.5,0.86) -- (-0.5,0.86);
\draw (-1,0) -- (-0.5,0.86);
\draw (-1,0) -- (-0.5,-0.86);
\draw (-0.5,-0.86) -- (0.5,-0.86);
\draw (1,0) -- (0.5,-0.86);
\draw (-1,0) -- (-2,0);
\draw (-2,0) -- (-4,0);
\draw [very thick] (-3,0) -- (-4,0);

\draw (1,0) node {$\bullet$};
\draw (0.5,0.86) node {$\bullet$};
\draw (-1,0) node {$\bullet$};
\draw (-0.5,0.86) node {$\bullet$};
\draw (-0.5,-0.86) node {$\bullet$};
\draw (0.5,-0.86)  node {$\bullet$};
\draw (-2,0) node {$\bullet$};
\draw (-3,0) node {$\bullet$};
\draw (-4,0) node {$\bullet$};

\draw (1,0) node  [right]{$A_{7}$};
\draw (0.5,0.86) node [above]{$A_{6}$};
\draw (-1,0) node  [right]{$A_{4}$};
\draw (-0.5,0.86) node  [above]{$A_{5}$};
\draw (-0.5,-0.86) node  [below]{$A_{9}$};
\draw (0.5,-0.86) node   [below]{$A_{8}$};
\draw (-2,0) node  [below]{$A_{3}$};
\draw (-3,0) node [below]{$A_{2}$};
\draw (-4,0) node [below]{$A_{1}$};

\end{tikzpicture}
\end{center} 

The curves $A_{1},\dots,A_{8}$ generate  the N\'eron--Severi lattice.
The divisor 
\[
D_{76}=(10,12,7,3,0,-2,-3,-2)
\]
 in the basis $A_{1},\dots,A_{8}$ is ample, of square $76$. The divisors
\[
F_{1}=A_{1}+A_{2},\quad F_{2}=A_{4}+A_{5}+A_{6}+A_{7}+A_{8}+A_{9}
\]
 are the reducible fibers of an elliptic fibration  such that $A_{3}$
is a section. By Theorem \ref{thm:SaintDonat-2}, case i) a), we have the following. 

\begin{prop}
The linear system $|4F_{1}+2A_{3}|$ defines a morphism $\varphi\colon X\to\mathbf{F}_{4}$
branched over the unique section $s$ with $s^2=-4$ and a curve $B\in|3s+12f|$. The curve
$B$ has a node $p$ and one $\mathbf{a_{5}}$ singularity $q$. The
pull-back of the fibers through \mbox{$p$, $q$} are the fibers $F_{1}$, $F_{2}$.
\end{prop}

\subsection{The lattice $\boldsymbol{U\oplus\mathbf{A}_{6}}$}

The surface $X$ contains eight $(-2)$-curves $A_{1},\dots,A_{8}$,
with dual graph

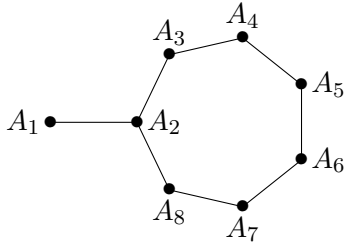
\begin{center}
\begin{tikzpicture}[scale=1]

\draw (-1.15,0) -- (-2.3,0);

\draw (-1.15,0) -- (-0.718,-0.9);
\draw (-0.718,-0.9) -- (0.256,-1.123);
\draw (0.256,-1.123) -- (1.038,-0.5);
\draw (1.038,-0.5) -- (1.038,0.5);
\draw (1.038,0.5) -- (0.256,1.123);
\draw (0.256,1.123) -- (-0.718,0.9);
\draw (-0.718,0.9) -- (-1.15,0);

\draw (-1.15,0) node {$\bullet$};
\draw (-2.3,0) node {$\bullet$};
\draw (-0.718,-0.9) node {$\bullet$};
\draw (0.256,-1.123) node {$\bullet$};
\draw (1.038,-0.5) node {$\bullet$};
\draw (1.038,0.5) node {$\bullet$};
\draw (0.256,1.123) node {$\bullet$};
\draw (-0.718,0.9) node {$\bullet$};

\draw (-2.3,0) node  [left]{$A_{1}$};
\draw (-1.15,0) node  [right]{$A_{2}$};
\draw (-0.718,-0.9) node [below]{$A_{8}$};
\draw (0.256,-1.123) node [below]{$A_{7}$};
\draw (1.038,-0.5) node [right]{$A_{6}$};
\draw (1.038,0.5) node [right]{$A_{5}$};
\draw (0.256,1.123) node [above]{$A_{4}$};
\draw (-0.718,0.9) node [above]{$A_{3}$};

\end{tikzpicture}
\end{center} 

These eight curves generate  the N\'eron--Severi lattice; the divisor
\[
D_{84}=(7,15,12,10,9,9,10,12)
\]
in the basis $A_{1},\dots,A_{8}$ has square $84$, and the curves $A_{j}$
have degree $1$ for $D_{84}$. The divisor 
\[
F=A_{2}+\dots+A_{8}
\]
 is a fiber of an elliptic fibration  for which $A_{1}$ is a section.
By Theorem \ref{thm:SaintDonat-2}, case i) a), we have the following. 

\begin{prop}
The linear system $|4F+2A_{1}|$ defines a morphism $\varphi\colon X\to\mathbf{F}_{4}$
branched over the unique section $s$ with $s^2=-4$ and a curve $B\in|3s+12f|$. The curve
$B$ has an $\mathbf{a_{6}}$ singularity $q$. The pull-back of the
fiber through $q$ is the fiber $F$.
\end{prop}

\begin{rem}
The divisor 
\[
D_{2}=(2,4,3,2,1,1,2,3)
\]
in the basis $A_{1},\dots,A_{8}$ is base-point free of square $2$. It
defines a double cover branched over a sextic with an $\mathbf{e}_{6}$
singularity.
\end{rem}

\subsection{The lattice $\boldsymbol{U\oplus\mathbf{A}_{2}\oplus\mathbf{D}_{4}}$}

The surface $X$ contains nine $(-2)$-curves, with dual graph
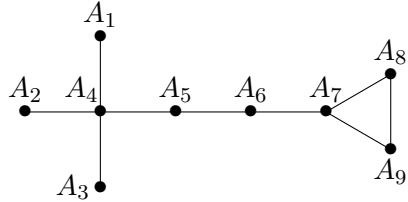
\begin{center}
\begin{tikzpicture}[scale=1]

\draw (3,0) -- (-1,0);
\draw (0,1) -- (0,-1);
\draw (3,0) -- (3+0.86,-0.5);
\draw (3+0.86,-0.5) -- (3+0.86,0.5);
\draw (3,0) -- (3+0.86,0.5);


\draw (-1,0) node {$\bullet$};
\draw (0,0) node {$\bullet$};
\draw (0,1) node {$\bullet$};
\draw (0,-1) node {$\bullet$};
\draw (1,0) node {$\bullet$};
\draw (2,0) node {$\bullet$};
\draw (3,0) node {$\bullet$};
\draw (3+0.86,-0.5) node {$\bullet$};
\draw (3+0.86,0.5) node {$\bullet$};

\draw (-1,0) node [above]{$A_{2}$};
\draw (-0.25,0) node [above]{$A_{4}$};
\draw (0,1) node [above]{$A_{1}$};
\draw (0,-1) node [left]{$A_{3}$};
\draw (1,0) node [above]{$A_{5}$};
\draw (2,0) node [above]{$A_{6}$};
\draw (3,0) node [above]{$A_{7}$};
\draw (3+0.86,-0.5) node [below]{$A_{9}$};
\draw (3+0.86,0.5) node [above]{$A_{8}$};

\end{tikzpicture}
\end{center} 

The curves $A_{1},\dots,A_{8}$ generate  the N\'eron--Severi lattice.
In that basis, the divisor
\[
D_{56}=(7,7,7,15,10,6,3,1)
\]
is ample of square $56$, with $D_{56} \cdot A_{j}=1$ for $j\in\{1,\dots,8\}$
and $D_{56} \cdot A_{9}=4$. The divisors 
\[
F_{1}=A_{7}+A_{8}+A_{9},\quad F_{2}=A_{1}+A_{2}+A_{3}+A_{5}+2A_{4}
\]
 are fibers of an elliptic fibration  for which $A_{6}$ is a section.
By Theorem \ref{thm:SaintDonat-2}, case i) a), we have the following. 

\begin{prop}
The linear system $|4F_{1}+2A_{6}|$ defines a morphism $\varphi\colon X\to\mathbf{F}_{4}$
branched over the unique section $s$ with $s^2=-4$ and a curve $B\in|3s+12f|$. The curve
$B$ has a cusp $p$ and one $\mathbf{d_{4}}$ singularity $q$. The
pull-backs of the fibers through \mbox{$p$, $q$} are the fibers $F_{1}$, $F_{2}$.
\end{prop}

\subsection{The lattice $\boldsymbol{U\oplus\mathbf{A}_{1}\oplus\mathbf{D}_{5}}$}

The surface $X$ contains nine $(-2)$-curves, with dual graph
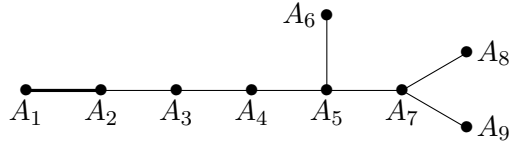
\begin{center}
\begin{tikzpicture}[scale=1]

\draw (5,0) -- (0,0);
\draw (5,0) -- (5+0.86,-0.5);
\draw (5,0) -- (5+0.86,0.5);
\draw (4,0) -- (4,1);
\draw [very thick] (0,0) -- (1,0);

\draw (0,0) node {$\bullet$};
\draw (1,0) node {$\bullet$};
\draw (2,0) node {$\bullet$};
\draw (3,0) node {$\bullet$};
\draw (4,0) node {$\bullet$};
\draw (5,0) node {$\bullet$};
\draw (4,1) node {$\bullet$};
\draw (5+0.86,-0.5) node {$\bullet$};
\draw (5+0.86,0.5) node {$\bullet$};

\draw (0,0) node [below]{$A_{1}$};
\draw (1,0) node [below]{$A_{2}$};
\draw (2,0) node [below]{$A_{3}$};
\draw (3,0) node [below]{$A_{4}$};
\draw (4,0) node [below]{$A_{5}$};
\draw (5,0) node [below]{$A_{7}$};
\draw (4,1) node [left]{$A_{6}$};
\draw (5+0.86,-0.5) node [right]{$A_{9}$};
\draw (5+0.86,0.5) node [right]{$A_{8}$};

\end{tikzpicture}
\end{center} 

The curves $A_{1},\dots,A_{8}$ generate  the N\'eron--Severi lattice.
In that basis, the divisor
\[
D_{96}=(9,12,8,5,3,1,1,0)
\]
is ample of square $96$. We have $D_{96} \cdot A_{1}=6$, $D_{96} \cdot A_{2}=2$
and $D_{96} \cdot A_{j}=1$ for $j\geq3$. The divisors 
\[
F_{1}=A_{1}+A_{2},\quad F_{2}=A_{4}+A_{6}+2A_{5}+2A_{7}+A_{8}+A_{9}
\]
are fibers of an elliptic fibration  for which $A_{3}$ is a section.
By Theorem \ref{thm:SaintDonat-2}, case i) a), we have the following. 

\begin{prop}
The linear system $|4F_{1}+2A_{3}|$ defines a morphism $\varphi\colon X\to\mathbf{F}_{4}$
branched over the unique section $s$ with $s^2=-4$ and a curve $B\in|3s+12f|$. The curve
$B$ has a node $p$ and one $\mathbf{d_{5}}$ singularity $q$. The
pull-backs of the fibers through \mbox{$p$, $q$} are the fibers $F_{1}$, $F_{2}$.
\end{prop}

\subsection{The lattice $\boldsymbol{U\oplus\mathbf{E}_{6}}$}

The surface $X$ contains eight $(-2)$-curves, with dual graph
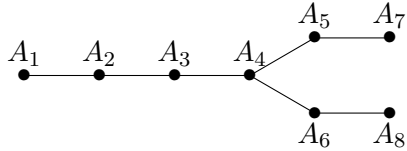
\begin{center}
\begin{tikzpicture}[scale=1]

\draw (3,0) -- (0,0);
\draw (3,0) -- (3+0.86,-0.5);
\draw (3,0) -- (3+0.86,0.5);
\draw (3+0.86,0.5) -- (4+0.86,0.5);
\draw (3+0.86,-0.5) -- (4+0.86,-0.5);


\draw (0,0) node {$\bullet$};
\draw (1,0) node {$\bullet$};
\draw (2,0) node {$\bullet$};
\draw (3,0) node {$\bullet$};
\draw (3+0.86,0.5) node {$\bullet$};
\draw (4+0.86,0.5) node {$\bullet$};
\draw (3+0.86,-0.5) node {$\bullet$};
\draw (4+0.86,-0.5) node {$\bullet$};

\draw (0,0) node [above]{$A_{1}$};
\draw (1,0) node [above]{$A_{2}$};
\draw (2,0) node [above]{$A_{3}$};
\draw (3,0) node [above]{$A_{4}$};
\draw (3+0.86,0.5) node [above]{$A_{5}$};
\draw (3+0.86,-0.5) node [below]{$A_{6}$};
\draw (4+0.86,0.5) node [above]{$A_{7}$};
\draw (4+0.86,-0.5) node [below]{$A_{8}$};

\end{tikzpicture}
\end{center} 

These eight curves $A_{1},\dots,A_{8}$ generate  the N\'eron--Severi
lattice. In that basis, the divisor
\[
D_{234}=(12,25,39,54,35,35,17,17)
\]
 is ample, of square $234$, with  $D_{234} \cdot A_{j}=1$ for $j\in\{1,\dots,8\}$.
The divisor 
\[
F=A_{2}+2A_{3}+3A_{4}+2A_{5}+2A_{6}+A_{7}+A_{8}
\]
is a fiber of an elliptic fibration  for which $A_{1}$ is a section.
By Theorem \ref{thm:SaintDonat-2}, case i) a), we have the following. 

\begin{prop}
The linear system $|4F+2A_{1}|$ defines a morphism $\varphi\colon X\to\mathbf{F}_{4}$
branched over the unique section $s$ with $s^2=-4$ and a curve $B\in|3s+12f|$. The curve
$B$ has an $\mathbf{e_{6}}$ singularity $q$. The pull-back of the
fiber through $q$ is the fiber $F$.
\end{prop}

\section{Rank 9 lattices }

\subsection{The lattice $\boldsymbol{U\oplus\mathbf{E}_{7}}$}

The surface $X$ contains nine $(-2)$-curves, with dual graph
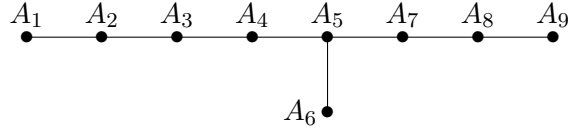
\begin{center}
\begin{tikzpicture}[scale=1]

\draw (7,0) -- (0,0);
\draw (4,0) -- (4,-1);


\draw (0,0) node {$\bullet$};
\draw (1,0) node {$\bullet$};
\draw (2,0) node {$\bullet$};
\draw (3,0) node {$\bullet$};
\draw (4,0) node {$\bullet$};
\draw (5,0) node {$\bullet$};
\draw (6,0) node {$\bullet$};
\draw (7,0) node {$\bullet$};
\draw (4,-1) node {$\bullet$};

\draw (0,0) node [above]{$A_{1}$};
\draw (1,0) node [above]{$A_{2}$};
\draw (2,0) node [above]{$A_{3}$};
\draw (3,0) node [above]{$A_{4}$};
\draw (4,0) node [above]{$A_{5}$};
\draw (5,0) node [above]{$A_{7}$};
\draw (6,0) node [above]{$A_{8}$};
\draw (7,0) node [above]{$A_{9}$};
\draw (4,-1) node [left]{$A_{6}$};

\end{tikzpicture}
\end{center} 

These nine curves $A_{1},\dots,A_{9}$ generate  the N\'eron--Severi
lattice. In that basis, the divisor
\[
D_{532}=(19,39,60,82,105,52,77,50,24)
\]
 is ample, of square $532$, with $D_{532} \cdot A_{j}=1$ for $j\leq8$ and $D_{532} \cdot A_{9}=2$.
The divisor 
\[
F=A_{2}+2A_{3}+3A_{4}+4A_{5}+2A_{6}+3A_{7}+2A_{8}+A_{9}
\]
is a fiber of an elliptic fibration, for which $A_{1}$ is a section.
 By Theorem \ref{thm:SaintDonat-2}, case i) a), we have the following. 

 \begin{prop}
The linear system $|4F+2A_{1}|$ defines a morphism $\varphi\colon X\to\mathbf{F}_{4}$
branched over the unique section $s$ with $s^2=-4$ and a curve $B\in|3s+12f|$. The curve
$B$ has an $\mathbf{e_{7}}$ singularity $q$. The pull-back of the
fiber through $q$ is the fiber $F$.
\end{prop}

\subsection{The lattice $\boldsymbol{U\oplus\mathbf{D}_{6}\oplus\mathbf{A}_{1}}$}

The surface $X$ contains $10$ $(-2)$-curves, with dual graph
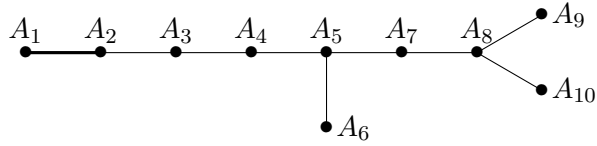
\begin{center}
\begin{tikzpicture}[scale=1]

\draw (6,0) -- (0,0);
\draw (4,0) -- (4,-1);
\draw (6,0) -- (6+0.86,0.5);
\draw (6,0) -- (6+0.86,-0.5);

\draw [very thick] (0,0) -- (1,0);

\draw (0,0) node {$\bullet$};
\draw (1,0) node {$\bullet$};
\draw (2,0) node {$\bullet$};
\draw (3,0) node {$\bullet$};
\draw (4,0) node {$\bullet$};
\draw (5,0) node {$\bullet$};
\draw (6,0) node {$\bullet$};
\draw (6+0.86,0.5) node {$\bullet$};
\draw (6+0.86,-0.5) node {$\bullet$};
\draw (4,-1) node {$\bullet$};

\draw (0,0) node [above]{$A_{1}$};
\draw (1,0) node [above]{$A_{2}$};
\draw (2,0) node [above]{$A_{3}$};
\draw (3,0) node [above]{$A_{4}$};
\draw (4,0) node [above]{$A_{5}$};
\draw (5,0) node [above]{$A_{7}$};
\draw (6,0) node [above]{$A_{8}$};
\draw (6+0.86,0.5) node [right]{$A_{9}$};
\draw (6+0.86,-0.5) node [right]{$A_{10}$};

\draw (4,-1) node [right]{$A_{6}$};

\end{tikzpicture}
\end{center} 

The nine curves $A_{1},\dots,A_{9}$ generate  the N\'eron--Severi lattice.
In that basis, the divisor
\[
D_{148}=(10,15,11,8,6,2,3,1,0)
\]
is ample, of square $148$, with $D_{148} \cdot A_{1}=10$, $D_{148} \cdot A_{6}=2$
and $D_{148} \cdot A_{j}=1$ for $j\neq1,6$. The divisors
\[
F_{1}=A_{1}+A_{2},\quad F_{2}=A_{4}+A_{6}+2A_{5}+2A_{7}+2A_{8}+A_{9}+A_{10}
\]
are fibers of an elliptic fibration. By Theorem \ref{thm:SaintDonat-2}, 
case i) a), we have the following. 

\begin{prop}
The linear system $|4F_{1}+2A_{3}|$ defines a morphism $\varphi\colon X\to\mathbf{F}_{4}$
branched over the unique section $s$ with $s^2=-4$ and a curve $B\in|3s+12f|$. The curve
$B$ has a node $p$ and one $\mathbf{d_{6}}$ singularity $q$. The
pull-backs of the fibers through \mbox{$p$, $q$} are the fibers $F_{1}$, $F_{2}$.
\end{prop}

\subsection{The lattice $\boldsymbol{U\oplus\mathbf{D}_{4}\oplus\mathbf{A}_{1}^{\oplus3}}$}

The surface $X$ contains $15$ $(-2)$-curves $A_{1},\dots,A_{15}$,
with dual graph
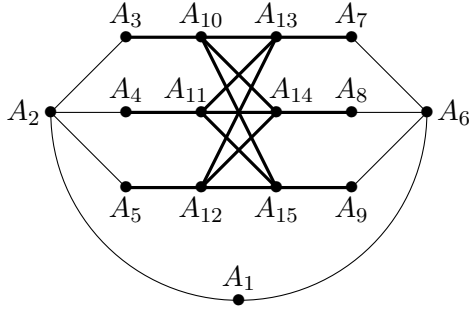
\begin{center}
\begin{tikzpicture}[scale=1]

\draw (5,0) -- (0,0);
\draw [very thick] (1,1) -- (4,1);
\draw [very thick] (1,0) -- (4,0);
\draw [very thick] (1,-1) -- (4,-1);
\draw (0,0) -- (1,1);
\draw (0,0) -- (1,-1);
\draw (5,0) -- (4,1);
\draw (5,0) -- (4,-1);
\draw [very thick] (2,1) -- (3,0);
\draw [very thick] (2,1) -- (3,-1);
\draw [very thick] (2,0) -- (3,1);
\draw [very thick] (2,0) -- (3,-1);
\draw [very thick] (2,-1) -- (3,1);
\draw [very thick] (2,-1) -- (3,0);

\draw (5,0) arc (0:-180:2.5);

\draw (0,0) node {$\bullet$};
\draw (1,0) node {$\bullet$};
\draw (2,0) node {$\bullet$};
\draw (3,0) node {$\bullet$};
\draw (4,0) node {$\bullet$};
\draw (5,0) node {$\bullet$};
\draw (1,1) node {$\bullet$};
\draw (2,1) node {$\bullet$};
\draw (3,1) node {$\bullet$};
\draw (4,1) node {$\bullet$};
\draw (1,-1) node {$\bullet$};
\draw (2,-1) node {$\bullet$};
\draw (3,-1) node {$\bullet$};
\draw (4,-1) node {$\bullet$};
\draw (2.5,-2.5) node {$\bullet$};

\draw (0,0) node [left]{$A_{2}$};
\draw (1,0) node [above]{$A_{4}$};
\draw (1.8,0) node [above]{$A_{11}$};
\draw (3.2,0) node [above]{$A_{14}$};
\draw (4,0) node [above]{$A_{8}$};
\draw (5,0) node [right]{$A_{6}$};

\draw (1,1) node [above]{$A_{3}$};
\draw (2,1) node [above]{$A_{10}$};
\draw (3,1) node [above]{$A_{13}$};
\draw (4,1) node [above]{$A_{7}$};
\draw (1,-1) node [below]{$A_{5}$};
\draw (2,-1) node [below]{$A_{12}$};
\draw (3,-1) node [below]{$A_{15}$};
\draw (4,-1) node [below]{$A_{9}$};
\draw (2.5,-2.5) node [above]{$A_{1}$};

\end{tikzpicture}
\end{center} 

The curves $A_{1},\dots,A_{9}$ generate  the N\'eron--Severi lattice.
In that basis, the divisor 
\[
D_{44}=(6,7,3,3,3,7,3,3,3)
\]
is ample, of square $44$, with $D_{44} \cdot A_{j}=1$ for $j\in\{2,\dots,9\}$,
$D_{44} \cdot A_{1}=2$ and $D_{44} \cdot A_{j}=6$ for $j\in\{10,\dots,15\}$.
 The divisor 
\[
D_{2}=(1,2,1,1,1,2,1,1,1)
\]
is nef, base-point free, with $D_{2} \cdot A_{j}=0$ for $j\in\{2,\dots,9\}$
and $D_{2} \cdot A_{j}=2$ for $j\not\in\{2,\dots,9\}$. Let $\eta\colon X\to\PP^{2}$
be the map associated to the linear system $|D_{2}|$. One thus has the following. 

\begin{prop}
The branch curve of $\eta$ is a sextic with two $\mathbf{d}_{4}$
singularities $p$, $q$. The image of the curve $A_{1}$ by the double
cover map is the line through $p$, $q$. The six lines $C_{10},\dots,C_{15}$
that are tangent to the six branches of the two $\mathbf{d}_{4}$
singularities are the images of $A_{10},\dots,A_{15}$. 
\end{prop}

\subsection{The lattice $\boldsymbol{U\oplus\mathbf{A}_{1}^{\oplus7}}$}

The K3 surface contains $37$ $\cu$-curves, which we denote by $A_{0},A_{1},\dots,A_{8}$ and $A_{ij}$ for $i,j\in\{1,\dots,8\}$ with $i<j$. We have $A_{0} \cdot A_{j}=1$ for $j\in\{1,\dots,8\}$, $A_{0} \cdot A_{ij}=0$ for  $j\in\{1,\dots,8\}$ and $i<j$, and we have $A_{j} \cdot A_{st}=2$ if and only if $j\in\{s,t\}$; else $A_{j} \cdot A_{st}=0$. Moreover, for $\{u,v\}\neq\{s,t\}$, we have
\[
A_{uv} \cdot A_{st}=2(1-|\{u,v\}\cap\{s,t\}|).
\]

\begin{prop}
The K3 surface is a double cover of $\,\PP^{2}$ branched over a sextic
curve which is the union of a smooth conic $C_{0}$ and a quartic
$Q$.
\end{prop}

The curves $A_{1},\dots,A_{8}$ are mapped to the $8$ nodes of the
sextic curve, the $28$ divisors $A_{ij}$ are mapped onto the $28$
lines through the $8$ nodes, and $A_{0}$ is mapped onto the conic
$C_{0}$. The moduli space of such surfaces is unirational.

\subsection{The lattice $\boldsymbol{U(2)\oplus\mathbf{A}_{1}^{\oplus7}}$ and degree 1 del Pezzo surfaces\label{subsec:U(2)(A1)7DelPezzoDeg1}}

Let $C_{6}$ be a sextic curve in $\PP^{2}$ with $8$ nodes in general
position. Let $Z\to\PP^{2}$ the  blow-up of the nodes; it is a degree
$1$ del Pezzo surface and contains $240$ $(-1)$-curves:
\begin{itemize}
\item the $8$ exceptional divisors $E_{i}$, $i=1,\dots,8$,
\item the strict transform $L_{ij}$ of the $28$ lines through $p_{i}$, $p_{j}$
  ($i\neq j$),
  \item the strict transforms $QO_{rst}$ of the $56$ conics that go through
points in $\{p_{1},\dots,p_{8}\}\setminus\{p_{r},p_{s},p_{t}\}$,
\item the strict transforms $CU_{rj}$ of the $56$ cubics that go through
$7$ points $p_{k}$ (with $k\neq r$) with a double point at $p_{j}$
($j\neq r$),
\item the strict transforms $QA_{rst}$ of the $56$ quartics through the
$8$ points $p_{j}$ with double points at $p_{r}$, $p_{s}$, $p_{t}$, 
\item the strict transforms $QI_{ij}$ of the $28$ quintics through the
$8$ points $p_{j}$ with double points at $p_{i}$, $p_{j}$ ($i\neq j$),
\item the strict transforms $S_{j}$ of the $8$ sextics through $8$ points
with double points at all except a single point $p_{j}$ with multiplicity
$3$.
\end{itemize}

The N\'eron--Severi lattice of $Z$ is generated by the pull-back $L'$
of a line and $E_{1},\dots,E_{8}$; it is the unimodular rank $9$
lattice 
\[
I_{1}\oplus I_{-1}^{\oplus8}.
\]
The anti-canonical divisor of the del Pezzo surface $Z$ of degree
$1$ is given by 
\[
-K_{Z}=3L'-\left(E_{1}+\dots+E_{8}\right);
\]
this is an ample divisor. Moreover, we remark that each of the $120$
divisors 
\[
E_{j}+S_{j},\,L_{ij}+QI_{ij},\,CU_{ij}+CU_{ji},\,QO_{ijk}+QA_{ijk},\quad\{i,j,k\}\subset\{1,\dots,8\}\text{ with }|\{i,j,k\}|=3,
\]
belongs to the base-point free (see \cite[Section 3, Theorem 1]{DemazureIV})
linear system $|-2K_{Z}|$. 

The double cover $f\colon Y\to Z$ branched over the strict transform of
$C_{6}$ is a smooth K3 surface. The pull-backs of the $240$ $(-1)$-curves
are $(-2)$-curves (the only negative curves on a K3 surface are $\cu$-curves).
We denote by $A_{1},\dots,A_{8}$ the pull-backs on $Y$ of the curves
$E_{i}$ and by $L$ the pull-back of $L$. Naturally, the lattice
$f^{*}\NS Z)$ is $\left(I_{1}\oplus I_{-1}^{\oplus8}\right)(2)$,
which is also the lattice generated by $L,A_{1},\dots,A_{8}$ and
is isometric to $U(2)\oplus\mathbf{A}_{1}^{\oplus7}$. 

Since $f\colon Y\to Z$ is finite, its pull-back $D_{2}=f^{*}(-K_{Z})$
is ample, with $D_{2}^{2}=2$. Since $|-2K_{Z}|$ is base-point free,
the system $|D_{2}|$ is non-hyperelliptic; thus it defines a double
cover 
\[
\pi\colon Y\to\PP^{2}
\]
branched over a sextic curve $\tilde{C}_{6}$, which is smooth since
$D_{2}$ is ample. Let $A$ be the pull-back on $Y$ of a $(-1)$-curve
$E$; this is a $(-2)$-curve. We have $D_{2}A=2(-K_{Z}E)=2$. Moreover,
we see from the above description of the $120$ divisors in $|-2K_{Z}|$
that there exists a $(-2)$-curve $B$ such that $A+B\equiv2D_{2}$;
in particular, the image of $A+B$ by $\pi$ is a conic which is $6$-tangent
to the sextic $\tilde{C}_{6}$. We thus obtain the following result. 

\begin{prop}
The branch curve $\tilde{C}_{6}$ of the morphism $\pi$ is a smooth
sextic curve which possesses $120$ $6$-tangent conics. 
\end{prop}

\begin{rem}
Let $B$ be the strict transform in $Z$ of $C_{6}$; this is a smooth
genus $2$ curve such that $B\equiv E_{j}+S_{j}$. Thus the ramification
locus $R$ of the double cover $f\colon Y\to Z$ is a smooth genus $2$ curve
with $2R\equiv f^{*}(E_{j}+S_{j})\equiv2D_{2}$, and (since $\NS X)=\Pic X)$)
the divisor $R$ is in the linear system $|D_{2}|$, and its image
by $\pi$ is a line in $\PP^{2}$. 
\end{rem}

\begin{rem}
The morphism $\pi$ is branched over a smooth sextic curve $\tilde{C}_{6}$
(thus of genus $10$), whereas $f$ is branched over a smooth genus
$2$ curve $B$, so that there exist (at least) two distinct non-symplectic
involutions $\iota_{1}$, $\iota_{2}$ on the same K3 surface. In fact,
 according to Kondo \cite{Kondo}, the automorphism group of such
a K3 surface is $\ZZ/2\ZZ\times\ZZ/2\ZZ$, so that we know the generators
of $\aut(X)$.
\end{rem}

The curve $\tilde{C}_{6}$ thus has $120$ conics that are $6$-tangent. 
As far as we know, this is the record for a smooth sextic. 

\begin{rem}
By a result of Degtyarev \cite{Degtyarev}, a smooth plane sextic
curve can have at most $72$ tritangent lines. Moreover, there is an
example of such a sextic curve by Mukai (see \cite{MukaiTri}). In \cite{Elkies},
Elkies gives an example of an irreducible sextic curve and $1240$
conics that are $6$-tangent to it; that curve has a unique node. 
\end{rem}

\subsection{The lattice $\boldsymbol{U\oplus\mathbf{A}_{7}}$}

The surface $X$ contains nine $(-2)$-curves $A_{1},\dots,A_{9}$,
with dual graph

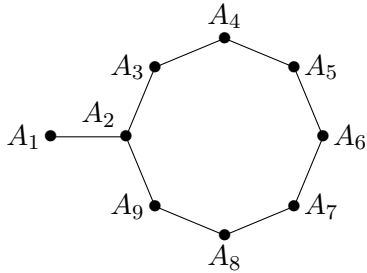
\begin{center}
\begin{tikzpicture}[scale=1]

\draw (-1.31,0) -- (-2.31,0);
\draw  (1.31, 0 ) -- (0.925, 0.925);
\draw  (0.925, 0.925) -- (0, 1.31);
\draw  (0, 1.31) -- (-0.925, 0.925);
\draw  (-0.925, 0.925) -- (-1.31, 0);
\draw  (-1.31, 0) -- (-0.925, -0.925);
\draw  (-0.925, -0.925) -- (0, -1.31);
\draw  (0, -1.31) -- (0.925, -0.925);
\draw  (0.925, -0.925) -- (1.31, 0 );

\draw  (1.31, 0 ) node {$\bullet$};
\draw  (0.925, 0.925) node {$\bullet$};
\draw  (0, 1.31) node {$\bullet$};
\draw  (-0.925, 0.925) node {$\bullet$};
\draw  (-1.31, 0) node {$\bullet$};
\draw  (-0.925, -0.925) node {$\bullet$};
\draw  (0, -1.31) node {$\bullet$};
\draw  (0.925, -0.925) node {$\bullet$};
\draw  (-2.31, 0 ) node {$\bullet$};

\draw (-2.31,0) node  [left]{$A_{1}$};
\draw  (1.31, 0 ) node [right]{$A_{6}$};
\draw  (0.925, 0.925) node [right]{$A_{5}$};
\draw  (0, 1.31) node [above]{$A_{4}$};
\draw  (-0.925, 0.925) node [left]{$A_{3}$};
\draw  (-1.31, 0) node [above left]{$A_{2}$};
\draw  (-0.925, -0.925) node [left]{$A_{9}$};
\draw  (0, -1.31) node [below]{$A_{8}$};
\draw  (0.925, -0.925) node [right]{$A_{7}$};

\end{tikzpicture}
\end{center} 

These nine curves generate  the N\'eron--Severi lattice; the divisor
\[
D_{120}=(9,19,15,12,10,9,10,12,15)
\]
 in that basis has square $120$, is ample, with $D_{120} \cdot A_{j}=1$
for $j\neq6$ and $D_{120} \cdot A_{6}=2$. The divisor $F=\sum_{j=2}^{9}A_{j}$
is a fiber of an elliptic fibration  for which $A_{1}$ is a section.
By Theorem \ref{thm:SaintDonat-2}, case i) a), we have the following. 

\begin{prop}
The linear system $|4F+2A_{1}|$ defines a morphism $\varphi\colon X\to\mathbf{F}_{4}$
branched over the unique section $s$ with $s^2=-4$ and a curve $B\in|3s+12f|$. The curve
$B$ has an $\mathbf{a_{7}}$ singularity $q$. The pull-back of the
fiber through $q$ is the fiber $F$.
\end{prop}

\subsection{The lattice $\boldsymbol{U\oplus\mathbf{D}_{4}\oplus\mathbf{A}_{3}}$}

The surface $X$ contains $10$ $(-2)$-curves, with dual graph

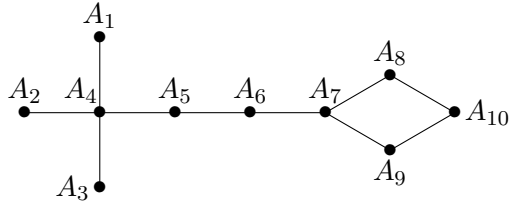
\begin{center}
\begin{tikzpicture}[scale=1]

\draw (3,0) -- (-1,0);
\draw (0,1) -- (0,-1);
\draw (3,0) -- (3+0.86,-0.5);
\draw (3+0.86,-0.5) -- (3+0.86*2,0);
\draw (3+0.86,0.5) -- (3+0.86*2,0);
\draw (3,0) -- (3+0.86,0.5);


\draw (-1,0) node {$\bullet$};
\draw (0,0) node {$\bullet$};
\draw (0,1) node {$\bullet$};
\draw (0,-1) node {$\bullet$};
\draw (1,0) node {$\bullet$};
\draw (2,0) node {$\bullet$};
\draw (3,0) node {$\bullet$};
\draw (3+0.86,-0.5) node {$\bullet$};
\draw (3+0.86,0.5) node {$\bullet$};
\draw (3+0.86*2,0) node {$\bullet$};

\draw (-1,0) node [above]{$A_{2}$};
\draw (-0.25,0) node [above]{$A_{4}$};
\draw (0,1) node [above]{$A_{1}$};
\draw (0,-1) node [left]{$A_{3}$};
\draw (1,0) node [above]{$A_{5}$};
\draw (2,0) node [above]{$A_{6}$};
\draw (3,0) node [above]{$A_{7}$};
\draw (3+0.86,-0.5) node [below]{$A_{9}$};
\draw (3+0.86,0.5) node [above]{$A_{8}$};
\draw (3+0.86*2,0) node [right]{$A_{10}$};

\end{tikzpicture}
\end{center} 

The curves $A_{1},\dots,A_{8}$ generate  the N\'eron--Severi lattice.
In that base, the divisor
\[
D_{56}=(7,7,7,15,10,6,3,1,0)
\]
is ample, with $D_{56} \cdot A_{j}=1$ for $j\in\{1,\dots,10\}\setminus\{9\}$ and
$D_{56} \cdot A_{9}=3$. The divisors 
\[
F_{1}=A_{4}+\sum_{j=1}^{5}  A_{j},\quad F_{2}=\sum_{j=7}^{10}A_{j}
\]
 are fibers of an elliptic fibration with section $A_{6}$. By Theorem
\ref{thm:SaintDonat-2}, case i) a), we have the following. 

\begin{prop}
The linear system $|4F_{1}+2A_{6}|$ defines a morphism $\varphi\colon X\to\mathbf{F}_{4}$
branched over the unique section $s$ with $s^2=-4$ and a curve $B\in|3s+12f|$. The curve
$B$ has a $\mathbf{d_{4}}$ singularity $p$ and one $\mathbf{a_{3}}$
singularity $q$. The pull-backs of the fibers through $p$, $q$ are the
fibers $F_{1}$, $F_{2}$.
\end{prop}

The divisor 
\[
D_{2}=A_{1}+2A_{2}+2A_{3}+4A_{4}+3A_{5}+2A_{6}+A_{7}
\]
is nef, of square $2$ and base-point free. We have $D_{2} \cdot A_{1}=2$,
$D_{2} \cdot A_{8}=D_{2} \cdot A_{9}=1$ and $D_{2} \cdot A_{j}=0$ for $j\notin\{1,8,9\}$. 
Moreover,
\[
D_{2}\equiv A_{2}+A_{3}+2A_{4}+2A_{5}+2A_{6}+2A_{7}+A_{8}+A_{9}+A_{10}. 
\]
Thus the K3 surface is a double cover of the plane branched over
a sextic curve with a $\mathbf{d}_{6}$ singularity and a node $\mathbf{a}_{1}$. 

\subsection{The lattice $\boldsymbol{U\oplus\mathbf{D}_{5}\oplus\mathbf{A}_{2}}$}

The surface $X$ contains $10$ $(-2)$-curves, with dual graph 

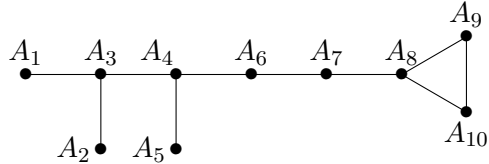
\begin{center}
\begin{tikzpicture}[scale=1]

\draw (3,0) -- (-2,0);
\draw (0,0) -- (0,-1);
\draw (-1,0) -- (-1,-1);
\draw (3,0) -- (3+0.86,-0.5);
\draw (3+0.86,0.5) -- (3+0.86,-0.5);
\draw (3,0) -- (3+0.86,0.5);


\draw (-2,0) node {$\bullet$}; 
\draw (-1,-1) node {$\bullet$}; 
\draw (-1,0) node {$\bullet$};
\draw (0,0) node {$\bullet$};
\draw (0,-1) node {$\bullet$};
\draw (1,0) node {$\bullet$};
\draw (2,0) node {$\bullet$};
\draw (3,0) node {$\bullet$};
\draw (3+0.86,-0.5) node {$\bullet$};
\draw (3+0.86,0.5) node {$\bullet$};

\draw (-2,0) node [above]{$A_{1}$};
\draw (-1,0) node [above]{$A_{3}$};
\draw (-0.25,0) node [above]{$A_{4}$};
\draw (-1,-1) node [left]{$A_{2}$};
\draw (0,-1) node [left]{$A_{5}$};
\draw (1,0) node [above]{$A_{6}$};
\draw (2,0) node [above]{$A_{7}$};
\draw (3,0) node [above]{$A_{8}$};
\draw (3+0.86,-0.5) node [below]{$A_{10}$};
\draw (3+0.86,0.5) node [above]{$A_{9}$};

\end{tikzpicture}
\end{center} 

The N\'eron--Severi lattice is generated by $A_{1},\dots,A_{9}$; moreover, 
\[
A_{10}=(1,1,2,2,1,1,0,-1,-1)
\]
in that base. The divisor 
\[
D_{88}=(8,8,17,19,9,13,8,4,1)
\]
is ample, of square $88$, with $D_{88} \cdot A_{j}=1$ for $j\leq8$, $D_{88} \cdot A_{9}=2$
and $D_{88} \cdot A_{10}=5$. The divisors 
\[
F_{1}=A_{1}+A_{2}+2A_{3}+2A_{4}+A_{5}+A_{6},\quad F_{2}=A_{8}+A_{9}+A_{10}
\]
are fibers of an elliptic fibration with section $A_{7}$. By Theorem
\ref{thm:SaintDonat-2}, case i) a), we have the following. 

\begin{prop}
The linear system $|4F_{1}+2A_{6}|$ defines a morphism $\varphi\colon X\to\mathbf{F}_{4}$
branched over the negative section $s$ and a curve $B\in|3s+12f|$.
The curve $B$ has a $\mathbf{d_{5}}$ singularity $p$ and one $\mathbf{a_{2}}$
singularity $q$. The pull-backs of the fibers through $p$, $q$ are the
fibers $F_{1}$, $F_{2}$.
\end{prop}

The divisor 
\[
D_{2}=A_{1}+A_{2}+3A_{3}+4A_{4}+2A_{5}+3A_{6}+2A_{7}+A_{8}
\]
is nef, of square $2$, with $D_{2} \cdot A_{1}=D_{2} \cdot A_{2}=D_{2} \cdot A_{9}=D_{2} \cdot A_{10}=1$ and 
$D_{2} \cdot A_{j}=0$ for $j\notin\{1,2,9,10\}$. Since 
\[
D_{2}\equiv A_{3}+2A_{4}+A_{5}+2A_{6}+2A_{7}+2A_{8}+A_{9}+A_{10},
\]
we get that the K3 surface is the minimal resolution of a double cover
of $\PP^{2}$ branched over a sextic curve with a $\mathbf{d}_{6}$
singularity. 

\subsection{The lattice $\boldsymbol{U\oplus\mathbf{D}_{7}}$}

The surface $X$ contains nine $(-2)$-curves, with dual graph 

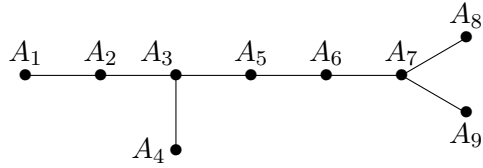
\begin{center}
\begin{tikzpicture}[scale=1]

\draw (3,0) -- (-2,0);
\draw (0,0) -- (0,-1);
\draw (3,0) -- (3+0.86,-0.5);
\draw (3,0) -- (3+0.86,0.5);


\draw (-2,0) node {$\bullet$}; 
\draw (-1,0) node {$\bullet$};
\draw (0,0) node {$\bullet$};
\draw (0,-1) node {$\bullet$};
\draw (1,0) node {$\bullet$};
\draw (2,0) node {$\bullet$};
\draw (3,0) node {$\bullet$};
\draw (3+0.86,-0.5) node {$\bullet$};
\draw (3+0.86,0.5) node {$\bullet$};

\draw (-2,0) node [above]{$A_{1}$};
\draw (-1,0) node [above]{$A_{2}$};
\draw (-0.25,0) node [above]{$A_{3}$};

\draw (0,-1) node [left]{$A_{4}$};
\draw (1,0) node [above]{$A_{5}$};
\draw (2,0) node [above]{$A_{6}$};
\draw (3,0) node [above]{$A_{7}$};
\draw (3+0.86,-0.5) node [below]{$A_{9}$};
\draw (3+0.86,0.5) node [above]{$A_{8}$};

\end{tikzpicture}
\end{center} 

The nine curves generate  the N\'eron--Severi lattice, and in that basis, 
the divisor 
\[
D_{260}=(13,27,42,20,38,35,33,16,16)
\]
is ample, of square $260$, with $D_{260} \cdot A_{2}=2$ and $D_{260} \cdot A_{j}=1$
for $j\neq4$. The divisor 
\[
F=A_{2}+A_{4}+2(A_{3}+A_{5}+A_{6}+A_{7})+A_{8}+A_{9}
\]
is a fiber of an elliptic fibration  with section $A_{1}$. By Theorem
\ref{thm:SaintDonat-2}, case i) a), we have the following. 

\begin{prop}
The linear system $|4F+2A_{1}|$ defines a morphism $\varphi\colon X\to\mathbf{F}_{4}$
branched over the unique section $s$ with $s^2=-4$ and a curve $B\in|3s+12f|$. The curve
$B$ has a $\mathbf{d_{7}}$ singularity $q$. The pull-back of the
fiber through $q$ is the fiber $F$.
\end{prop}

One can also construct that surface as a double plane: the divisor 
\[
D_{2}=2A_{1}+4A_{2}+6A_{3}+3A_{4}+5A_{5}+4A_{6}+3A_{7}+A_{8}+A_{9}
\]
is nef, of square $2$, with $D_{2} \cdot A_{j}=0$ for $j\leq7$ and 
$D_{2} \cdot A_{8}=D_{2} \cdot A_{9}=1$. The K3 surface is a double cover of $\PP^{2}$
branched over a sextic curve with an $\mathbf{e}_{7}$ singularity. 

\subsection{The lattice $\boldsymbol{U\oplus\mathbf{A}_{1}\oplus\mathbf{E}_{6}}$}

The surface $X$ contains $10$ $(-2)$-curves, with dual graph 

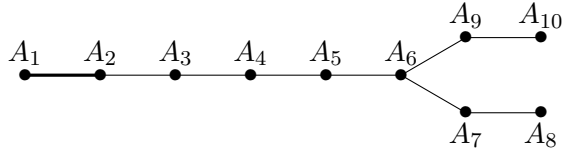
\begin{center}
\begin{tikzpicture}[scale=1]

\draw (3,0) -- (-2,0);
\draw (3,0) -- (3+0.86,-0.5);
\draw (3+0.86,-0.5) -- (4+0.86,-0.5);
\draw (3+0.86,0.5) -- (4+0.86,0.5);
\draw (3,0) -- (3+0.86,0.5);

\draw [very thick] (-2,0) -- (-1,0);

\draw (-2,0) node {$\bullet$}; 
\draw (-1,0) node {$\bullet$};
\draw (0,0) node {$\bullet$};
\draw (1,0) node {$\bullet$};
\draw (2,0) node {$\bullet$};
\draw (3,0) node {$\bullet$};
\draw (3+0.86,-0.5) node {$\bullet$};
\draw (3+0.86,0.5) node {$\bullet$};
\draw (4+0.86,-0.5) node {$\bullet$};
\draw (4+0.86,0.5) node {$\bullet$};

\draw (-2,0) node [above]{$A_{1}$};
\draw (-1,0) node [above]{$A_{2}$};
\draw (0,0) node [above]{$A_{3}$};
\draw (1,0) node [above]{$A_{4}$};
\draw (2,0) node [above]{$A_{5}$};
\draw (3,0) node [above]{$A_{6}$};
\draw (3+0.86,-0.5) node [below]{$A_{7}$};
\draw (3+0.86,0.5) node [above]{$A_{9}$};
\draw (4+0.86,-0.5) node [below]{$A_{8}$};
\draw (4+0.86,0.5) node [above]{$A_{10}$};

\end{tikzpicture}
\end{center} 

The nine curves $A_{1},\dots,A_{9}$ generate  the N\'eron--Severi lattice,
and in that basis, the divisor 
\[
D_{184}=(12,17,12,8,5,3,1,0,1)
\]
is ample, of square $184$, with $D_{184} \cdot A_{1}=10$, $D_{184} \cdot A_{2}=2$
and $D_{184} \cdot A_{j}=1$ for $j\geq3$. The divisors
\[
F_{1}=A_{1}+A_{2},\quad F_{2}=A_{4}+A_{8}+A_{10}+2(A_{5}+A_{7}+A_{9})+3A_{6}
\]
are fibers of an elliptic fibration  of $X$ with a section $A_{3}$.
By Theorem \ref{thm:SaintDonat-2}, case i) a), we have the following. 

\begin{prop}
The linear system $|4F_{1}+2A_{6}|$ defines a morphism $\varphi\colon X\to\mathbf{F}_{4}$
branched over the unique section $s$ with $s^2=-4$ and a curve $B\in|3s+12f|$. The curve
$B$ has a node $p$ and one $\mathbf{e}_{6}$ singularity $q$. The
pull-backs of the fibers through \mbox{$p$, $q$} are the fibers $F_{1}$, $F_{2}$.
\end{prop}

One can also construct that surface as a double plane: the divisor 
\[
D_{2}=A_{1}+2(A_{2}+A_{3}+A_{4}+A_{5}+A_{6})+A_{7}+A_{9}
\]
is nef, base-point free, of square $2$. We have $D_{2} \cdot A_{1}=2$,
$D_{2} \cdot A_{8}=D_{2} \cdot A_{10}=1$ and $D_{2} \cdot A_{j}=0$ for the other curves.
The K3 surface is a double cover of $\PP^{2}$ branched over a sextic
curve with a $\mathbf{d}_{7}$ singularity.

\section{Rank 10 lattices }

\subsection{The lattice $\boldsymbol{U\oplus\mathbf{E}_{8}}$}

The K3 surface $X$ contains $10$ $\cu$-curves with dual graph

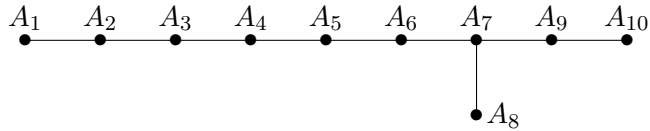
\begin{center}
\begin{tikzpicture}[scale=1]

\draw (8,0) -- (0,0);
\draw (6,0) -- (6,-1);


\draw (0,0) node {$\bullet$}; 
\draw (1,0) node {$\bullet$};
\draw (2,0) node {$\bullet$};
\draw (3,0) node {$\bullet$};
\draw (4,0) node {$\bullet$};
\draw (5,0) node {$\bullet$};
\draw (6,0) node {$\bullet$};
\draw (7,0) node {$\bullet$};
\draw (8,0) node {$\bullet$};
\draw (6,-1) node {$\bullet$};

\draw (0,0) node [above]{$A_{1}$};
\draw (1,0) node [above]{$A_{2}$};
\draw (2,0) node [above]{$A_{3}$};
\draw (3,0) node [above]{$A_{4}$};
\draw (4,0) node [above]{$A_{5}$};
\draw (5,0) node [above]{$A_{6}$};
\draw (6,0) node [above]{$A_{7}$};
\draw (7,0) node [above]{$A_{9}$};
\draw (8,0) node [above]{$A_{10}$};
\draw (6,-1) node [right]{$A_{8}$};

\end{tikzpicture}
\end{center} 

These curves generate  the N\'eron--Severi lattice. In that base, the
divisor 
\[
D_{1240}=(30,61,93,126,160,195,231,115,153,76)
\]
is ample, of square $1240$, with $D_{1240} \cdot A_{j}=1$ for $j\in\{1,\dots,10\}$.
The divisor 
\[
F=A_{2}+2A_{3}+3A_{4}+4A_{5}+5A_{6}+6A_{7}+3A_{8}+4A_{9}+2A_{10}
\]
is a fiber of an elliptic fibration  and $A_{1}$ is a section. By
Theorem \ref{thm:SaintDonat-2}, case i) a), we have the following. 

\begin{prop}
The linear system $|4F+2A_{1}|$ defines a morphism $\varphi\colon X\to\mathbf{F}_{4}$
branched over the unique section $s$ with $s^2=-4$ and a curve $B\in|3s+12f|$. The curve
$B$ has an $\mathbf{e_{8}}$ singularity $q$. The pull-back of the
fiber through $q$ is the fiber $F$.
\end{prop}

\begin{rem}
Using that $U\oplus\mathbf{E}_{8}$ is unimodular, one can prove
that if $D$ is an ample divisor, then $D^{2}\geq1240$.
\end{rem}

\subsection{The lattice $\boldsymbol{U\oplus\mathbf{D}_{8}}$}

The K3 surface $X$ contains $10$ $\cu$-curves with dual graph

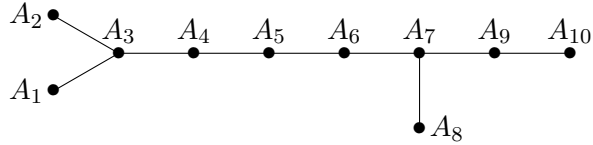
\begin{center}
\begin{tikzpicture}[scale=1]

\draw (2-0.866,0.5) -- (2,0);
\draw (2-0.866,-0.5) -- (2,0);
\draw (2,0) -- (8,0);
\draw (6,0) -- (6,-1);


\draw (2-0.866,-0.5) node {$\bullet$}; 
\draw (2-0.866,0.5) node {$\bullet$};
\draw (2,0) node {$\bullet$};
\draw (3,0) node {$\bullet$};
\draw (4,0) node {$\bullet$};
\draw (5,0) node {$\bullet$};
\draw (6,0) node {$\bullet$};
\draw (7,0) node {$\bullet$};
\draw (8,0) node {$\bullet$};
\draw (6,-1) node {$\bullet$};

\draw (2-0.866,-0.5) node [left]{$A_{1}$};
\draw (2-0.866,0.5) node [left]{$A_{2}$};
\draw (2,0) node [above]{$A_{3}$};
\draw (3,0) node [above]{$A_{4}$};
\draw (4,0) node [above]{$A_{5}$};
\draw (5,0) node [above]{$A_{6}$};
\draw (6,0) node [above]{$A_{7}$};
\draw (7,0) node [above]{$A_{9}$};
\draw (8,0) node [above]{$A_{10}$};
\draw (6,-1) node [right]{$A_{8}$};

\end{tikzpicture}
\end{center} 

These curves generate  the N\'eron--Severi lattice. In that base, the
divisor 
\[
D_{280}=(15,15,31,33,36,40,45,22,29,14)
\]
is ample, with  $D_{280} \cdot A_{j}=1$ for $j\in\{1,\dots,10\}$. The divisor
\[
F=A_{1}+A_{2}+2(A_{3}+\dots+A_{7})+A_{8}+A_{9}
\]
is a fiber of an elliptic fibration  with section $A_{10}$. By Theorem
\ref{thm:SaintDonat-2}, case i) a), we have the following. 

\begin{prop}
The linear system $|4F+2A_{10}|$ defines a morphism $\varphi\colon X\to\mathbf{F}_{4}$
branched over the unique section $s$ with $s^2=-4$ and a curve $B\in|3s+12f|$. The curve
$B$ has a $\mathbf{d_{8}}$ singularity $q$. The pull-back of the
fiber through $q$ is the fiber $F$.
\end{prop}

\subsection{The lattice $\boldsymbol{U\oplus\mathbf{E}_{7}\oplus\mathbf{A}_{1}}$}

The K3 surface $X$ contains $11$ $\cu$-curves with dual graph

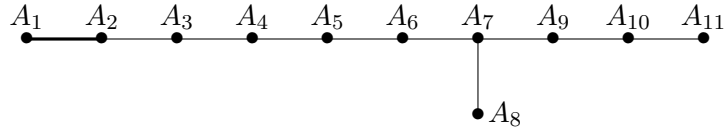
\begin{center}
\begin{tikzpicture}[scale=1]

\draw (9,0) -- (0,0);
\draw (6,0) -- (6,-1);

\draw [very thick] (0,0) -- (1,0);

\draw (0,0) node {$\bullet$}; 
\draw (1,0) node {$\bullet$};
\draw (2,0) node {$\bullet$};
\draw (3,0) node {$\bullet$};
\draw (4,0) node {$\bullet$};
\draw (5,0) node {$\bullet$};
\draw (6,0) node {$\bullet$};
\draw (7,0) node {$\bullet$};
\draw (8,0) node {$\bullet$};
\draw (9,0) node {$\bullet$};
\draw (6,-1) node {$\bullet$};

\draw (0,0) node [above]{$A_{1}$};
\draw (1,0) node [above]{$A_{2}$};
\draw (2,0) node [above]{$A_{3}$};
\draw (3,0) node [above]{$A_{4}$};
\draw (4,0) node [above]{$A_{5}$};
\draw (5,0) node [above]{$A_{6}$};
\draw (6,0) node [above]{$A_{7}$};
\draw (7,0) node [above]{$A_{9}$};
\draw (8,0) node [above]{$A_{10}$};
\draw (9,0) node [above]{$A_{11}$};
\draw (6,-1) node [right]{$A_{8}$};

\end{tikzpicture}
\end{center} 

The curves $A_{1},\dots,A_{10}$ generate  the N\'eron--Severi lattice.
In that base, the divisor 
\[
D_{370}=(15,24,19,15,12,10,9,4,5,2)
\]
is ample, of square $370$, with $D_{370} \cdot A_{j}=1$ for $j\in\{2,\dots,10\}$,
$D_{370} \cdot A_{1}=18$ and $D_{370} \cdot A_{11}=2$. The divisors 
\[
F_{1}=A_{1}+A_{2},\quad F_{2}=A_{4}+2A_{5}+3A_{6}+4A_{7}+2A_{8}+3A_{9}+2A_{10}+A_{11}
\]
are fibers of an elliptic fibration  with section $A_{3}$. By Theorem
\ref{thm:SaintDonat-2}, case i) a), we have the following. 

\begin{prop}
The linear system $|4F_{1}+2A_{3}|$ defines a morphism $\varphi\colon X\to\mathbf{F}_{4}$
branched over the unique section $s$ with $s^2=-4$ and a curve $B\in|3s+12f|$. The curve
$B$ has a node $p$ and one $\mathbf{e}_{7}$ singularity $q$. The
pull-backs of the fibers through \mbox{$p$, $q$} are the fibers $F_{1}$, $F_{2}$.
\end{prop}

\subsection{The lattice $\boldsymbol{U\oplus\mathbf{D}_{4}^{\oplus2}}$}

The K3 surface $X$ contains $11$ $\cu$-curves with dual graph

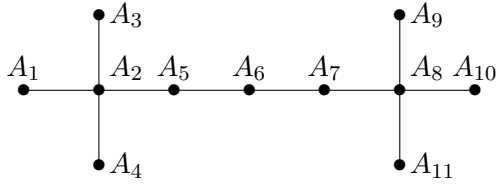
\begin{center}
\begin{tikzpicture}[scale=1]

\draw (6,0) -- (0,0);
\draw (1,1) -- (1,-1);
\draw (5,1) -- (5,-1);


\draw (0,0) node {$\bullet$}; 
\draw (1,0) node {$\bullet$};
\draw (2,0) node {$\bullet$};
\draw (3,0) node {$\bullet$};
\draw (4,0) node {$\bullet$};
\draw (5,0) node {$\bullet$};
\draw (6,0) node {$\bullet$};
\draw (1,1) node {$\bullet$};
\draw (1,-1) node {$\bullet$};
\draw (5,1) node {$\bullet$};
\draw (5,-1) node {$\bullet$};

\draw (0,0) node [above]{$A_{1}$};
\draw (1,0) node [above right]{$A_{2}$};
\draw (2,0) node [above]{$A_{5}$};
\draw (3,0) node [above]{$A_{6}$};
\draw (4,0) node [above]{$A_{7}$};
\draw (5,0) node [above right]{$A_{8}$};
\draw (6,0) node [above]{$A_{10}$};
\draw (1,1) node [right]{$A_{3}$};
\draw (1,-1) node [right]{$A_{4}$};
\draw (5,1) node [right]{$A_{9}$};
\draw (5,-1) node [right]{$A_{11}$};

\end{tikzpicture}
\end{center} 

The curves $A_{1},\dots,A_{10}$ generate the N\'eron--Severi lattice.
In that base, we have 
\[
A_{11}=(1,2,1,1,1,0,-1,-2,-1,-1).
\]
 The divisor 
\[
D_{56}\equiv(7,15,7,7,10,6,3,1,0,0)
\]
is ample of square $56$, and the $(-2)$-curves on $X$ have degree
$1$ for $D_{56}$. The divisors 
\[
F_{1}=A_{2}+\sum_{j=1}^{5}A_{j},\quad F_{2}=A_{8}+\sum_{j=7}^{11}A_{j},
\]
are fibers of an elliptic fibration  with section $A_{6}$. By Theorem
\ref{thm:SaintDonat-2}, case i) a), we have the following. 

\begin{prop}
The linear system $|4F_{1}+2A_{6}|$ defines a morphism $\varphi\colon X\to\mathbf{F}_{4}$
branched over the unique section $s$ with $s^2=-4$ and a curve $B\in|3s+12f|$. The curve
$B$ has two $\mathbf{d}_{4}$ singularities. The pull-backs of the
fibers through $p$, $q$ are the fibers $F_{1}$, $F_{2}$.
\end{prop}

One can give another construction as follows. The divisor 
\[
D_{2}=A_{1}+4A_{2}+2A_{3}+2A_{4}+3A_{5}+2A_{6}+A_{7}
\]
 is nef, of square $2$, base-point free, with $D_{2} \cdot A_{1}=2$, $D_{2} \cdot A_{8}=1$ 
and $D_{2} \cdot A_{j}=0$ for $j\notin\{2,8\}$. We have 
\[
D_{2}\equiv2A_{2}+A_{3}+A_{4}+2A_{5}+2A_{6}+2A_{7}+2A_{8}+A_{9}+A_{10}+A_{11}.
\]
The branch curve of the double cover $\varphi\colon X\to\PP^{2}$ associated
to $D_{2}$ is a sextic with three nodes and a $\mathbf{d}_{6}$ singularity
$q$. The curve $A_{8}$ is in the ramification locus; its image by
$\varphi$ is a line $L$. Let $Q$ be the residual quintic curve
of the sextic branch locus. The line $L$ cuts $Q$ transversally
in an $\mathbf{a}_{3}$ singularity and in three other points $p_{1}$, $p_{2}$, $p_{3}$.
The curves $A_{9}$, $A_{10}$, $A_{11}$ are mapped to these three points.
The image by $\varphi$ of $A_{1}$ is the line that is the branch
of the $\mathbf{a}_{3}$ singularity. The curves $A_{2},\dots,A_{7}$
are mapped to $q$.

\subsection{The lattice $\boldsymbol{U\oplus\mathbf{D}_{6}\oplus\mathbf{A}_{1}^{\oplus2}}$}

The K3 surface $X$ contains $14$ $\cu$-curves with dual graph

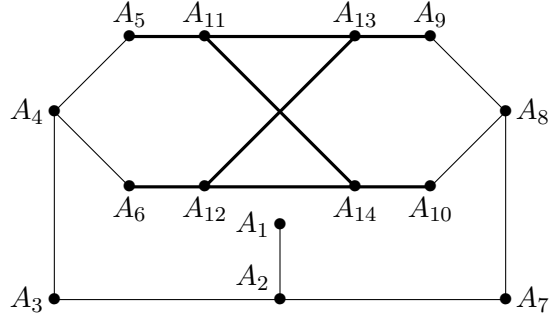
\begin{center}
\begin{tikzpicture}[scale=1]

\draw [very thick] (-1,-1) -- (1,1);
\draw [very thick] (1,-1) -- (-1,1);
\draw [very thick] (-2,1) -- (-1,1);
\draw [very thick] (-2,-1) -- (-1,-1);
\draw [very thick] (2,1) -- (1,1);
\draw [very thick] (2,-1) -- (1,-1);
\draw (-2,1) -- (-3,0);
\draw (-2,-1) -- (-3,0);
\draw (2,1) -- (3,0);
\draw (2,-1) -- (3,0);
\draw (-3,0) -- (-3,-2.5);
\draw (3,0) -- (3,-2.5);
\draw (-3,-2.5) -- (3,-2.5);
\draw (0,-2.5) -- (0,-1.5);
\draw [very thick] (-1,1) -- (1,1); 
\draw [very thick] (-1,-1) -- (1,-1); 

\draw (1,1) node {$\bullet$}; 
\draw (1,-1) node {$\bullet$}; 
\draw (-1,1) node {$\bullet$}; 
\draw (-1,-1) node {$\bullet$}; 
\draw (2,1) node {$\bullet$}; 
\draw (2,-1) node {$\bullet$}; 
\draw (-2,1) node {$\bullet$}; 
\draw (-2,-1) node {$\bullet$}; 
\draw (-3,0) node {$\bullet$}; 
\draw (3,0) node {$\bullet$}; 
\draw (-3,-2.5) node {$\bullet$}; 
\draw (3,-2.5) node {$\bullet$}; 
\draw (0,-1.5) node {$\bullet$}; 
\draw (0,-2.5) node {$\bullet$}; 


\draw (1,1) node [above]{$A_{13}$};
\draw (1,-1) node [below]{$A_{14}$};
\draw (-1,1) node [above]{$A_{11}$};
\draw (-1,-1) node [below]{$A_{12}$};
\draw (2,1) node [above]{$A_{9}$};
\draw (2,-1) node [below]{$A_{10}$};
\draw (-2,1) node [above]{$A_{5}$};
\draw (-2,-1) node [below]{$A_{6}$};
\draw (-3,0) node [left]{$A_{4}$};
\draw (3,0) node [right]{$A_{8}$};
\draw (-3,-2.5) node [left]{$A_{3}$};
\draw (3,-2.5) node [right]{$A_{7}$};
\draw (0,-1.5) node [left]{$A_{1}$};
\draw (0,-2.5) node [above left]{$A_{2}$};

\end{tikzpicture}
\end{center} 

The curves $A_{1},\dots,A_{10}$ generate  the N\'eron--Severi lattice.
In that base, the divisor 
\[
D_{98}=(7,16,13,11,5,5,13,11,5,5)
\]
is ample, of square $98$, with $D_{98} \cdot A_{j}=1$ for $j\in\{2,\dots,10\}$,
$D_{98} \cdot A_{1}=2$ and $D_{98} \cdot A_{j}=10$ for $j\in\{11,\dots,14\}$.
The divisor 
\[
D_{2}=(1,2,2,2,1,1,2,2,1,1)
\]
is nef, base-point free, of square $2$, with $D_{2} \cdot A_{2}=1$, $D_{2} \cdot A_{j}=2$
for $j\in\{11,\dots,14\}$ and $D_{2} \cdot A_{j}=0$ for the remaining curves.
By using the linear system $|D_{2}|$, we obtain the following. 

\begin{prop}
The branch curve has a node and two $\mathbf{d}_{4}$ singularities.
It is the union of a quintic curve with two nodes $q$, $q'$ and the
line through the points $q$, $q'$. The line cuts the quintic in a third
point $p$, forming a node. 
\end{prop}

The curve $A_{1}$ is sent to $p$, the curves $A_{3}$, $A_{4}$, $A_{5}$, $A_{6}$
are sent to $q$, and the curves $A_{7}$, $A_{8}$, $A_{9}$, $A_{10}$ are sent
to $q'$. The curves $A_{11}$ and $A_{12}$ are sent to the two branches of the singularity $q$ that are distinct from $L$; the curves $A_{13}$ and
$A_{14}$ are sent to the two branches of the singularity $q'$ that are distinct from $L$. 

\subsection{The lattice $\boldsymbol{U(2)\oplus\mathbf{D}_{4}^{\oplus2}}$}

\subsubsection{First involution}

There exist $18$ $(-2)$-curves $A_{1},\dots,A_{18}$ on $X$, with
dual graph

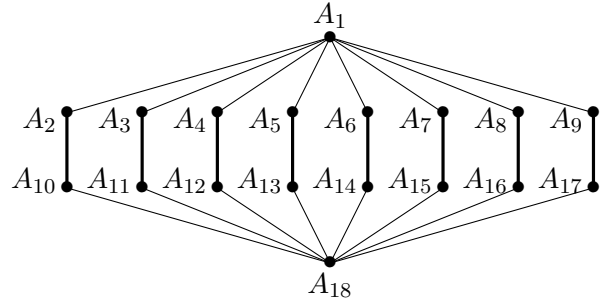
\begin{center}
\begin{tikzpicture}[scale=1]
\draw (0,1.5) -- (3.5,0.5);
\draw (0,1.5) -- (2.5,0.5);
\draw (0,1.5) -- (1.5,0.5);
\draw (0,1.5) -- (0.5,0.5);
\draw (0,1.5) -- (-0.5,0.5);
\draw (0,1.5) -- (-1.5,0.5);
\draw (0,1.5) -- (-2.5,0.5);
\draw (0,1.5) -- (-3.5,0.5);

\draw [very thick] (3.5,0.5) -- (3.5,-0.5);
\draw [very thick] (2.5,0.5) -- (2.5,-0.5);
\draw [very thick] (1.5,0.5) -- (1.5,-0.5);
\draw [very thick] (0.5,0.5) -- (0.5,-0.5);
\draw [very thick] (-3.5,0.5) -- (-3.5,-0.5);
\draw [very thick] (-2.5,0.5) -- (-2.5,-0.5);
\draw [very thick] (-1.5,0.5) -- (-1.5,-0.5);
\draw [very thick] (-0.5,0.5) -- (-0.5,-0.5);

\draw (0,-1.5) -- (3.5,-0.5);
\draw (0,-1.5) -- (2.5,-0.5);
\draw (0,-1.5) -- (1.5,-0.5);
\draw (0,-1.5) -- (0.5,-0.5);
\draw (0,-1.5) -- (-0.5,-0.5);
\draw (0,-1.5) -- (-1.5,-0.5);
\draw (0,-1.5) -- (-2.5,-0.5);
\draw (0,-1.5) -- (-3.5,-0.5);

\draw (0,1.5) node {$\bullet$};
\draw (3.5,0.5) node  {$\bullet$};
\draw (2.5,0.5) node  {$\bullet$};
\draw (1.5,0.5) node  {$\bullet$};
\draw (0.5,0.5) node  {$\bullet$};
\draw (-0.5,0.5) node  {$\bullet$};
\draw (-1.5,0.5) node  {$\bullet$};
\draw (-2.5,0.5) node  {$\bullet$};
\draw (-3.5,0.5) node  {$\bullet$};

\draw (0,1.5) node [above]{$A_{1}$};
\draw (3.5,0.4) node [left]{$A_{9}$};
\draw (2.5,0.4) node [left]{$A_{8}$};
\draw (1.5,0.4) node [left]{$A_{7}$};
\draw (0.5,0.4) node [left]{$A_{6}$};
\draw (-0.5,0.4) node [left]{$A_{5}$};
\draw (-1.5,0.4) node [left]{$A_{4}$};
\draw (-2.5,0.4) node [left]{$A_{3}$};
\draw (-3.5,0.4) node [left]{$A_{2}$};

\draw (0,-1.5) node {$\bullet$};
\draw (3.5,-0.5) node  {$\bullet$};
\draw (2.5,-0.5) node  {$\bullet$};
\draw (1.5,-0.5) node  {$\bullet$};
\draw (0.5,-0.5) node  {$\bullet$};
\draw (-0.5,-0.5) node  {$\bullet$};
\draw (-1.5,-0.5) node  {$\bullet$};
\draw (-2.5,-0.5) node  {$\bullet$};
\draw (-3.5,-0.5) node  {$\bullet$};

\draw (0,-1.5) node [below]{$A_{18}$};
\draw (3.5,-0.4) node [left]{$A_{17}$};
\draw (2.5,-0.4) node [left]{$A_{16}$};
\draw (1.5,-0.4) node [left]{$A_{15}$};
\draw (0.5,-0.4) node [left]{$A_{14}$};
\draw (-0.5,-0.4) node [left]{$A_{13}$};
\draw (-1.5,-0.4) node [left]{$A_{12}$};
\draw (-2.5,-0.4) node [left]{$A_{11}$};
\draw (-3.5,-0.4) node [left]{$A_{10}$};


\end{tikzpicture}
\end{center} 

The curves $A_{1},\dots,A_{8},A_{10},A_{18}$ generate the N\'eron--Severi
lattice. The divisor 
\[
D_{8}=A_{1}+3A_{2}+3A_{10}+A_{18}
\]
is ample, base-point free, of square $8$, with $D_{8} \cdot A_{j}=1$ for all $j\in\{1,\dots,18\}$. The divisors $A_{k}+A_{8+k}\equiv F$, $k\in\{2,\dots,9\}$, are singular fibers of an elliptic fibration, for which $A_{1}$ and $A_{18}$ are sections. Since $FD_{8}=2$, we see that $D_{8}$ is hyperelliptic; thus the image by the associated map $\varphi$ of $X$ is a rational normal scroll of degree $4$ in $\PP^{5}$, \textit{i.e.}, the Hirzebruch surface $\mathbf{F}_{2}$. The pull-back by $\varphi$ of the unique negative curve $s$ (the section) is the union of two disjoint $(-2)$-curves. The branch curve $B$ must satisfy $Bs=0$; thus $B=v(s+2f)$, where $f$ is a fiber of the unique fibration of $\mathbf{F}_{2}$. We have $F=\varphi^{*}f$; thus the branch curve cuts a  general  fiber $f$ in four points and $v=4$.

\begin{prop}
The K3 surface $X$ is the double cover of $\,\mathbf{F}_{2}$ branched
on a smooth curve $B$ of genus $9$ in $|4s+8f|$. The curve $B$
is such that there are $8$ fibers $f_{1},\dots,f_{8}$ that meet
$B$ with even multiplicities at each intersection point. The $18$
curves on $X$ are in the pull-back of $s$ $($which is $A_{1}+A_{18})$
and the pull-backs of the fibers $f_{i}$, $i\in\{1,\dots,8\}$. 
\end{prop}

\subsubsection{Second involution}

The divisor 
\[
D_{2}=2A_{1}+2A_{2}+A_{3}+A_{4}+A_{10}
\]
is nef, of square $2$, with $D_{2} \cdot A_{j}=2$ for $j\in\{5,\dots,12\}$,
$D_{2} \cdot A_{18}=1$ and else $D_{2} \cdot A_{j}=0$. We have 
\[
D_{2}\equiv2A_{1}+A_{2}+A_{3}+A_{4}+A_{k}+A_{8+k},\quad k\in\{2,\dots,9\},
\]
and 
\[
D_{2}\equiv A_{13}+A_{14}+A_{15}+A_{16}+A_{17}+2A_{18}.
\]

\begin{prop}
The K3 surface $X$ is the double cover of $\,\PP^{2}$ branched over
a sextic curve $C_{6}$ with a $\mathbf{d}_{4}$ singularity $q$
$($onto which $A_{1}$, $A_{2}$, $A_{3}$, $A_{4}$ are mapped$)$ and five nodes
$($onto which the five curves $A_{13},\dots,A_{17}$ are mapped$)$. 
\end{prop}

The curve $C_{6}$ is the union of a line $L$ and a quintic with
a $\mathbf{d}_{4}$ singularity. The double cover is ramified over
$A_{18}$, and the image of $A_{18}$ is $L$. The images of $A_{10}$, $A_{11}$, $A_{12}$
are conics tangent to the three branches of the $\mathbf{d}_{4}$
singularity; the images of $A_{13},\dots,A_{17}$ are the five lines
through the five nodes (intersection of $L$ and the quintic) and
$q$.

\begin{rem}
According to \cite{Kondo}, the automorphism group of $X$ is $(\ZZ/2\ZZ)^{2}$.
The fixed loci of the two involutions we described are distinct; thus
we obtained generators of the automorphism group.
\end{rem}

\subsection{The lattice $\boldsymbol{U\oplus\mathbf{D}_{4}\oplus\mathbf{A}_{1}^{\oplus4}}$}

The K3 surface $X$ contains $27$ $(-2)$-curves $A_{1},\dots,A_{27}$;
the curves $A_{1},\dots,A_{11}$ have dual graph

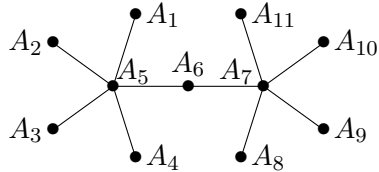
\begin{center}
\begin{tikzpicture}[scale=1]

\draw (0,0) -- (1,0);
\draw (0,0) -- (0.3,-0.95);
\draw (0,0) -- (0.3,0.95);
\draw (0,0) -- (-0.8,-0.58);
\draw (0,0) -- (-0.8,0.58);

\draw (2,0) -- (2-1,0);
\draw (2,0) -- (2-0.3,-0.95);
\draw (2,0) -- (2-0.3,0.95);
\draw (2,0) -- (2+0.8,-0.58);
\draw (2,0) -- (2+0.8,0.58);


\draw (0,0) node {$\bullet$};
\draw (1,0) node {$\bullet$};
\draw (0.3,-0.95) node {$\bullet$};
\draw (0.3,0.95) node {$\bullet$};
\draw (-0.8,-0.58) node {$\bullet$};
\draw (-0.8,0.58) node {$\bullet$};
\draw (2-0.3,-0.95) node {$\bullet$};
\draw (2-0.3,0.95) node {$\bullet$};
\draw (2+0.8,-0.58) node {$\bullet$};
\draw (2+0.8,0.58) node {$\bullet$};
\draw (2,0) node {$\bullet$};

\draw (1,0) node [above]{$A_{6}$};
\draw (-0.1,-0.1) node [above right]{$A_{5}$};
\draw (0.3,-0.95) node [right]{$A_{4}$};
\draw (0.3,0.95) node [right]{$A_{1}$};
\draw (-0.8,-0.58) node [left]{$A_{3}$};
\draw (-0.8,0.58) node [left]{$A_{2}$};
\draw (2-0.3,-0.95) node [right]{$A_{8}$};
\draw (2-0.3,0.95) node [right]{$A_{11}$};
\draw (2+0.8,-0.58) node [right]{$A_{9}$};
\draw (2+0.8,0.58) node [right]{$A_{10}$};
\draw (2,-0.1) node [above left]{$A_{7}$};

\end{tikzpicture}
\end{center} 

and $A_{1},\dots,A_{10}$ generate the N\'eron--Severi lattice. In that
base, the divisor 
\[
D_{26}=(3,3,3,3,7,3,1,0,0,0)
\]
is ample, of square $26$, with $D_{26} \cdot A_{j}=1$ for $j\in\{1,\dots,11\},\,j\neq5$,
$D_{26} \cdot A_{5}=2$ and $D_{26} \cdot A_{k}=6$ for the remaining $\cu$-curves.
The divisors 
\[
A_{1}+A_{2}+A_{3}+A_{4}+2A_{5},\,\,2A_{7}+A_{8}+A_{9}+A_{10}+A_{11}
\]
are fibers of the same fibration. That gives the class of $A_{11}$
in the basis $A_{1},\dots,A_{10}$. The divisor 
\[
F_{5}=A_{6}+2A_{7}+A_{8}+A_{9}+A_{10}
\]
is a fiber of an elliptic fibration. For $j\in\{1,\dots,4\}$, one
has $A_{11+j}=F_{5}-A_{j}$ (as a class). The divisor 
\[
F_{1}=A_{2}+A_{3}+A_{4}+2A_{5}+A_{6}
\]
is a fiber of an elliptic fibration. For $j\in\{1,\dots,3\}$, one
has $A_{15+j}=F_{1}-A_{7+j}$ (as a class). The divisor 
\[
F_{2}=A_{1}+A_{3}+A_{4}+2A_{5}+A_{6}
\]
is a fiber of an elliptic fibration. For $j\in\{1,\dots,3\}$, one
has $A_{18+j}=F_{2}-A_{7+j}$ (as a class). The divisor 
\[
F_{3}=A_{1}+A_{2}+A_{4}+2A_{5}+A_{6}
\]
is a fiber of an elliptic fibration. For $j\in\{1,\dots,3\}$, one
has $A_{21+j}=F_{3}-A_{7+j}$ (as a class). The divisor 
\[
F_{4}=A_{1}+A_{2}+A_{3}+2A_{5}+A_{6}
\]
is a fiber of an elliptic fibration. For $j\in\{1,\dots,3\}$, one
has $A_{24+j}=F_{4}-A_{7+j}$ (as a class). 

The divisor 
\[
D_{2}=A_{1}+A_{2}+A_{3}+A_{4}+2A_{5}+A_{6}
\]
is nef, of square $2$, with $D_{2} \cdot A_{5}=D_{2} \cdot A_{7}=1$, $D_{2} \cdot A_{k}=0$
for $k\in\{1,\dots,4,6,8,\dots,11\}$ and $D_{2} \cdot A_{j}=2$ for the
remaining curves. We also have 
\[
D_{2}\equiv A_{6}+2A_{7}+A_{8}+A_{9}+A_{10}+A_{11}.
\]
By using the linear system $|D_{2}|$, we obtain the following. 

\begin{prop}
\label{prop:section10.7}The surface $X$ is the double cover of the
plane branched over a sextic curve which is the union of two lines
$L,L'$ and a quartic $Q$. 
\end{prop}

The images of $A_{5}$, $A_{7}$ are the two lines. We observe that 
\[
\begin{array}{l}
D_{2}\equiv A_{16}+A_{1}+A_{8}\equiv A_{17}+A_{1}+A_{9}\equiv A_{18}+A_{1}+A_{10},\\
D_{2}\equiv A_{19}+A_{2}+A_{8}\equiv A_{20}+A_{2}+A_{9}\equiv A_{21}+A_{2}+A_{10},\\
D_{2}\equiv A_{22}+A_{3}+A_{8}\equiv A_{23}+A_{3}+A_{9}\equiv A_{24}+A_{3}+A_{10},\\
D_{2}\equiv A_{25}+A_{4}+A_{8}\equiv A_{26}+A_{4}+A_{9}\equiv A_{27}+A_{4}+A_{10}.
\end{array}
\]
Thus $A_{1}$, $A_{2}$, $A_{3}$, $A_{4}$ are contracted to the nodes in the
intersection of $Q$ and $L$, and so are $A_{8}$, $A_{9}$, $A_{10}$. Since
\[
D_{2}\equiv A_{1}+A_{11}+A_{12}\equiv A_{2}+A_{11}+A_{13}\equiv A_{3}+A_{11}+A_{14}\equiv A_{4}+A_{11}+A_{15},
\]
the curve $A_{11}$ is contracted to the fourth intersection point
of $L'$ and $Q$. Moreover, we see from these equivalences that the
$16$ $(-2)$-curves $A_{12},\dots,A_{27}$ are sent to the $16$
lines  between points in $L\cap Q$ and points in $L'\cap Q$.
The curve $A_{6}$ is contracted onto the intersection point of $L$
and $L'$. 

Proposition \ref{prop:section10.7} implies  that the moduli space
of K3 surfaces with N\'eron--Severi group isometric to $U\oplus\mathbf{D}_{4}\oplus\mathbf{A}_{1}^{\oplus4}$
is unirational.

\subsection{The lattice $\boldsymbol{U\oplus\mathbf{A}_{1}^{\oplus8}}$}

\subsubsection{First involution}

Let us denote by $f_{1},f_{2},e_{1},\dots,e_{8}$ the canonical basis
of $U\oplus\mathbf{A}_{1}^{\oplus8}$. In that basis, let  
\[
D_{6}=(7,5,-2,-2,-2,-2,-2,-2,-2,-2).
\]
It has square $6$, and no $\cu$-classes are perpendicular to it.
We thus have a marking such that $U\oplus{\bf A}_{1}^{\oplus6}\simeq\NS X)$
 which maps $D_{6}$ to an ample class. The K3 surface $X$ contains
$145$ $\cu$-curves; with respect to $D_{6}$, the curves $A_{1},\dots,A_{16}$
have degree $1$, the curve $A_{0}$ has degree $2$, and the remaining
curves have degree $4$. Let us describe these curves. 

For $j\in\{1,\dots,8\}$, one has $A_{j}=f_{1}-e_{j}$; the divisors
\[
A_{1}+A_{9},\dots,A_{8}+A_{16}
\]
are fibers of an elliptic fibration  $\varphi\colon X\to\PP^{1}$, where
the class of a fiber $F$ is 
\[
F=(4,2,-1,-1,-1,-1,-1,-1,-1,-1)
\]
in the canonical basis. The curve $A_{0}$ is $A_{0}=-f_{1}+f_{2}$, 
and $A_{0} \cdot A_{k}=1$ for $1\leq k\leq16$. 

For the choice of any three elements $\{i,j,k\}$ in $\{1,\dots,8\}$
($56$ possibilities), the classes 
\[
\begin{array}{l}
A_{i,j,k}=4f_{1}+4f_{2}-e_{i}-e_{j}-e_{k}-\sum_{l=1}^{8}e_{l},\\
B_{i,j,k}=2f_{1}+2f_{2}+e_{i}+e_{j}+e_{k}-\sum_{l=1}^{8}e_{l}
\end{array}
\]
are classes of $(-2)$-curves. The $16$ classes
\[
\begin{array}{c}
C_{j}=6f_{1}+6f_{2}-e_{j}-2\sum_{l=1}^{8}e_{l}\quad\text{and}\quad E_{j}=e_{j},\end{array}j\in\{1,\dots,8\}
\]
are classes of $(-2)$-curves. Thus in total, we get $145$ $\cu$-curves. 

The divisor 
\[
D_{2}=(3,3,-1,-1,-1,-1,-1,-1,-1,-1)=F+A_{0}
\]
has square $2$, and $D_{2} \cdot A_{j}=1$ for $j\in\{1,\dots,16\}$, $D_{2} \cdot A_{0}=0$
and $D_{2}A=2$ for the remaining $\cu$-curves $A$. Let $C_{6}$
be the sextic branch curve of the associated double cover $X\to\PP^{2}$.
The curve $C_{6}$ has a node $q$ onto which the curve $A_{0}$ is
contracted. For $j\in\{1,\dots,8\}$, we have 
\[
A_{j}+A_{8+j}+A_{0}\equiv D_{2}; 
\]
thus the divisor $A_{j}+A_{8+j}+A_{0}$ is the pull-back of a line
through the nodal point $q$, and that line is tangent to $C_{6}$
at every other intersection points. For a set of three elements $\{i,j,k\}$
in $\{1,\dots,8\}$, we have
\[
A_{i,j,k}+B_{i,j,k}\equiv2D_{2}; 
\]
thus the $56$ divisors $A_{i,j,k}+B_{i,j,k}$ are pull-backs of $56$
conics that are $6$-tangent to $C_{6}$ (and not containing the nodal
point of $C_{6}$). We have 
\[
C_{j}+E_{j}\equiv2D_{2};
\]
thus the eight divisors $C_{j}+E_{j}$, $j\in\{1,\dots,8\}$, are also
pull-backs of eight conics that are $6$-tangent to $C_{6}$. Summing
up, we have the following. 

\begin{prop}
The K3 surface $X$ is the double cover of $\,\PP^{2}$ branched over
a nodal sextic curve. Through the node of the sextic, there are eight 
lines that are tangent to the sextic at other intersection points.
Moreover, there are $64$ conics that are $6$-tangent to the sextic. 
\end{prop}

The N\'eron--Severi lattice is generated by the classes of $A_{0},A_{1},\dots,A_{8}$
and $E_{8}$. 

\begin{rem}
Let $S$ be the set of the $128$ $(-2)$-curves that are above the
$64$ conics. One can prove that there exists a partition of $S$
into eight sets $S_{1},\dots,S_{8}$ of $16$ curves such that for
any such set $S_{i}=\{B_{1},\dots,B_{16}\}$, one has (up to permutation
of the indices) $B_{2k-1} \cdot B_{2k}=6$ for $k\in\{1,\dots,8\}$, and
for $s$, $t$ ($s\neq t$) such that $\{s,t\}\neq\{2k-1,2k\}$, one has
$B_{s} \cdot B_{t}=4$ if $s+t=1$ mod $2$; else $B_{s} \cdot B_{t}=0$. 

For $j\neq t$ in $\{1,\dots,8\}$, if $B$ is a curve in $S_{j}$,
there are $10$ $(-2)$-curves $B'$ in $S_{t}$ such that $BB'=2$,
$3$ $(-2)$-curves such that $BB'=0$ and $3$ $(-2)$-curves such
that $BB'=4$. 
\end{rem}

\begin{rem}
The ample divisor $D_{6}$ satisfies $D_{6}\equiv2F+A_{0}$ and $FD_{6}=2$;
thus it is hyperelliptic. According to Theorem \cite{SaintDonat},
the morphism associated to $|D_{6}|$ is a degree $2$ map onto a
degree $3$ rational normal scroll in $\PP^{4}$. That surface is
the Hirzebruch surface $\mathbb{F}_{1}$, which is the blow-up of $\PP^{2}$
in one point, embedded by $|2L-E_{0}|$, where $L$ is the pull-back
of a line and $E_{0}$ is the exceptional divisor. The composite of
$X\to\mathbb{F}_{1}$ with the natural map $\mathbb{F}_{1}\to\PP^{2}$
is given by the linear system $|D_{2}|$. 
\end{rem}

\subsubsection{Second involution: A more geometric interpretation of the ($-$2)-curves}

Let $f_{1},f_{2},e_{1},\dots,e_{8}$ be the canonical basis of $U\oplus\mathbf{A}_{1}^{\oplus8}$.
The divisor $D_{2}'=f_{1}+f_{2}$ is nef, of square $2$. We recall
that $A_{0}=-f_{1}+f_{2}$ is a $\cu$-curve; moreover, the divisor
$F'=f_{1}$ is a fiber of an elliptic fibration; thus 
\[
D'_{2}=2F'+A_{0}.
\]
One has $FA_{0}=1$, $D'_{2}F'=1$, and therefore the linear system
$|D_{2}'|$ has base points. We have $D'_{2} \cdot A_{0}=0$ and
\begin{alignat*}{3}
D_{2}'\cdot A_{j}&=1,\; D_{2}'\cdot A_{j+8}=5,\quad&&\forall j\in\{1,\dots,8\},\\[-1ex]
D'_{2}\cdot A_{i,j,k}&=8,\quad&&\forall\{i,j,k\}\subset\{1,\dots,8\}\text{ of order }3,\\[-1ex]
D'_{2}\cdot B_{i,j,k}&=4,\quad&&\forall\{i,j,k\}\subset\{1,\dots,8\}\text{ of order }3,\\[-1ex]
D'_{2}\cdot C_{j}&=12,\; D'_{2}\cdot E_{j}=0,\quad&&\forall k\in\{1,\dots,8\}.
\end{alignat*}
Let us define $D_{8}=2D'_{2}$. The linear system $|D_{8}|$ is base-point
free, and it is hyperelliptic since $D_{8}F'=2$. By Theorem \ref{thm:SaintDonat-2},
case iii) v), the associated map $\varphi_{D_{8}}\colon X\to\PP^{5}$ has
image a cone over a rational normal curve in $\PP^{4}$; it factorizes
through a surjective map $\varphi'\colon X\to\mathbf{F}_{4}$, where $\mathbf{F}_{4}$
is the Hirzebruch surface with a section $s$ such that $s^{4}=-4$.
The section $s$ is mapped to the vertex of the cone by the map $\mathbf{F}_{4}\to\PP^{5}$.
Let $f$ denote a fiber of the unique fibration $\mathbf{F}_{4}\to\PP^{1}$.
The branch locus of $\varphi'$ is the union of $s$ and a curve $C$
such that $Cs=0$ and $C\in|3(s+4f)|$ (so that $s+C\in|-2K_{\mathbf{F}_{4}}|$).
The curves $E_{1},\dots,E_{8}$ being contracted by $\varphi'$, the
curve $C$ has eight nodes $p_{1},\dots,p_{8}$, which are the images
of $E_{1},\dots,E_{8}$. We have 
\[
|D_{8}|=\varphi^{*}|4f+s|.
\]
We moreover have the relations 
\[
\begin{array}{c}
A_{j}+E_{j}\equiv F'=\varphi^{*}f,\quad\forall j\in\{1,\dots,8\};\end{array}
\]
therefore, the curves $A_{1},\dots,A_{8}$ are mapped by $\varphi'$
to the eight fibers going through the nodes $p_{1},\dots,p_{8}$. Since
$D_{8}\equiv4(A_{1}+E_{1})+2A_{0}$, the curve $A_{0}$ is in the
ramification locus, with image $s$. Since 
\[
A_{8+k}+\sum_{j=1,j\neq k}^{j=8}E_{j}\equiv D_{8}+F',\quad\forall k\in\{1,\dots,8\},
\]
the image of $A_{8+k}$ belongs in $|5f+s|$; it is a curve which
goes through the seven points $\{p_{1},\dots,p_{8}\}\setminus\{p_{k}\}$.
We moreover have the relations 
\begin{equation}
\begin{array}{l}
\hphantom{2}D_{8}\equiv B_{ijk}-(E_{i}+E_{j}+E_{k})+\sum_{t=1}^{8}E_{t},\hfill\\
2D_{8}\equiv A_{ijk}+E_{i}+E_{j}+E_{k}+\sum_{t=1}^{8}E_{s},\hfill\\
3D_{8}\equiv C_{k}+E_{k}+2\sum_{t=1}^{8}E_{s},\hfill
\end{array}\label{eq:240}
\end{equation}
and therefore
\begin{itemize}
\item the image of $B_{ijk}$ is a curve in the linear system $|4f+s|$
which goes through the five points distinct from $p_{i}$, $p_{j}$, $p_{k}$; 
\item the image of $A_{ijk}$ is a curve in the linear system $|2(4f+s)|$
which goes through the eight points $p_{1},\dots,p_{8}$ with double
points at $p_{i}$, $p_{j}$, $p_{k}$;
\item the image of $C_{k}$ is a curve in the linear system $|3(4f+s)|$
which goes through  eight points with double points at all except
at the single point $p_{k}$ with multiplicity $3$.
\end{itemize}
Let $J,J'$ be subsets of order $3$ of $\{1,\dots,8\}$. Using the
relations in \eqref{eq:240} and $2D_{2}\equiv A_{J}+B_{J}\equiv C_{k}+E_{k}$,
one finds that
\begin{align*}
A_{J}\cdot A_{J'}&=B_{J} \cdot B_{J'}=4-2\#(J\cap J'),\\[-1ex]
A_{J} \cdot B_{J'}&=2\#(J\cap J'),\hfill\\[-1ex]
C_{k}\cdot C_{k'}&=E_{k}E_{k'}=-2\delta_{kk'},\hfill\\[-1ex]
C_{k}\cdot E_{k'}&=4+2\d_{kk'},\hfill\\[-1ex]
C_{k} \cdot B_{J}&=4\text{ if }k\in J,\quad C_{k} \cdot B_{J}=2\text{ if }k\notin J,\\[-1ex]
E_{k} \cdot B_{J}&=0\text{ if }k\in J,\quad E_{k}\cdot B_{J}=2\text{ if }k\notin J.
\end{align*}

The situation is very much similar to what happens for K3 surfaces
$Y$ with N\'eron--Severi lattice $U(2)\oplus\mathbf{A}_{1}^{\oplus7}$,
which are double covers of del Pezzo surfaces of degree $1$, where
the $240$ $(-2)$-curves of $Y$ come from lines, conics, cubics,
quartics, quintics and sextics going through $8$ points in the plane
with various multiplicities. In particular, the $112$ $\cu$-curves
$A_{J},B_{J}$ with $J\subset\{1,\dots,8\}$ of order $3$ has the
same configurations as the $112$ $\cu$-curves on a K3 surface $Y$
which are pull-backs of conics and quartics. 

\begin{rem}
According to Kondo \cite{Kondo}, the automorphism group of the
surface $X$ is $(\ZZ/2\ZZ)^{2}$. The branch loci of the two involutions
associated to $D_{2}$ and $D_{8}$ have genus $9$ and
$2$, respectively. Thus these two involutions generate the automorphism
group of $X$. 
\end{rem}

\subsection{The lattice $\boldsymbol{U\oplus\mathbf{A}_{2}\oplus\mathbf{E}_{6}}$}

The K3 surface $X$ contains $11$ $(-2)$-curves $A_{1},\dots,A_{11}$; 
their dual graph is 

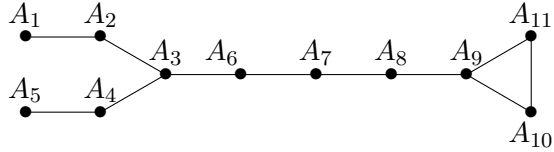
\begin{center}
\begin{tikzpicture}[scale=1]

\draw (3,0) -- (-1,0);
\draw (3,0) -- (3+0.86,-0.5);
\draw (3,0) -- (3+0.86,0.5);
\draw (3+0.86,-0.5) -- (3+0.86,0.5);
\draw (-1-0.866,0.5) -- (-1,0);
\draw (-1-0.866,-0.5) -- (-1,0);
\draw (-1-0.896*2,0.5) -- (-1-0.866,0.5);
\draw (-1-0.895*2,-0.5) -- (-1-0.866,-0.5);


\draw (-1-0.866,0.5) node {$\bullet$}; 
\draw (-1-0.866,-0.5) node {$\bullet$}; 
\draw (-2-0.86,-0.5) node {$\bullet$}; 
\draw (-2-0.86,0.5) node {$\bullet$}; 
\draw (-1,0) node {$\bullet$};
\draw (0,0) node {$\bullet$};
\draw (1,0) node {$\bullet$};
\draw (2,0) node {$\bullet$};
\draw (3,0) node {$\bullet$};
\draw (3+0.86,-0.5) node {$\bullet$};
\draw (3+0.86,0.5) node {$\bullet$};

\draw (-1-0.866,0.5)  node [above]{$A_{2}$};
\draw (-2-0.866,0.5)  node [above]{$A_{1}$};
\draw (-1-0.866,-0.5)  node [above]{$A_{4}$};
\draw (-2-0.866,-0.5)  node [above]{$A_{5}$};
\draw (-1,0) node [above]{$A_{3}$};
\draw (-0.25,0) node [above]{$A_{6}$};
\draw (1,0) node [above]{$A_{7}$};
\draw (2,0) node [above]{$A_{8}$};
\draw (3,0) node [above]{$A_{9}$};
\draw (3+0.86,-0.5) node [below]{$A_{10}$};
\draw (3+0.86,0.5) node [above]{$A_{11}$};

\end{tikzpicture}
\end{center} 

The curves $A_{1},\dots,A_{10}$ generate the N\'eron--Severi lattice.
In that base, the divisor 
\[
D_{160}=(10,21,33,21,10,25,18,12,7,3)
\]
is ample, of square $160$, with $D_{160} \cdot A_{j}=1$ for $j\leq10$
and $D_{160} \cdot A_{11}=10$. The divisors
\[
F_{1}=A_{9}+A_{10}+A_{11},\quad F_{2}=A_{1}+A_{7}+A_{5}+2(A_{2}+A_{4}+A_{6})+3A_{3}
\]
are fibers of an elliptic fibration  and $A_{8}$ is a section. By
Theorem \ref{thm:SaintDonat-2}, case i) a), we have the following. 

\begin{prop}
The linear system $|4F_{1}+2A_{8}|$ defines a morphism $\varphi\colon X\to\mathbf{F}_{4}$
branched over the unique section $s$ with $s^2=-4$ and a curve $B\in|3s+12f|$. The curve
$B$ has one $\mathbf{a}_{2}$ singularity $p$ and one $\mathbf{e}_{6}$
singularity $q$. The pull-backs of the fibers through $p$, $q$ are the
fibers $F_{1}$, $F_{2}$.
\end{prop}

\section{Rank 11 lattices }

\subsection{The lattice $\boldsymbol{U\oplus\mathbf{E}_{8}\oplus\mathbf{A}_{1}}$}

The K3 surface $X$ contains $12$ $(-2)$-curves $A_{1},\dots,A_{12}$; 
their dual graph is 

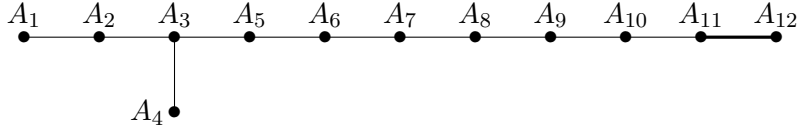
\begin{center}
\begin{tikzpicture}[scale=1]

\draw (0,0) -- (9,0);
\draw [very thick] (9,0) -- (10,0);
\draw (2,0) -- (2,-1);

\draw (0,0) node {$\bullet$};
\draw (1,0) node {$\bullet$};
\draw (2,0) node {$\bullet$};
\draw (3,0) node {$\bullet$};
\draw (4,0) node {$\bullet$};
\draw (5,0) node {$\bullet$};
\draw (6,0) node {$\bullet$};
\draw (7,0) node {$\bullet$};
\draw (8,0) node {$\bullet$};
\draw (9,0) node {$\bullet$};
\draw (10,0) node {$\bullet$};
\draw (2,-1) node {$\bullet$};

\draw (2,-1) node [left]{$A_{4}$};
\draw (0,0) node [above]{$A_{1}$};
\draw (1,0) node [above]{$A_{2}$};
\draw (2,0) node [above]{$A_{3}$};
\draw (3,0) node [above]{$A_{5}$};
\draw (4,0) node [above]{$A_{6}$};
\draw (5,0) node [above]{$A_{7}$};
\draw (6,0) node [above]{$A_{8}$};
\draw (7,0) node [above]{$A_{9}$};
\draw (8,0) node [above]{$A_{10}$};
\draw (9,0) node [above]{$A_{11}$};
\draw (10,0) node [above]{$A_{12}$};

\end{tikzpicture}
\end{center} 

The curves $A_{1},\dots,A_{11}$ generate the N\'eron--Severi lattice.
In that base, the divisor 
\[
D_{848}=(48,97,147,73,125,104,84,65,47,30,14)
\]
 is ample, of square $848$, with $D_{848} \cdot A_{j}=1$ for $j\in\{1,\dots,10\}$,
$D_{848} \cdot A_{11}=2$ and $D_{848} \cdot A_{12}=28.$ The divisors 
\[
F_{1}=2A_{1}+4A_{2}+6A_{3}+3A_{4}+5A_{5}+4A_{6}+3A_{7}+2A_{8}+A_{9},\quad F_{2}=A_{11}+A_{12}
\]
are fibers of an elliptic fibration with section $A_{10}$. By Theorem
\ref{thm:SaintDonat-2}, case i) a), we have the following. 

\begin{prop}
The linear system $|4F_{1}+2A_{10}|$ defines a morphism $\varphi\colon X\to\mathbf{F}_{4}$
branched over the unique section $s$ with $s^2=-4$ and a curve $B\in|3s+12f|$. The curve
$B$ has one $\mathbf{e}_{8}$ singularity $p$ and one node $q$.
The pull-backs of the fibers through $p$, $q$ are the fibers $F_{1}$, $F_{2}$.
\end{prop}

\subsection{The lattice $\boldsymbol{U\oplus\mathbf{D}_{8}\oplus\mathbf{A}_{1}}$}

The K3 surface $X$ contains $14$ $\cu$-curves with dual graph

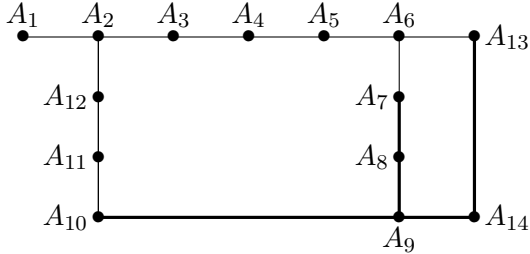
\begin{center}
\begin{tikzpicture}[scale=1]

\draw (0,0) -- (6,0);
\draw (1,0) -- (1,-2.4);
\draw (1,0) -- (1,-2.4);
\draw (5,0) -- (5,-2.4);
\draw [very thick] (5,-0.8) -- (5,-2.4);
\draw [very thick] (6,0) -- (6,-2.4);
\draw [very thick] (1,-2.4) -- (6,-2.4);


\draw (0,0) node {$\bullet$}; 
\draw (1,0) node {$\bullet$};
\draw (2,0) node {$\bullet$};
\draw (3,0) node {$\bullet$};
\draw (4,0) node {$\bullet$};
\draw (5,0) node {$\bullet$};
\draw (6,0) node {$\bullet$};
\draw (5,-0.8) node {$\bullet$};
\draw (5,-0.8*2) node {$\bullet$};
\draw (5,-0.8*3) node {$\bullet$};
\draw (1,-0.8) node {$\bullet$};
\draw (1,-0.8*2) node {$\bullet$};
\draw (1,-0.8*3) node {$\bullet$};
\draw (6,-0.8*3) node {$\bullet$};

\draw (0,0) node [above]{$A_{1}$};
\draw (1,0) node [above]{$A_{2}$};
\draw (2,0) node [above]{$A_{3}$};
\draw (3,0) node [above]{$A_{4}$};
\draw (4,0) node [above]{$A_{5}$};
\draw (5,0) node [above]{$A_{6}$};
\draw (6,0) node [right]{$A_{13}$};
\draw (5,-0.8) node [left]{$A_{7}$};
\draw (5,-0.8*2) node [left]{$A_{8}$};
\draw (5,-0.8*3) node [below]{$A_{9}$};
\draw (1,-0.8) node [left]{$A_{12}$};
\draw (1,-0.8*2) node [left]{$A_{11}$};
\draw (1,-0.8*3) node [left]{$A_{10}$};
\draw (6,-0.8*3) node [right]{$A_{14}$};

\end{tikzpicture}
\end{center} 

The curves $A_{1},\dots,A_{9},A_{11},A_{12}$ generate  the N\'eron--Severi
lattice; in that basis, the divisor 
\[
D_{208}=(-23,-45,-35,-24,-12,1,15,15,9,-16,-31)
\]
is ample, of square $208$; the degrees $D_{208} \cdot A_{j}$ of the $\cu$-curves
$A_{j}$, $j=1,\dots,14$, are 
\[
1,1,1,1,1,1,1,18,12,2,1,1,1,18.
\]
The divisor $F_{1}=A_{9}+A_{10}$ is the fiber of an elliptic fibration, 
 and there is a second fiber $F_{2}$ supported on $A_{1},\dots,A_{7},A_{12},A_{13}$
of type $\tilde{\mathbf{D}_{8}}$. The curve $A_{11}$ is a section.
By Theorem \ref{thm:SaintDonat-2}, case i) a), we have the following. 

\begin{prop}
The linear system $|4F_{1}+2A_{11}|$ defines a morphism $\varphi\colon X\to\mathbf{F}_{4}$
branched over the unique section $s$ with $s^2=-4$ and a curve $B\in|3s+12f|$. The curve
$B$ has one $\mathbf{d}_{8}$ singularity $p$ and one node $q$.
The pull-backs of the fibers through $p$, $q$ are the fibers $F_{2}$, $F_{1}$.
\end{prop}

One can also find a construction using the double cover associated
to the divisor 
\[
D_{2}=A_{7}+A_{8}+A_{9}
\]
which is nef, base-point free, of square $2$, with intersections $D_{2} \cdot A_{j}$,
$j=1,\dots,14$, equal to, respectively,  
\[
0,0,0,0,0,1,0,2,0,2,0,0,0,2.
\]

\subsection{The lattice $\boldsymbol{U\oplus\mathbf{D}_{4}\oplus\mathbf{D}_{4}\oplus\mathbf{A}_{1}}$}

The K3 surface contains $22$ $(-2)$-curves $A_{1},\dots,A_{22}$.
The configuration of the curves $A_{1},\dots,A_{15}$ is 

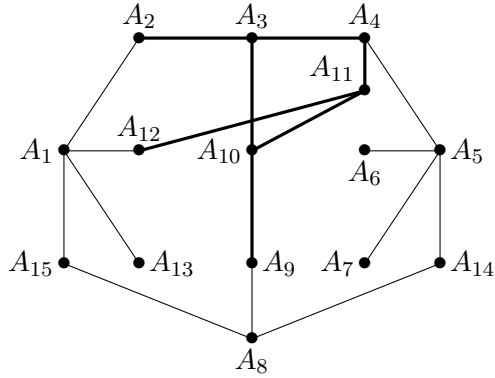
\begin{center}
\begin{tikzpicture}[scale=1]

\draw (0,0) -- (2.5,1);
\draw (0,0) -- (-2.5,1);
\draw (0,0) -- (0,1);
\draw [very thick] (0,1) -- (0,4);
\draw [very thick] (-1.5,4) -- (1.5,4);
\draw [very thick] (1.5,4) -- (1.5,3.3);
\draw [very thick] (-1.5,2.5) -- (1.5,3.3);
\draw [very thick] (0,2.5) -- (1.5,3.3);
\draw (2.5,1) -- (2.5,2.5);
\draw (-2.5,1) -- (-2.5,2.5);
\draw (-2.5,2.5) -- (-1.5,2.5);
\draw (2.5,2.5) -- (1.5,2.5);
\draw (-2.5,2.5) -- (-1.5,4);
\draw (2.5,2.5) -- (1.5,4);
\draw (2.5,2.5) -- (1.5,1);
\draw (-2.5,2.5) -- (-1.5,1);

\draw (0,0) node [below]{$A_{8}$};
\draw (0,1) node [right]{$A_{9}$};
\draw (0,2.5) node [left]{$A_{10}$};
\draw (1.5,1) node [left]{$A_{7}$};
\draw (2.5,1) node [right]{$A_{14}$};
\draw (-1.5,1) node [right]{$A_{13}$};
\draw (-2.5,1) node [left]{$A_{15}$};
\draw (1.5,2.5) node [below]{$A_{6}$};
\draw (2.5,2.5) node [right]{$A_{5}$};
\draw (-1.5,2.5) node  [above]{$A_{12}$};
\draw (-2.5,2.5) node [left]{$A_{1}$};
\draw (1.5,4) node [above]{$A_{4}$};
\draw (-1.5,4) node [above]{$A_{2}$};
\draw (0,4) node [above]{$A_{3}$};
\draw (1.5,3.3) node [above left]{$A_{11}$};

\draw (0,0) node {$\bullet$};
\draw (0,1) node {$\bullet$};
\draw (1.5,1) node {$\bullet$};
\draw (2.5,1) node {$\bullet$};
\draw (-1.5,1) node {$\bullet$};
\draw (-2.5,1) node {$\bullet$};
\draw (1.5,2.5) node {$\bullet$};
\draw (2.5,2.5) node {$\bullet$};
\draw (-1.5,2.5) node {$\bullet$};
\draw (-2.5,2.5) node {$\bullet$};
\draw (1.5,4) node {$\bullet$};
\draw (-1.5,4) node {$\bullet$};
\draw (0,4) node {$\bullet$};
\draw (1.5,3.3) node {$\bullet$};
\draw (0,2.5) node {$\bullet$};

\end{tikzpicture}
\end{center} 

The curves $A_{1},\dots,A_{11}$ generate  the rank $11$ N\'eron--Severi
lattice. In that base, the divisor 
\[
D_{48}=(1,3,3,3,1,0,0,0,1,2,0)
\]
 is ample, of square $48$. We have 
\[
\begin{array}{l}
A_{13}\equiv(-2,-3,-2,-2,-2,-1,-1,2,4,3,1),\\
A_{14}\equiv(0,0,0,-1,-2,-1,-1,0,1,1,0),\\
A_{15}\equiv(0,1,1,2,2,1,1,-2,-3,-2,0).
\end{array}
\]
The divisor $D_{2}=A_{2}+A_{3}+A_{4}$ is nef, of square $2$, base-point
free, with $D_{2} \cdot A_{1}=D_{2} \cdot A_{5}=1$, $D_{2} \cdot A_{j}=0$ for $j\in\{2,4,6,7,8,9,12,13,14,15\}$
and else $D_{2} \cdot A_{j}=2$. We have 
\[
\begin{array}{l}
D_{2}\equiv2A_{1}+(A_{2}+2A_{8}+A_{9}+A_{12}+A_{13}+A_{14}+2A_{15}),\hfill\\
D_{2}\equiv2A_{5}+(A_{4}+A_{6}+A_{7}+2A_{8}+A_{9}+2A_{14}+A_{15}),\hfill\\
D_{2}\equiv A_{10}+(2A_{8}+2A_{9}+A_{14}+A_{15})\equiv A_{11}+(A_{4}+A_{12}),\\
D_{2}\equiv A_{16}+A_{4}+A_{13}\equiv A_{17}+A_{2}+A_{6}\equiv A_{18}+A_{6}+A_{12},\\
D_{2}\equiv A_{19}+A_{13}+A_{6}\equiv A_{20}+A_{2}+A_{7}\equiv A_{21}+A_{7}+A_{12},\\
D_{2}\equiv A_{22}+A_{7}+A_{13}.\hfill
\end{array}
\]
Therefore, the K3 surface is a double cover of $\PP^{2}$ branched
over a sextic curve $C_{6}$, and the following holds. 

\begin{prop}
The curves $A_{1}$, $A_{5}$ are in the ramification locus, and their
images are lines $L$, $L'$ that are components of $\,C_{6}$. Let $Q$
be the residual quartic curve. The lines $L$, $L'$ meet in a point $q$
which is on $Q$, so that the sextic has a $\mathbf{d}_{4}$ singularity
at $q$.
\end{prop}

The curves $A_{8}$, $A_{9}$, $A_{14}$, $A_{15}$ are contracted to $q$. The
line $L$ meets the quartic in three other points $p_{2}$, $p_{12}$, $p_{13}$,
which are the images of $A_{2}$, $A_{12}$, $A_{13}$. The line $L'$ meet
the quartic in three other points $p_{4}$, $p_{6}$, $p_{7}$, which are the
images of $A_{4}$, $A_{6}$, $A_{7}$. The images of the nine curves $A_{3}$, $A_{11}$, $A_{16},\dots,A_{22}$
are lines through the points $(p_{2},p_{4})$, $(p_{4},p_{12})$, $(p_{4},p_{13})$, $(p_{2},p_{6})$,
$(p_{6},p_{12})$, $(p_{6},p_{13})$, $(p_{2},p_{7})$, $(p_{7},p_{12})$,
$(p_{7},p_{13})$, respectively. Finally, the image of $A_{10}$ is
a line tangent to the third branch of the $\mathbf{d}_{4}$ singularity.

\subsection{The lattice $\boldsymbol{U\oplus\mathbf{D}_{4}\oplus\mathbf{A}_{1}^{\oplus5}}$}

\subsubsection{First involution}

The K3 surface contains $90$ $\cu$-curves. The first $20$  curves
$A_{1},\dots,A_{20}$ have the following configuration: 

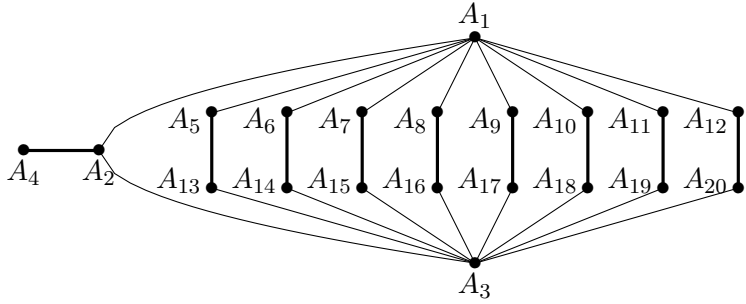
\begin{center}
\begin{tikzpicture}[scale=1]
\draw [domain=-5:0] plot(\x,{0.67*(5-abs(\x))^0.5});
\draw [domain=-5:0] plot(\x,{-0.67*(5-abs(\x))^0.5});

\draw [very thick] (-6,0) -- (-5,0);

\draw (0,1.5) -- (3.5,0.5);
\draw (0,1.5) -- (2.5,0.5);
\draw (0,1.5) -- (1.5,0.5);
\draw (0,1.5) -- (0.5,0.5);
\draw (0,1.5) -- (-0.5,0.5);
\draw (0,1.5) -- (-1.5,0.5);
\draw (0,1.5) -- (-2.5,0.5);
\draw (0,1.5) -- (-3.5,0.5);

\draw [very thick] (3.5,0.5) -- (3.5,-0.5);
\draw [very thick] (2.5,0.5) -- (2.5,-0.5);
\draw [very thick] (1.5,0.5) -- (1.5,-0.5);
\draw [very thick] (0.5,0.5) -- (0.5,-0.5);
\draw [very thick] (-3.5,0.5) -- (-3.5,-0.5);
\draw [very thick] (-2.5,0.5) -- (-2.5,-0.5);
\draw [very thick] (-1.5,0.5) -- (-1.5,-0.5);
\draw [very thick] (-0.5,0.5) -- (-0.5,-0.5);

\draw (0,-1.5) -- (3.5,-0.5);
\draw (0,-1.5) -- (2.5,-0.5);
\draw (0,-1.5) -- (1.5,-0.5);
\draw (0,-1.5) -- (0.5,-0.5);
\draw (0,-1.5) -- (-0.5,-0.5);
\draw (0,-1.5) -- (-1.5,-0.5);
\draw (0,-1.5) -- (-2.5,-0.5);
\draw (0,-1.5) -- (-3.5,-0.5);

\draw (0,1.5) node {$\bullet$};
\draw (3.5,0.5) node  {$\bullet$};
\draw (2.5,0.5) node  {$\bullet$};
\draw (1.5,0.5) node  {$\bullet$};
\draw (0.5,0.5) node  {$\bullet$};
\draw (-0.5,0.5) node  {$\bullet$};
\draw (-1.5,0.5) node  {$\bullet$};
\draw (-2.5,0.5) node  {$\bullet$};
\draw (-3.5,0.5) node  {$\bullet$};

\draw (-5,0) node  {$\bullet$};
\draw (-6,0) node  {$\bullet$};

\draw (0,1.5) node [above]{$A_{1}$};
\draw (3.5,0.4) node [left]{$A_{12}$};
\draw (2.5,0.4) node [left]{$A_{11}$};
\draw (1.5,0.4) node [left]{$A_{10}$};
\draw (0.5,0.4) node [left]{$A_{9}$};
\draw (-0.5,0.4) node [left]{$A_{8}$};
\draw (-1.5,0.4) node [left]{$A_{7}$};
\draw (-2.5,0.4) node [left]{$A_{6}$};
\draw (-3.5,0.4) node [left]{$A_{5}$};

\draw (0,-1.5) node {$\bullet$};
\draw (3.5,-0.5) node  {$\bullet$};
\draw (2.5,-0.5) node  {$\bullet$};
\draw (1.5,-0.5) node  {$\bullet$};
\draw (0.5,-0.5) node  {$\bullet$};
\draw (-0.5,-0.5) node  {$\bullet$};
\draw (-1.5,-0.5) node  {$\bullet$};
\draw (-2.5,-0.5) node  {$\bullet$};
\draw (-3.5,-0.5) node  {$\bullet$};

\draw (0,-1.5) node [below]{$A_{3}$};
\draw (3.5,-0.4) node [left]{$A_{20}$};
\draw (2.5,-0.4) node [left]{$A_{19}$};
\draw (1.5,-0.4) node [left]{$A_{18}$};
\draw (0.5,-0.4) node [left]{$A_{17}$};
\draw (-0.5,-0.4) node [left]{$A_{16}$};
\draw (-1.5,-0.4) node [left]{$A_{15}$};
\draw (-2.5,-0.4) node [left]{$A_{14}$};
\draw (-3.5,-0.4) node [left]{$A_{13}$};
\draw (-6,0) node [below]{$A_{4}$};
\draw (-5,0) node [below]{$A_{2}$};


\end{tikzpicture}
\end{center} 

The curves $A_{1},\dots,A_{11}$ generate the N\'eron--Severi lattice.
The divisor 
\[
D_{18}=A_{1}+5A_{2}+2A_{3}+4A_{4}
\]
is ample, of square $18$, with $D_{18} \cdot A_{1}=3$, $D_{18} \cdot A_{k}=1$
for $k\in\{2,3,5,\dots,12\}$, $D_{18} \cdot A_{k}=2$ for $k\in\{4,13,\dots,20\}$
and $D_{18} \cdot A_{k}=8$ for $k>20$. 

The divisor 
\[
D_{2}=A_{1}+2A_{2}+A_{3}+A_{4}
\]
is nef, base-point free, of square $2$, with $D_{2} \cdot A_{1}=D_{2} \cdot A_{2}=D_{2} \cdot A_{3}=0$,
$D_{2} \cdot A_{j}=1$ for $j\in\{5,\dots,20\}$ and $D_{2} \cdot A_{j}=2$ for
$j=4$ or $j>20$. In the basis $A_{1},\dots,A_{11}$, we have 
\[
A_{12}=(-2,4,2,4,-1,-1,-1,-1,-1,-1,-1).
\]
Moreover, for $k\in\{5,\dots,12\}$, we have $A_{k}+A_{k+8}\equiv A_{2}+A_{4}$
(these are fibers of an elliptic fibration );  thus 
\[
D_{2}\equiv A_{1}+A_{2}+A_{3}+A_{k}+A_{k+8},\quad\forall k\in\{5,\dots,12\}, 
\]
and we obtain in that way the classes of $A_{1},\dots,A_{20}.$ Moreover,
we see that the surface $X$ is a double cover of $\PP^{2}$ branched
over a sextic curve $C_{6}$ which has an $\mathbf{a}_{3}$ singularity
$q$. The curves $A_{1}$, $A_{2}$, $A_{3}$ are contracted to $q$, the
curve $A_{4}$ is mapped onto a line that is tangent to the branch
of $C_{6}$ at $q$, and the curves $A_{k}$, $A_{k+8}$ with $k\in\{5,\dots,12\}$
are mapped to eight lines going through $q$  which are tangent
to the sextic at any other intersection point. For any subset $J=\{i,j,k,l\}$
of $\{5,\dots,11\}$ of order $4$ (there are $35$ such choices), let us define
\[
\begin{array}{l}
A_{J}=2A_{1}-A_{4}+\sum_{t\in J}A_{t},\hfill\\
B_{J}=4A_{2}+2A_{3}+3A_{4}-\sum_{t\in J}A_{t}.
\end{array}
\]
The classes $A_{J}$ and $B_{J}$ are the classes of the remaining
$70$ $(-2)$-curves $A_{21},\dots,A_{90}$. Moreover, we see that
\[
2D_{2}\equiv A_{J}+B_{J},\quad\forall J=\{i,j,k,l\}\subset\{5,\dots,11\},\,\,\#\{i,j,k,l\}=4,
\]
and therefore there exist $35$ conics that are $6$-tangent to $C_{6}$. 

Let $J$, $J'$ be two subsets of order $4$ of $\{5,\dots,11\}$. The
configuration of the curves $A_{J}$, $A_{J'}$, $B_{J}$, $B_{J'}$ is as follows:
\[
\begin{array}{l}
A_{J} \cdot A_{J'}=B_{J}\cdot B_{J'}=6-2\#(J\cap J'),\\
A_{J}\cdot B_{J'}=-2+2\#(J\cap J').
\end{array}
\]

\subsubsection{Second involution}

The divisor $D_{2}'=2A_{2}+A_{3}+2A_{4}$ is nef, of square $2$,
with $D_{2} \cdot A_{1}=2$, $D_{2} \cdot A_{j}=1$ for $j\in\{2,13,\dots,20\}$,
$D_{2} \cdot A_{j}=0$ for $j\in\{3,\dots,12\}$ and $D_{2} \cdot A_{j}=4$ for
$j\geq21$. We have 
\[
D_{2}'=2F+A_{3},
\]
where $F=A_{2}+A_{4}$ is a fiber of an elliptic fibration  such that
$FA_{3}=1$; thus the linear system $|D_{2}'|$ has base points. Let
$D_{8}=2D_{2}'$. The linear system $|D_{8}|$ is base-point free,
and it is hyperelliptic since $D_{8}F=2$. One can check easily that
\begin{equation}
\begin{array}{l}
D_{8}\equiv2A_{1}+\sum_{j=5}^{12}A_{j},\hfill\\
D_{8}\equiv4(A_{k}+A_{k+8})+2A_{3},\quad k\in\{5,\dots,12\},\\
D_{8}\equiv B_{J}+A_{4}+\sum_{t\in J}A_{t},\hfill\\
D_{8}\equiv A_{J}+A_{4}+A_{12}+\sum_{t\in J^{c}}A_{t},\hfill
\end{array}\label{eq:equivRel90curves}
\end{equation}
where $J=\{i,j,k,l\}$ is a subset of order $4$ of $\{5,\dots,11\}$
and $J^{c}$ is its complement. 

By \cite[Equation~(5.9.1)]{SaintDonat}, the associated map $\varphi_{D_{8}}\colon X\to\PP^{5}$
has image a cone over a rational normal curve in $\PP^{4}$; it factorizes
through a surjective map $\varphi'\colon X\to\mathbf{F}_{4}$, where $\mathbf{F}_{4}$
is the Hirzebruch surface with a section $s$ such that $s^{4}=-4$.
The section $s$ is mapped to the vertex of the cone by the map $\mathbf{F}_{4}\to\PP^{5}$.
Let $f$ denote a fiber of the unique fibration $\mathbf{F}_{4}\to\PP^{1}$.
By \cite[Equation~(5.9.1)]{SaintDonat}, the branch locus of $\varphi'$ is
the union of $s$ and a curve $C$ such that $Cs=0$ and $C\in|3(s+4f)|$
(so that $b+C\in|-2K_{\mathbf{F}_{4}}|$). We have 
\[
|D_{8}|=\varphi'^{*}|4f+s|, 
\]
and therefore from the equivalence relations in \eqref{eq:equivRel90curves},
we get the following:
\begin{itemize}
\item The curve $A_{3}$ is in the ramification locus; the image by $\varphi'$
of the curve $A_{3}$ is the section $s$. 
\item The curve $A_{1}$ is in the ramification locus. Since $A_{1} \cdot A_{3}=0$
and $A_{1}$ is a section, the image $C_{1}$ of $A_{1}$ is in the
linear system $|s+4f|$. Let $B'$ be the curve $B'\in|2s+8f|$ such
that the branch locus of the double cover $\varphi\colon X\to\mathbf{F}_{4}$
is 
\[
B=s+C_{1}+B'\in|4s+12f|.
\]
The singular points of the branch locus are nodes $p_{4},\dots,p_{12}$.
The eight points $p_{5},\dots,p_{12}$ are the intersection points
of $C_{1}$ and $B'$; the image by $\varphi'$ of $A_{5},\dots,A_{12}$
are the eight points $p_{5},\dots,p_{12}$. The curve $B'$ has a node
at the point $p_{4}$ onto which the curve $A_{4}$ is contracted by $\varphi'$.
The curves $A_{2},A_{13},\dots,A_{20}$ are sent, respectively, to the
fibers passing through $p_{4},p_{5},\dots,p_{12}$. 
\item The image of the curve $A_{J}$ ($J=\{i,j,k,l\}$),  is a curve in $|4f+s|$
passing through the points $p_{t}$, for $t\in\{4,m,n,o,12\}$, where
$\{i,j,k,l,m,n,o\}=\{5,\dots,11\}$. 
\item The image of the curve $B_{J}$ ($J=\{i,j,k,l\}$),  is a curve in $|4f+s|$
passing through the points $p_{t}$, for $t\in\{4,i,j,k,l\}$.
\end{itemize}

\begin{rem}
By \cite{Kondo}, the automorphism group of $X$ is $(\ZZ/2\ZZ)^{2}$.
It is generated by the involutions associated to the two double covers
we described.
\end{rem}

\section{Rank 12 lattices }

\subsection{The lattice $\boldsymbol{U\oplus\mathbf{E}_{8}\oplus\mathbf{A}_{1}^{\oplus2}}$}

The K3 surface $X$ contains $14$ $(-2)$-curves $A_{1},\dots,A_{14}$; 
their dual graph is 

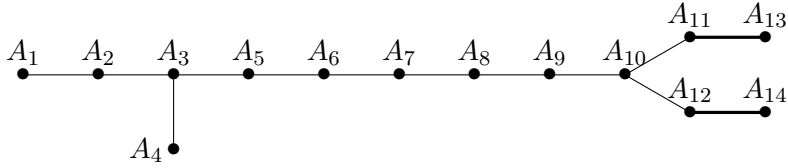
\begin{center}
\begin{tikzpicture}[scale=1]

\draw (0,0) -- (8,0);
\draw (8,0) --  (8+0.866,0.5); 
\draw [very thick] (8+0.866,0.5) -- (9+0.866,0.5);
\draw (8,0) --  (8+0.866,-0.5); 
\draw [very thick] (8+0.866,-0.5) -- (9+0.866,-0.5);

\draw (2,0) -- (2,-1);

\draw (0,0) node {$\bullet$};
\draw (1,0) node {$\bullet$};
\draw (2,0) node {$\bullet$};
\draw (3,0) node {$\bullet$};
\draw (4,0) node {$\bullet$};
\draw (5,0) node {$\bullet$};
\draw (6,0) node {$\bullet$};
\draw (7,0) node {$\bullet$};
\draw (8,0) node {$\bullet$};
\draw (8+0.866,0.5) node {$\bullet$};
\draw (9+0.866,0.5) node {$\bullet$};
\draw (8+0.866,-0.5) node {$\bullet$};
\draw (9+0.866,-0.5) node {$\bullet$};
\draw (2,-1) node {$\bullet$};

\draw (2,-1) node [left]{$A_{4}$};
\draw (0,0) node [above]{$A_{1}$};
\draw (1,0) node [above]{$A_{2}$};
\draw (2,0) node [above]{$A_{3}$};
\draw (3,0) node [above]{$A_{5}$};
\draw (4,0) node [above]{$A_{6}$};
\draw (5,0) node [above]{$A_{7}$};
\draw (6,0) node [above]{$A_{8}$};
\draw (7,0) node [above]{$A_{9}$};
\draw (8,0) node [above]{$A_{10}$};
\draw (8+0.866,0.5) node [above]{$A_{11}$};
\draw (9+0.866,0.5) node [above]{$A_{13}$};
\draw (8+0.866,-0.5) node [above]{$A_{12}$};
\draw (9+0.866,-0.5) node [above]{$A_{14}$};

\end{tikzpicture}
\end{center} 

The curves $A_{1},\dots,A_{12}$ generate the N\'eron--Severi lattice.
In that base, the divisor 
\[
D_{456}=(20,41,63,31,55,48,42,37,33,30,14,14)
\]
is ample, of square $456$, with $D_{456} \cdot A_{j}=1$ for $j\leq12$ and 
$D_{456} \cdot A_{j}=28$ for $j\in\{13,14\}$. 

\begin{rem}
The divisors
\[
\begin{array}{l}
F_{1}=2A_{1}+4A_{2}+6A_{3}+3A_{4}+5A_{5}+4A_{6}+3A_{7}+2A_{8}+A_{9},\\
F_{2}=A_{11}+A_{13},\,\\F_{3}=A_{12}+A_{14}
\end{array}
\]
are fibers of an elliptic fibration. The number of curves in $F_{1}$
counted with multiplicities is $30$; thus it is impossible to find
an ample divisor $D$ such that $D(A_{11}+A_{13})<30$. 
\end{rem}

By Theorem \ref{thm:SaintDonat-2}, case i) a), we have the following. 

\begin{prop}
The linear system $|4F_{1}+2A_{10}|$ defines a morphism $\varphi\colon X\to\mathbf{F}_{4}$
branched over the unique section $s$ with $s^2=-4$ and a curve $B\in|3s+12f|$. The curve
$B$ has one $\mathbf{e}_{8}$ singularity $p$ and two nodes $q$, $q^{'}$.
The pull-backs of the fibers through $p$, $q$, $q'$ are the fibers $F_{1}$, $F_{2}$, $F_{3}$.
\end{prop}

We can also construct this surface as follows. In the basis $A_{1},\dots,A_{12}$, the divisor 
\[
D_{2}=(2,4,6,3,5,4,3,2,2,2,1,1)
\]
is nef, base-point free, of square $2$, with $D_{2} \cdot A_{8}=1$, $D_{2} \cdot A_{13}=D_{2} \cdot A_{14}=2$
and $D_{2} \cdot A_{j}=0$ for $j\notin\{8,13,14\}$. Therefore, the K3 surface
$X$ is the double cover of $\PP^{2}$ branched over a sextic curve
which has an $\mathbf{e}_{7}$ singularity and a $\mathbf{d}_{4}$
singularity. The map $X\to\PP^{2}$ is ramified over $A_{8}$. Since
\[
D_{2}\equiv A_{9}+2A_{10}+2A_{11}+A_{12}+A_{13}\equiv A_{9}+2A_{10}+A_{11}+2A_{12}+A_{14},
\]
we see that the images of $A_{13}$ and $A_{14}$ are the two tangent
lines through the two remaining branches of the $\mathbf{d}_{4}$
singularities. The sextic curve is the union of a quintic and the
line $L$. The quintic has a nodal and a cusp singularity, so that
with the line $L$, they become an $\mathbf{e}_{7}$ singularity and
a $\mathbf{d}_{4}$ singularity.

\subsection{The lattice $\boldsymbol{U\oplus\mathbf{D}_{8}\oplus\mathbf{A}_{1}^{\oplus2}}$}

The K3 surface $X$ contains $19$ $(-2)$-curves $A_{1},\dots,A_{19}$; 
their dual graph is 

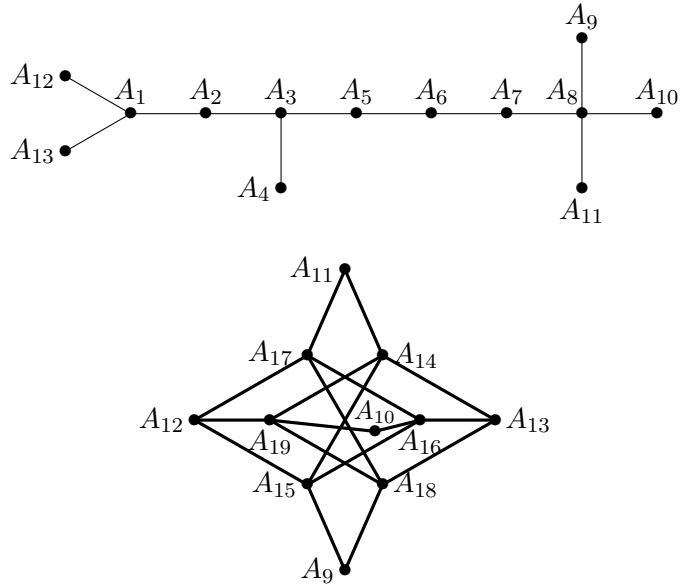
\begin{center}
\begin{tikzpicture}[scale=1]

\draw (-1,0) -- (6,0);
\draw (-1,0) -- (-1-0.866,0.5);
\draw (-1,0) -- (-1-0.866,-0.5);
\draw (5,-1) --  (5,1); 

\draw (1,0) -- (1,-1);

\draw (-1-0.866,0.5) node {$\bullet$};
\draw (-1-0.866,-0.5) node {$\bullet$};
\draw (-1,0) node {$\bullet$};
\draw (0,0) node {$\bullet$};
\draw (1,0) node {$\bullet$};
\draw (2,0) node {$\bullet$};
\draw (3,0) node {$\bullet$};
\draw (4,0) node {$\bullet$};
\draw (5,0) node {$\bullet$};
\draw (6,0) node {$\bullet$};
\draw (5,1) node {$\bullet$};
\draw (5,-1) node {$\bullet$};
\draw (1,-1) node {$\bullet$};

\draw (-1-0.866,0.5) node [left]{$A_{12}$};
\draw (-1-0.866,-0.5) node [left] {$A_{13}$};
\draw (-1,0) node [above]{$A_{1}$};
\draw (1,-1) node [left]{$A_{4}$};
\draw (0,0) node [above]{$A_{2}$};
\draw (1,0) node [above]{$A_{3}$};
\draw (2,0) node [above]{$A_{5}$};
\draw (3,0) node [above]{$A_{6}$};
\draw (4,0) node [above]{$A_{7}$};
\draw (5.1,0) node [above left]{$A_{8}$};
\draw (5,1) node [above]{$A_{9}$};
\draw (6,0) node [above]{$A_{10}$};
\draw (5,-1) node [below]{$A_{11}$};

\end{tikzpicture}
\end{center} 

\begin{center}
\begin{tikzpicture}[scale=1]

\draw [very thick] (1,0) -- (-0.5,0.86);
\draw [very thick] (-0.5,0.86) -- (0.5,-0.86);
\draw [very thick] (-1,0) -- (0.5,-0.86);
\draw [very thick] (-1,0) -- (0.5,0.86);
\draw [very thick] (0.5,0.86) -- (-0.5,-0.86);
\draw [very thick] (1,0) -- (-0.5,-0.86);
\draw [very thick] (-2,0) -- (-1,0);
\draw [very thick] (-1,0) -- (0.4,-0.15);
\draw [very thick] (1,0) -- (0.4,-0.15);
\draw [very thick] (1,0) -- (2,0);
\draw [very thick] (-2,0) -- (-0.5,0.86);
\draw [very thick] (2,0) -- (0.5,0.86);
\draw [very thick] (-2,0) -- (-0.5,-0.86);
\draw [very thick] (2,0) -- (0.5,-0.86);
\draw [very thick] (0,2) -- (0.5,0.86);
\draw [very thick] (0,2) -- (-0.5,0.86);
\draw [very thick] (0,-2) -- (0.5,-0.86);
\draw [very thick] (0,-2) -- (-0.5,-0.86);

\draw (0.4,-0.15) node {$\bullet$};
\draw (0,-2) node {$\bullet$};
\draw (0,2) node {$\bullet$};
\draw (-2,0) node {$\bullet$};
\draw (2,0) node {$\bullet$};
\draw (1,0) node {$\bullet$};
\draw (0.5,0.86) node {$\bullet$};
\draw (-1,0) node {$\bullet$};
\draw (-0.5,0.86) node {$\bullet$};
\draw (-0.5,-0.86) node {$\bullet$};
\draw (0.5,-0.86)  node {$\bullet$};

\draw (0.4,-0.17) node [above]{\small $A_{10}$};
\draw (1,0) node  [below]{$A_{16}$};
\draw (0.5,0.86) node [right]{$A_{14}$};
\draw (-1,0) node  [below]{$A_{19}$};
\draw (-0.55,0.89) node  [left]{$A_{17}$};
\draw (-0.5,-0.86) node  [left]{$A_{15}$};
\draw (0.5,-0.86) node  [right]{$A_{18}$};
\draw (0,-2) node [left]{$A_{9}$};
\draw (0,2) node [left]{$A_{11}$}; 
\draw (-2,0) node [left]{$A_{12}$};
\draw (2,0) node [right]{$A_{13}$};

\end{tikzpicture}
\end{center} 

The curves $A_{1},\dots,A_{12}$ generate the N\'eron--Severi lattice, 
and in that basis the divisor 
\[
D_{136}=(2,5,9,4,10,12,15,19,9,9,6,0)
\]
 is ample, of square $136$, with $D_{136} \cdot A_{j}=1$ for $j\leq10$,
$D_{136} \cdot A_{j}=7,2,2,12,18,18,12,18,18$ for $j=11,\dots,19$.

The divisor 
\[
D_{2}=(0,1,2,1,2,2,2,2,1,1,1,0)
\]
is base-point free, of square $2$, with $D_{2} \cdot A_{1}=D_{2} \cdot A_{8}=1$,
$D_{2} \cdot A_{j}=0$ for $j\in\{2,\dots,7,9,\dots,13\}$ and $D_{2} \cdot A_{j}=2$
for $j\in\{14,\dots,19\}$. Moreover, 
\[
\begin{array}{l}
D_{2}\equiv A_{11}+A_{12}+A_{17}\equiv A_{11}+A_{13}+A_{14},\\
D_{2}\equiv A_{10}+A_{12}+A_{19}\equiv A_{10}+A_{13}+A_{16},\\
D_{2}\equiv A_{9}+A_{12}+A_{15}\equiv A_{9}+A_{13}+A_{18},\\
D_{2}\equiv2A_{1}+3A_{2}+4A_{3}+2A_{4}+3A_{5}+2A_{6}+A_{7}+A_{12}+A_{13}.
\end{array}
\]
By using the linear system $|D_{2}|$, we obtain the following. 

\begin{prop}
The K3 surface $X$ is a double cover of $\,\PP^{2}$ ramified over
a sextic curve $C_{6}$; the curves $A_{1}$ and $A_{8}$ are in the
ramification locus, and their images are two lines. We denote by $Q_{4}$
the residual quartic. The sextic curve has a $\mathbf{d_{6}}$ singularity 
$($the curves $A_{2},\dots,A_{7}$ are contracted to that singularity$)$
and five nodal singularities $p_{j}$ to which the curves $A_{j}$
are contracted, for $j\in\{9,\dots,13\}$. 
\end{prop}

The image of $A_{15}$ (resp.\ $A_{17}$, $A_{19}$) is a line
through the node $p_{12}$ and the node $p_{9}$ (resp.\  $p_{11}$, $p_{10}$)
which is tangent to $Q_{4}$. The image of $A_{14}$ (resp.\
$A_{16}$, $A_{18}$) is a line through the node $p_{13}$ and the node $p_{11}$
(resp.\ $p_{10}$, $p_{9}$) which is tangent to $Q_{4}$.

\subsection{The lattice $\boldsymbol{U\oplus\mathbf{D}_{4}^{\oplus2}\oplus\mathbf{A}_{1}^{\oplus2}}$}

The lattice $U\oplus\mathbf{D}_{4}^{\oplus2}\oplus\mathbf{A}_{1}^{\oplus2}$
is isometric to $U\oplus{\bf D}_{6}\oplus\mathbf{A}_{1}^{\oplus4}$; 
see \cite{Kondo}. 

\subsubsection{First involution}

The K3 surface $X$ contains $59$ $(-2)$-curves. The configuration
of the first $19$  $\cu$-curves $A_{1},\dots,A_{19}$ is as follows:

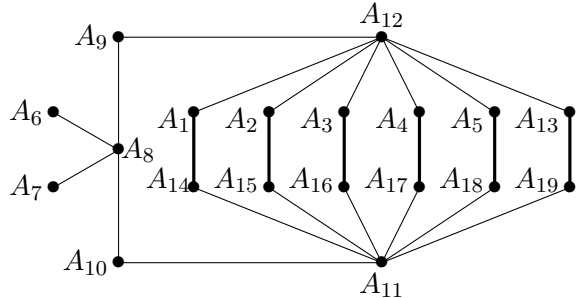
\begin{center}
\begin{tikzpicture}[scale=1]

\draw (0,1.5) -- (2.5,0.5);
\draw (0,1.5) -- (1.5,0.5);
\draw (0,1.5) -- (0.5,0.5);
\draw (0,1.5) -- (-0.5,0.5);
\draw (0,1.5) -- (-1.5,0.5);
\draw (0,1.5) -- (-2.5,0.5);

\draw [very thick] (2.5,0.5) -- (2.5,-0.5);
\draw [very thick] (1.5,0.5) -- (1.5,-0.5);
\draw [very thick] (0.5,0.5) -- (0.5,-0.5);
\draw [very thick] (-2.5,0.5) -- (-2.5,-0.5);
\draw [very thick] (-1.5,0.5) -- (-1.5,-0.5);
\draw [very thick] (-0.5,0.5) -- (-0.5,-0.5);

\draw (-3.5,0) -- (-3.5-0.866,0.5);
\draw (-3.5,0) -- (-3.5-0.866,-0.5);
\draw (0,1.5) -- (-3.5,1.5);
\draw (-3.5,1.5) -- (-3.5,-1.5);
\draw (0,-1.5) -- (-3.5,-1.5);
\draw (0,-1.5) -- (2.5,-0.5);
\draw (0,-1.5) -- (1.5,-0.5);
\draw (0,-1.5) -- (0.5,-0.5);
\draw (0,-1.5) -- (-0.5,-0.5);
\draw (0,-1.5) -- (-1.5,-0.5);
\draw (0,-1.5) -- (-2.5,-0.5);

\draw (-3.5-0.866,0.5) node {$\bullet$};
\draw (-3.5-0.866,-0.5) node {$\bullet$};
\draw (0,1.5) node {$\bullet$};
\draw (-3.5,1.5) node  {$\bullet$};
\draw (2.5,0.5) node  {$\bullet$};
\draw (1.5,0.5) node  {$\bullet$};
\draw (0.5,0.5) node  {$\bullet$};
\draw (-0.5,0.5) node  {$\bullet$};
\draw (-1.5,0.5) node  {$\bullet$};
\draw (-2.5,0.5) node  {$\bullet$};
\draw (-3.5,-1.5) node  {$\bullet$};

\draw (0,1.5) node [above]{$A_{ 12}$};
\draw (2.5,0.4) node [left]{$A_{ 13}$};
\draw (1.5,0.4) node [left]{$A_{ 5}$};
\draw (0.5,0.4) node [left]{$A_{ 4}$};
\draw (-0.5,0.4) node [left]{$A_{ 3}$};
\draw (-1.5,0.4) node [left]{$A_{ 2}$};
\draw (-2.4,0.4) node [left]{$A_{ 1}$};
\draw (-3.5,-1.5) node [left]{$A_{ 10}$};

\draw (0,-1.5) node {$\bullet$};
\draw (2.5,-0.5) node  {$\bullet$};
\draw (1.5,-0.5) node  {$\bullet$};
\draw (0.5,-0.5) node  {$\bullet$};
\draw (-0.5,-0.5) node  {$\bullet$};
\draw (-1.5,-0.5) node  {$\bullet$};
\draw (-2.5,-0.5) node  {$\bullet$};
\draw (-3.5,0) node  {$\bullet$};

\draw (0,-1.5) node [below]{$A_{ 11}$};
\draw (2.5,-0.4) node [left]{$A_{ 19}$};
\draw (1.5,-0.4) node [left]{$A_{ 18}$};
\draw (0.5,-0.4) node [left]{$A_{ 17}$};
\draw (-0.5,-0.4) node [left]{$A_{ 16}$};
\draw (-1.5,-0.4) node [left]{$A_{ 15}$};
\draw (-2.4,-0.4) node [left]{$A_{ 14}$};
\draw (-3.5,1.5) node [left]{$A_{ 9}$};

\draw (-3.5-0.866,0.5) node [left]{$A_{ 6}$};
\draw (-3.5-0.866,-0.5) node [left]{$A_{ 7}$};
\draw (-3.6,0) node [right]{$A_{ 8}$};


\end{tikzpicture}
\end{center} 

The curves $A_{1},\dots,A_{12}$ generate  the N\'eron--Severi lattice; 
in that basis, the divisor 
\[
D_{40}=(3,3,3,0,0,2,2,6,6,3,1,7)
\]
is ample, of square $40$, with $D_{40} \cdot A_{6}=D_{40} \cdot A_{7}=2$, $D_{40} \cdot A_{j}=1$
for $j\in\{1,2,3,8,\dots,12,17,18,19\}$ and $D_{40} \cdot A_{j}=7$ for
$j\in\{4,5,13,\dots,16\}$. 

The divisor 
\[
D_{2}=(0,0,0,0,0,1,1,3,2,2,1,1)
\]
is nef, base-point free, of square $2$, with $D_{2} \cdot A_{j}=0$ for
$j\in\{8,\dots,12\}$, $D_{2} \cdot A_{j}=1$ for $j\in\{1,\dots,7,13,\dots,19\}$
and $D_{2} \cdot A_{j}=2$ for $j\geq20$. The K3 surface $X$ is a double
cover of the plane branched over a sextic curve $C_{6}$ which has
an $\mathbf{a}_{5}$ singularity $q$; the curves $A_{8},\dots,A_{12}$
are contracted to $q$. We have 
\[
D_{2}\equiv A_{8}+A_{9}+A_{10}+A_{11}+A_{12}+F, 
\]
where 
\[\begin{array}{l}
F\equiv A_{6}+A_{7}+2A_{8}+A_{9}+A_{10}\\
\hphantom{F}\equiv A_{1}+A_{14}\equiv A_{2}+A_{15}\equiv A_{3}+A_{16}\\
\hphantom{F}\equiv A_{4}+A_{17}\equiv A_{5}+A_{18}\equiv A_{13}+A_{19}; 
\end{array}\]
thus there exist seven lines $L$, $L_{1},\dots,L_{6}$ through the $\mathbf{a}_{5}$
singularity such that $L$ intersects $C_{6}$ in that point only (so 
$L$ is the tangent to the branch curve of the singularity; the strict
transform of $L$ is $A_{6}+A_{7}$) and that the lines $L_{i}$ have even
intersection multiplicities at their other intersection points with
the sextic.

The $\cu$-curves on $X$ are $A_{1},\dots,A_{19}$ and the curves
\begin{align*}
B_{ij}^{(1)}&=-A_{r}-A_{s}-A_{t}+3A_{6}+2A_{7}+6A_{8}+3A_{9}+4A_{10}+2A_{11},\\[-1ex]
B_{ij}^{(2)}&=-A_{r}-A_{s}-A_{t}+2A_{6}+3A_{7}+6A_{8}+3A_{9}+4A_{10}+2A_{11},\\[-1ex]
C_{ij}^{(1)}&=-A_{i}-A_{j}-A_{13}+2A_{6}+3A_{7}+6A_{8}+3A_{9}+4A_{10}+2A_{11},\\[-1ex]
C_{ij}^{(2)}&=-A_{i}-A_{j}-A_{13}+3A_{6}+2A_{7}+6A_{8}+3A_{9}+4A_{10}+2A_{11},
\end{align*}
where $\{i,j,r,s,t\}=\{1,2,3,4,5\}$. Using the relation
\[
A_{13}=(-1,-1,-1,-1,-1,3,3,6,2,4,2,-2),
\]
one can check that $B_{ij}^{(a)}+C_{ij}^{(a)}=2D_{2}$; thus the curves
$B_{ij}^{(a)}$, $C_{ij}^{(a)}$ are pull-backs of conics which are $6$-tangent
to the sextic curve $C_{6}$.

\subsubsection{Second involution}

The divisor $F'=A_{6}+A_{7}+2A_{8}+A_{9}+A_{10}$ is a fiber of an
elliptic fibration. The divisor $D_{2}'=2F'+A_{11}$ is nef, of square
$2$, with
\[
\begin{array}{l}
D_{2} \cdot A_{j}=0\quad\text{for }j\in\{1,2,3,4,5,6,7,8,9,11,13\},\\
D_{2} \cdot A_{j}=1\quad\text{for }j\in\{10,14,15,16,17,18,19\},\hfill\\
D_{2} \cdot A_{j}=4\quad\text{for }j\geq20\hfill
\end{array}
\]
and $D'_{2} \cdot A_{12}=2$. The linear system $|D_{2}'|$ has base points.
The linear system $|D_{8}|$ (where $D_{8}=2D_{2}'$) is base-point
free, and it is hyperelliptic since $D_{8}F'=2$. One can check easily
that for $i$, $j$, $r$, $s$, $t$ such that $\{i,j,r,s,t\}=\{1,\dots,5\}$, one
has 
\[
\begin{array}{l}
D_{8}\equiv B_{ij}^{(1)}+A_{r}+A_{s}+A_{t}+A_{6}+2A_{7}+2A_{8}+A_{9},\,\,\,\\
D_{8}\equiv B_{ij}^{(2)}+A_{r}+A_{s}+A_{t}+2A_{6}+A_{7}+2A_{8}+A_{9},\,\,\,\\
D_{8}\equiv C_{ij}^{(1)}+A_{i}+A_{j}+A_{13}+2A_{6}+A_{7}+2A_{8}+A_{9},\,\\
D_{8}\equiv C_{ij}^{(1)}+A_{i}+A_{j}+A_{13}+A_{6}+2A_{7}+2A_{8}+A_{9}.
\end{array}
\]
Moreover, 
\[
\begin{array}{l}
\hphantom{2}D_{8}\equiv3A_{10}+A_{19}+A_{13}+3A_{6}+3A_{7}+6A_{8}+3A_{9}+2A_{11},\hfill\\
2D_{8}\equiv2A_{10}+\sum_{k=14}^{19}  A_{k}+\sum_{k=1}^{5}A_{k}+2(A_{6}+A_{7}+2A_{8}+A_{9}+2A_{11})+A_{13},\\
\hphantom{2}D_{8}\equiv2A_{12}+\sum_{k=1}^{7}A_{k}+2A_{8}+2A_{9}+A_{13}.\hfill
\end{array}
\]
By \cite[Equation~(5.9.1)]{SaintDonat}, the map $\varphi_{D_{8}}\colon X\to\PP^{5}$ associated to $|D_{8}|$ has image a cone over a rational normal curve in $\PP^{4}$; it factorizes through a surjective map $\varphi'\colon X\to\mathbf{F}_{4}$, where $\mathbf{F}_{4}$ is the Hirzebruch surface with a section $s$ such that $s^{4}=-4$. The section $s$ is mapped to the vertex of the cone by the map $\mathbf{F}_{4}\to\PP^{5}$. Let $f$ denote a fiber of the unique fibration $\mathbf{F}_{4}\to\PP^{1}$. By \cite[Equation~(5.9.1)]{SaintDonat}, the branch locus of $\varphi'$ is the union of $s$ and a curve $C$ such that $Cs=0$ and $C\in|3(s+4f)|$ (so that $s+C\in|-2K_{\mathbf{F}_{4}}|$).  We have
\[
|D_{8}|=\varphi'^{*}|4f+s|, 
\]
and therefore from the above equivalence relations, we get the following:
\pagebreak
\begin{itemize}
\item The image of curve $A_{11}$ by $\varphi'$ is the section $s$. 
\item The branch curve is the union of three components: $s$, $B'$ and
$C_{12}$, where $B'\in|2(s+4f)|$, $C_{12}\in|s+4f|$. The curve
$B'$ has a node $q$, and the curves $C_{12}$ and $B'$ meet at $q$
and at six other points $p_{1},\dots,p_{5},p_{13}$. The singularity
at $q$ of $B'+C_{12}$ has type $\mathbf{d}_{4}$; the other singular
points are nodes. 
\item The curves $A_{6},\dots,A_{9}$ are mapped by $\varphi'$ to $q$. 
\item The curves $A_{1},\dots,A_{5}, A_{13}$ are sent to $p_{1},\dots,p_{5},p_{13}$. 
\item The curve $A_{12}$ is part of the ramification locus; its image is
$C_{12}$.
\item The curve $A_{11}$ is part of the ramification locus; its image is
$s$.
\item The curves $A_{10}$, $A_{14},\dots,A_{19}$ are mapped to the fibers
through $q, p_{1},\dots,p_{5}, p_{13}$. 
\item The images of the curves $B_{ij}^{(1)}$ and $B_{ij}^{(2)}$ are curves in $|s+4f|$ passing through $p_{r}$,  $p_{s}$, $p_{t}$ and through $q$ with certain
tangency properties at the branches of the singularity $q$.
\item The images of the curves $C_{ij}^{(1)}$ and $C_{ij}^{(2)}$ are curves in $|s+4f|$ passing through $p_{i}$, $p_{j}$, $p_{13}$ and through $q$ with certain
tangency properties at the branches.
\end{itemize}

The $10$ curves $B_{ij}^{(1)}$ (resp.\  $B_{ij}^{(2)}$, $C_{ij}^{(1)}$, $C_{ij}^{(2)}$)
for $\{i,j\}\subset\{1,2,3,4,5\}$ have the configuration of the Petersen
graph, with weight $2$ on the edges. Moreover, for $1\leq i<j\leq5$
and $1\leq s<t\leq5$, the intersections between the four types of
curves are as follows:
\begin{align*}
B_{ij}^{(1)}\cdot B_{st}^{(2)}&=C_{ij}^{(1)}\cdot C_{st}^{(2)}=2+B_{ij}^{(1)}\cdot B_{st}^{(1)},\\
B_{ij}^{(1)}\cdot C_{st}^{(2)}&=B_{ij}^{(2)}\cdot C_{st}^{(1)}=2-B_{ij}^{(1)}\cdot B_{st}^{(1)},\\
B_{ij}^{(1)}\cdot C_{st}^{(1)}&=B_{ij}^{(2)}\cdot C_{st}^{(2)}=4-B_{ij}^{(1)}\cdot B_{st}^{(1)}.
\end{align*}

\begin{rem}
In \cite[Remark 1]{Kondo2}, Kondo constructed specific surfaces with
N\'eron--Severi lattice isometric to $U\oplus\mathbf{D}_{4}^{\oplus2}\oplus\mathbf{A}_{1}^{\oplus2}$
as follows. Let $C$ be a smooth curve of genus $2$ and $q$ be a point on $C$.
The linear system $|K_{C}+2q|$ gives a plane quartic curve with a
cusp. The minimal resolution $Y$ of the cyclic degree $4$ cover
of $\PP^{2}$ branched over that curve has an elliptic fibration  (obtained
by blowing up the cusp) with a $\tilde{\mathbf{D}}_{4}$ and six $\tilde{\mathbf{A}}_{1}$
fibers. The automorphism group of such a surface is larger than the
 general  one.
\end{rem}

\subsection{The lattice $\boldsymbol{U\oplus\mathbf{A}_{2}\oplus\mathbf{E}_{8}}$\label{subsec:The-lattice124-UA2E8}}

The K3 surface $X$ contains $13$ $(-2)$-curves $A_{1},\dots,A_{13}$; 
their configuration is as follows:

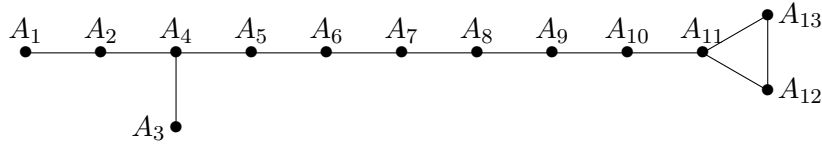
\begin{center}
\begin{tikzpicture}[scale=1]

\draw (0,0) -- (9,0);
\draw (9,0) --  (9+0.866,0.5); 
\draw  (9+0.866,0.5) -- (9+0.866,-0.5);
\draw (9,0) --  (9+0.866,-0.5); 

\draw (2,0) -- (2,-1);

\draw (0,0) node {$\bullet$};
\draw (1,0) node {$\bullet$};
\draw (2,0) node {$\bullet$};
\draw (3,0) node {$\bullet$};
\draw (4,0) node {$\bullet$};
\draw (5,0) node {$\bullet$};
\draw (6,0) node {$\bullet$};
\draw (7,0) node {$\bullet$};
\draw (8,0) node {$\bullet$};
\draw (9,0) node {$\bullet$};
\draw (9+0.866,0.5) node {$\bullet$};
\draw (9+0.866,-0.5) node {$\bullet$};
\draw (2,-1) node {$\bullet$};

\draw (2,-1) node [left]{$A_{3}$};
\draw (0,0) node [above]{$A_{1}$};
\draw (1,0) node [above]{$A_{2}$};
\draw (2,0) node [above]{$A_{4}$};
\draw (3,0) node [above]{$A_{5}$};
\draw (4,0) node [above]{$A_{6}$};
\draw (5,0) node [above]{$A_{7}$};
\draw (6,0) node [above]{$A_{8}$};
\draw (7,0) node [above]{$A_{9}$};
\draw (8,0) node [above]{$A_{10}$};
\draw (9,0) node [above]{$A_{11}$};
\draw (9+0.866,0.5) node [right]{$A_{13}$};
\draw (9+0.866,-0.5) node [right]{$A_{12}$};

\end{tikzpicture}
\end{center} 

The curves $A_{1},\dots,A_{12}$ generate  the N\'eron--Severi lattice;
in that basis, the divisor 
\[
D_{698}=(38,77,58,117,100,84,69,55,42,30,19,9)
\]
 is ample, of square $698$, with $D_{698} \cdot A_{j}=1$ for $j\leq12$
and $D_{698} \cdot A_{13}=28$. The divisor 
\[
F_{1}=A_{11}+A_{12}+A_{13}
\]
 is a fiber of an elliptic fibration  with section $A_{10}$. That
fibration has another singular fiber $F_{2}$ of type $\tilde{\mathbf{E}_{8}}$.
By Theorem \ref{thm:SaintDonat-2}, case i) a), we have the following. 

\begin{prop}
The linear system $|4F_{1}+2A_{10}|$ defines a morphism $\varphi\colon X\to\mathbf{F}_{4}$
branched over the unique section $s$ with $s^2=-4$ and a curve $B\in|3s+12f|$. The curve
$B$ has one cusp $p$ and one $\mathbf{e}_{8}$ singularity $q$.
The pull-backs of the fibers through $p$, $q$ are the fibers $F_{1}$, $F_{2}$.
\end{prop}

We can also give another construction as a double cover of $\PP^{2}$:
The divisor 
\[
D_{2}=(2,5,4,8,7,6,5,4,3,2,1,0)
\]
 is nef, of square $2$, base-point free, with $D_{2} \cdot A_{1}=D_{2} \cdot A_{12}=D_{2} \cdot A_{13}=1$
and $D_{2} \cdot A_{j}=0$ for $j\in\{2,\dots,11\}$. The K3 surface is the
double cover of $\PP^{2}$ branched over a sextic curve with a $\mathbf{d}_{10}$
singularity. The curve $A_{1}$ is in the ramification locus; we denote
by $L$ its image and by $Q$ the residual quintic curve. The quintic
$Q$ has a node, and $L$ is tangent with multiplicity $4$ at a branch
of that node. We have 
\[
D_{2}\equiv A_{2}+A_{3}+2A_{4}+2A_{5}+2A_{6}+2A_{7}+2A_{8}+2A_{9}+2A_{10}+2A_{11}+A_{12}+A_{13}; 
\]
thus the image of $A_{12}$ and $A_{13}$ is the line which is the
tangent to the other branch of the node, which is moreover tangent
to $Q$ at another point. 

\section{Rank 13 lattices }

\subsection{The lattice $\boldsymbol{U\oplus\mathbf{E}_{8}\oplus\mathbf{A}_{1}^{\oplus3}}$}

The K3 surface $X$ contains $17$ $(-2)$-curves $A_{1},\dots,A_{17}$;
their configuration is as follows:

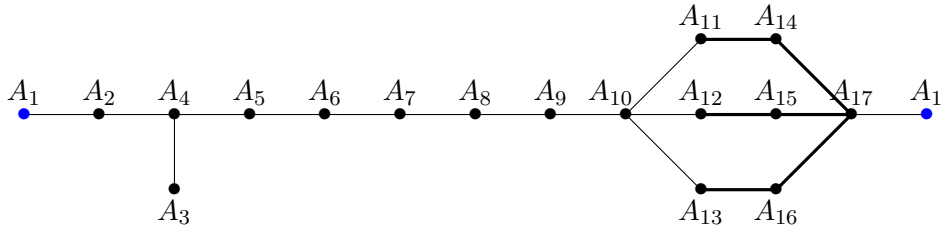
\begin{center}
\begin{tikzpicture}[scale=1]

\draw (0,0) -- (12,0);
\draw [very thick] (9,0) -- (11,0);
\draw  (9,1) -- (8,0);
\draw  (9,-1) --  (8,0); 
\draw [very thick]  (9,1) -- (10,1);
\draw [very thick] (9,-1) --  (10,-1); 
\draw [very thick]  (10,1) -- (11,0);
\draw [very thick] (10,-1) --  (11,0);

\draw (2,0) -- (2,-1);

\draw [color=blue] (0,0) node {$\bullet$};
\draw (1,0) node {$\bullet$};
\draw (2,0) node {$\bullet$};
\draw (3,0) node {$\bullet$};
\draw (4,0) node {$\bullet$};
\draw (5,0) node {$\bullet$};
\draw (6,0) node {$\bullet$};
\draw (7,0) node {$\bullet$};
\draw (8,0) node {$\bullet$};
\draw (9,0) node {$\bullet$};
\draw (10,0) node {$\bullet$};
\draw (11,0) node {$\bullet$};
\draw [color=blue] (12,0) node {$\bullet$};
\draw (2,-1) node {$\bullet$};
\draw (9,1) node {$\bullet$};
\draw (10,1) node {$\bullet$};
\draw (9,-1) node {$\bullet$};
\draw (10,-1) node {$\bullet$};

\draw (2,-1) node [below]{$A_{3}$};
\draw (0,0) node [above]{$A_{1}$};
\draw (12,0) node [above]{$A_{1}$};
\draw (1,0) node [above]{$A_{2}$};
\draw (2,0) node [above]{$A_{4}$};
\draw (3,0) node [above]{$A_{5}$};
\draw (4,0) node [above]{$A_{6}$};
\draw (5,0) node [above]{$A_{7}$};
\draw (6,0) node [above]{$A_{8}$};
\draw (7,0) node [above]{$A_{9}$};
\draw (7.8,0) node [above]{$A_{10}$};
\draw (9,0) node [above]{$A_{12}$};
\draw (10,0) node [above]{$A_{15}$};
\draw (11,0) node [above]{$A_{17}$};
\draw (9,1) node [above]{$A_{11}$};
\draw (10,1) node [above]{$A_{14}$};
\draw (9,-1) node [below]{$A_{13}$};
\draw (10,-1) node [below]{$A_{16}$};

\end{tikzpicture}
\end{center} 

The curves $A_{1},\dots,A_{13}$ generate  the N\'eron--Severi lattice;
in that basis, the divisor 
\[
D_{294}=(2,5,4,9,10,12,15,19,24,30,14,14,9)
\]
is ample, of square $294$, with $D_{294} \cdot A_{j}=1$ for $j\leq10$,
$D_{294} \cdot A_{j}=2$ for $j\in\{11,12,17\}$, $D_{294} \cdot A_{j}=28$ for
$j\in\{14,15\}$, $D_{294} \cdot A_{13}=12$ and $D_{294} \cdot A_{16}=18$. By considering
the elliptic fibration s with section $A_{10}$ and the fibers 
\[
F=A_{11}+A_{14},\,A_{12}+A_{15},\,A_{13}+A_{16}
\]
plus the fiber of type $\tilde{\mathbf{E}_{8}}$ supported on $A_{1},\dots,A_{9}$,
we get the following. 

\begin{prop}
The surface is a double cover of the Hirzebruch surface $\mathbf{F}_{4}$
branched over the negative section $s$ and a curve $B\in|3(s+4f)|$
with three nodes and an $\mathbf{e}_{8}$ singularity.
\end{prop}

The double cover $\eta$ is given by the linear system $|D_{8}|$, 
where $D_{8}=4F+2A_{10}$. One has 
\[
D_{8}=A_{17}+4A_{1}+7A_{2}+5A_{3}+10A_{4}+8A_{5}+6A_{6}+4A_{7}+2A_{8}+A_{14}+A_{15}+A_{16}; 
\]
thus the image by $\eta$ of the curve $A_{17}$ is a curve in $|s+4f|$
going through the three nodes plus the $\mathbf{e}_{8}$ singularity
and infinitely near points of it. 

We can also construct the surface $X$ as follows. The divisor
\[
D_{2}=(0,1,1,2,2,2,2,2,2,2,1,1,1)
\]
is nef, of square $2$, base-point free, with $D_{2} \cdot A_{j}=0$ for $j\in\{2,\dots,9,11,12,13,17\}.$
We have the relations 
\[
D_{2}\equiv A_{13}+A_{16}+A_{17}\equiv A_{12}+A_{15}+A_{17}\equiv A_{11}+A_{14}+A_{17}
\]
and 
\[
D_{2}\equiv2A_{1}+4A_{2}+3A_{3}+6A_{4}+5A_{5}+4A_{6}+3A_{7}+2A_{8}+A_{9}+A_{17}.
\]
The curves $A_{1}$ and $A_{10}$ are part of the ramification locus; 
their images are lines $L$, $L'$. Let us denote by $Q$ the residual
quartic. The images of the curves $A_{14}$, $A_{15}$, $A_{16}$ are lines
$L_{1}$, $L_{2}$, $L_{3}$ going through the images of $A_{17}$ and of $A_{11}$, $A_{12}$, $A_{13}$, 
respectively. The quartic curve is smooth and tangent to the line $L$
(the image of $A_{1}$) at the intersection point of $L$ and $L'$,
so that the singularity at $q$ is $\mathbf{d}_{8}$. The branch locus
has four nodes, three on the line $L'$ (the image of $A_{10}$) and
one on $L$ (the image of $A_{1}$). 

\subsection{The lattice $\boldsymbol{U\oplus\mathbf{D}_{8}\oplus\mathbf{A}_{1}^{\oplus3}}$}

The lattice $U\oplus\mathbf{D}_{8}\oplus\mathbf{A}_{1}^{\oplus3}$
is also isometric to $U\oplus\mathbf{E}_{7}\oplus\mathbf{A}_{1}^{\oplus4}$
and $U\oplus\mathbf{D}_{6}\oplus\mathbf{D}_{4}\oplus\mathbf{A}_{1}$; 
see \cite{Kondo}. 

\subsubsection{The first involution}

The K3 surface $X$ contains $39$ $(-2)$-curves $A_{1},\dots,A_{39}.$
The configuration of the curves $A_{1},\dots,A_{19}$ is as follows:

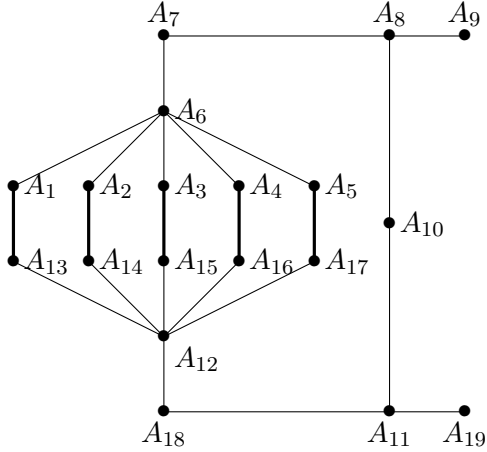
\begin{center}
\begin{tikzpicture}[scale=1]

\draw (0,0) -- (-2,1);
\draw (0,0) -- (-1,1);
\draw (0,-1) -- (0,4);
\draw (0,0) -- (1,1);
\draw (0,0) -- (2,1);

\draw (0,4) -- (4,4);
\draw (3,4) -- (3,-1);
\draw (0,-1) -- (4,-1);

\draw (0,3) -- (-2,2);
\draw (0,3) -- (-1,2);
\draw (0,3) -- (1,2);
\draw (0,3) -- (2,2);

\draw [very thick] (-2,2) -- (-2,1);
\draw [very thick] (-1,2) -- (-1,1);
\draw [very thick] (0,1) -- (0,2);
\draw [very thick] (1,2) -- (1,1);
\draw [very thick] (2,2) -- (2,1);

\draw (0,0) node {$\bullet$};
\draw (-2,1) node {$\bullet$};
\draw (-1,1) node {$\bullet$};
\draw (0,1) node {$\bullet$};
\draw (1,1) node {$\bullet$};
\draw (2,1) node {$\bullet$};
\draw (-2,2) node {$\bullet$};
\draw (-1,2) node {$\bullet$};
\draw (0,2) node {$\bullet$};
\draw (1,2) node {$\bullet$};
\draw (2,2) node {$\bullet$};
\draw (0,3) node {$\bullet$};
\draw (0,4) node {$\bullet$};
\draw (3,4) node {$\bullet$};
\draw (4,4) node {$\bullet$};
\draw (3,1.5) node {$\bullet$};
\draw (3,-1) node {$\bullet$};
\draw (4,-1) node {$\bullet$};
\draw (0,-1) node {$\bullet$};

\draw (0,0) node [below right]{$A_{12}$};
\draw (-2,1) node [right]{$A_{13}$};
\draw (-1,1) node [right]{$A_{14}$};
\draw (0,1) node [right]{$A_{15}$};
\draw (1,1) node [right]{$A_{16}$};
\draw (2,1) node [right]{$A_{17}$};
\draw (-2,2) node [right]{$A_{1}$};
\draw (-1,2) node [right]{$A_{2}$};
\draw (0,2) node [right]{$A_{3}$};
\draw (1,2) node [right]{$A_{4}$};
\draw (2,2) node [right]{$A_{5}$};
\draw (0,3) node [right]{$A_{6}$};
\draw (0,4) node [above]{$A_{7}$};
\draw (3,4) node [above]{$A_{8}$};
\draw (4,4) node [above]{$A_{9}$};
\draw (3,1.5) node [right]{$A_{10}$};
\draw (3,-1) node [below]{$A_{11}$};
\draw (4,-1) node [below]{$A_{19}$};
\draw (0,-1) node [below]{$A_{18}$};

\end{tikzpicture}
\end{center} 

The curves $A_{1},\dots,A_{13}$ generate the N\'eron--Severi lattice;
in that basis, the divisor 
\[
D_{86}=(2,1,2,5,5,11,8,6,2,3,1,0,1)
\]
is ample, of square $86$. The divisor 
\[
D_{2}=(-1,1,1,1,1,3,3,3,1,2,1,-1,-2)
\]
is nef, of square $2$, base-point free, with $D_{2} \cdot A_{j}=0$ for $j\in\{6,7,8,10,11,12,18\}$,
$D_{2} \cdot A_{j}=2$ for $j\geq20$ and else $D_{2} \cdot A_{j}=1$. For $k\in\{10,\dots,19\}$,
one has 
\[
2D_{2}\equiv A_{2k}+A_{2k+1},
\]
and using the equivalences from the elliptic fibration  with fiber
$A_{1}+A_{13}$, we get 
\[
\begin{array}{l}
D_{2}\equiv A_{13}+A_{1}+A_{6}+A_{7}+A_{8}+A_{10}+A_{11}+A_{12}+A_{18},\\
D_{2}\equiv A_{14}+A_{2}+A_{6}+A_{7}+A_{8}+A_{10}+A_{11}+A_{12}+A_{18},\\
D_{2}\equiv A_{15}+A_{3}+A_{6}+A_{7}+A_{8}+A_{10}+A_{11}+A_{12}+A_{18},\\
D_{2}\equiv A_{16}+A_{4}+A_{6}+A_{7}+A_{8}+A_{10}+A_{11}+A_{12}+A_{18},\\
D_{2}\equiv A_{17}+A_{5}+A_{6}+A_{7}+A_{8}+A_{10}+A_{11}+A_{12}+A_{18},\\
D_{2}\equiv A_{9}+A_{19}+A_{6}+2A_{7}+3A_{8}+3A_{10}+3A_{11}+A_{12}+2A_{18}.
\end{array}
\]
The K3 surface is the double cover of $\PP^{2}$ branched over a sextic
curve $C_{6}$ with a singularity $q$ of type~$\mathbf{a}_{7}$.
The line tangent to the branch of $C_{6}$ has no other intersection
points with $C_{6}$, and the curves $A_{9}$, $A_{19}$ map onto that
line. Moreover, there exist five lines through $q$ that are bitangent
to the sextic at their other intersection points. The images of $A_{j}$, $A_{j+12}$
for $j\in\{1,\dots,5\}$ are these lines. 

For $J=\{i,j\}\subset\{1,\dots,5\}$, let 
\[
A_{J}=-A_{i}-A_{j}+\sum_{t=1}^{5}  A_{t}+2A_{6}+A_{7}-A_{9}.
\]
Then $A_{J}$ is the class of a $(-2)$-curve, and $B_{J}=2D_{2}-A_{J}$
is also a $(-2)$-curve. These are the $20$ curves $A_{20},\dots,A_{39}$; 
thus the images by the double cover map of curves $A_{J},B_{J}$ are
conics that are $6$-tangent to the sextic branch curve.

We have $A_{J} \cdot A_{J'}=2$ if and only if $\#J\cap J'=0$, and else
$A_{J} \cdot A_{J'}\in\{-2,0\}$. The dual graph of the $10$ curves $A_{J}$ 
(for $J\subset\{1,\dots,5\}$, $\#J=2$) is the Petersen graph with
weight $2$ on the edges:

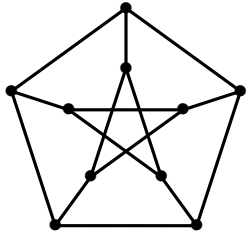
\begin{center}
\begin{tikzpicture}[scale=0.8]



\draw [very thick] ( 0.95106, 0.30901 )  -- ( -0.95103, 0.30910); 
\draw [very thick] ( 0, 1.0000 ) -- ( -0.58788, -0.80894 );
\draw [very thick] ( -0.95103, 0.30910 ) -- ( 0.58765, -0.80911);     
\draw [very thick] ( -0.58788, -0.80894 ) -- ( 0.95106, 0.30901); 
\draw [very thick] ( 0.58765, -0.80911 ) --( 0, 1.0000 );

\draw [very thick] ( 1.9021, 0.61802 )--( 0, 2 );
\draw [very thick] ( 0, 2 )--( -1.9021, 0.61819 );
\draw [very thick] ( -1.9021, 0.61819 )--( -1.1758, -1.6179 );
\draw [very thick] ( -1.1758, -1.6179 )--( 1.1753, -1.6182 );
\draw [very thick] ( 1.1753, -1.6182 )--( 1.9021, 0.61802 );

\draw [very thick] ( 0.95106, 0.30901 ) -- ( 1.9021, 0.61802 );
\draw [very thick] ( 0, 1.0000 ) -- ( 0, 2 );
\draw [very thick] ( -0.95103, 0.30910 ) -- ( -1.9021, 0.61819 );
\draw [very thick] ( -0.58788, -0.80894 ) -- ( -1.1758, -1.6179 );
\draw [very thick] ( 0.58765, -0.80911 ) -- ( 1.1753, -1.6182 );

\draw ( 0.95106, 0.30901 )     node {$\bullet$};
\draw ( 0, 1.0000 ) node {$\bullet$};
\draw ( -0.95103, 0.30910 )      node {$\bullet$};
\draw ( -0.58788, -0.80894 )      node {$\bullet$};
\draw ( 0.58765, -0.80911 ) node {$\bullet$};

\draw ( 1.9021, 0.61802 ) node {$\bullet$};
\draw ( 0, 2 ) node {$\bullet$};
\draw ( -1.9021, 0.61819 ) node {$\bullet$};
\draw ( -1.1758, -1.6179 ) node {$\bullet$};
\draw ( 1.1753, -1.6182 ) node {$\bullet$};


\end{tikzpicture}
\end{center} 
The configuration of the curves $B_{J}$ ($J\subset\{1,\dots,5\}$, $\#J=2$)
is also the Petersen graph with weight $2$ on the edges. Using that
$D_{2} \cdot A_{J}=D_{2} \cdot B_{J}=2$, we get that $A_{J} \cdot A_{J'}=B_{J} \cdot B_{J'}$
and $A_{J} \cdot B_{J'}=4-A_{J} \cdot A_{J'}$ for any subsets $J$, $J'$ of order $2$
of $\{1,\dots,5\}$. 

\subsubsection{The second involution}

The divisor 
\[
D_{2}'=2A_{1}+A_{12}+2A_{13}
\]
is nef, of square $2$, with base points. One has $D_{2} \cdot A_{j}=0$
for $j\in\{1,\dots5,7,\dots,12,19\}$, $D_{2}'\cdot A_{6}=2$, $D_{2}'\cdot A_{j}=1$
for $j\in\{13,\dots,18\}$ and $D_{2}'\cdot A_{j}=4$ for $j\geq20$. Since
$D_{2}'=2F+A_{12}$, where $F=A_{1}+A_{13}$ is a fiber of an elliptic
fibration  with $FA_{12}=1$, the linear system $|D_{2}'|$ has base
points, and $|D_{8}|=|2D_{2}'|$ is base-point free and hyperelliptic.
We have 
\begin{equation}
\begin{array}{l}
D_{8}\equiv A_{13}+A_{14}+A_{15}+A_{16}+A_{1}+A_{2}+A_{3}+A_{4}+2A_{12},\hfill\\
D_{8}\equiv A_{15}+A_{16}+A_{17}+A_{18}+A_{3}+A_{4}+A_{5}+A_{7}+2A_{8}\hfill\\
\hphantom{D_{8}\equiv\;}+A_{9}+2A_{10}+2A_{11}+2A_{12}+A_{19},\\
D_{8}\equiv B_{J}+\sum_{t=1,t\notin J}^{5}  A_{j}+A_{7}+2A_{8}+A_{9}+2A_{10}+2A_{11}+2A_{19},\hfill\\
D_{8}\equiv A_{J}+\sum_{t\in J}A_{t}+2A_{7}+4A_{8}+3A_{9}+3A_{10}+2A_{11},\hfill\\
D_{8}=2A_{6}+\sum_{t=1}^{5}A_{t}+3A_{7}+4A_{8}+2A_{9}+3A_{10}+2A_{11}+A_{19}.\hfill
\end{array}\label{eq:UD8A1exp3}
\end{equation}
The map $\varphi_{D_{8}}\colon X\to\PP^{5}$ associated to $|D_{8}|$ has  
image a cone over a rational normal curve in $\PP^{4}$; it factorizes
through a surjective map $\varphi'\colon X\to\mathbf{F}_{4}$, where $\mathbf{F}_{4}$
is the Hirzebruch surface with a section $s$ such that $s^{4}=-4$.
The section $s$ is mapped to the vertex of the cone by the map $\mathbf{F}_{4}\to\PP^{5}$.
Let $f$ denote a fiber of the unique fibration $\mathbf{F}_{4}\to\PP^{1}$.
By Theorem \ref{thm:SaintDonat-2}, the branch locus of $\varphi'$
is the union of $s$ and a curve $C$ such that $Cs=0$ and $C\in|3(s+4f)|$
(so that $b+C\in|-2K_{\mathbf{F}_{4}}|$). We have 
\[
|D_{8}|=\varphi'^{*}|4f+s|, 
\]
and therefore from the equivalence relations in \eqref{eq:UD8A1exp3},
we get the following claims.
\begin{itemize}
\item The curve $A_{12}$ is in the ramification locus;  its image is $s$. 
\item The curves $A_{7}$, $A_{8}$, $A_{9}$, $A_{10}$, $A_{11}$, $A_{19}$ are contracted to a $\mathbf{d}_{6}$ singularity $q$ of $C$.
\item The curve $C$ is the union of a curve $s'$ in $|s+4f|$ and a curve
$B'$ in $|2(s+4f)|$ (thus $s'B'=8$). The curve $B'$ has a node
at $q$, and $s'$ is tangent to one of the branches of $B'$ at $q$,
so that the intersection multiplicity of $s'$ and $B'$ at $q$ is
$3$ and the singularity of $C=B'+s'$ has type $\mathbf{d}_{6}$.
\item The other intersection points of $s'$ and $B'$ are nodes $p_{1},\dots,p_{5}$ to which the curves $A_{1},\dots,A_{5}$ are mapped. 
\item The curve $A_{6}$ is in the ramification locus; its image is $s'$.
\item The curves $A_{13},\dots,A_{17},A_{18}$ are mapped to the fibers
through $p_{1},\dots,p_{5}$ and $q$, respectively. 
\item The curves $A_{J}$ and $B_{J}$ are mapped to curves in the linear system
$|s+4f|$ passing through the points $p_{1},\dots,p_{5}$ and points infinitely
near $q$ with certain multiplicities. 
\end{itemize}

\subsection{The lattice $\boldsymbol{U\oplus\mathbf{E}_{8}\oplus\mathbf{A}_{3}}$}

The K3 surface $X$ contains $14$ $(-2)$-curves $A_{1},\dots,A_{14}$; 
their configuration is as follows:

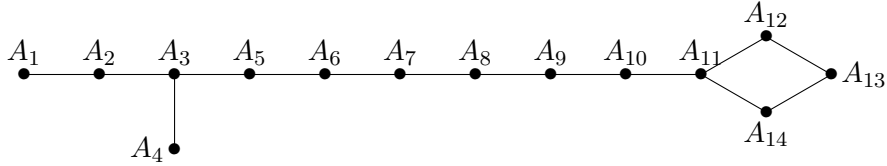
\begin{center}
\begin{tikzpicture}[scale=1]

\draw (0,0) -- (9,0);
\draw (9,0) --  (9+0.866,0.5); 
\draw  (9+0.866,0.5) -- (9+0.866*2,0);
\draw  (9+0.866,-0.5) -- (9+0.866*2,0);
\draw (9,0) --  (9+0.866,-0.5); 
\draw (2,0) -- (2,-1);

\draw (0,0) node {$\bullet$};
\draw (1,0) node {$\bullet$};
\draw (2,0) node {$\bullet$};
\draw (3,0) node {$\bullet$};
\draw (4,0) node {$\bullet$};
\draw (5,0) node {$\bullet$};
\draw (6,0) node {$\bullet$};
\draw (7,0) node {$\bullet$};
\draw (8,0) node {$\bullet$};
\draw (9,0) node {$\bullet$};
\draw (9+0.866,0.5) node {$\bullet$};
\draw (9+0.866,-0.5) node {$\bullet$};
\draw (2,-1) node {$\bullet$};
\draw (9+0.866*2,0) node {$\bullet$};

\draw (2,-1) node [left]{$A_{4}$};
\draw (0,0) node [above]{$A_{1}$};
\draw (1,0) node [above]{$A_{2}$};
\draw (2,0) node [above]{$A_{3}$};
\draw (3,0) node [above]{$A_{5}$};
\draw (4,0) node [above]{$A_{6}$};
\draw (5,0) node [above]{$A_{7}$};
\draw (6,0) node [above]{$A_{8}$};
\draw (7,0) node [above]{$A_{9}$};
\draw (8,0) node [above]{$A_{10}$};
\draw (9,0) node [above]{$A_{11}$};
\draw (9+0.866,0.5) node [above]{$A_{12}$};
\draw (9+0.866,-0.5) node [below]{$A_{14}$};
\draw (9+0.866*2,0) node [right]{$A_{13}$};

\end{tikzpicture}
\end{center} 

The curves $A_{1},\dots,A_{13}$ generate  the N\'eron--Severi lattice;
in that basis, the divisor
\[
D_{506}=(46,93,141,70,120,100,81,63,46,30,15,1,-12)
\]
is ample, of square $506$, with $D_{506} \cdot A_{j}=1$ for $j\leq12$,
$D_{506} \cdot A_{13}=25$ and $D_{506} \cdot A_{14}=3$. The surface has an elliptic
fibration  such that $A_{10}$ is a section with singular fiber  
\[
F_{1}=A_{11}+A_{12}+A_{13}+A_{14}
\]
and a fiber $F_{2}$ of type $\tilde{\mathbf{E}_{8}}$. By Theorem
\ref{thm:SaintDonat-2}, case i) a), we have the following. 

\begin{prop}
The linear system $|4F_{1}+2A_{8}|$ defines a morphism $\varphi\colon X\to\mathbf{F}_{4}$
branched over the unique section $s$ with $s^2=-4$ with $s^{2}=-4$ and a curve $B\in|3s+12f|$.
The curve $B$ has one $\mathbf{a}_{3}$ singularity $p$ and one
$\mathbf{e}_{8}$ singularity $q$. The pull-backs of the fibers through
$p$, $q$ are the fibers $F_{1}$, $F_{2}$.
\end{prop}

In order to construct $X$, one can also use the divisor 
\[
D_{2}=(2,5,8,4,7,6,5,4,3,2,1,0,0), 
\]
which is nef, of square $2$, base-point free, with $D_{2} \cdot A_{j}=0$
for $j\in\{2,\dots,11,13\}$ and $D_{2} \cdot A_{j}=1$ for $j\in\{1,12,14\}$.
Since 
\[
D_{2}\equiv A_{2}+2A_{3}+A_{4}+2A_{5}+2A_{6}+2A_{7}+2A_{8}+A_{9}+2A_{10}+2A_{11}+A_{12}+A_{13}+A_{14},
\]
we obtain that the K3 surface is the double cover of $\PP^{2}$ branched
over a sextic curve which is the union of a line $L$ and a quintic
$Q$ such that $A_{1}$ is in the ramification locus and its image
is $L$. The sextic curve has a singularity $q$ of type $\mathbf{d}_{10}$
and a node. The image of $A_{12}$ and $A_{14}$ is the line $L$.
The situation is a specialization of Section \ref{subsec:The-lattice124-UA2E8}.

\section{Rank 14 lattices }

\subsection{The lattice $\boldsymbol{U\oplus\mathbf{E}_{8}\oplus\mathbf{D}_{4}}$}

The K3 surface $X$ contains $15$ $(-2)$-curves $A_{1},\dots,A_{15}$; 
their dual graph is 

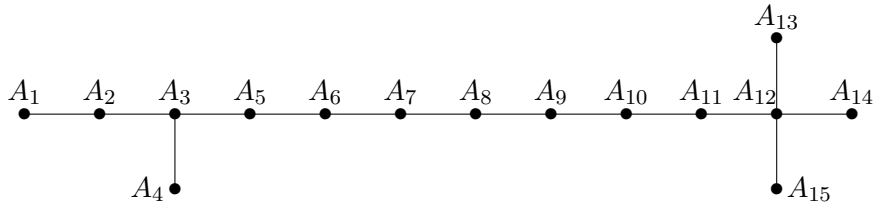
\begin{center}
\begin{tikzpicture}[scale=1]

\draw (0,0) -- (11,0);
\draw (10,1) -- (10,-1);
\draw (2,0) -- (2,-1);

\draw (0,0) node {$\bullet$};
\draw (1,0) node {$\bullet$};
\draw (2,0) node {$\bullet$};
\draw (3,0) node {$\bullet$};
\draw (4,0) node {$\bullet$};
\draw (5,0) node {$\bullet$};
\draw (6,0) node {$\bullet$};
\draw (7,0) node {$\bullet$};
\draw (8,0) node {$\bullet$};
\draw (9,0) node {$\bullet$};
\draw (10,0) node {$\bullet$};
\draw (10,1) node {$\bullet$};
\draw (10,-1) node {$\bullet$};
\draw (11,0) node {$\bullet$};
\draw (2,-1) node {$\bullet$};

\draw (2,-1) node [left]{$A_{4}$};
\draw (0,0) node [above]{$A_{1}$};
\draw (1,0) node [above]{$A_{2}$};
\draw (2,0) node [above]{$A_{3}$};
\draw (3,0) node [above]{$A_{5}$};
\draw (4,0) node [above]{$A_{6}$};
\draw (5,0) node [above]{$A_{7}$};
\draw (6,0) node [above]{$A_{8}$};
\draw (7,0) node [above]{$A_{9}$};
\draw (8,0) node [above]{$A_{10}$};
\draw (9,0) node [above]{$A_{11}$};
\draw (10.15,0) node [above left]{$A_{12}$};

\draw (10,1) node [above]{$A_{13}$};
\draw (10,-1) node [right]{$A_{15}$};
\draw (11,0) node [above]{$A_{14}$};

\end{tikzpicture}
\end{center} 

The curves $A_{1},\dots,A_{14}$ generate the N\'eron--Severi lattice.
In that base, the divisor 
\[
D_{506}=(46,93,141,70,120,100,81,63,46,30,15,1,-12,0)
\]
 is ample, of square $506$, with $D_{506} \cdot A_{13}=25$ and $D_{506} \cdot A_{j}=1$
for $j\neq13$. The divisor 
\[
D_{2}=(2,5,8,4,7,6,5,4,3,2,1,0,0,0)
\]
is nef, base-point free, of square $2$, with $D_{2} \cdot A_{1}=D_{2} \cdot A_{12}=1$
and $D_{2} \cdot A_{j}=0$ for $j\notin\{1,12\}$. We have
\[
D_{2}\equiv A_{2}+2A_{3}+A_{4}+2(A_{5}+A_{6}+A_{7}+A_{8}+A_{9}+A_{10}+A_{11}+A_{12})+A_{13}+A_{14}+A_{15}; 
\]
thus, using the linear system $|D_{2}|$, we obtain the following. 

\begin{prop}
The K3 surface is a double cover of $\,\PP^{2}$ which is ramified over
$A_{1}$ and $A_{12}$; the images of these curves are lines in the
sextic branch curve. The sextic curve has a $\mathbf{d}_{10}$ singularity
at a point $q$ and three nodal singularities. The residual quartic
curve is smooth at $q$. The tangent at $q$ is the line $L$, and 
the order of tangency at $q$ is $4$. The other nodes are the intersections 
of the line $L'$ with the quartic. 
\end{prop}

\subsection{The lattice $\boldsymbol{U\oplus\mathbf{D}_{8}\oplus\mathbf{D}_{4}}$}

\subsubsection{First involution}

The K3 surface $X$ contains $20$ $(-2)$-curves $A_{1},\dots,A_{20}$;
their configuration is as follows:

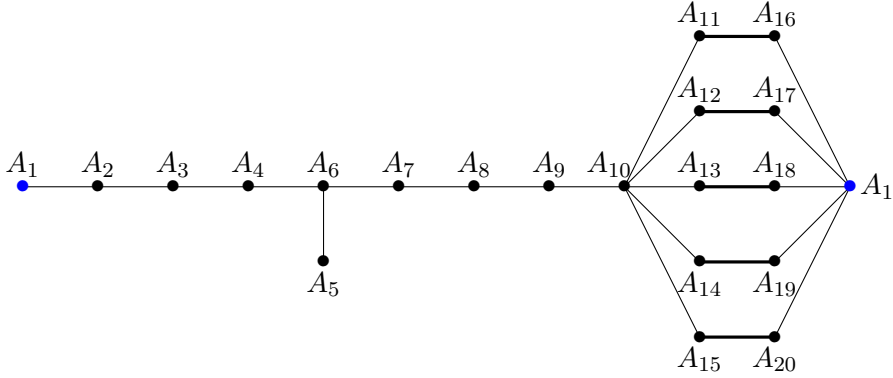
\begin{center}
\begin{tikzpicture}[scale=1]

\draw (0,0) -- (11,0);
\draw [very thick] (9,0) -- (10,0);
\draw  (9,1) -- (8,0);
\draw  (9,-1) --  (8,0); 
\draw  (9,2) -- (8,0);
\draw  (9,-2) --  (8,0); 

\draw [very thick]  (9,2) -- (10,2);
\draw [very thick]  (9,1) -- (10,1);
\draw [very thick] (9,-1) --  (10,-1); 
\draw [very thick]  (9,-2) -- (10,-2);
\draw  (10,1) -- (11,0);
\draw  (10,-1) --  (11,0); 
\draw  (10,2) -- (11,0);
\draw  (10,-2) --  (11,0);

\draw (4,0) -- (4,-1);

\draw [color=blue] (0,0) node {$\bullet$};
\draw (1,0) node {$\bullet$};
\draw (2,0) node {$\bullet$};
\draw (3,0) node {$\bullet$};
\draw (4,0) node {$\bullet$};
\draw (5,0) node {$\bullet$};
\draw (6,0) node {$\bullet$};
\draw (7,0) node {$\bullet$};
\draw (8,0) node {$\bullet$};
\draw (9,0) node {$\bullet$};
\draw (10,0) node {$\bullet$};
\draw [color=blue] (11,0) node {$\bullet$};
\draw (4,-1) node {$\bullet$};
\draw (9,1) node {$\bullet$};
\draw (10,1) node {$\bullet$};
\draw (9,-1) node {$\bullet$};
\draw (10,-1) node {$\bullet$};
\draw (9,2) node {$\bullet$};
\draw (10,2) node {$\bullet$};
\draw (9,-2) node {$\bullet$};
\draw (10,-2) node {$\bullet$};

\draw (4,-1) node [below]{$A_{5}$};
\draw (0,0) node [above]{$A_{1}$};
\draw (1,0) node [above]{$A_{2}$};
\draw (2,0) node [above]{$A_{3}$};
\draw (3,0) node [above]{$A_{4}$};
\draw (4,0) node [above]{$A_{6}$};
\draw (5,0) node [above]{$A_{7}$};
\draw (6,0) node [above]{$A_{8}$};
\draw (7,0) node [above]{$A_{9}$};
\draw (7.8,0) node [above]{$A_{10}$};
\draw (9,0) node [above]{$A_{13}$};
\draw (10,0) node [above]{$A_{18}$};
\draw (11,0) node [right]{$A_{1}$};
\draw (9,1) node [above]{$A_{12}$};
\draw (10,1) node [above]{$A_{17}$};
\draw (9,-1) node [below]{$A_{14}$};
\draw (10,-1) node [below]{$A_{19}$};
\draw (9,2) node [above]{$A_{11}$};
\draw (10,2) node [above]{$A_{16}$};
\draw (9,-2) node [below]{$A_{15}$};
\draw (10,-2) node [below]{$A_{20}$};

\end{tikzpicture}
\end{center} 

The curves $A_{1},\dots,A_{14}$ generate  the N\'eron--Severi lattice;
in that basis, the divisor 
\[
D_{154}=(1,3,6,10,7,15,14,14,15,17,8,8,0,4)
\]
is ample, of square $154$, with $D_{154} \cdot A_{j}=1$ for $j\in\{1,1,\dots,12,18,20\}$,
$D_{154} \cdot A_{j}=9$ for $j\in\{14,19\}$ and $D_{154} \cdot A_{j}=17$ for
$j\in\{15,16,17\}$. The divisor 
\[
D_{2}=(1,2,3,4,2,5,4,3,2,1,0,0,0,0)
\]
is nef, base-point free, of square $2$, with $D_{2} \cdot A_{j}=0$ for
$j\in\{1,\dots,4,6,\dots,10\}$ and else $D_{2} \cdot A_{j}=1$. Using the
equivalences from the elliptic fibration  with fiber $A_{11}+A_{16}$,
we obtain that 
\[
D_{2}\equiv A_{10+k}+A_{15+k}+\sum_{j=1,j\neq5}^{10}A_{j},\,\,\forall k\in\{1,2,3,4,5\}.
\]
The linear system $|D_{2}|$ defines the K3 surface $X$ as the double
cover of $\PP^{2}$ branched over a sextic curve $C_{6}$ with an
$\mathbf{a}_{9}$ singularity $q$. The cover is ramified above $A_{5}$;
the image of $A_{5}$ is a line, a component of $C_{6}$ which has
no other intersection points with the residual quintic curve $Q$.
The images of $A_{10+k}+A_{15+k},\,k\in\{1,\dots,5\}$, are lines through
$q$ which are bitangent to $Q$. The quintic is smooth, thus of genus
$6$.

\subsubsection{Second involution}

The divisor 
\[
D_{2}'=(0,1,2,4,3,6,5,4,3,2,1,0,0,0)
\]
is nef, base-point free, of square $2$. It satisfies $D_{2} \cdot A_{j}=0$
for $j\in\{2,4,\dots,11,17,\dots,20\}.$ Using equivalence relations
obtained from the elliptic fibration  with fiber $A_{11}+A_{16}$,
we have
\[
\begin{array}{c}
D_{2}'\equiv A_{12}+(A_{17}+A_{4}+A_{5}+2\sum_{t=6}^{10}A_{t}+A_{11}),\\
D_{2}'\equiv A_{13}+(A_{18}+A_{4}+A_{5}+2\sum_{t=6}^{10}A_{t}+A_{11}),\\
D_{2}'\equiv A_{14}+(A_{19}+A_{4}+A_{5}+2\sum_{t=6}^{10}A_{t}+A_{11}),\\
D_{2}'\equiv A_{15}+(A_{20}+A_{4}+A_{5}+2\sum_{t=6}^{10}A_{t}+A_{11}),\\
D_{2}'\equiv A_{16}+(A_{11}+A_{4}+A_{5}+2\sum_{t=6}^{10}A_{t}+A_{11}); 
\end{array}
\]
moreover,  
\[
D_{2}'\equiv2A_{1}+A_{2}+\sum_{j=17}^{20}A_{j}.
\]
We have the following assertions.
\begin{itemize}
\item The branch curve of the corresponding double cover is a sextic with
a $\mathbf{d}_{8}$ singularity at a point $q$ onto which the curves
$A_{4},\dots,A_{11}$ are contracted and five nodes $p_{2}, p_{17},\dots,p_{20}$
onto which $A_{2}, A_{17},\dots,A_{20}$ are contracted. The curves
$A_{1}$, $A_{3}$ are in the ramification locus; their images are lines
$L_{1}$, $L_{3}$. Let us denote by $Q'$ the residual quartic curve.

\item The line $L_{1}$ cuts $Q'$ (resp.\ $L_{3}$) at the points $p_{17},\dots,p_{20}$ (resp.\ the point $p_{2}$). 
\item The quartic $Q'$ has a node at $q$, and the line $L_{3}$ is tangent
with multiplicity $3$ at one of the branches of the node.
\item For $k\in\{12,\dots,15\}$, the image of $A_{k}$ is a line through
$q$ and $p_{k+5}$. 
\item The image of $A_{16}$ is a line passing through $q$ and tangent
  to $Q'$ at another point.
  \end{itemize}

\begin{rem}
We constructed the surface $X$ as two double covers. The associated
involutions fix curves of geometric genus of $6$ and $2$, respectively; 
thus these involution generate  the automorphism group of $X$, which
is $(\ZZ/2\ZZ)^{2}$ according to \cite{Kondo}. 
\end{rem}

\subsection{The lattice $\boldsymbol{U\oplus\mathbf{E}_{8}\oplus\mathbf{A}_{1}^{\oplus4}}$}

The lattice $U\oplus\mathbf{E}_{8}\oplus\mathbf{A}_{1}^{\oplus4}$
is also isometric to $U\oplus\mathbf{E}_{7}\oplus{\bf D}_{4}\oplus\mathbf{A}_{1}$; 
see \cite{Kondo}. 

\subsubsection{First involution}

The K3 surface $X$ contains $27$ $(-2)$-curves $A_{1},\dots,A_{27}$.
The dual graph of $A_{1},\dots,A_{19}$ and the intersections of the curves
$A_{1},\dots,A_{11}$ with the curves $A_{j}$ for $j\geq20$ is given
by

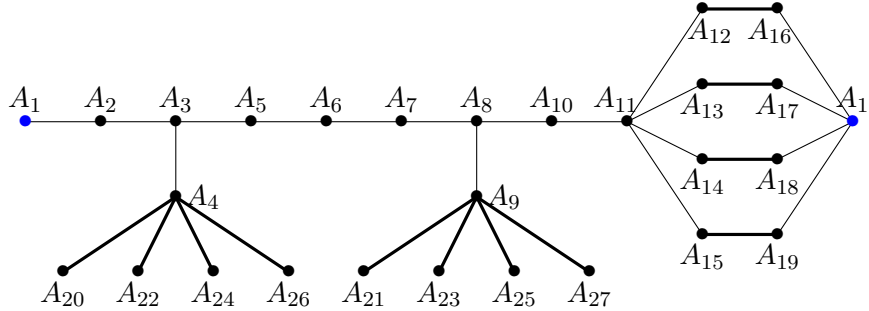
\begin{center}
\begin{tikzpicture}[scale=1]

\draw (0,0) -- (8,0);
\draw (6,0) -- (6,-1);
\draw (2,0) -- (2,-1);
\draw (8,0) -- (9,1.5);
\draw (8,0) -- (9,0.5);
\draw (8,0) -- (9,-0.5);
\draw (8,0) -- (9,-1.5);

\draw (11,0) -- (10,1.5);
\draw (11,0) -- (10,0.5);
\draw (11,0) -- (10,-0.5);
\draw (11,0) -- (10,-1.5);

\draw [very thick] (9,1.5) -- (10,1.5);
\draw [very thick] (9,0.5) -- (10,0.5);
\draw [very thick] (9,-0.5) -- (10,-0.5);
\draw [very thick] (9,-1.5) -- (10,-1.5);

\draw [very thick] (2,-1) -- (2+1.5,-2);
\draw [very thick] (2,-1) -- (2+0.5,-2);
\draw [very thick] (2,-1) -- (2-0.5,-2);
\draw [very thick] (2,-1) -- (2-1.5,-2);
\draw [very thick] (6,-1) -- (6+1.5,-2);
\draw [very thick] (6,-1) -- (6+0.5,-2);
\draw [very thick] (6,-1) -- (6-0.5,-2);
\draw [very thick] (6,-1) -- (6-1.5,-2);

\draw (10,1.5) node {$\bullet$};
\draw (10,0.5) node {$\bullet$};
\draw (10,-1.5) node {$\bullet$};
\draw (10,-0.5) node {$\bullet$};
\draw [color=blue] (11,0)  node {$\bullet$};
\draw  (2+1.5,-2) node {$\bullet$};
\draw  (2+0.5,-2) node {$\bullet$};
\draw  (2-0.5,-2) node {$\bullet$};
\draw  (2-1.5,-2) node {$\bullet$};
\draw  (6+1.5,-2) node {$\bullet$};
\draw  (6+0.5,-2) node {$\bullet$};
\draw  (6-0.5,-2) node {$\bullet$};
\draw  (6-1.5,-2) node {$\bullet$};

\draw [color=blue] (0,0) node {$\bullet$};
\draw (1,0) node {$\bullet$};
\draw (2,0) node {$\bullet$};
\draw (3,0) node {$\bullet$};
\draw (4,0) node {$\bullet$};
\draw (5,0) node {$\bullet$};
\draw (6,0) node {$\bullet$};
\draw (7,0) node {$\bullet$};
\draw (8,0) node {$\bullet$};
\draw (9,1.5) node {$\bullet$};
\draw (9,0.5) node {$\bullet$};
\draw (9,-1.5) node {$\bullet$};
\draw (9,-0.5) node {$\bullet$};
\draw (2,-1) node {$\bullet$};
\draw (6,-1) node {$\bullet$};

\draw (2,-1) node [right]{$A_{4}$};
\draw (0,0) node [above]{$A_{1}$};
\draw (11,0) node [above]{$A_{1}$};
\draw (1,0) node [above]{$A_{2}$};
\draw (2,0) node [above]{$A_{3}$};
\draw (3,0) node [above]{$A_{5}$};
\draw (4,0) node [above]{$A_{6}$};
\draw (5,0) node [above]{$A_{7}$};
\draw (6,0) node [above]{$A_{8}$};
\draw (6,-1) node [right]{$A_{9}$};
\draw (7,0) node [above]{$A_{10}$};
\draw (7.85,0) node [above]{$A_{11}$};
\draw (9.1,1.5) node [below]{$A_{12}$};
\draw (9,0.5) node [below]{$A_{13}$};
\draw (9,-1.5) node [below]{$A_{15}$};
\draw (9,-0.5) node [below]{$A_{14}$};

\draw (9.9,1.5) node [below]{$A_{16}$};
\draw (10,0.5) node [below]{$A_{17}$};
\draw (10,-1.5) node [below]{$A_{19}$};
\draw (10,-0.5) node [below]{$A_{18}$};

\draw  (2+1.5,-2) node [below]{$A_{26}$};
\draw  (2+0.5,-2) node [below]{$A_{24}$};
\draw  (2-0.5,-2) node [below]{$A_{22}$};
\draw  (2-1.5,-2) node [below]{$A_{20}$};
\draw  (6+1.5,-2) node [below]{$A_{27}$};
\draw  (6+0.5,-2) node [below]{$A_{25}$};
\draw  (6-0.5,-2) node [below]{$A_{23}$};
\draw  (6-1.5,-2) node [below]{$A_{21}$};

\end{tikzpicture}
\end{center} 
The curves $A_{1},\dots,A_{14}$ generate the N\'eron--Severi lattice.
In that base, the divisor 
\[
D_{132}=(2,5,9,4,10,12,15,19,9,15,12,5,5,0)
\]
is ample, of square $132$, with $D_{132} \cdot A_{j}=1$ for $j\leq11$,
$D_{132} \cdot A_{j}=2$ for $j\in\{12,13,18,19\}$, $D_{132} \cdot A_{j}=12$ for
$j\in\{15,\dots,18\}$, $D_{132} \cdot A_{j}=18$ for $j\in\{21,23,24,26\}$ and 
$D_{132} \cdot A_{j}=28$ for $j\in\{20,22,25,27\}$. The divisor 
\[
D_{2}=(1,2,3,1,3,3,3,3,1,2,1,0,0,0)
\]
is nef, of square $2$, base-point free. We have  $D_{2} \cdot A_{j}=0$ for $j\in\{1,2,3,5,6,7,8,10,11\}$, $D_{2} \cdot A_{j}=1$ for $j\in\{4,9,12,\dots,19\}$ and $D_{2} \cdot A_{j}=2$ for $j\in\{20,\dots,27\}$. Using the elliptic fibration  with fiber
$A_{12}+A_{16}$, we get 
\[
D_{2}\equiv A_{11+k}+A_{15+k}+\left(-A_{4}-A_{9}+\sum_{j=1}^{11}A_{j}\right),\,k\in\{1,2,3,4\}.
\]
Moreover, one has 
\[
2D_{2}\equiv A_{2k}+A_{2k+1}\quad\text{ for }k\in\{10,11,12,13\};
\]
thus the K3 surface is the double cover branched over a sextic curve
with an $\mathbf{a}_{9}$ singularity $q$. The image of $A_{4}$
and $A_{9}$ is the line tangent to the branch of that singularity.
The curves $A_{11+k}+A_{15+k}$, $k\in\{1,2,3,4\}$, are mapped onto
lines going through $q$  that are tangent to the sextic curve
in two other points. The curves $A_{2k}+A_{2k+1}\text{ for }k\in\{10,11,12,13\}$
are mapped onto conics that are $6$-tangent to the sextic curve. 

The classes of the curves $A_{20}$, $A_{22}$, $A_{24}$, $A_{26}$ are 
\[
\begin{array}{ll}
A_{20}=(0,0,0,-1,1,2,3,4,2,3,2,0,0,1), & A_{22}=(2,4,6,2,6,6,6,6,3,3,0,-1,-1,-1)\\
A_{24}=(0,0,0,-1,1,2,3,4,2,3,2,1,0,0), & A_{26}=(0,0,0,-1,1,2,3,4,2,3,2,0,1,0),
\end{array}
\]
and the dual graph of the curves $A_{20},\dots,A_{27}$ is 

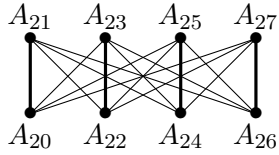
\begin{center}
\begin{tikzpicture}[scale=1]

\draw [very thick] (0,0) -- (0,1);
\draw [very thick] (1,0) -- (1,1);
\draw [very thick] (2,0) -- (2,1);
\draw [very thick] (3,0) -- (3,1);

\draw (0,0) -- (1,1);
\draw (0,0) -- (2,1);
\draw (0,0) -- (3,1);
\draw (1,0) -- (0,1);
\draw (1,0) -- (2,1);
\draw (1,0) -- (3,1);
\draw (2,0) -- (1,1);
\draw (2,0) -- (2,1);
\draw (2,0) -- (3,1);
\draw (3,0) -- (1,1);
\draw (3,0) -- (2,1);
\draw (3,0) -- (0,1);

\draw (2,0) -- (0,1);

\draw (0,0) node {$\bullet$};
\draw (1,0) node {$\bullet$};
\draw (2,0) node {$\bullet$};
\draw (3,0) node {$\bullet$};
\draw (0,1) node {$\bullet$};
\draw (1,1) node {$\bullet$};
\draw (2,1) node {$\bullet$};
\draw (3,1) node {$\bullet$};

\draw (0,0) node  [below]{$A_{20}$};
\draw (1,0) node  [below]{$A_{22}$};
\draw (2,0) node  [below]{$A_{24}$};
\draw (3,0) node  [below]{$A_{26}$};
\draw (0,1) node [above]{$A_{21}$};
\draw (1,1) node [above]{$A_{23}$};
\draw (2,1) node [above]{$A_{25}$};
\draw (3,1) node [above]{$A_{27}$};

\end{tikzpicture}
\end{center} 
where a thick line has weight $6$ and a thin line has weight $4$.
For completeness, we also give the intersections between the curves
$A_{12},\dots,A_{19}$ and the curves $A_{20},\dots,A_{27}$:\\
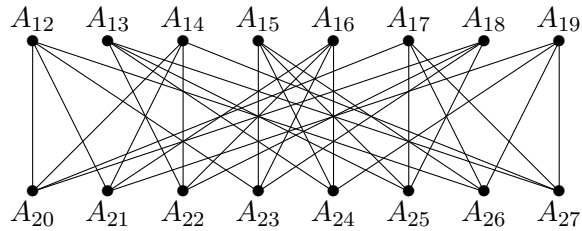
\begin{center}
\begin{tikzpicture}[scale=1]

\draw  (0,0) -- (0,2);
\draw  (0,2) -- (1,0);
\draw  (0,2) -- (3,0);
\draw   (0,2) -- (6,0);

\draw  (1,2) -- (2,0);
\draw  (1,2) -- (4,0);
\draw  (1,2) -- (5,0);
\draw  (1,2) -- (7,0);

\draw  (2,2) -- (0,0);
\draw  (2,2) -- (1,0);
\draw  (2,2) -- (2,0);
\draw  (2,2) -- (7,0);

\draw  (3,2) -- (3,0);
\draw  (3,2) -- (4,0);
\draw  (3,2) -- (5,0);
\draw  (3,2) -- (6,0);

\draw  (4,2) -- (3,0);
\draw  (4,2) -- (4,0);
\draw  (4,2) -- (1,0);
\draw  (4,2) -- (2,0);

\draw  (5,2) -- (0,0);
\draw  (5,2) -- (5,0);
\draw  (5,2) -- (6,0);
\draw  (5,2) -- (7,0);

\draw  (6,2) -- (0,0);
\draw  (6,2) -- (2,0);
\draw  (6,2) -- (3,0);
\draw  (6,2) -- (5,0);

\draw  (7,2) -- (1,0);
\draw  (7,2) -- (4,0);
\draw  (7,2) -- (6,0);
\draw  (7,2) -- (7,0);

\draw (0,0) node {$\bullet$};
\draw (1,0) node {$\bullet$};
\draw (2,0) node {$\bullet$};
\draw (3,0) node {$\bullet$};
\draw (4,0) node {$\bullet$};
\draw (5,0) node {$\bullet$};
\draw (6,0) node {$\bullet$};
\draw (7,0) node {$\bullet$};
\draw (0,2) node {$\bullet$};
\draw (1,2) node {$\bullet$};
\draw (2,2) node {$\bullet$};
\draw (3,2) node {$\bullet$};
\draw (4,2) node {$\bullet$};
\draw (5,2) node {$\bullet$};
\draw (6,2) node {$\bullet$};
\draw (7,2) node {$\bullet$};

\draw (0,0) node  [below]{$A_{20}$};
\draw (1,0) node  [below]{$A_{21}$};
\draw (2,0) node  [below]{$A_{22}$};
\draw (3,0) node  [below]{$A_{23}$};
\draw (4,0) node  [below]{$A_{24}$};
\draw (5,0) node  [below]{$A_{25}$};
\draw (6,0) node  [below]{$A_{26}$};
\draw (7,0) node  [below]{$A_{27}$};

\draw (0,2) node [above]{$A_{12}$};
\draw (1,2) node [above]{$A_{13}$};
\draw (2,2) node [above]{$A_{14}$};
\draw (3,2) node [above]{$A_{15}$};
\draw (4,2) node [above]{$A_{16}$};
\draw (5,2) node [above]{$A_{17}$};
\draw (6,2) node [above]{$A_{18}$};
\draw (7,2) node [above]{$A_{19}$};

\end{tikzpicture}
\end{center} 
where a thin line has weight $2$. In fact, one can check that the
above graph is another occurrence of the Levi graph of the M\"obius
configuration.

\subsubsection{Second involution}

The divisor 
\[
D_{2}'=(1,2,4,2,4,4,4,4,2,2,0,0,0,0)
\]
is nef, of square $2$, with base points and with $D_{2} \cdot A_{j}=0$
for $j\in\{1,3,\dots,10,12,\dots,15\}$ and $D_{2} \cdot A_{j}=1$ for $j\in\{2,16,17,18,19\}$,
$D_{2} \cdot A_{j}=4$ for $j\geq20$. Let $D_{8}=2D_{2}'$; it is base-point
free and hyperelliptic. It defines a morphism $\varphi'\colon X\to\PP^{5}$
onto a degree $4$ surface. That morphism factors through the Hirzebruch
surface $\mathbf{F}_{4}$ by a map denoted by $\varphi$; the map
$\mathbf{F}_{4}\to\PP^{5}$ contracts the section $s$ of the Hirzebruch
surface. The divisor 
\[
F_{1}=A_{2}+A_{4}+2(A_{3}+A_{5}+A_{6}+A_{7}+A_{8})+A_{9}+A_{10}
\]
is the fiber of an elliptic fibration  for which the curves $A_{11+k}+A_{15+k}$,
$k\in\{1,2,3,4\}$, are also fibers. One has $D_{2}'=2F_{1}+A_{1}$; 
thus
\[
D_{8}\equiv\sum_{j=12}^{19}A_{j}+2A_{1}.
\]
Moreover, we have 
\[
\begin{array}{l}
D_{8}\equiv2A_{11}+2A_{3}+A_{4}+3A_{5}+4A_{6}+5A_{7}+6A_{8}+3A_{9}+4A_{10},\hfill\\
D_{8}\equiv A_{20}+2\sum_{j=3}^{8}A_{j}+A_{9}+A_{10}+A_{12}+A_{13}+A_{15},\hfill\\
D_{8}\equiv A_{22}+2\sum_{j=3}^{8}A_{j}+A_{9}+A_{10}+A_{12}+A_{13}+A_{14},\hfill\\
D_{8}\equiv A_{24}+2\sum_{j=3}^{8}A_{j}+A_{9}+A_{10}+A_{13}+A_{14}+A_{15},\hfill\\
D_{8}\equiv A_{26}+2\sum_{j=3}^{8}A_{j}+A_{9}+A_{10}+A_{12}+A_{14}+A_{15},\hfill\\
D_{8}\equiv A_{21}+2A_{3}+A_{4}+3A_{5}+4A_{6}+5A_{7}+6A_{8}+4A_{9}+3A_{10}+A_{14},\\
D_{8}\equiv A_{23}+2A_{3}+A_{4}+3A_{5}+4A_{6}+5A_{7}+6A_{8}+4A_{9}+3A_{10}+A_{15},\\
D_{8}\equiv A_{25}+2A_{3}+A_{4}+3A_{5}+4A_{6}+5A_{7}+6A_{8}+4A_{9}+3A_{10}+A_{12},\\
D_{8}\equiv A_{27}+2A_{3}+A_{4}+3A_{5}+4A_{6}+5A_{7}+6A_{8}+4A_{9}+3A_{10}+A_{13}.
\end{array}
\]
We therefore have the following claims.
\begin{itemize}
\item The curves $A_{3},\dots,A_{10}$ are contracted to a $\mathbf{d}_{8}$
singularity $q$ of the branch curve $B$ of $\varphi$; the curves
$A_{12},\dots,A_{15}$ are mapped to nodes $p_{12},\dots, p_{15}$
of $B$. 
\item The curve $B$ is the union of $s$ and a curve $C$ such that $Cs=0$ and 
$C\in|3(s+4f)|$ (so that $s+C\in|-2K_{\mathbf{F}_{4}}|$). We have
\[
|D_{8}|=\varphi{}^{*}|4f+s|
\]
and therefore from the above equivalence relations, we get the following: 
\item The image of curve $A_{1}$ by $\varphi$ is the section $s$.  
\item The curve $B$ is the union of three components: $s$, $B'$ and $C_{11}$,
where $B'\in|2(s+4f)|$ and $C_{11}\in|s+4f|$. The curve $B'$ has a
node $q$, and the curves $C_{11}$ and $B'$ meet at $q$ in such
a way that $C_{11}$ intersect one of the branches with multiplicity
$3$ (so that the singularity at $q$ of $B'+C_{11}$ has type $\mathbf{d}_{8}$).
The four points $p_{12},\dots,p_{15}$ are nodes; they are the remaining
intersection points of $C_{11}$ and $B'$ (so that $B'C_{11}=8$).
\item The images of the curves $A_{16},\dots,A_{19}$ are the fibers through
$p_{12},\dots,p_{15}$,
\item The images of the curves $A_{20}$, $A_{22}$, $A_{24}$, $A_{26}$ by $\varphi$ are curves in the linear system $|s+4f|$ passing through three of
the four points $p_{12},\dots,p_{15}$, and through $q$ and points infinitely
near  $q$ with certain multiplicities.
\item The images of the curves $A_{21}$, $A_{23}$, $A_{25}$, $A_{27}$ by $\varphi$ are curves in the linear system $|s+4f|$ passing through one of
the four points $p_{12},\dots,p_{15}$ and through $q$ and points infinitely
near  $q$ with certain multiplicities.
\end{itemize}

  \section{Rank at least~15 lattices }

\subsection{The lattice $\boldsymbol{U\oplus\mathbf{E}_{8}\oplus\mathbf{D}_{4}\oplus\mathbf{A}_{1}}$}

The K3 surface $X$ contains $21$ $(-2)$-curves $A_{1},\dots,A_{21}$.
The dual graph of these curves is given in \cite[Section 4]{Kondo}; 
we reproduce it here:

\begin{center}
\begin{tikzpicture}[scale=1]

\draw (0,0) -- (6,0);
\draw (1,1.5) -- (5,1.5);
\draw (0,0) -- (5,1.5);
\draw (1,1.5) -- (6,0);
\draw [very thick] (2,1.5) -- (4,1.5);
\draw [very thick] (2,2.5) -- (4,2.5);
\draw [very thick] (2,3.5) -- (4,3.5);
\draw [ultra thick] (0,4.5) -- (6,4.5);
\draw [very thick] (0,4.5) -- (0,2.5);
\draw [very thick] (6,4.5) -- (6,2.5);
\draw (0,2.5) -- (0,0);
\draw (6,2.5) -- (6,0);
\draw (0,4.5) -- (2,1.5);
\draw (0,4.5) -- (2,2.5);
\draw (0,4.5) -- (2,3.5);
\draw (6,4.5) -- (4,1.5);
\draw (6,4.5) -- (4,2.5);
\draw (6,4.5) -- (4,3.5);
\draw (1,1.5) -- (2,1.5);
\draw (1,1.5) -- (2,2.5);
\draw (1,1.5) -- (2,3.5);
\draw (5,1.5) -- (4,1.5);
\draw (5,1.5) -- (4,2.5);
\draw (5,1.5) -- (4,3.5);

\draw (0,0) node {$\bullet$};
\draw (1,0) node {$\bullet$};
\draw (2,0) node {$\bullet$};
\draw (3,0) node {$\bullet$};
\draw (4,0) node {$\bullet$};
\draw (5,0) node {$\bullet$};
\draw (6,0) node {$\bullet$};
\draw (1.5,0.45) node {$\bullet$};
\draw (4.5,0.45) node {$\bullet$};
\draw (0,2.5)  node {$\bullet$};
\draw (6,2.5)  node {$\bullet$};
\draw (0,4.5)  node {$\bullet$};
\draw (2,1.5) node {$\bullet$};
\draw (2,2.5) node {$\bullet$};
\draw (2,3.5) node {$\bullet$};
\draw (6,4.5)  node {$\bullet$};
\draw (4,1.5) node {$\bullet$};
\draw (4,2.5) node {$\bullet$};
\draw (4,3.5) node {$\bullet$};
\draw (1,1.5)  node {$\bullet$};
\draw (5,1.5)  node {$\bullet$};

\draw (0,0) node [below]{$A_{10}$};
\draw (1,0) node [below]{$A_{9}$};
\draw (2,0) node [below]{$A_{8}$};
\draw (3,0) node [below]{$A_{7}$};
\draw (4,0) node [below]{$A_{6}$};
\draw (5,0) node [below]{$A_{5}$};
\draw (6,0) node [below]{$A_{3}$};
\draw (0,2.5) node [left]{$A_{11}$}; 
\draw (6,2.5)  node [right]{$A_{4}$};
\draw (0,4.5)  node [left]{$A_{20}$};
\draw (2,1.5) node [above]{$A_{17}$};
\draw (2,2.5) node [above]{$A_{18}$};
\draw (2,3.5) node [above]{$A_{19}$};
\draw (6,4.5)  node [right]{$A_{21}$};
\draw (4-0.15,1.5) node [above]{$A_{14}$};
\draw (4-0.15,2.5) node [above]{$A_{15}$};
\draw (4-0.15,3.5) node [above]{$A_{16}$};
\draw (1,1.5)  node [left]{$A_{1}$};
\draw (5,1.5)  node [right]{$A_{13}$};
\draw (1.5,0.45) node [above]{$A_{12}$};
\draw (4.5,0.45) node [above]{$A_{2}$};

\draw (3,4.5) node [above]{\small $6$};

\end{tikzpicture}
\end{center} 

The curves $A_{1},\dots,A_{15}$ generate the N\'eron--Severi lattice;
in that basis, the divisor 
\[
D_{242}=(2,5,9,4,10,12,15,19,24,30,14,23,17,8,4)
\]
is ample, of square $242$, with $D_{242}A_{j}=1$ for $j\in\{1,\dots,14\}\setminus\{11\}$. 

The divisor 
\[
D_{2}=(1,2,3,1,3,3,3,3,3,3,1,2,1,0,0)
\]
 is nef, base-point free, with $D_{2}^{2}=2$, $D_{2} \cdot A_{20}=D_{2} \cdot A_{21}=2$,
$D_{2} \cdot A_{j}=1$ for $j\in\{4,11,14,\dots,19\}$ and else $D_{2} \cdot A_{j}=0$.
We have 
\[
2D_{2}\equiv A_{20}+A_{21}
\]
and 
\[
D_{2}\equiv\sum_{k\in\{1,\dots,13\}\setminus\{4,11\}}A_{k}+A_{13+j}+A_{16+j},\,\,j\in\{1,2,3\}.
\]
By using the linear system $|D_{2}|$, we obtain the following. 

\begin{prop}
The K3 surface $X$ is the double cover of $\,\PP^{2}$ branched over
a sextic curve $C_{6}$ with an $\mathbf{a}_{11}$ singularity $q$; 
the curves $A_{j}$ with $j\in\{1,\dots,13\}\setminus\{4,11\}$ are
mapped to $q$. The image of the curves $A_{4}$ and $A_{11}$ is the line
$L$ tangent to the branch of $C_{6}$ at $q$. The images of $A_{13+j}$ and $A_{16+j}$, $j\in\{1,2,3\}$, 
are lines $L_{1}$, $L_{2}$, $L_{3}$ through $q$ that are bitangent to
the sextic at other intersection points. The image of $A_{20}$, $A_{21}$
is a conic which is $6$-tangent to the sextic. 
\end{prop}

\begin{rem}
The branch curve is irreducible, with geometric genus $4$. 
\end{rem}

The divisor $D'_{2}=(0,1,2,1,2,2,2,2,3,4,2,3,2,1,0)$ is nef, base-point
free of square $2$, with 
\[
D_{2} \cdot A_{1}=D_{2} \cdot A_{8}=1,\,D_{2} \cdot A_{15}=D_{2} \cdot A_{16}=D_{2} \cdot A_{17}=2,\,D_{2} \cdot A_{20}=D_{2} \cdot A_{21}=4.
\]
Therefore, the linear system $|D_{2}'|$ induces a double cover $\varphi'\colon X\to\PP^{2}$
which is ramified above the line $L'=\varphi'(A_{8})$; the curves
$A_{2}$, $A_{3}$, $A_{5}$, $A_{6}$, $A_{7}$ are contracted to a singularity of
the branch curve, and the curves $A_{9},\dots,A_{14}$ are contracted
to another singularity. In that way, we obtain a second involution
acting on the K3 surface. By \cite{Kondo}, we know that the automorphism
group of the  general  K3 is $(\ZZ/2\ZZ)^{2}$. 

\subsection{The lattice $\boldsymbol{U\oplus\mathbf{E}_{8}\oplus\mathbf{D}_{6}}$}

The K3 surface $X$ contains $19$ $(-2)$-curves $A_{1},\dots,A_{19}$.
The dual graph of these curves is (see \cite[Section 4]{Kondo}) 

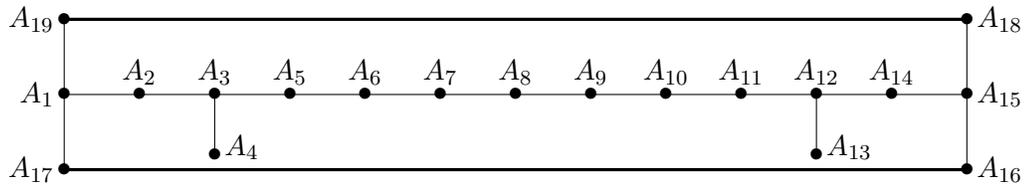
\begin{center}
\begin{tikzpicture}[scale=1]

\draw [very thick] (0,1) -- (12,1);
\draw [very thick] (0,-1) -- (12,-1);
\draw (0,0) -- (12,0);
\draw (0,1) -- (0,-1);
\draw (12,1) -- (12,-1);

\draw (2,-0.8) -- (2,0); 
\draw (10,-0.8) -- (10,0);

\draw (0,1) node {$\bullet$};
\draw (0,-1) node {$\bullet$};
\draw (12,1) node {$\bullet$};
\draw (12,-1) node {$\bullet$};

\draw (0,0) node {$\bullet$};
\draw (1,0) node {$\bullet$};
\draw (2,0) node {$\bullet$};
\draw (3,0) node {$\bullet$};
\draw (4,0) node {$\bullet$};
\draw (5,0) node {$\bullet$};
\draw (6,0) node {$\bullet$};
\draw (7,0) node {$\bullet$};
\draw (8,0) node {$\bullet$};
\draw (9,0) node {$\bullet$};
\draw (10,0) node {$\bullet$};
\draw (11,0) node {$\bullet$};
\draw (12,0) node {$\bullet$};
\draw (2,-0.8) node {$\bullet$};
\draw (10,-0.8) node {$\bullet$};

\draw (0,0) node [left]{$A_{1}$};
\draw (1,0) node [above]{$A_{2}$};
\draw (2,0) node [above]{$A_{3}$};
\draw (3,0) node [above]{$A_{5}$};
\draw (4,0) node [above]{$A_{6}$};
\draw (5,0) node [above]{$A_{7}$};
\draw (6,0) node [above]{$A_{8}$};
\draw (7,0) node [above]{$A_{9}$};
\draw (8,0) node [above]{$A_{10}$};
\draw (9,0) node [above]{$A_{11}$};
\draw (10,0) node [above]{$A_{12}$};
\draw (11,0) node [above]{$A_{14}$};
\draw (12,0) node [right]{$A_{15}$};

\draw (0,1) node [left]{$A_{19}$};
\draw (0,-1) node [left]{$A_{17}$};
\draw (12,1) node [right]{$A_{18}$};
\draw (12,-1) node [right]{$A_{16}$};
\draw (2,-0.7) node [right]{$A_{4}$};
\draw (10,-0.7) node [right]{$A_{13}$};

\end{tikzpicture}
\end{center} 

The curves $A_{1},\dots,A_{16}$ generate the N\'eron--Severi lattice;
in that basis, the divisor 
\[
D_{304}=(20,41,63,31,55,48,42,37,33,30,28,27,13,14,2,-9)
\]
is ample, of square $304$, with $D_{304} \cdot A_{j}=1$ for $j\leq15$,
$D_{304} \cdot A_{16}=D_{304} \cdot A_{19}=20$ and $D_{304} \cdot A_{17}=D_{304} \cdot A_{18}=2$.
The divisor 
\[
D_{2}=(1,2,3,1,3,3,3,3,3,3,3,3,1,2,1,0)
\]
is nef, base-point free, of square $2$, with $D_{2} \cdot A_{j}=1$ for
$j\in\{4,13,16,17,18,19\}$ and else $D_{2} \cdot A_{j}=0$. One has 
\[
\begin{array}{c}
D_{2}\equiv\left(-A_{4}-A_{13}+\sum_{j=1}^{15}A_{j}\right)+A_{16}+A_{17},\\
D_{2}\equiv\left(-A_{4}-A_{13}+\sum_{j=1}^{15}A_{j}\right)+A_{18}+A_{19}.
\end{array}
\]
By using the linear system $|D_{2}|$, we obtain the following. 

\begin{prop}
The K3 surface is the double cover of $\,\PP^{2}$ branched over a sextic
curve with an $\mathbf{a}_{13}$ singularity. The image of $A_{4}$
and $A_{13}$ is a line, so are the images of $A_{16}$ and $A_{17}$ and of 
$A_{18}$ and $A_{19}$.
\end{prop}

The automorphism group of the  general  K3 surface is $(\ZZ/2\ZZ)^{2}$
(see \cite[Section 4]{Kondo}).

\subsection{The lattice $\boldsymbol{U\oplus\mathbf{E}_{8}\oplus\mathbf{E}_{7}}$}

The K3 surface $X$ contains $19$ $(-2)$-curves $A_{1},\dots,A_{19}$.
The dual graph of these curves is (see \cite[Section 4]{Kondo})

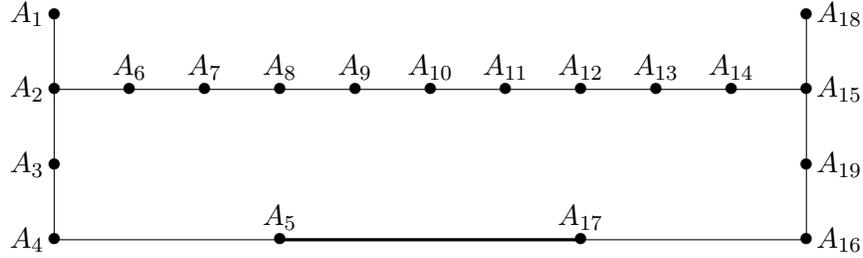
\begin{center}
\begin{tikzpicture}[scale=1]

\draw [very thick] (3,-2) -- (7,-2);
\draw (0,-2) -- (10,-2);
\draw (0,0) -- (10,0);
\draw (0,1) -- (0,-2);
\draw (10,1) -- (10,-2);

\draw (0,0) node {$\bullet$};
\draw (1,0) node {$\bullet$};
\draw (2,0) node {$\bullet$};
\draw (3,0) node {$\bullet$};
\draw (4,0) node {$\bullet$};
\draw (5,0) node {$\bullet$};
\draw (6,0) node {$\bullet$};
\draw (7,0) node {$\bullet$};
\draw (8,0) node {$\bullet$};
\draw (9,0) node {$\bullet$};
\draw (10,0) node {$\bullet$};

\draw (0,1) node {$\bullet$};
\draw (0,-1) node {$\bullet$};
\draw (0,-2) node {$\bullet$};
\draw (3,-2) node {$\bullet$};

\draw (10,1) node {$\bullet$};
\draw (10,-1) node {$\bullet$};
\draw (10,-2) node {$\bullet$};
\draw (7,-2) node {$\bullet$};

\draw (0,0) node [left]{$A_{2}$};
\draw (1,0) node [above]{$A_{6}$};
\draw (2,0) node [above]{$A_{7}$};
\draw (3,0) node [above]{$A_{8}$};
\draw (4,0) node [above]{$A_{9}$};
\draw (5,0) node [above]{$A_{10}$};
\draw (6,0) node [above]{$A_{11}$};
\draw (7,0) node [above]{$A_{12}$};
\draw (8,0) node [above]{$A_{13}$};
\draw (9,0) node [above]{$A_{14}$};

\draw (0,1) node [left]{$A_{1}$};
\draw (0,-1) node [left]{$A_{3}$};
\draw (0,-2) node [left]{$A_{4}$};
\draw (3,-2) node [above]{$A_{5}$};

\draw (10,0) node [right]{$A_{15}$};
\draw (10,1) node [right]{$A_{18}$};
\draw (10,-1) node [right]{$A_{19}$}; 
\draw (10,-2) node [right]{$A_{16}$}; 
\draw (7,-2) node [above]{$A_{17}$};

\end{tikzpicture}
\end{center} 

The curves $A_{1},\dots,A_{17}$ generate the N\'eron--Severi lattice;
in that basis, the divisor 
\[
D_{538}=(32,66,46,27,9,55,45,36,28,21,15,10,6,3,1,0,1)
\]
is ample, of square $538$, with $D_{538} \cdot A_{j}=1$ for $j\neq1,4,17$
and $D_{538} \cdot A_{j}=2,11,16$ for $j=1,4,17$. 

The automorphism group of the  general  K3 surface is $(\ZZ/2\ZZ)^{2}$
(see \cite[Section 4]{Kondo}).

The coarse moduli space $\mathcal{M}^{N}$ of $N=U\oplus\mathbf{E}_{8}\oplus\mathbf{E}_{7}$-polarized
K3 surfaces is studied in \cite{Clingher1}, where it is proved that
this moduli space can be seen naturally as the open set 
\[
\mathcal{M}^{N}=\{[\a,\b,\g,\d]\in\mathbb{WP}^{3}(2,3,5,6)\,|\,\g\neq0\text{ or }\d\neq0\}
\]
of a weighted projective space; thus in particular that moduli is
rational. In \cite{Clingher1} is also given a (singular) model of
the K3 surfaces in $\mathcal{M}^{N}$ as a quartic surface in $\PP^{3}$.
The K3 surfaces with lattice $U\oplus\mathbf{E}_{8}\oplus\mathbf{E}_{7}$
also belong among the ``famous 95'' families discussed in Section
\ref{subsec:About-the-famous95}. 

\subsection{The lattice $\boldsymbol{U\oplus\mathbf{E}_{8}\oplus\mathbf{E}_{8}}$\label{subsec:The-latticeUE8E8}}

The set of $(-2)$-curves $A_{1},\dots,A_{19}$ on the K3 surface $X$
and their configuration had been determined in the classical work
of Vinberg \cite{Vinberg2} (see also \cite[Section 4]{Kondo}):

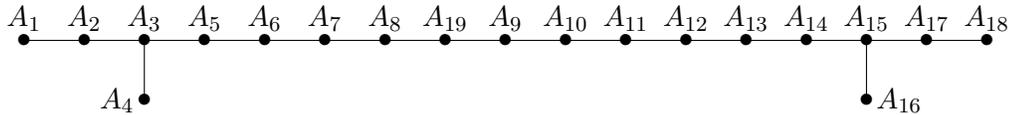
\begin{center}
\begin{tikzpicture}[scale=1]

\draw (0,0) -- (16*0.8,0);
\draw (1.6,0) -- (1.6,-0.8);
\draw (0.8*14,0) -- (0.8*14,-0.8);

\draw (0,0) node {$\bullet$};
\draw (0.8,0) node {$\bullet$};
\draw (2*0.8,0) node {$\bullet$};
\draw (3*0.8,0) node {$\bullet$};
\draw (4*0.8,0) node {$\bullet$};
\draw (5*0.8,0) node {$\bullet$};
\draw (6*0.8,0) node {$\bullet$};
\draw (7*0.8,0) node {$\bullet$};
\draw (8*0.8,0) node {$\bullet$};
\draw (9*0.8,0) node {$\bullet$};
\draw (10*0.8,0) node {$\bullet$};
\draw (11*0.8,0) node {$\bullet$};
\draw (12*0.8,0) node {$\bullet$};
\draw (13*0.8,0) node {$\bullet$};
\draw (14*0.8,0) node {$\bullet$};
\draw (15*0.8,0) node {$\bullet$};
\draw (16*0.8,0) node {$\bullet$};
\draw (0.8*2,-0.8) node {$\bullet$};
\draw (0.8*14,-0.8) node {$\bullet$};

\draw (0,0) node [above]{$A_{1}$};
\draw (0.8,0) node [above]{$A_{2}$};
\draw (2*0.8,0) node [above]{$A_{3}$};
\draw (3*0.8,0) node [above]{$A_{5}$};
\draw (4*0.8,0) node [above]{$A_{6}$};
\draw (5*0.8,0) node [above]{$A_{7}$};
\draw (6*0.8,0) node [above]{$A_{8}$};
\draw (7*0.8,0) node [above]{$A_{19}$};
\draw (8*0.8,0) node [above]{$A_{9}$};
\draw (9*0.8,0) node [above]{$A_{10}$};
\draw (10*0.8,0) node [above]{$A_{11}$};
\draw (11*0.8,0) node [above]{$A_{12}$};
\draw (12*0.8,0) node [above]{$A_{13}$};
\draw (13*0.8,0) node [above]{$A_{14}$};
\draw (14*0.8,0) node [above]{$A_{15}$};
\draw (15*0.8,0) node [above]{$A_{17}$};
\draw (16*0.8,0) node [above]{$A_{18}$};
\draw (0.8*2,-0.8) node [left]{$A_{4}$};
\draw (0.8*14,-0.8) node [right]{$A_{16}$};

\end{tikzpicture}
\end{center} 

The curves $A_{1},\dots,A_{18}$ generate the N\'eron--Severi lattice;
in that base, the divisor 
\[
D_{620}=(-46,-91,-135,-68,-110,-84,-57,-29,30,61,93,126,160,195,231,115,153,76)
\]
is ample, of square $620$, with $D_{620} \cdot A_{j}=1$ for $1\leq j\leq19$.
The divisor 
\[
D_{2}=(-5,-10,-15,-8,-12,-9,-6,-3,3,6,9,12,15,18,21,10,14,7)
\]
is nef, of square $2$, with $D_{2} \cdot A_{4}=D_{2} \cdot A_{16}=1$ and $D_{2} \cdot A_{j}=0$
for $j\neq4,16$. Using $|D_{2}|$, we see that the K3 surface $X$
is the double cover of $\PP^{2}$ branched over a sextic curve with
an $\mathbf{a}_{17}$ singularity. We have 
\[
D_{2}\equiv A_{1}+2A_{2}+3A_{3}+A_{4}+3\left(\sum_{j=5}^{15}A_{j}\right)+A_{16}+2A_{17}+A_{18}+3A_{19};
\]
thus the image  of $A_{4}$, $A_{16}$ by the double cover map is a line
through the singularity. 

The automorphism group of the  general  K3 surface is $(\ZZ/2\ZZ)^{2}$
(see \cite[Section 4]{Kondo}).

The rich geometry of $U\oplus\mathbf{E}_{8}\oplus\mathbf{E}_{8}$-polarized
K3 surfaces is studied in \cite{Clingher2}. These K3 surfaces also
belong among the ``famous 95'' families discussed in Section \ref{subsec:About-the-famous95}. 

\subsection{The lattice $\boldsymbol{U\oplus\mathbf{E}_{8}\oplus\mathbf{E}_{8}\oplus\mathbf{A}_{1}}$}

The K3 surface $X$ contains $24$ $(-2)$-curves $A_{1},\dots,A_{24}$.
The dual graph of these curves is (see \cite[Section 4]{Kondo}) 
\begin{center}
\begin{tikzpicture}[scale=1]

\draw (-3,0) -- (3,0);
\draw (0,5) -- (3,0);
\draw (0,5) -- (-3,0);
\draw (-3/2,5/2) -- (3,0);
\draw (3/2,5/2) -- (-3,0);
\draw (0,0) -- (0,5);

\draw (0,5) node {$\bullet$};
\draw (1/2,-5/3*1/2 + 5) node {$\bullet$};
\draw (1, -5/3*2/2 + 5) node {$\bullet$};
\draw (3/2,-5/3*3/2 + 5) node {$\bullet$};
\draw (2,-5/3*4/2 + 5) node {$\bullet$};
\draw (5/2, -5/3*5/2 + 5) node {$\bullet$};
\draw (3, -5/3*6/2 + 5) node {$\bullet$};

\draw (-1/2,5/3*-1/2 + 5) node {$\bullet$};
\draw (-1, 5/3*-2/2 + 5) node {$\bullet$};
\draw (-3/2,5/3*-3/2 + 5) node {$\bullet$};
\draw (-2,5/3*-4/2 + 5) node {$\bullet$};
\draw (-5/2, 5/3*-5/2 + 5) node {$\bullet$};
\draw (-3, 5/3*-6/2 + 5) node {$\bullet$};

\draw (0.75,5/9*0.75+5/3) node {$\bullet$};
\draw (-0.75,5/9*0.75+5/3) node {$\bullet$};
\draw (0, 5/6) node {$\bullet$};
\draw (0.75,5/9*0.75+5/3) node [above]{$A_{22}$};
\draw (-0.75,5/9*0.75+5/3) node [above]{$A_{23}$};
\draw (0, 5/6) node  [right]{$A_{24}$};

\draw [dotted][very thick] (0, 5/6) -- (-0.75,5/9*0.75+5/3);
\draw [dotted][very thick] (0, 5/6) -- (0.75,5/9*0.75+5/3);
\draw [dotted][very thick] (0.75,5/9*0.75+5/3) -- (-0.75,5/9*0.75+5/3);

\draw (-1.5, 5/9*-1.5+5/3) node {$\bullet$};
\draw (1.5, 5/9*-1.5+5/3) node {$\bullet$};
\draw (0, -5/3*2/2 + 5) node {$\bullet$};
\draw [very thick] (0, -5/3*2/2 + 5) -- (0,0);
\draw [very thick] (-1.5, 5/9*-1.5+5/3) -- (1.5,-5/3*3/2 + 5);
\draw [very thick] (1.5, 5/9*-1.5+5/3) -- (-1.5,-5/3*3/2 + 5);
\draw (-1.5, 5/9*-1.5+5/3) node [below]{$A_{19}$};
\draw (1.5, 5/9*-1.5+5/3) node [below]{$A_{20}$};
\draw (0, -5/3*2/2 + 5) node [right]{$A_{21}$};

\draw (-3,0) node {$\bullet$};
\draw (-2,0) node {$\bullet$};
\draw (-1,0) node {$\bullet$};
\draw (0,0) node {$\bullet$};
\draw (1,0) node {$\bullet$};
\draw (2,0) node {$\bullet$};
\draw (3,0) node {$\bullet$};

\draw (-3,0) node [below]{$A_{1}$};
\draw (-2,0) node [below]{$A_{2}$};
\draw (-1,0) node [below]{$A_{3}$};
\draw (0,0) node [below]{$A_{4}$};
\draw (1,0) node [below]{$A_{5}$};
\draw (2,0) node [below]{$A_{6}$};
\draw (3,0) node [below]{$A_{7}$};

\draw (1/2,-5/3*1/2 + 5) node [right]{$A_{12}$};
\draw (1, -5/3*2/2 + 5) node  [right]{$A_{11}$};
\draw (3/2,-5/3*3/2 + 5) node  [right]{$A_{10}$};
\draw (2,-5/3*4/2 + 5) node  [right]{$A_{9}$};
\draw (5/2, -5/3*5/2 + 5) node  [right]{$A_{8}$};
\draw (0,5) node [above]{$A_{13}$};
\draw (-1/2,5/3*-1/2 + 5) node [left]{$A_{14}$};
\draw (-1, 5/3*-2/2 + 5) node  [left]{$A_{15}$};
\draw (-3/2,5/3*-3/2 + 5) node  [left]{$A_{16}$};
\draw (-2,5/3*-4/2 + 5) node  [left]{$A_{17}$};
\draw (-5/2, 5/3*-5/2 + 5) node  [left]{$A_{18}$};

\end{tikzpicture}
\end{center} 

Here the dotted thick segments indicate an intersection number equal
to $6$ between the curves. The automorphism group of the  general 
K3 surface is $\mathfrak{S}_{3}\times\ZZ/2\ZZ$ (see \cite[Section 4]{Kondo}).

\newpage

\section{Table: Number of ($-$2)-curves}\label{sec:table}

{ \setlength\extrarowheight{3.2pt}
\begin{xltabular}{\linewidth}{|c|c|c|c|c|c|c|c|c|c|}
\hline 
Lattice & $\#\cu\,$ & $\,\,\,$ & $\aut$ &  & Lattice & $\#\cu$ & $\,\,\,$ & $\aut$ & \tabularnewline
\hline
\textbf{Rank} $3$ &  &  &  &  & \textbf{Rank} $3$ &  &  &  & \tabularnewline\hline 
$S_{1}$ & $6$ & $\,\mathbf{u}\,$ & $\zd$ & $n$ & $S_{1,3,1}$ & $3$ & $\,\mathbf{u}\,$ & $\triv$ & \tabularnewline
\hline 
$S_{2}$ & $6$ & $\,\mathbf{u}\,$ & $\triv$ &  & $S_{1,4,1}$ & $4$ & $\,\mathbf{u}\,$ & $\zd$ & $n$\tabularnewline
\hline 
$S_{3}$ & $4$ & $\,\mathbf{u}\,$ & $\triv$ &  & $S_{1,5,1}$ & $6$ & $\,\mathbf{u}\,$ & $\zd$ & $n$\tabularnewline
\hline 
$S_{4}$ & $4$ &  & $\zd$ & $n$ & $S_{1,6,1}$ & $4$ & $\,\mathbf{u}\,$ & $\triv$ & \tabularnewline
\hline 
$S_{5}$ & $4$ & $\,\mathbf{u}\,$ & $\zd$ & $n$ & $S_{1,9,1}$ & $9$ &  & $\triv$ & \tabularnewline
\hline 
$S_{6}$ & $6$ & $\,\mathbf{u}\,$ & $\zd$ & $n$ & $S_{4,1,1}$ & $3$ & $\,\mathbf{u}\,$ & $\triv$ & \tabularnewline
\hline 
$S_{1,1,1}$ & $3$ & $\,\aleph\,\,\,\mathbf{u}\,$ & $\zd$ & $t$ & $S_{5,1,1}$ & $4$ & $\,\mathbf{u}\,$ & $\zd$ & $n$\tabularnewline
\hline 
$S_{1,1,2}$ & $3$ & $\,\aleph\,\,\,\mathbf{u}\,$ & $\zd$ & $t$ & $S_{6,1,1}$ & $4$ & $\,\mathbf{u}\,$ & $\triv$ & \tabularnewline
\hline 
$S_{1,1,3}$ & $4$ & $\,\mathbf{u}\,$ & $\zd$ & $n$ & $S_{7,1,1}$ & $6$ &  & $\triv$ & \tabularnewline
\hline 
$S_{1,1,4}$ & $4$ & $\,\mathbf{u}\,$ & $\triv$ &  & $S_{8,1,1}$ & $4$ & $\,\mathbf{u}\,$ & $\triv$ & \tabularnewline
\hline 
$S_{1,1,6}$ & $6$ &  & $\zd$ & $n$ & $S_{10,1,1}$ & $8$ &  & $\zd$ & $n$\tabularnewline
\hline 
$S_{1,1,8}$ & $8$ &  & $\triv$ &  & $S_{12,1,1}$ & $6$ &  & $\triv$ & \tabularnewline
\hline 
$S_{1,2,1}$ & $3$ & $\,\mathbf{u}\,$ & $\zd$ & $n$ & $S_{4,1,2}'$ & $4$ &  & $\zd$ & $n$\tabularnewline
\hline
&  &  &  &  &  &  &  &  & \tabularnewline
\hline
\textbf{Rank} $4$ &  &  &  &  & \textbf{Rank} $5$ &  &  &  & \tabularnewline
\hline 
$L(24)$ & $6$ &  & $\zd$ & $n$ & 
$U\oplus\mathbf{A}_{1}^{\oplus3}$ & $7$ &  & $\zd$ & $t$ 
\tabularnewline
\hline 
$L(27)$ & $8$ &  & $\zd$ & $n$ &
$U(2)\oplus\mathbf{A}_{1}^{\oplus3}$ & $10$ &  & $\zd$ & $t$ 
\tabularnewline
\hline 
$[8]\oplus\mathbf{A}_{1}^{\oplus3}$ & $12$ & $\,\mathbf{u}\,$ & $\zd$ & $n$ &
$U(4)\oplus\mathbf{A}_{1}^{\oplus3}$ & $24$ &  & $\zd$ & $n$
\tabularnewline
\hline 
$U\oplus\mathbf{A}_{1}^{\oplus2}$ & $5$ & $\,\angle\,\,\mathbf{u}\,$ & $\zd$ & $t$ &
$U\oplus\mathbf{A}_{1}\oplus\mathbf{A}_{2}$ & $6$ & $\angle\star$ & $\zd$ & $n$ 
\tabularnewline
\hline 
$U(2)\oplus\mathbf{A}_{1}^{\oplus2}$ & $6$ & $\,\mathbf{u}\,$ & $\zd$ & $t$ &
$U\oplus\mathbf{A}_{3}$ & $5$ & $\angle\star$ & $\zd$ & $n$ 
\tabularnewline
\hline 
$U(3)\oplus\mathbf{A}_{1}^{\oplus2}$ & $8$ & $\,\mathbf{u}\,$ & $\zd$ & $n$ &
$[4]\oplus\mathbf{D}_{4}$ & $5$ &  & $\zd$ & $n$ 
\tabularnewline
\hline 
$U(4)\oplus\mathbf{A}_{1}^{\oplus2}$ & $8$ & $\,\mathbf{u}\,$ & $\triv$ &  &
$[8]\oplus\mathbf{D}_{4}$ & $7$ &  & $\zd$ & $n$ 
\tabularnewline
\hline 
$U\oplus\mathbf{A}_{2}$ & $4$ & $\aleph\star\,\mathbf{u}\,$ & $\zd$ & $n$ &
$[16]\oplus\mathbf{D}_{4}$ & $8$ &  & $\zd$ & $n$ 
\tabularnewline
\hline 
$U(2)\oplus\mathbf{A}_{2}$ & $4$ & $\,\mathbf{u}\,$ & $\zd$ & $n$ &
$[6]\oplus\mathbf{A}_{2}^{\oplus2}$ & $10$ & $\,\mathbf{u}\,$ & $\zd$ & $n$ 
\tabularnewline
\hline 
$U(3)\oplus\mathbf{A}_{2}$ & $4$ & $\,\angle\,\,\mathbf{u}\,$ & $\triv$ &  &
& & & & 
\tabularnewline
\hline 
$U(6)\oplus\mathbf{A}_{2}$ & $6$ & $\,\mathbf{u}\,$ & $\triv$ &  &
&  &  &  &
\tabularnewline
\hline 
$L_{12}=S_{0}\oplus\mathbf{A}_{2}$ & $6$ & $\,\mathbf{u}\,$ & $\zd$ & $n$ &
&  &  &  &
\tabularnewline
\hline 
$[4]\oplus[-4]\oplus\mathbf{A}_{2}$ & $6$ &  & $\zd$ & $n$ &
&  &  &  &
\tabularnewline
\hline 
$[4]\oplus\mathbf{A}_{3}$ & $5$ & $\,\mathbf{u}\,$ & $\zd$ & $n$ &
&  &  &  &
\tabularnewline
\hline 
\end{xltabular}
}

\pagebreak

{ \setlength\extrarowheight{3.2pt}
\begin{xltabular}{\linewidth}{|c|c|c|c|c|c|c|c|c|c|}
\hline 
Lattice & $\#\cu\,$ & $\,\,\,$ & $\aut$ &  & Lattice & $\#\cu$ & $\,\,\,$ & $\aut$ & \tabularnewline
\hline
\textbf{Rank} $6$ &  &  &  &  &
\textbf{Rank} $9$ &  &  &  & \tabularnewline
\hline 
$U(3)\oplus\mathbf{A}_{2}^{\oplus2}$ & $12$ & $\,\mathbf{u}\,$ & $\zd$ & $n$&
$U\oplus\mathbf{E}_{7}$ & $9$ & $\,\aleph\,\,\mathbf{u}\,$ & $\zd$ & $t$
\tabularnewline
\hline 
$U(4)\oplus\mathbf{D}_{4}$ & $9$ &  & $\zd$ & $n$& 
$U\oplus\mathbf{D}_{6}\oplus\mathbf{A}_{1}$ & $10$ &  & $\zd$ & $t$
\tabularnewline
\hline 
$U\oplus\mathbf{A}_{4}$ & $6$ & $\,\star\,$ & $\zd$ & $n$& 
$U\oplus\mathbf{D}_{4}\oplus\mathbf{A}_{1}^{\oplus3}$ & $15$ &  & $\zd$ & $t$ 
\tabularnewline
\hline 
$U\oplus\mathbf{A}_{1}\oplus\mathbf{A}_{3}$ & $7$ & $\,\star\,$ & $\zd$ & $n$& 
$U\oplus\mathbf{A}_{1}^{\oplus7}$ & $37$ & $\,\mathbf{u}\,$ & $\zd$ & $t$ 
\tabularnewline
\hline 
$U\oplus\mathbf{A}_{2}^{\oplus2}$ & $7$ & $\angle\star$ & $\zd$ & $n$& 
$U(2)\oplus\mathbf{A}_{1}^{\oplus7}$ & $240$ & $\,\mathbf{u}\,$ & $\zde$ & $tn$ 
\tabularnewline
\hline 
$U\oplus\mathbf{A}_{1}^{\oplus2}\oplus\mathbf{A}_{2}$ & $8$ & $\,\star\,$ & $\zd$ & $n$& 
$U\oplus\mathbf{A}_{7}$ & $9$ & $\,\star\,$ & $\zd$ & $n$
\tabularnewline
\hline 
$U(2)\oplus\mathbf{A}_{1}^{\oplus4}$ & $16$ & $\,\mathbf{u}\,$ & $\zd$ & $t$& 
$U\oplus\mathbf{D}_{4}\oplus\mathbf{A}_{3}$ & $10$ & $\,\star\,$ & $\zd$ & $n$
\tabularnewline
\hline 
$U\oplus\mathbf{A}_{1}^{\oplus4}$ & $9$ &  & $\zd$ & $t$& 
$U\oplus\mathbf{D}_{5}\oplus\mathbf{A}_{2}$ & $10$ & $\,\star\,$ & $\zd$ & $n$

\tabularnewline
\hline 
$U(2)\oplus\mathbf{D}_{4}$ & $6$ & $\,\aleph\,\,\mathbf{u}\,$ & $\zd$ & $t$& 
$U\oplus\mathbf{D}_{7}$ & $9$ & $\,\star\,$ & $\zd$ & $t$
\tabularnewline
\hline
$U\oplus\mathbf{D}_{4}$ & $6$ & $\,\aleph\,\,\mathbf{u}\,$ & $\zd$ & $t$&
$U\oplus\mathbf{A}_{1}\oplus\mathbf{E}_{6}$ & $10$ & $\aleph\star\,\mathbf{u}\,$ & $\zd$ & $n$
\tabularnewline
\hline
&  &  &  &  &  &  &  &  & \tabularnewline
\hline
\textbf{Rank} $7$ &  &  &  &  & \textbf{Rank} $10$ &  &  &  & \tabularnewline
\hline 
$U\oplus\mathbf{D}_{4}\oplus\mathbf{A}_{1}$ & $8$ &  & $\zd$ & $t$ &
 $U\oplus\mathbf{E}_{8}$ & $10$ & $\aleph\,\mathbf{u}\,$ & $\zd$ & $t$
\\
\hline
$U\oplus\mathbf{A}_{1}^{\oplus5}$ & $12$ & $\,\mathbf{u}\,$ & $\zd$ & $t$ &
 $U\oplus\mathbf{D}_{8}$ & $10$ &  & $\zd$ & $t$ 
\\
\hline
$U(2)\oplus\mathbf{A}_{1}^{\oplus5}$ & $27$ & $\,\mathbf{u}\,$ & $\zd$ & $t$ &
 $U\oplus\mathbf{E}_{7}\oplus\mathbf{A}_{1}$ & $11$ & $\,\aleph\,\,\mathbf{u}\,$ & $\zd$ & $t$ 
\\
\hline
$U\oplus\mathbf{A}_{1}\oplus\mathbf{A}_{2}^{\oplus2}$ & $9$ & $\,\star\,$ & $\zd$ & $n$ &
 $U\oplus\mathbf{D}_{4}^{\oplus2}$ & $11$ & $\,\aleph\,\,\mathbf{u}\,$ & $\zd$ & $t$ 
\\
\hline
$U\oplus\mathbf{A}_{1}^{\oplus2}\oplus\mathbf{A}_{3}$ & $9$ & $\,\star\,$ & $\zd$ & $n$ &
$U\oplus\mathbf{D}_{6}\oplus\mathbf{A}_{1}^{\oplus2}$ & $14$ &  & $\zd$ & $t$  
\\
\hline

$U\oplus\mathbf{A}_{2}\oplus\mathbf{A}_{3}$ & $8$ & $\angle\star$ & $\zd$ & $n$&
$U(2)\oplus\mathbf{D}_{4}^{\oplus2}$ & $18$ & $\,\aleph\,\,\mathbf{u}\,$ & $\zde$ & $tn$ 
\\
\hline
$U\oplus\mathbf{A}_{1}\oplus\mathbf{A}_{4}$ & $8$ & $\,\star\,$ & $\zd$ & $n$&
$U\oplus\mathbf{D}_{4}\oplus\mathbf{A}_{1}^{\oplus4}$ & $27$ & $\,\mathbf{u}\,$ & $\zd$ & $t$ 
\\
\hline

$U\oplus\mathbf{A}_{5}$ & $7$ & $\angle\star$ & $\zd$ & $n$&
 $U\oplus\mathbf{A}_{1}^{\oplus8}$ & $145$ &  & $\zde$ & $tn$ 
\\
\hline
$U\oplus\mathbf{D}_{5}$ & $7$ & $\aleph\star\,\mathbf{u}\,$ & $\zd$ & $n$&
$U\oplus\mathbf{A}_{2}\oplus\mathbf{E}_{6}$ & $11$ & $\aleph\star\,\mathbf{u}\,$ & $\zd$ & $n$
\\
\hline
&  &  &  &  &  &  &  &  & \tabularnewline
\hline
\textbf{Rank} $8$ &  &  &  &  & \textbf{Rank} $11$ &  &  &  & \tabularnewline
\hline 
$U\oplus\mathbf{D}_{6}$ & $8$ &  & $\zd$ & $t$ & $U\oplus\mathbf{E}_{8}\oplus\mathbf{A}_{1}$ & $12$ & $\,\aleph\,\,\mathbf{u}\,$ & $\zd$ & $t$\tabularnewline
\hline 
$U\oplus\mathbf{D}_{4}\oplus\mathbf{A}_{1}^{\oplus2}$ & $10$ &  & $\zd$ & $t$ & $U\oplus\mathbf{D}_{8}\oplus\mathbf{A}_{1}$ & $14$ &  & $\zd$ & $t$\tabularnewline
\hline 
$U\oplus\mathbf{A}_{1}^{\oplus6}$ & $19$ & $\,\mathbf{u}\,$ & $\zd$ & $t$ & $U\oplus\mathbf{D}_{4}^{\oplus2}\oplus\mathbf{A}_{1}$ & $22$ &  & $\zd$ & $t$\tabularnewline
\hline 
$U(2)\oplus\mathbf{A}_{1}^{\oplus6}$ & $56$ & $\,\mathbf{u}\,$ & $\zd$ & $t$ & $U\oplus\mathbf{D}_{4}\oplus\mathbf{A}_{1}^{\oplus5}$ & $90$ &  & $\zde$ & $tn$\tabularnewline
\hline 
$U\oplus\mathbf{A}_{2}^{\oplus3}$ & $10$ & $\,\star\,$ & $\zd$ & $n$ &  &  &  &  & \tabularnewline
\hline 
$U\oplus\mathbf{A}_{3}^{\oplus2}$ & $9$ & $\,\star\,$ & $\zd$ & $n$ &  &  &  &  & \tabularnewline
\hline 
$U\oplus\mathbf{A}_{2}\oplus\mathbf{A}_{4}$ & $9$ & $\,\star\,$ & $\zd$ & $n$ & \textbf{Rank} $12$ &  &  &  & \tabularnewline
\hline 
$U\oplus\mathbf{A}_{1}\oplus\mathbf{A}_{5}$ & $9$ & $\,\star\,$ & $\zd$ & $t$ & $U\oplus\mathbf{E}_{8}\oplus\mathbf{A}_{1}^{\oplus2}$ & $14$ &  & $\zd$ & $t$\tabularnewline
\hline 
$U\oplus\mathbf{A}_{6}$ & $8$ & $\,\star\,$ & $\zd$ & $n$ & $U\oplus\mathbf{D}_{8}\oplus\mathbf{A}_{1}^{\oplus2}$ & $19$ &  & $\zd$ & $t$\tabularnewline
\hline 
$U\oplus\mathbf{A}_{2}\oplus\mathbf{D}_{4}$ & $9$ & $\aleph\star\,\mathbf{u}\,$ & $\zd$ & $n$ & $U\oplus\mathbf{D}_{4}^{\oplus2}\oplus\mathbf{A}_{1}^{\oplus2}$ & $59$ &  & $\zde$ & $tn$\tabularnewline
\hline 
$U\oplus\mathbf{A}_{1}\oplus\mathbf{D}_{5}$ & $9$ & $\,\star\,$ & $\zd$ & $n$ & $U\oplus\mathbf{A}_{2}\oplus\mathbf{E}_{8}$ & $13$ & $\aleph\star\,\mathbf{u}\,$ & $\zd$ & $n$\tabularnewline
\hline 
$U\oplus\mathbf{E}_{6}$ & $8$ & $\aleph\star\,\mathbf{u}\,$ & $\zd$ & $n$ &  &  &  &  & \tabularnewline
\hline 
Lattice & $\#\cu\,$ & $\,\,\,$ & $\aut$ &  & Lattice & $\#\cu$ & $\,\,\,$ & $\aut$ &
\tabularnewline
\hline
\textbf{Rank} $13$ &  &  &  &  & \textbf{Rank} $\geq15$ &  &  &  & \\
\hline
$U\oplus\mathbf{E}_{8}\oplus\mathbf{A}_{1}^{\oplus3}$ & $17$ & $\,\aleph\,\,\mathbf{u}\,$ & $\zd$ & $t$ &
$U\oplus\mathbf{E}_{8}\oplus\mathbf{D}_{4}\oplus\mathbf{A}_{1}$ & $21$ &  & $\zde$ & $tn$ 
\tabularnewline
\hline
$U\oplus\mathbf{D}_{8}\oplus\mathbf{A}_{1}^{\oplus3}$ & $39$ & $\aleph\,\mathbf{u}\,$ & $\zde$ & $tn$ &
$U\oplus\mathbf{E}_{8}\oplus\mathbf{D}_{6}$ & $19$ & $\,\aleph\,\,\mathbf{u}\,$ & $\zde$ & $tn$ 
\tabularnewline
\hline
$U\oplus\mathbf{E}_{8}\oplus\mathbf{A}_{3}$ & $14$ & $\,\star\,$ & $\zd$ & $n$ &
$U\oplus\mathbf{E}_{8}\oplus\mathbf{E}_{7}$ & $19$ & $\,\aleph\,\,\mathbf{u}\,$ & $\zde$ & $tn$ 
\tabularnewline
\hline
& & & & &
$U\oplus\mathbf{E}_{8}\oplus\mathbf{E}_{8}$ & $19$ & $\,\aleph\,\,\mathbf{u}\,$ & $\zde$ & $tn$ 
\tabularnewline
\hline
\textbf{Rank} $14$ &  &  & & &
$U\oplus\mathbf{E}_{8}\oplus\mathbf{E}_{8}\oplus\mathbf{A}_{1}$ & $24$ & $\,\aleph\,\,\mathbf{u}\,$ & $\tss$ & $tn$
\\
\hline
$U\oplus\mathbf{E}_{8}\oplus\mathbf{D}_{4}$ & $15$ & $\,\aleph\,\,\mathbf{u}\,$ & $\zd$ & $t$& 

&&&&
\tabularnewline
\hline
$U\oplus\mathbf{D}_{8}\oplus\mathbf{D}_{4}$ & $20$ & $\,\aleph\,\,\mathbf{u}\,$ & $\zde$ & $tn$& 
&&&&
\tabularnewline
\hline
$U\oplus\mathbf{E}_{8}\oplus\mathbf{A}_{1}^{\oplus4}$ & $27$ &  & $\zde$ & $t$& 
&&&&
\tabularnewline
\hline

\end{xltabular}
}

This table give the number of $(-2)$-curves. A $\,\aleph\,$ means
that the lattice is among the $95$ famous families; a $\,\angle\,$
means that the lattice is a mirror of one of the $95$ famous ones, but
not one of the $95$ (see Section \ref{subsec:About-the-famous95}).
A $\star$ means that their $\cu$-curve configuration is predicted
by Theorem \ref{thm:Not-2-elementaryAllMinus2}. A $\,\mathbf{u}\,$
means that the unirationality of the moduli space is proved (the absence
of $\,\mathbf{u}\,$ does not exclude the possibility of unirationality).
The column $\aut$ gives the automorphism group of the general K3
surface. A $t$ means that the action on the N\'eron--Severi lattice of
a hyperelliptic involution is trivial, an $n$ means that the action
of a hyperelliptic involution is not trivial, and a $tn$ means that both
cases exists. From that data, one can recover the kernel of the map
$\aut(X)\to O(\NS X))$.


\newcommand{\etalchar}[1]{$^{#1}$}


\begin{thebibliography}{BHP{\etalchar{+}}04+++}

\bibitem[AN06]{AN} V.~Alexeev and V.\,V.~Nikulin, \emph{Del Pezzo and K3 surfaces}, MSJ Mem.\ 15, Math.\ Soc.\ Japan, Tokyo, 2006.

\bibitem[ACL21]{ACL} M.~Artebani, C.~Correa Diesler and  A.~Laface, \emph{Cox rings of K3 surfaces of Picard number three},  J.~Algebra \textbf{565} (2021), 598--626 

\bibitem[ACR20]{ACR} M.~Artebani,  C.~Correa Deisler and X.~Roulleau, \emph{Mori dream K3 surfaces of Picard number four: projective models and Cox rings}, Israel J.\ of Math.\ (to appear), \arXiv{2011.00475}.
  

\bibitem[AHL10]{AHL} M.~Artebani, J.~Hausen and A.~Laface, \emph{On Cox rings of K3 surfaces}, Compos.\ Math.\ \textbf{146} (2010), no.~4, 964--998. 

\bibitem[AK11]{AK} M.~Artebani and S.~Kondo, \emph{The moduli of curves of genus six and K3 surfaces},  Trans.\ Amer.\ Math.\ Soc.\ \textbf{363} (2011), no.~3, 1445--1462. 


\bibitem[BHP{\etalchar{+}}04]{BHPV} W.~Barth, K.~Hulek, C.~Peters and A.~Van De Ven, \emph{Complex Compact Surfaces}, 2nd ed., Ergebn.\  Math\ Grenzgeb.~(2), Springer-Verlag, Berlin, 2004.

\bibitem[Bel02]{Belcastro} S.-M.~Belcastro, \emph{Picard lattices of families of K3 surfaces}, Comm.\ Algebra \textbf{30} (2002), no.~1, 61--82. 

\bibitem[BCP97]{Magma} W.~Bosma, J.~Cannon and C.~Playoust, \emph{The Magma algebra system. I. The user language}, In: \emph{Comput.\ alg.\ and num.\ th.} (London, 1993), J.~Symbolic Comput.\ \textbf{24}, 1997, no.~3--4, 235--265. 

\bibitem[CD07]{Clingher2} A.~Clingher and C.~Doran, \emph{Modular invariants for lattice polarized K3 surfaces},  Michigan Math.~J.\ \textbf{55} (2007), no.~2, 355--393. 


\bibitem[CD12]{Clingher1} \bysame, \emph{Lattice Polarized K3 Surfaces and Siegel Modular Forms}, Adv.\ Math.\ \textbf{231} (2012), no.~1, 172--212. 


\bibitem[Cox50]{Coxeter} H.\,S.\,M.~Coxeter, \emph{Self-dual configurations and regular graphs}, Bull.\ Amer.\ Math.\ Soc.\ \textbf{56} (1950), 413--455. 

\bibitem[Deg19]{Degtyarev} A.~Degtyarev, \emph{Tritangents to smooth sextic curves},  Ann.\ Inst.\ Fourier (to appear), \arXiv{1909.05657}. 

\bibitem[Dem76]{DemazureIV} M.~Demazure, \emph{Surfaces de Del Pezzo : IV - Syst\`emes anticanoniques}, S\'eminaire sur les singularit\'es des surfaces (Polytechnique) (1976-1977), exp.\ no.~6, p. 1--11. Available from \url{http://www.numdam.org/item/?id=SSS\_1976-1977\_\_\_\_A7\_0}. 

\bibitem[Dol96]{Dolga} I.~Dolgachev, \emph{Mirror symmetry for lattice polarized K3 surfaces}, In: \emph{Algebraic geometry, 4},  J.~Math.\ Sci.\ \textbf{81} (1996), no.~3, 2599--2630.
  
\bibitem[DS20]{DS} I.~Dolgachev and I.~Shimada, \emph{$15$-nodal quartic surfaces. Part II: The automorphism group}, Rend.\ Circ.\ Mat.\ Palermo (2) \textbf{69} (2020), 1165--1191. 

\bibitem[Elk09]{Elkies} N.~Elkies, \emph{About the cover: Rational curves on a K3 surface}, In: \emph{Arithmetic geometry, 1--4},  Clay Math.\ Proc.\ 8, Amer.\ Math.\ Soc., Providence, RI, 2009. 
  
\bibitem[GLP10]{GLP} F.~Galluzzi, G.~Lombardo and C.~Peters, \emph{Automorphs of indefinite binary quadratic forms and K3-surfaces with Picard number 2}, Rend.\ Semin.\ Mat.\ Univ.\ Politec.\ Torino \textbf{68} (2010), no.~1, 57--77. 

\bibitem[Har92]{Harris} J.~Harris, \emph{Algebraic Geometry, A First Course}, Grad.\ Texts in Math.\ 133, Springer-Verlag, New York, 1992.

\bibitem[Har77]{Hartshorne} R.~Hartshorne, \emph{Algebraic geometry}, Grad.\ Texts in Math.\ 52. Springer-Verlag, New York--Heidelberg, 1977. 

\bibitem[Huy16]{Huybrechts} D.~Huybrechts, \emph{Lectures on K3 surfaces}, Cambridge Stud.\ Adv.\ Math.\ 158, Cambridge Univ.\ Press, Cambridge, 2016.  

\bibitem[KM12]{KM} M.~Kobayashi and M.~Mase, \emph{Isomorphism among families of weighted K3 hypersurfaces},  Tokyo J.~Math.\ \textbf{35} (2012), no.~2, 461--467. 

\bibitem[Kon86]{Kondo86} S.~Kondo, \emph{On automorphisms of algebraic K3 surfaces which act trivially on Picard groups}, Proc.\ Japan Acad., Ser. A \textbf{62}, 356--359 (1986).

\bibitem[Kon89]{Kondo} \bysame, \emph{Algebraic K3 surfaces with finite automorphism groups},  Nagoya Math.~J.\ \textbf{116} (1989), 1--15. 

\bibitem[Kon00]{Kondo3} \bysame, \emph{A complex hyperbolic structure for the moduli space of curves of genus three}, J.~reine angew.\ Math.\ \textbf{525} (2000), 219--232. 

\bibitem[Kon02]{Kondo2} \bysame, \emph{The Moduli Space of Curves of Genus 4 and Deligne-Mostow's Complex Reflection Groups}, In: \emph{Algebraic Geometry 2000} (Azumino, 2000), pp.~383--400, Adv.\ Stud.\ Pure Math.\ 36, Math.\ Soc.\ Japan, Tokyo, 2002.    
  
\bibitem[Kov94]{Kovacs} S.\,J.~Kov\'{a}cs, \emph{The cone of curves on a K3 surface}, Math.\ Annalen \textbf{300}, (1994), no. 4, 681--692.

\bibitem[vLu07]{vL} R.~van Luijk, \emph{K3 surfaces with Picard number one and infinitely many rational points}, Algebra Number Theory \textbf{1} (2007), no.~1, 1--15. 

\bibitem[McK10]{McK} J.~McKernan, \emph{Mori dream spaces}, Japan.\ J.~Math.\ \textbf{5} (2010), no.~1, 127--151. 

\bibitem[Mor84]{Morrison} D.\,R.~Morrison, \emph{On K3 surfaces with large Picard number},  Invent.\ Math.\ \textbf{75} (1984), no.~1, 105--121. 


\bibitem[Muk88a]{Mukai2} S.~Mukai, \emph{Curves, K3 Surfaces and Fano 3-folds of Genus $\leq 10$}, In: \emph{Algebraic geometry and commutative algebra, Vol. I}, pp.~357--377, Kinokuniya, Tokyo, 1988.
  
\bibitem[Muk88b]{MukaiTri} \bysame, \emph{Finite groups of automorphisms of K3 surfaces and the Mathieu group}, Invent.\ Math.\ \textbf{94} (1988), no.~1, 183--221.

\bibitem[Muk92]{Mukai3} \bysame, \emph{Polarized K3 surfaces of genus $18$ and $20$}, In: \emph{Complex projective geometry}, pp.~264--276, London Math.\ Soc.\ Lecture Note Ser.\ 179, Cambridge Univ.\ Press, Cambridge, 1992.

\bibitem[Nik79]{Nikulin3} V.\,V.~Nikulin, \emph{Integer symmetric bilinear forms and some of their geometric applications}, Izv.\ Akad.\ Nauk SSSR Ser.\ Mat.\ \textbf{43} (1979), no.~1, 111--177. 

\bibitem[Nik80]{NikulinFinite} \bysame, \emph{Finite groups of automorphisms of K\"ahler K3 surfaces}, Tr.\ Mosk.\ Mat.\ Obs. \textbf{38} (1980), 71--135. 

\bibitem[Nik83]{Nikulin} \bysame, \emph{Factor groups of automorphisms of hyperbolic forms with respect to subgroup generated by $2$-reflections, algebro-geometric applications}, J.\ Soviet Math.\ \textbf{22} (1983), 1401--1475.
  
\bibitem[Nik85]{Nikulin2} \bysame, \emph{Surfaces of type K3 with a finite automorphism group and a Picard group of rank three}, Proc.\ Steklov Inst.\ Math.\ \textbf{3} (1985), 131--155.

\bibitem[Nik14]{NikulinElliptic} \bysame, \emph{Elliptic fibrations on K3 surfaces},  Proc.\ Edinb.\ Math.\ Soc.\ (2) \textbf{57} (2014), no.~1, 253--267.

\bibitem[Ott13]{Ottem} J.\,C.~Ottem, \emph{Cox rings of K3 surfaces with Picard number $2$}, J.~Pure Appl.\ Algebra \textbf{217} (2013), no.~4, 709--715. 

\bibitem[PS71]{PS} I.\,I.~Pjateckii-Sapiro and I.\,R.~Shafarevich, \emph{Torelli's theorem for algebraic surfaces of type K3} (Russian), Izv.\ Akad.\ Nauk SSSR Ser.\ Mat.\ \textbf{35} (1971), 530--572.

\bibitem[Rat06]{Ratcliffe} J.~Ratcliffe, \emph{Foundations of Hyperbolic Manifolds}, Grad.\ Texts Math.\ 149, Springer, New York, 2006.


\bibitem[Rou19]{Roulleau} X.~Roulleau, \emph{On the geometry of K3 surfaces with finite automorphism group and no elliptic fibrations}, Internat.\ J.\ Math.\ 33 (2022), no.\ \textbf{6}, Paper No. 2250040. 

\bibitem[Shi10]{ShimadaLatt} I.~Shimada, \emph{Lattice Zariski k-ples of plane sextic curves and Z-splitting curves for double plane sextics},  Michigan Math.~J.\ \textbf{59} (2010), 621--665. 

\bibitem[Shi14]{Shimada} \bysame, \emph{Projective models of the supersingular K3 surface with Artin invariant $1$ in characteristic $5$}, J.~Algebra \textbf{403} (2014), 273--299. 

\bibitem[Shi15]{ShimadaIMRN} \bysame, \emph{An algorithm to compute automorphism groups of K3 surfaces and an application to singular K3 surfaces},  Int.\ Math.\ Res.\ Not.\ IMRN (2015), no.~22, 11961--12014.

\bibitem[Sai74]{SaintDonat} B.~Saint-Donat, \emph{Projective models of K3 surfaces}, Amer.\ J.~Math.\ \textbf{96} (1974), no.~4, 602--639. 


\bibitem[Ste85]{Sterk} H.~Sterk, \emph{Finiteness results for algebraic K3 surfaces}, Math.~Z.\ \textbf{189} (1985), no.~4, 507--513. 
  
\bibitem[Vin75]{Vinberg2} E.\,B.~Vinberg, \emph{Some arithmetical discrete groups in Lobacevskii spaces}, In: \emph{Discrete subgroups of Lie groups and applications to moduli} (Internat.\ Colloq., Bombay, 1973), pp.~323--348. Oxford Univ.\ Press, Bombay, 1975.

\bibitem[Vin07]{Vinberg} \bysame, \emph{Classification of $2$-reflective hyperbolic lattices of rank 4},  Trans.\ Moscow Math.\ Soc.\ 2007, 39--66.  

\bibitem[Yon90]{Yonemura} T.~Yonemura, \emph{Hypersurface K3 Singularities}, Tohoku Math J.\ \textbf{42} (1990), 351--380. 

\bibitem[Zha98]{Zhang} D.-Q.~Zhang, \emph{Quotients of K3 surfaces modulo involutions}, Japan J.~Math.\ \textbf{24}, no.~2 (1998). 

\end{thebibliography}
\end{document}